\AddToHook{package/hyperref/before}{\RequirePackage{imakeidx}}
\documentclass[english]{shinybook}
\makeindex[title=index,intoc,columns=2]
\indexsetup{headers={\indexname}{\indexname}}
\newcommand{\term}[2][]{\emph{#2}\if\relax\detokenize{#1}\relax\index{#2}\else\index{#1}\fi}

\newcommand{\N}{\mathbb{N}}
\newcommand{\Z}{\mathbb{Z}}
\newcommand{\R}{\mathbb{R}}
\newcommand{\1}{\mathbb{1}}
\newcommand{\linop}{\mathcal{L}}

\newcommand{\calD}{\mathcal{D}}
\newcommand{\calE}{\mathcal{E}}
\newcommand{\calK}{\mathcal{K}}
\newcommand{\calM}{\mathcal{M}}
\newcommand{\calN}{\mathcal{N}}
\newcommand{\calR}{\mathcal{R}}

\newcommand{\eps}{\varepsilon}
\renewcommand{\phi}{\varphi}
\newcommand{\dom}{\operatorname{\mathrm{dom}}}
\newcommand{\sign}{\operatorname{\mathrm{sign}}}
\DeclareMathOperator{\Id}{\mathrm{Id}}
\newcommand{\spann}{\mathrm{span}}

\newcommand{\norm}[1]{\|#1\|}
\newcommand{\abs}[1]{\left|#1\right|}
\newcommand{\inner}[2]{\left(#1 \,\middle|\, #2\right)}
\newcommand{\dual}[1]{\langle #1 \rangle}
\newcommand{\setof}[2]{\left\{#1 \;\middle|\; #2\right\}}
\newcommand{\wkto}{\rightharpoonup}

\DeclareMathOperator{\Exp}{\mathbb{E}}
\DeclareMathOperator{\Cov}{\mathrm{Cov}}

\newcommand{\Rbar}{\overline\R}

\title{Regularization of Inverse Problems}
\subtitle{lecture notes}
\author{Christian Clason}
\publishers{\small\email{c.clason@uni-graz.at}\\\url{https://imsc.uni-graz.at/clason}}
\hypersetup{
    pdftitle={Regularization of Inverse Problems},
    pdfauthor={Christian Clason},
}
\date{%
    \today
    \\
    {\small\sffamily\textsc{arxiv}:\,\href{https://arxiv.org/abs/2001.00617}{\nolinkurl{2001.00617v3}}}
}

\begin{document}
\index{generalized inverse|see{Moore--Penrose inverse}}
\index{pseudoinverse|see{Moore--Penrose inverse}}
\index{Hanke--Raus rule|see{parameter choice rule}}
\index{L-curve criterion|see{parameter choice rule}}
\index{quasi-optimality principle|see{parameter choice rule}}
\index{differentiable|see{derivative}}
\index{subderivative|see{derivative}}
\index{subdifferential|see{derivative}}
\index{discrepancy term|see{Tikhonov functional}}
\index{penalty term|see{Tikhonov functional}}
\index{white noise|see{noise}}
\index{singular system|see{singular value decomposition}}
\index{truncated singular value decomposition|see{singular value decomposition}}
\index{credible set|see{estimator, region}}
\index{highest posterior density set|see{estimator, region}}
\index{source representer|see{source condition}}
\maketitle

\frontmatter
\tableofcontents

\mainmatter

\chapter*{Preface}
\markboth{Preface}{Preface}
\enlargethispage*{2cm}

Inverse problems occur wherever a quantity cannot be directly measured but only inferred through comparing observations to the output of mathematical models.
Examples of such problems are ubiquitous in biomedical imaging, non-destructive testing, and calibration of financial models. The name \term{inverse problem} is due to the fact that it contains as a \term{direct problem} the evaluation of the model given an estimate of the sought-for quantity. However, it is more relevant from a mathematical point of view that such problems are \emph{ill-posed} and cannot be treated by standard methods for solving (non)linear equations.\footnote{Otherwise there would be no need for a dedicated lecture. In fact, a more fitting title would have been \emph{Ill-posed Problems}, but the term \emph{inverse problem} has become widely accepted, especially in applications.}

The mathematical theory of inverse problems is therefore a part of functional analysis: in the same way that the latter is concerned with the question of when an equation $F(x)=y$ between infinite-dimensional vector spaces admits a unique solution $x$ that depends continuously on $y$, the former is concerned with conditions under which this is \emph{not} the case and with methods to at least obtain a reasonable approximation to $x$. In finite dimensions, this essentially corresponds to the step from regular to inconsistent, under- or overdetermined, or ill-conditioned systems of linear equations.

Although inverse problems are increasingly studied in Banach spaces, we will mostly restrict ourselves in these notes to Hilbert spaces as there the theory is essentially complete and allows for full characterizations in many cases. Here we only give a brief introduction to convex regularization in Banach spaces and refer to \cite{Scherzer:2009,ItoJin,HKKS} for a full treatment of such problems.
Similarly, we will only cursorily give an outlook to statistical (frequentist and Bayesian) inverse problems, which have also become prominent in recent years. Here a broad and elementary exposition aimed at a mathematical audience is still missing.

\medskip

These notes are based on graduate lectures given 2014--2020 at the University of Duisburg-Essen. As such, no claim is made of originality (beyond possibly the selection and arrangement of the material). Rather, like a magpie, I have tried to collect the shiniest results and proofs I could find. Here I mainly followed the seminal work \cite{Engl} (with simplifications by considering only compact instead of bounded linear operators on Hilbert spaces), with additional material from \cite{Hohage,Kindermann,Scherzer:2014,Kirsch,ItoJin,Kaltenbacher}. Further literature consulted during the writing of these notes was \cite{Louis,Hofmann,Rieder,Harrach,Burger}. The outlook on frequentist and Bayesian statistical inverse problems is based on \cite{Cavalier,Kekkonen}, respectively.

Thanks are due to Sascha Beutler and Florian Wittbold for helpful comments on the lecture notes and to Stefan Kindermann for pointing out the reference \cite{Du:2008}.
I am also grateful to Tapio Helin and Lauri Oksanen for identifying an issue with \cref{thm:regularisierung:ordnung} and its application to a priori choice rules and the discrepancy principle in an earlier version of these notes; see \cite[Section 5.4]{FormalizingIP}. All remaining errors are my own.

\paragraph{AI declaration} Beside the above-mentioned issue, Claude Opus 5.5 was asked to check for typos, grammatical errors, and mathematical inaccuracies. It found a distressing number of them and -- without being asked -- suggested fixes, some of which I accepted as-is. It also generated the index and created a first draft of Part IV by extracting the relevant parts from an earlier set of lecture notes \cite{Clason:2019}. All changes were verified and cleaned up manually.

\part{Basics of functional analysis}

\chapter{Linear operators in normed spaces}

In this and the following chapter, we collect the basic concepts and results (and, more importantly, fix notations) from linear functional analysis that will be used throughout these notes. For details and proofs, the reader is referred to the standard literature, e.g., \cite{Alt,Brezis:2010a}, or to \cite{Clason}.

\section{Normed vector spaces}

In the following, $X$ will denote a vector space over the field $\mathbb{K}$, where we restrict ourselves to the case $\mathbb{K}=\R$. A mapping $\norm{\cdot}:X\to \R^+:= [0,\infty)$ is called a \term{norm} (on $X$) if for all $x\in X$ there holds
\begin{enumerate}
    \item $\norm{\lambda x} = |\lambda| \norm{x}$ for all $\lambda\in\mathbb{K}$,
    \item $\norm{x+y} \leq \norm{x} + \norm{y}$ for all $y\in X$,
    \item $\norm{x} = 0$ if and only if $x = 0\in X$.
\end{enumerate}
\begin{example}\label{ex:funktan:norm}
    \begin{enumerate}
        \item Norms on $X = \R^N$ are defined by
            \begin{equation*}
                \begin{aligned}
                    \norm{x}_p &= \left(\sum_{i=1}^N |x_i|^p\right)^{1/p},\qquad1\leq p<\infty,\\
                    \norm{x}_\infty &= \max_{i=1,\dots,N} |x_i|.
                \end{aligned}
            \end{equation*}
        \item Norms on $X = \ell^p$ (the space of real-valued sequences for which the corresponding terms are finite) are defined by
            \begin{equation*}
                \begin{aligned}
                    \norm{x}_p &= \left(\sum_{i=1}^\infty |x_i|^p\right)^{1/p},\qquad 1\leq p<\infty,\\
                    \norm{x}_\infty &= \sup_{i=1,\dots,\infty} |x_i|.
                \end{aligned}
            \end{equation*}
        \item Norms on $X = L^p(\Omega)$ (the space of real-valued measurable functions on the domain $\Omega\subset \R^d$ for which the corresponding terms are finite) are defined by
            \begin{equation*}
                \begin{aligned}
                    \norm{u}_p &= \left(\int_\Omega |u(x)|^p\,dx\right)^{1/p},\qquad 1\leq p<\infty,\\
                    \norm{u}_\infty &= \mathop\mathrm{ess\,\sup}_{x\in \Omega} |u(x)|.
                \end{aligned}
            \end{equation*}
        \item A norm on $X = C(\overline\Omega)$ (the space of continuous functions on $\overline\Omega$) is defined by
            \begin{equation*}
                \norm{u}_C = \sup_{x\in \Omega} |u(x)|.
            \end{equation*}
            Similarly, a norm on the space $C^k(\overline\Omega)$ of $k$ times continuously differentiable functions is defined by $\norm{u}_{C^k}=\sum_{j=0}^k\norm{u^{(j)}}_{C}$.
    \end{enumerate}
\end{example}
If $\norm{\cdot}$ is a norm on $X$, the pair $(X,\norm{\cdot})$ is called a \term[space!normed vector]{normed vector space}, and one frequently denotes this by writing $\norm{\cdot}_X$. If the norm is canonical (as in \cref{ex:funktan:norm}\,(ii)--(iv)), it is often omitted, and one speaks simply of ``the normed vector space $X$''.

Two norms $\norm{\cdot}_1$, $\norm{\cdot}_2$ are called \term[norm!equivalent]{equivalent} if there exist constants $c_1,c_2 >0 $ such that
\begin{equation*}
    c_1 \norm{x}_2 \leq \norm{x}_1 \leq c_2 \norm{x}_2 \qquad\text{for all } x\in X.
\end{equation*}
If $X$ is finite-dimensional, all norms on $X$ are equivalent. However, in this case the constants $c_1, c_2$ may depend on the dimension of $X$; in particular, it may be the case that $c_1(N)\to 0$ or $c_2(N)\to \infty$ for $\dim X=N\to \infty$, making the corresponding inequality useless for growing dimensions. Avoiding such dimension-dependent constants is therefore one of the main reasons for studying inverse problems in infinite-dimensional spaces.

If $(X,\norm{\cdot}_X)$ and $(Y,\norm{\cdot}_Y)$ are normed vector spaces with $X\subset Y$, then $X$ is called \term{continuously embedded} in $Y$, denoted by $X\hookrightarrow Y$, if there exists a $C>0$ such that
\begin{equation*}
    \norm{x}_Y \leq C\norm{x}_X \qquad\text{for all } x\in X.
\end{equation*}

\bigskip

A norm directly induces a notion of convergence, the so-called \term[convergence!strong]{strong convergence}:
A sequence $\{x_n\}_{n\in\N}\subset X$ converges (strongly in $X$) to an $x\in X$, denoted by $x_n\to x$, if
\begin{equation*}
    \lim_{n\to\infty} \norm{x_n-x}_X = 0.
\end{equation*}

A set $U\subset X$ is called
\begin{itemize}
    \item \term[set!closed]{closed} if for every convergent sequence $\{x_n\}_{n\in\N}\subset U$ the limit $x\in X$ lies in $U$ as well;
    \item \term[set!compact]{compact} if every sequence $\{x_n\}_{n\in\N}\subset U$ contains a convergent subsequence $\{x_{n_k}\}_{k\in\N}$ with limit $x\in U$;
    \item \term[set!dense]{dense} in $X$ if for all $x\in X$ there exists a sequence $\{x_n\}_{n\in\N}\subset U$ with $x_n\to x$.
\end{itemize}
The union of $U$ with the set of all limits of convergent sequences in $U$ is called the \term{closure} $\overline U$ of $U$; obviously, $U$ is dense in $\overline U$.

A normed vector space $X$ is called \term[space!complete]{complete}, if every Cauchy sequence in $X$ converges; in this case, $X$ is called a \term[space!Banach]{Banach space}. All spaces in \cref{ex:funktan:norm} are Banach spaces. If $X$ is an incomplete normed space, we denote by $\overline X$ the \term{completion} of $X$ (with respect to the norm $\|\cdot\|_X$).

Finally, we define for later use for given $x\in X$ and $r>0$
\begin{itemize}
    \item the \term[ball!open]{open ball} $U_r(x) := \setof{z\in X}{\norm{x-z}_X< r}$ and
    \item the \term[ball!closed]{closed ball} $B_r(x) := \setof{z\in X}{\norm{x-z}_X\leq r}$.
\end{itemize}
The closed ball around $x=0$ with radius $r=1$ is also referred to as the \term[ball!unit]{unit ball} $B_X$.
A set $U\subset X$ is called
\begin{itemize}
    \item \term[set!open]{open} if for all $x\in U$ there exists an $r>0$ such that $U_r(x)\subset U$ (i.e., all $x\in U$ are \term[point, interior]{interior points} of $U$);
    \item \term[set!bounded]{bounded} if it is contained in a closed ball $B_r(0)$ for an $r>0$;
    \item \term[convex!set]{convex} if for all $x,y\in U$ and $\lambda\in [0,1]$ also $\lambda x + (1-\lambda)y\in U$.
\end{itemize}
In normed spaces, the complement of an open set is also closed and vice versa (i.e., the closed sets in the sense of topology are exactly the (sequentially) closed sets in the sense of the above definition).
The definition of a norm directly implies that open and closed balls are convex.
On the other hand, the unit ball is compact if \emph{and only if} $X$ is finite-dimensional; this will be of fundamental importance throughout these notes.

\section{Bounded operators}

We now consider mappings between normed vector spaces. In the following, let $(X,\norm{\cdot}_X)$ and $(Y,\norm{\cdot}_Y)$ be normed vector spaces, $U\subset X$, and $F: U\to Y$ be a mapping. We denote by
\begin{itemize}
    \item $\calD(F) := U$ the \term{domain} of $F$;
    \item $\calN(F) := \setof{x\in U}{F(x) = 0}$ the \term{kernel} or \term[space!null]{null space} of $F$;
    \item $\calR(F) := \setof{F(x)\in Y}{x\in U}$ the \term{range} of $F$.
\end{itemize}
We call $F:U\to Y$
\begin{itemize}
    \item \term[mapping!continuous]{continuous} in $x\in U$ if for all $\eps>0$ there exists a $\delta >0$ with
        \begin{equation*}
            \norm{F(x)-F(z)}_Y \leq \eps\qquad \text{for all } z\in U \text{ with } \norm{x-z}_X \leq \delta;
        \end{equation*}
    \item \term{Lipschitz continuous} if there exists a \term{Lipschitz constant} $L>0$ with
        \begin{equation*}
            \norm{F(x_1)-F(x_2)}_Y \leq L \norm{x_1-x_2}_X \qquad\text{for all } x_1,x_2 \in U.
        \end{equation*}
\end{itemize}
A mapping $F:X\to Y$ is thus continuous if and only if $x_n\to x$ implies $F(x_n)\to F(x)$; it is \term[mapping!closed]{closed} if both $x_n\to x$ and $F(x_n)\to y$ imply $F(x) = y$.

If $F:X\to Y$ is linear (i.e., $F(\lambda_1 x_1 +\lambda_2 x_2) = \lambda_1 F(x_1) + \lambda_2 F(x_2)$ for all $\lambda_1,\lambda_2 \in \R$ and $x_1,x_2\in X$), continuity of $F$ is equivalent to the existence of a $C>0$ such that
\begin{equation*}
    \norm{F(x)}_Y \leq C\norm{x}_X \qquad\text{for all }x\in X.
\end{equation*}
For this reason, continuous linear mappings are called \term[operator!bounded]{bounded}; one also speaks of a bounded linear \term[operator!bounded linear]{operator}.
(In the following, we generically denote these by $T$ and omit the parentheses around the argument to indicate this.) If $Y$ is complete, the vector space $\linop(X,Y)$ of bounded linear operators becomes a Banach space when endowed with the \term[norm!operator]{operator norm}
\begin{equation*}
    \norm{T}_{\linop(X,Y)} = \sup_{x\in X\setminus\{0\}}\frac{\norm{Tx}_Y}{\norm{x}_X}
    = \sup_{\norm{x}_X\leq 1} \norm{Tx}_Y = \sup_{\norm{x}_X = 1} \norm{Tx}_Y,
\end{equation*}
which is equal to the minimal constant $C$ in the definition of continuity.
This immediately implies that
\begin{equation*}
    \norm{Tx}_Y \leq \norm{T}_{\linop(X,Y)} \norm{x}_X \qquad\text{for all }x\in X.
\end{equation*}

As in linear algebra, we call $T$
\begin{itemize}
    \item \term{injective} if $\calN(T) = \{0\}$;
    \item \term{surjective} if $\calR(T)= Y$;
    \item \term{bijective} if $T$ is injective and surjective.
\end{itemize}
If $T\in \linop(X,Y)$ is bijective, the \term[operator!inverse]{inverse} $T^{-1}:Y\to X$, $Tx\mapsto x$, is continuous if and only if there exists a $c>0$ with
\begin{equation}
    \label{eq:funktan:inverse_bd}
    c\norm{x}_X \leq \norm{Tx}_Y \qquad\text{for all }x\in X;
\end{equation}
in this case, $\norm{T^{-1}}_{\linop(Y,X)} = c^{-1}$ holds for the maximal $c$ satisfying \eqref{eq:funktan:inverse_bd}.
The question of when this is the case is answered by the following three main theorems of functional analysis (that all are more or less direct consequences of the open mapping theorem).
\begin{theorem}[continuous inverse]\label{cor:funktan:inverse}\index{theorem!continuous inverse}
    If $X,Y$ are Banach spaces and $T\in \linop(X,Y)$ is bijective, then $T^{-1}:Y\to X$ is continuous.
\end{theorem}
Of particular relevance for inverse problems is the situation that $T$ is injective but not surjective; in this case, one would like to at least have a continuous inverse on the range of $T$. However, this does not hold in general, which is one of the fundamental issues in infinite-dimensional inverse problems.
\begin{theorem}[closed range]\label{thm:closedrange}\index{theorem!closed range}
    If $X,Y$ are Banach spaces and $T\in\linop(X,Y)$ is injective, then $T^{-1}:\calR(T)\to X$ is continuous if and only if $\calR(T)$ is closed.
\end{theorem}
The following theorem completes the trio.
\begin{theorem}[closed graph]\label{thm:closedgraph}\index{theorem!closed graph}
    Let $X,Y$ be Banach spaces. Then $T:X\to Y$ is continuous if and only if $T$ is closed.
\end{theorem}

\bigskip

We now consider sequences of linear operators. Here we distinguish two notions of convergence: A sequence $\{T_n\}_{n\in\N}\subset \linop(X,Y)$ converges to $T\in \linop(X,Y)$
\begin{enumerate}
    \item \term[convergence!pointwise]{pointwise} if $T_nx\to Tx$ (strongly in $Y$) for all $x\in X$;
    \item \term[convergence!uniform]{uniformly} if $T_n\to T$ (strongly in $\linop(X,Y)$).
\end{enumerate}
Obviously, uniform convergence implies pointwise convergence; weaker conditions are provided by another main theorem of functional analysis.
\begin{theorem}[Banach--Steinhaus]\label{thm:banach-steinhaus}\index{Banach--Steinhaus theorem}
    Let $X$ be a Banach space, $Y$ be a normed vector space, and $\{T_i\}_{i\in I}\subset \linop(X,Y)$ be a family of pointwise bounded operators, i.e., for all $x\in X$ there exists an $M_x>0$ with $\sup_{i\in I}\|T_ix\|_Y \leq M_x$. Then
    \begin{equation*}
        \sup_{i\in I} \|T_i\|_{\linop(X,Y)} < \infty.
    \end{equation*}
\end{theorem}
\begin{corollary}\label{cor:banach-steinhaus1}
    Let $X,Y$ be Banach spaces and $\{T_n\}_{n\in\N}\subset \linop(X,Y)$. Then the following statements are equivalent:
    \begin{enumerate}
        \item $\{T_n\}_{n\in\N}$ converges uniformly on compact subsets of $X$;
        \item $\{T_n\}_{n\in\N}$ converges pointwise on $X$;
        \item $\{T_n\}_{n\in\N}$ converges pointwise on a dense subset $U\subset X$ and
            \begin{equation*}
                \sup_{n\in \N} \|T_n\|_{\linop(X,Y)} < \infty.
            \end{equation*}
    \end{enumerate}
\end{corollary}
\begin{corollary}\label{cor:banach-steinhaus2}
    Let $X,Y$ be Banach spaces and $\{T_n\}_{n\in\N}\subset \linop(X,Y)$. If $T_n$ converges pointwise to a $T:X\to Y$, then $T$ is bounded.
\end{corollary}

\chapter{Compact operators in Hilbert spaces}\label{chap:funktan_hilbert}

As mentioned in the preface, the theory of linear inverse problems can be stated most fully in Hilbert spaces. There, the analogy to ill-conditioned linear systems of equations is also particularly evident.

\section{Inner products and weak convergence}

Hilbert spaces are characterized by an additional structure: a mapping $\inner{\cdot}{\cdot}:X\times X\to \R$ on a normed vector space $X$ over the field $\R$ is called an \term[product, inner]{inner product} if
\begin{enumerate}
    \item $\inner{\alpha x+\beta y}{z} = \alpha\inner{x}{z} + \beta \inner{y}{z}$ for all $x,y,z\in X$ and $\alpha,\beta\in \R$;
    \item $\inner{x}{y} = \inner{y}{x}$ for all $x,y\in X$;
    \item $\inner{x}{x} \geq 0$ for all $x\in X$ with $\inner{x}{x}=0$ if and only if $x=0$.
\end{enumerate}
An inner product induces a norm
\begin{equation*}
    \norm{x}_X :=  \sqrt{\inner{x}{x}_X}
\end{equation*}
which satisfies the \term{Cauchy--Schwarz inequality}
\begin{equation*}
    |\inner{x}{y}_X| \leq \norm{x}_X\norm{y}_X.
\end{equation*}
(If one argument is fixed, the inner product is hence continuous in the other with respect to the induced norm.)
If $X$ is complete with respect to the induced norm (i.e., $(X,\norm{\cdot}_X)$ is a Banach space), then $X$ is called a \term[space!Hilbert]{Hilbert space}; if the inner product and hence the induced norm is canonical, it is frequently omitted.
\begin{example}
    \Cref{ex:funktan:norm}\,(i--iii) for $p=2$ are Hilbert spaces, where the inner product is defined by
    \begin{enumerate}
        \item for $X=\R^N$:\quad $\displaystyle \inner{x}{y}_X = \sum_{i=1}^N x_iy_i$,
        \item for $X=\ell^2$:\quad $\displaystyle \inner{x}{y}_X = \sum_{i=1}^\infty x_iy_i$,
        \item for $X=L^2(\Omega)$:\quad $\displaystyle \inner{u}{v}_X = \int_\Omega u(x)v(x)\,dx$.
    \end{enumerate}
    In all cases, the inner product induces the canonical norm.
\end{example}

The inner product induces an additional notion of convergence: the \term[convergence!weak]{weak convergence}.
A sequence $\{x_n\}_{n\in\N}\subset X$ converges weakly (in $X$) to $x\in X$, denoted by $x_n\wkto x$, if
\begin{equation*}
    \inner{x_n}{z}_X\to  \inner{x}{z}_X \qquad\text{for all }z\in X.
\end{equation*}
This notion generalizes the componentwise convergence in $\R^N$ (choose $z=e_i$, the $i$th unit vector); hence weak and strong convergence coincide in finite dimensions.
In infinite-dimensional spaces, strong convergence implies weak convergence but not vice versa.
However, if a sequence $\{x_n\}_{n\in\N}$ converges weakly to $x\in X$ and additionally  $\norm{x_n}_X\to\norm{x}_X$, then $x_n$ converges even strongly to $x$.
Furthermore, the norm is \term[semicontinuous!weakly lower]{weakly lower semicontinuous}:  If $x_n\wkto x$, then
\begin{equation}\label{eq:funktan:lsc}
    \norm{x}_X\leq \liminf_{n\to\infty}\norm{x_n}_X.
\end{equation}

This notion of convergence is useful in particular because the Bolzano--Weierstraß theorem holds for it (in contrast to the strong convergence) even in infinite dimensions: Every bounded sequence in a Hilbert space contains a weakly convergent subsequence. Conversely, every weakly convergent sequence is bounded.

\bigskip

We now consider linear operators $T\in \linop(X,Y)$ between Hilbert spaces $X$ and $Y$.
Of particular importance is the special case $Y=\R$, i.e., the space $\linop(X,\R)$ of \term[functional!bounded linear]{bounded linear functionals} on $X$. These can be identified with elements of $X$.
\begin{theorem}[Fréchet--Riesz]\label{thm:frechetriesz}\index{Frechet--Riesz theorem@Fréchet--Riesz theorem}
    Let $X$ be a Hilbert space and $\lambda\in \linop(X,\R)$. Then there exists a unique $z_\lambda\in X$ with $\norm{\lambda}_{\linop(X,\R)} = \norm{z_\lambda}_X$ and
    \begin{equation*}
        \lambda(x) = \inner{z_\lambda}{x}_X \qquad\text{for all } x\in X.
    \end{equation*}
\end{theorem}

This theorem allows us to define for any linear operator $T\in \linop(X,Y)$ an \term[operator!adjoint]{adjoint operator} $T^*\in\linop(Y,X)$ via
\begin{equation*}
    \inner{T^* y}{x}_X = \inner{Tx}{y}_Y \qquad\text{for all } x\in X, y\in Y,
\end{equation*}
which satisfies
\begin{enumerate}
    \item $(T^*)^* = T$;
    \item $\norm{T^*}_{\linop(Y,X)} = \norm{T}_{\linop(X,Y)}$;
    \item $\norm{T^*T}_{\linop(X,X)} = \norm{T}_{\linop(X,Y)}^2$.
\end{enumerate}
If $T^* = T$, then $T$ is called \term{selfadjoint}.

\section{Orthogonality and orthonormal systems}

An inner product induces the notion of \term[orthogonal]{orthogonality}: If $X$ is a Hilbert space, then $x,y\in X$ are called \term{orthogonal} if $\inner{x}{y}_X = 0$.
For a set $U\subset X$,
\begin{equation*}
    U^\bot := \setof{x\in X}{\inner{x}{u}_X = 0 \text{ for all } u\in U}
\end{equation*}
is called the \term[orthogonal!complement]{orthogonal complement} of $U$ in $X$; the definition immediately implies that $U^\bot$ is a closed subspace. In particular, $X^\bot = \{0\}$.
Furthermore, $U\subset (U^\bot)^\bot$. If $U$ is a closed subspace, it even holds that $U=(U^\bot)^\bot$ (and hence $\{0\}^\bot = X$). In this case, we have the \term[orthogonal!decomposition]{orthogonal decomposition}
\begin{equation*}
    X = U\oplus U^\bot,
\end{equation*}
i.e., every element $x\in X$ can be represented uniquely as
\begin{equation*}
    x = u +u_\bot\qquad\text{with}\qquad u\in U,\ u_\bot\in U^\bot.
\end{equation*}
The mapping $x\mapsto u$ defines a linear operator $P_U\in \linop(X,X)$, called the \term[orthogonal!projection]{orthogonal projection} on $U$, which has the following properties:
\begin{enumerate}
    \item $P_U$ is selfadjoint;
    \item $\norm{P_U}_{\linop(X,X)} = 1$;
    \item $\Id - P_U = P_{U^\bot}$;
    \item $\norm{x-P_U x}_X = \min_{u\in U} \norm{x-u}_X$;
    \item $z = P_U x$ if and only if $z\in U$ and $z-x \in U^\bot$.
\end{enumerate}

If $U$ is not a closed subspace, only $(U^\bot)^\bot=\overline{U}\supset U$ holds. Hence, for any $T\in \linop(X,Y)$ we have
\begin{enumerate}
    \item $\calR(T)^\bot = \calN(T^*)$ and hence $\calN(T^*)^\bot = \overline{\calR(T)}$;
    \item $\calR(T^*)^\bot = \calN(T)$ and hence $\calN(T)^\bot = \overline{\calR(T^*)}$.
\end{enumerate}
In particular, the null space of a bounded linear operator is always closed; furthermore, $T$ is injective if and only if $\calR(T^*)$ is dense in $X$.

\bigskip

A set $U\subset X$ whose elements are pairwise orthogonal is called an \term[orthogonal!system]{orthogonal system}. If in addition
\begin{equation*}
    \inner{x}{y}_X =
    \begin{cases}
        1 & \text{if } x =y, \\
        0 &\text{else},
    \end{cases}
\end{equation*}
for all $x,y\in U$, then $U$ is called an \term[orthonormal!system]{orthonormal system}; an orthonormal system is called \term[orthonormal!system!complete]{complete}, if there is no orthonormal system $V\subset X$ with $U\subsetneq V$.
Every orthonormal system $U\subset X$ satisfies the \term{Bessel inequality}
\begin{equation}\label{eq:funktan:bessel}
    \sum_{u\in U} |\inner{x}{u}_X|^2 \leq \norm{x}_X^2\qquad\text{for all }x\in X,
\end{equation}
where at most countably many terms are not equal to zero.
If equality holds in \eqref{eq:funktan:bessel}, then $U$ is called an \term[orthonormal!basis]{orthonormal basis}; in this case, $U$ is complete and
\begin{equation}\label{eq:funktan:ons_reihe}
    x  = \sum_{u\in U} \inner{x}{u}_X u \qquad\text{for all } x\in X.
\end{equation}

Every Hilbert space contains an orthonormal basis. If one of them is at most countable, the Hilbert space is called \term{separable}. The Bessel inequality then implies that the sequence $\{u_n\}_{n\in\N} = U$ converges weakly to $0$ (but not strongly due to $\norm{u_n}_X=1$!)
\begin{example}
    For $X=L^2((0,1))$, an orthonormal basis is given by $\{u_n\}_{n\in\Z}$ for
    \begin{equation*}
        u_n(x) = \begin{cases}
            \sqrt{2}\sin(2 \pi\,n\,x) & n>0,\\
            \sqrt{2}\cos(2 \pi\,n\,x) & n<0,\\
            1 & n=0.
        \end{cases}
    \end{equation*}
\end{example}

Finally, every closed subspace $U\subset X$ contains an orthonormal basis $\{u_n\}_{n\in \N}$ for which the orthogonal projection on $U$ can be written as
\begin{equation*}
    P_U x = \sum_{j=1}^\infty \inner{x}{u_j}_X u_j.
\end{equation*}

\section{The spectral theorem for compact operators}

Just as Hilbert spaces can be considered as generalizations of finite-dimensional vector spaces, compact operators furnish an analog to matrices.
Here a linear operator $T:X\to Y$ is called \term[operator!compact]{compact} if the image of every bounded sequence $\{x_n\}_{n\in\N}\subset X$ contains a convergent subsequence $\{Tx_{n_k}\}_{k\in\N}\subset Y$.
A linear operator $T$ is compact if and only if $T$ maps weakly convergent sequences in $X$ to strongly convergent sequences in $Y$. (This property is also called \term[continuity, complete]{complete continuity}.) We will generically denote compact operators by $K$.

Obviously, every linear operator with finite-dimensional range is compact. In particular, the \term[operator!identity]{identity}
$\Id:X\to X$ -- like the unit ball $B_X$ -- is compact if \emph{and only if} $X$ is finite-dimensional.
Furthermore, the space $\calK(X,Y)$ of linear compact operators from $X$ to $Y$ is a closed subspace of $\linop(X,Y)$ (and hence a Banach space when endowed with the operator norm).
This implies that the limit of any sequence of linear operators with finite-dimensional range is compact.
If $T\in \linop(X,Y)$ and $S\in \linop(Y,Z)$ and at least one of the two is compact, then $S\circ T$ is compact as well. Furthermore, $T^*$ is compact if and only if $T$ is compact (which is known as the Schauder theorem\index{Schauder theorem}).

\begin{example}\label{ex:integral-operator}
    Canonical examples of compact operators are \term[operator!integral]{integral operators}.
    We consider for $X=Y=L^2(\Omega)$ with $\Omega=(0,1)$ and for a given \term[operator!integral!kernel]{kernel} $k\in L^2(\Omega\times \Omega)$ the operator $K:L^2(\Omega)\to L^2(\Omega)$ defined pointwise via
    \begin{equation*}
        [K x](t) = \int_0^1 k(s,t)x(s)\,ds \qquad\text{for almost every }t\in \Omega
    \end{equation*}
    (where $Kx\in L^2(\Omega)$ by Fubini's theorem).
    The Cauchy--Schwarz inequality and Fubini's theorem immediately yield
    \begin{equation*}
        \norm{K}_{\linop(X,Y)} \leq \norm{k}_{L^2(\Omega^2)},
    \end{equation*}
    which also implies that $K$ is a bounded operator from $L^2(\Omega)$ to $L^2(\Omega)$.

    Since $k\in L^2(\Omega^2)$ is in particular measurable, there is a sequence $\{k_n\}_{n\in \N}$ of simple functions (i.e., attaining only finitely many different values) with $k_n\to k$ in $L^2(\Omega^2)$. These can be written as
    \begin{equation*}
        k_n(s,t) = \sum_{i,j=1}^n \alpha_{ij}\1_{E_i}(s)\1_{E_j}(t),
    \end{equation*}
    where $\1_E$ is the characteristic function of the measurable interval $E\subset \Omega$ and $E_i$ are a finite disjoint decomposition of $\Omega$. The corresponding integral operators $K_n$ with kernel $k_n$ by linearity of the integral therefore satisfy
    \begin{equation*}
        \norm{K_n -K}_{\linop(X,Y)} \leq \norm{k_n-k}_{L^2(\Omega^2)} \to 0,
    \end{equation*}
    i.e., $K_n\to K$. Furthermore,
    \begin{equation*}
        [K_n x](t) = \int_0^1 k_n(s,t)x(s)\,ds = \sum_{j=1}^{n} \left(\sum_{i=1}^n \alpha_{ij}  \int_{E_i} x(s)\,ds\right)\1_{E_j}(t)
    \end{equation*}
    and hence $K_n x$ is a linear combination of the $\{\1_{E_j}\}_{1\leq j\leq n}$. This implies that $K$ is the limit of the sequence $\{K_n\}_{n\in\N}$ of operators with finite-dimensional range and therefore compact.

    For the adjoint operator $K^*\in \linop(X,X)$, one can use the definition of the inner product on $L^2(\Omega)$ together with Fubini's theorem to show that
    \begin{equation*}
        [K^*y](s) = \int_0^1 k(s,t)y(t)\,dt \qquad\text{for almost every }s\in \Omega.
    \end{equation*}
    Hence an integral operator is selfadjoint if and only if the kernel is symmetric, i.e., $k(s,t) = k(t,s)$ for almost every $s,t\in \Omega$.

    For example, solution operators to (partial) differential equations or convolution operators -- and thus a large class of practically relevant operators -- can be represented as integral operators and thus shown to be compact.
\end{example}

\bigskip

The central analogy between compact operators and matrices consists in the fact that compact linear operators have at most countably many eigenvalues (which is not necessarily the case for bounded linear operators).
Correspondingly, we have the following variant of the Schur factorization, which will be the crucial tool allowing the thorough investigation of linear inverse problems in Hilbert spaces.
\begin{theorem}[spectral theorem]\label{thm:spektral}\index{theorem!spectral}
    Let $X$ be a Hilbert space and $K\in \calK(X,X)$ be selfadjoint. Then there exists a (possibly finitely terminating) orthonormal system $\{u_n\}_{n\in \N}\subset X$ and a (in this case also finitely terminating) null sequence $\{\lambda_n\}_{n\in\N}\subset \R\setminus\{0\}$ with
    \begin{equation*}
        Kx = \sum_{n\in \N} \lambda_n \inner{x}{u_n}_X u_n \qquad \text{for all }x\in X.
    \end{equation*}
    Furthermore, $\{u_n\}_{n\in\N}$ forms an orthonormal basis of $\overline{\calR(K)}$.
\end{theorem}
Setting $x=u_n$ immediately implies that $u_n$ is an eigenvector for the eigenvalue $\lambda_n$, i.e., $Ku_n = \lambda_n u_n$. By convention, the eigenvalues are sorted by decreasing magnitude, i.e.,
\begin{equation*}
    |\lambda_1|\geq |\lambda_2|\geq \cdots > 0.
\end{equation*}
With this ordering, the eigenvalues can also be characterized by the \term{Courant--Fischer min--max principle}
\begin{equation}\label{eq:courant}
    \begin{aligned}[t]
        \lambda_n &= \min_{V\subset X}\max_{x\in V}\setof{\inner{Kx}{x}_X}{\norm{x}_X = 1,\ \dim V^\bot = n-1}\\
        &=\max_{V\subset X}\min_{x\in V}\setof{\inner{Kx}{x}_X}{\norm{x}_X = 1,\ \dim V\vphantom{^\bot} = n}.
    \end{aligned}
\end{equation}
In particular, $\norm{K}_{\linop(X,X)} = |\lambda_1|$.

\part{Linear inverse problems}

\chapter{Ill-posed operator equations}

We now start our study of operator equations that cannot be solved by standard methods.
We first consider a linear operator $T$ between two normed vector spaces $X$ and $Y$. Following \href{http://mathshistory.st-andrews.ac.uk/Biographies/Hadamard.html}{Jacques Hadamard}, we call the equation $Tx=y$ \term{well-posed}, if for all $y\in Y$
\begin{enumerate}
    \item there exists an $x\in X$ with $Tx = y$;
    \item this solution is unique, i.e., $z\neq x$ implies $Tz\neq y$;
    \item this solution depends continuously on $y$, i.e., for all $\{x_n\}_{n\in\N}$ with $Tx_n\to y$ we also have $x_n\to x$.
\end{enumerate}
If one of these conditions is violated, the equation is called \term{ill-posed}.

In practice, a violation of the first two conditions often occurs due to insufficient knowledge of reality and can be handled by extending the mathematical model giving rise to the equation. It can also be handled by extending the concept of a solution such that a generalized solution exists for arbitrary $y\in Y$; if this is not unique, one can use additional information on the sought-for $x$ to select a specific solution.
For finite-dimensional Hilbert spaces, this leads to the well-known \term[method!least squares]{least squares method}; since all linear operators are then continuous, the problem is solved in principle (even if the details and in particular the efficient numerical implementation may still take significant effort). However, in infinite dimensions this is not the case, as the following example illustrates.
\begin{example}\label{ex:illposed_diff}
    We want to compute for given $y\in Y:= C^1([0,1])$ the derivative $x:= y'\in C([0,1])$, where we assume that the function $y$ to be differentiated is only given by measurements subject to additive noise, i.e., we only have at our disposal
    \begin{equation*}
        \tilde y = y + \eta.
    \end{equation*}
    In general, we cannot assume that the measurement error $\eta$ is continuously differentiable; but for the sake of simplicity, we assume that it is at least continuous. In this case, $\tilde y\in C([0,1])$ as well, and we have to consider the mapping $x=y'\mapsto y$ as a (linear) operator $T:C([0,1])\to C([0,1])$. Obviously, condition (i) is then violated. But the problem is not well-posed even if the error is continuously differentiable by coincidence:
    Consider a sequence $\{\delta_n\}_{n\in\N}$ with $\delta_n\to 0$, choose $k\in \N$ arbitrary, and set
    \begin{equation*}
        \eta_n(t) := \delta_n\sin\left(\tfrac{k t}{\delta_n}\right)
    \end{equation*}
    as well as $\tilde y_n:= y +\eta_n$.
    Then, $\eta_n\in C^1([0,1])$ and
    \begin{equation*}
        \norm{\tilde y_n-y}_{C} = \norm{\eta_n}_{C}=\delta_n\to 0,
    \end{equation*}
    but
    \begin{equation*}
        \tilde x_n(t):= \tilde y_n'(t) = y'(t) + k\cos\left(\tfrac{k t}{\delta_n}\right),
    \end{equation*}
    i.e., $x := y'$ satisfies
    \begin{equation*}
        \norm{ x-\tilde x_n}_{C} = \norm{\eta'_n}_{C} = k \qquad\text{for all }n\in \N.
    \end{equation*}
    Hence the error in the derivative $x$ can (depending on $k$) be arbitrarily large, even if the error in $y$ is arbitrarily small.

    (In contrast, the problem is of course well-posed for $T:C([0,1])\to C^1([0,1])$, since then $\norm{\eta_n}_{C^1}\to 0$ implies by definition that $\norm{x-\tilde x_n}_{C} \leq \norm{\eta_n}_{C^1}\to 0$. The occurring norms thus decide the well-posedness of the problem; these are however usually given by the problem setting. In our example, taking $C^1([0,1])$ as image space implies that besides $y$ also $y'$ is measured -- and that is precisely the quantity we are interested in, so that we are no longer considering an inverse problem.)
\end{example}
Note that the three conditions for well-posedness are not completely independent. For example, if $T\in \linop(X,Y)$ satisfies the first two conditions, and $X$ and $Y$ are Banach spaces, then $T$ is bijective and thus has by \cref{cor:funktan:inverse} a continuous inverse, satisfying also the third condition.

\section{Generalized inverses}

We now try to handle the first two conditions for linear operators between Hilbert spaces by generalizing the concept of solution in analogy to the least squares method in $\R^N$.
Let $X$ and $Y$ be Hilbert spaces (which we always assume from now on) and consider for $T\in \linop(X,Y)$ the equation $Tx=y$. If $y\notin\calR(T)$, this equation has no solution. In this case it is reasonable to look for an $x\in X$ that minimizes the distance $\norm{Tx-y}_Y$. On the other hand, if $\calN(T)\neq \{0\}$, then there exist infinitely many solutions; in this case, we choose the one with minimal norm.
This leads to the following definition.
\newpage
\begin{definition}
    An element $x^\dag\in X$ is called
    \begin{enumerate}
        \item \term[solution!least squares]{least squares solution} of $Tx=y$ if
            \begin{equation*}
                \norm{Tx^\dag-y}_Y = \min_{z\in X}\norm{Tz-y}_Y;
            \end{equation*}
        \item \term[solution!minimum norm]{minimum norm solution} of $Tx=y$ if
            \begin{equation*}
                \norm{x^\dag}_X = \min \setof{\norm{z}_X}{\text{$z$ is a least squares solution of $Tx=y$}}.
            \end{equation*}
    \end{enumerate}
\end{definition}
If $T$ is bijective, $x=T^{-1}y$ is obviously the only least squares and hence minimum norm solution. A least squares solution need not exist, however, if $\calR(T)$ is not closed (since in this case, the minimum in the definition need not be attained).
To answer the question for which $y\in Y$ a minimum norm solution exists, we introduce an operator -- called \term[Moore--Penrose inverse]{generalized inverse} or \term[Moore--Penrose inverse]{pseudoinverse} -- mapping $y$ to the corresponding minimum norm solution.
We do this by first restricting the domain and range of $T$ such that the operator is invertible and then extending the inverse of the restricted operator to its maximal domain.
\begin{theorem}\label{thm:pseudoinverse}
    Let $T\in \linop(X,Y)$ and set
    \begin{equation*}
        \tilde T := T|_{\calN(T)^{\bot}}:\calN(T)^\bot \to \calR(T).
    \end{equation*}
    Then there exists a unique linear extension $T^\dag$ of $\tilde T^{-1}$ with
    \begin{align*}
        \calD(T^\dag)&=\calR(T) \oplus \calR(T)^\bot,\\
        \calN(T^\dag)&=\calR(T)^\bot.
    \end{align*}
\end{theorem}
\begin{proof}
    Due to the restriction to $\calN(T)^\bot$ and $\calR(T)$, the operator $\tilde T$ is injective and surjective, and hence there exists a (linear) inverse $\tilde T^{-1}$. Thus $T^{\dag}$ is well-defined and linear on $\calR(T)$. For any $y\in \calD(T^\dag)$, we obtain by orthogonal decomposition unique $y_1 \in \calR(T)$ and $y_2\in \calR(T)^\bot$ with $y=y_1+y_2$. Since $\calN(T^\dag) = \calR(T)^\bot$,
    \begin{equation}\label{eq:inverse:proj1}
        T^\dag y := T^\dag y_1 + T^\dag y_2 = T^\dag y_1 = \tilde T^{-1} y_1
    \end{equation}
    defines a unique linear extension.
    Hence $T^\dag$ is well-defined on its whole domain $\calD(T^\dag)$.
\end{proof}
The operator $T^\dag$ constructed in \cref{thm:pseudoinverse} is called the \term{Moore--Penrose inverse} of $T$.
If $T$ is bijective, we obviously have $T^\dag = T^{-1}$. However, it is important to note that $T^{\dag}$ need not be a \emph{continuous} extension.

\bigskip

In the following, we will need the following properties of the Moore--Penrose inverse.
\begin{lemma}\label{lem:inverse}
    The Moore--Penrose inverse $T^\dag$ satisfies $\calR(T^\dag) = \calN(T)^\bot$ as well as the \term[Moore--Penrose inverse!Moore--Penrose equations]{Moore--Penrose equations}
    \begin{enumerate}
        \item $TT^\dag T = T$,
        \item $T^\dag T T^\dag = T^\dag$,
        \item $T^\dag T = \Id-P_{\calN}$,
        \item $TT^\dag = (P_{\overline\calR})|_{\calD(T^\dag)}$,
    \end{enumerate}
    where $P_{\calN}$ and $P_{\overline\calR}$ denote the orthogonal projections on $\calN(T)$ and $\overline{\calR(T)}$, respectively.
\end{lemma}
\begin{proof}
    We first show that $\calR(T^\dag) = \calN(T)^\bot$.
    By the definition of $T^\dag$ and \eqref{eq:inverse:proj1}, we have for all $y\in \calD(T^\dag)$ that
    \begin{equation}
        \label{eq:inverse:mpgl1}
        T^\dag y = T^\dag P_{\overline\calR}y = \tilde T^{-1} P_{\overline\calR}y
    \end{equation}
    since $y\in \calD(T^\dag)=\calR(T)\oplus \calR(T)^\bot$ implies that $P_{\overline{\calR}}y \in \calR(T)$ (and not only in $\overline{\calR(T)}$ -- this fundamental property will be used repeatedly in the following).
    Hence $T^\dag y \in \calR(\tilde T^{-1}) = \calN(T)^\bot$, i.e., $\calR(T^\dag)\subset \calN(T)^\bot$. Conversely, $T^\dag T x = \tilde T^{-1}\tilde T x = x$ for all $x\in\calN(T)^\bot$, i.e., $x\in \calR(T^\dag)$. This shows that $\calR(T^\dag) = \calN(T)^\bot$ as claimed.

    Ad (iv): For $y\in\calD(T^\dag)$, we have from \eqref{eq:inverse:mpgl1} and $\calR(T^\dag) = \calN(T)^\bot$ that
    \begin{equation*}
        TT^\dag y = T\tilde T^{-1} P_{\overline\calR}y = \tilde T\tilde T^{-1} P_{\overline\calR}y = P_{\overline\calR}y
    \end{equation*}
    since $\tilde T^{-1} P_{\overline\calR}y\in \calN(T)^\bot$ and $T=\tilde T$ on $\calN(T)^\bot$.

    Ad (iii): The definition of $T^\dag$ implies that $T^\dag Tx = \tilde T^{-1} Tx$ for all $x\in X$ and hence that
    \begin{equation*}
        T^\dag T x = \tilde T^{-1} T\left(P_{\calN} x + (\Id - P_{\calN})x\right)
        = \tilde T^{-1} T P_{\calN} x + \tilde T^{-1} \tilde T (\Id - P_{\calN}) x = (\Id - P_{\calN})x.
    \end{equation*}

    Ad (ii): Using (iv) and \eqref{eq:inverse:mpgl1} yields for $y\in \calD(T^\dag)$ that
    \begin{equation*}
        T^\dag TT^\dag y = T^\dag P_{\overline\calR}y = T^\dag y.
    \end{equation*}

    Ad (i): Directly from (iii) it follows that
    \begin{equation*}
        TT^\dag Tx = T(\Id - P_{\calN})x = Tx - TP_{\calN} x = Tx \qquad\text{for all } x\in X.
        \qedhere
    \end{equation*}
\end{proof}
(In fact, the Moore--Penrose equations are an equivalent characterization of $T^\dag$.)

We can now show that the Moore--Penrose inverse indeed yields the minimum norm solution; in passing, we also characterize the least squares solutions.
\begin{theorem}\label{thm:inverse:minnorm}
    For any $y\in \calD(T^\dag)$, the equation $Tx=y$ admits
    \begin{enumerate}
        \item least squares solutions, which are exactly the solutions of
            \begin{equation}\label{eq:inverse:ausgleichung}
                Tx = P_{\overline\calR} y;
            \end{equation}
        \item a unique minimum norm solution $x^\dag\in X$, which is given by
            \begin{equation*}
                x^\dag = T^\dag y.
            \end{equation*}
            The set of all least squares solutions is given by $x^\dag + \calN(T)$.
    \end{enumerate}
\end{theorem}
\begin{proof}
    First, $P_{\overline\calR}y \in \calR(T)$ for $y\in\calD(T^\dag)$ implies that \eqref{eq:inverse:ausgleichung} admits at least one solution. The optimality of the orthogonal projection further implies that any such solution $z\in X$ satisfies
    \begin{equation*}
        \norm{Tz-y}_Y = \norm{P_{\overline\calR}y-y}_Y = \min_{w\in\overline{\calR(T)}}\norm{w-y}_Y \leq \norm{Tx-y}_Y\qquad\text{for all }x\in X,
    \end{equation*}
    i.e., all solutions of \eqref{eq:inverse:ausgleichung} are least squares solutions of $Tx=y$. Conversely, any least squares solution $z\in X$ satisfies
    \begin{equation*}
        \norm{P_{\overline\calR}y-y}_Y \leq \norm{Tz-y}_Y = \min_{x\in X}\norm{Tx-y}_Y = \min_{w\in{\calR(T)}}\norm{w-y}_Y \leq \norm{P_{\overline\calR}y-y}_Y
    \end{equation*}
    since $P_{\overline\calR}y\in\calR(T)$ and hence $Tz=P_{\overline\calR}y$. This shows (i).

    The least squares solutions are thus exactly the solutions of $Tx=P_{\overline\calR}y$, which can be uniquely represented as $x=\bar x+x_0$ with $\bar x\in \calN(T)^\bot$ and $x_0\in\calN(T)$. Since $T$ is injective on $\calN(T)^\bot$, the element $\bar x$ must be independent of $x$ (otherwise $Tx' = T\bar x' \neq T\bar x = P_{\overline\calR}y$ for $x' =\bar x' + x_0$ with $\bar x' \neq \bar x$).
    It then follows from
    \begin{equation*}
        \norm{x}_{X}^2 = \norm{\bar x+x_0}_X^2 = \norm{\bar x}_X^2 + 2\inner{\bar x}{x_0}_X + \norm{x_0}_X^2 = \norm{\bar x}_X^2 + \norm{x_0}_X^2 \geq \norm{\bar x}^2_X
    \end{equation*}
    that $x^\dag:= \bar x\in \calN(T)^\bot$ is the unique minimum norm solution.

    Finally, $x^\dag \in \calN(T)^\bot$ and $Tx^\dag = P_{\overline \calR}y$ together with \cref{lem:inverse}\,(iii) and (ii) imply that
    \begin{equation*}
        x^\dag = P_{\calN^\bot}x^\dag = (\Id-P_{\calN})x^\dag = T^\dag T x^\dag = T^\dag P_{\overline\calR}y = T^\dag TT^\dag y = T^\dag y,
    \end{equation*}
    which shows (ii).
\end{proof}

We can give an alternative characterization that will later be useful.
\begin{corollary}\label{thm:inverse:normalen}
    Let $y\in \calD(T^\dag)$. Then $x\in X$ is a least squares solution of $Tx =y$ if and only if $x$ satisfies the \term[equation, normal]{normal equation}
    \begin{equation}
        \label{eq:inverse:normal}
        T^* T x = T^*y.
    \end{equation}
    If additionally $x\in\calN(T)^\bot$, then $x=x^\dag$.
\end{corollary}
\begin{proof}
    \Cref{thm:inverse:minnorm}\,(i) states that $x\in X$ is a least squares solution if and only if $Tx = P_{\overline\calR} y$, which is equivalent to $Tx\in \overline{\calR(T)}$ and $Tx-y\in \overline{\calR(T)}{}^\bot = \calN(T^*)$, i.e., $T^*(Tx-y) = 0$.

    Similarly, \cref{thm:inverse:minnorm}\,(ii) implies that a least squares solution $x$ has minimal norm if and only if $x=T^\dag y\in\calN(T)^\bot$.
\end{proof}
The minimum norm solution $x^\dag$ of $Tx=y$ is therefore also the solution -- and hence, in particular, the least squares solution -- of \eqref{eq:inverse:normal} with minimal norm, i.e.,
\begin{equation}\label{eq:inverse:normal_pseudo}
    x^\dag = (T^*T)^\dag T^*y.
\end{equation}
We can therefore characterize $x^\dag$ as the minimum norm solution of \eqref{eq:inverse:normal} as well as of $Tx=y$, which can sometimes be advantageous.

\bigskip

Until now, we have considered the pseudoinverse on its domain without characterizing this further; this we now catch up on.
First, by construction $\calD(T^\dag) = \calR(T)\oplus \calR(T)^\bot$. Since orthogonal complements are always closed,
\begin{equation*}
    \overline{\calD(T^\dag)} = \overline{\calR(T)} \oplus \calR(T)^\bot = \calN(T^*)^\bot \oplus \calN(T^*) = Y,
\end{equation*}
i.e., $\calD(T^\dag)$ is dense in $Y$. If $\calR(T)$ is closed, this implies that $\calD(T^\dag)=Y$ (which conversely implies that $\calR(T)$ is closed).
Furthermore, for $y\in \calR(T)^\bot = \calN(T^\dag)$ the minimum norm solution is always $x^\dag=0$.
The central question is therefore whether a given $y\in\overline{\calR(T)}$ is in fact an element of $\calR(T)$. If this always holds, $T^\dag$ is even continuous. Conversely, the existence of a single $y\in\overline{\calR(T)}\setminus\calR(T)$ already suffices for $T^\dag$ not to be continuous.
\begin{theorem}\label{thm:inverse:cont}
    Let $T\in \linop(X,Y)$. Then $T^\dag:\calD(T^\dag)\to X$ is continuous if and only if $\calR(T)$ is closed.
\end{theorem}
\begin{proof}
    We apply the closed graph theorem (\cref{thm:closedgraph}), for which we have to show that $T^\dag$ is closed. Let $\{y_n\}_{n\in\N}\subset \calD(T^\dag)$ be a sequence with $y_n\to y\in Y$ and $T^\dag y_n \to x\in X$. \Cref{lem:inverse}\,(iv) then implies that
    \begin{equation*}
        TT^\dag y_n = P_{\overline\calR} y_n \to P_{\overline\calR} y
    \end{equation*}
    due to the continuity of the orthogonal projection. It follows from this and the continuity of $T$ that
    \begin{equation}\label{eq:inverse_cont1}
        P_{\overline\calR} y = \lim_{n\to\infty} P_{\overline\calR} y_n = \lim_{n\to\infty} TT^\dag y_n = Tx,
    \end{equation}
    i.e., $x$ is a least squares solution.
    Furthermore, $T^\dag y_n \in \calR(T^\dag)=\calN(T)^\bot$ also implies that
    \begin{equation*}
        T^\dag y_n \to x\in \calN(T)^\bot
    \end{equation*}
    since $\calN(T)^\bot=\overline{\calR(T^*)}$ is closed. By \cref{thm:inverse:minnorm}\,(ii), $x$ is thus the minimum norm solution of $Tx=y$, i.e., $x = T^\dag y$. Hence $T^\dag$ is closed.

    If $\calR(T)$ is now closed, we have that $\calD(T^\dag)=Y$ and thus that $T^\dag:Y\to X$ is continuous by \cref{thm:closedgraph}. Conversely, if $T^\dag$ is continuous on $\calD(T^\dag)$, the density of $\calD(T^\dag)$ in $Y$ ensures that $T^\dag$ can be extended continuously to $Y$ by
    \begin{equation*}
        \overline{T^\dag} y := \lim_{n\to\infty} T^\dag y_n \quad \text{for a sequence } \{y_n\}_{n\in\N} \subset \calD(T^\dag) \text{ with } y_n\to y\in Y.
    \end{equation*}
    (Since $T^\dag$ is bounded, it maps Cauchy sequences to Cauchy sequences, and hence $\overline{T^\dag}$ is well-defined and continuous.)
    Let now $y\in\overline{\calR(T)}$ and $\{y_n\}_{n\in\N} \subset \calR(T)$ with $y_n\to y$. As for \eqref{eq:inverse_cont1}, we then have
    \begin{equation*}
        y = P_{\overline{\calR}}y = \lim_{n\to\infty} P_{\overline{\calR}}y_n = \lim_{n\to\infty}T{T^\dag}y_n = T\overline{T^\dag}y \in \calR(T)
    \end{equation*}
    and hence that $\overline{\calR(T)}=\calR(T)$.
\end{proof}

Accordingly, the operator equation $Tx=y$ is called \term[ill-posed!in the sense of Nashed]{ill-posed in the sense of \href{https://en.wikipedia.org/wiki/Zuhair_Nashed}{Nashed}} if $\calR(T)$ is not closed.
Unfortunately, this already excludes many interesting compact operators.
\begin{corollary}\label{thm:inverse:discont}
    If $K\in\calK(X,Y)$ has infinite-dimensional range $\calR(K)$, then $K^\dag$ is not continuous.
\end{corollary}
\begin{proof}
    Assume to the contrary that $K^\dag$ is continuous. Then $\calR(K)$ is closed by \cref{thm:inverse:cont}, and thus the operator $\tilde K:=K:\calN(K)^\bot \to \calR(K)$ from \cref{thm:pseudoinverse} has a continuous inverse $\tilde K^{-1}\in\linop(\calR(K),\calN(K)^\bot)$.
    Now, $K$ and therefore also $K\circ\tilde K^{-1}$ are compact. By
    \begin{equation*}
        K \tilde K^{-1} y = y \qquad \text{for all } y\in \calR(K),
    \end{equation*}
    this implies that the identity $\Id:\calR(K)\to \calR(K)$ is compact as well, which is only possible if $\calR(K)$ is finite-dimensional.
\end{proof}
For compact operators, the third condition for well-posedness in the sense of Hadamard therefore has to be handled by other methods, which we will study in the following chapters.

\section{Singular value decomposition of compact operators}

We now characterize the Moore--Penrose inverse of a compact operator $K\in \calK(X,Y)$ via orthonormal systems. We would like to do this using a spectral decomposition, which however exists only for selfadjoint operators. But by \cref{thm:inverse:normalen}, we can equivalently consider the Moore--Penrose inverse of $K^*K$, which \emph{is} selfadjoint; this leads to the \term{singular value decomposition}.
\begin{theorem}\label{thm:svd}
    For every $K\in\calK(X,Y)$, there exist
    \begin{enumerate}
        \item a null sequence $\{\sigma_n\}_{n\in\N}$ with $\sigma_1\geq \sigma_2\geq \dots >0$,
        \item an orthonormal basis $\{u_n\}_{n\in\N}\subset Y$ of $\overline{\calR(K)}$,
        \item an orthonormal basis $\{v_n\}_{n\in\N}\subset X$ of $\overline{\calR(K^*)}$
    \end{enumerate}
    (possibly finitely terminating) with
    \begin{equation}\label{eq:sing_vec}
        K v_n = \sigma_n u_n\quad\text{and}\quad K^* u_n = \sigma_n v_n \qquad\text{for all }n\in\N
    \end{equation}
    and
    \begin{equation}\label{eq:svd}
        Kx = \sum_{n\in\N} \sigma_n \inner{x}{v_n}_X u_n \qquad\text{for all }x\in X.
    \end{equation}
    A sequence $\{(\sigma_n,u_n,v_n)\}_{n\in\N}$ satisfying the \emph{singular value decomposition} \eqref{eq:svd} is called a \term[singular value decomposition!singular system]{singular system}.
\end{theorem}
\begin{proof}
    Since $K^*K:X\to X$ is compact and selfadjoint, the Spectral \cref{thm:spektral} yields a null sequence $\{\lambda_n\}_{n\in\N}\subset [0,\infty)$ (ordered by decreasing magnitude) and an orthonormal system $\{v_n\}_{n\in\N}\subset X$ of corresponding eigenvectors with
    \begin{equation*}
        K^*Kx = \sum_{n\in\N} \lambda_n \inner{x}{v_n}_X v_n \qquad \text{for all }x\in X.
    \end{equation*}
    Since $\lambda_n = \lambda_n \norm{v_n}_X^2 = \inner{\lambda_n v_n}{v_n}_X =\inner{K^*K v_n}{v_n}_X = \norm{Kv_n}_Y^2 >0$, we can define for all $n\in\N$
    \begin{equation*}
        \sigma_n := \sqrt{\lambda_n} >0 \qquad\text{ and }\qquad u_n:= \sigma_n^{-1}{Kv_n} \in Y.
    \end{equation*}
    (If the sequence $\{v_n\}_{n\in\N}$ terminates finitely, we set $\sigma_n:=0$.)
    The latter form an orthonormal system due to
    \begin{equation*}
        \inner{u_i}{u_j}_Y = \frac1{\sigma_i\sigma_j}\inner{Kv_i}{Kv_j}_Y = \frac1{\sigma_i\sigma_j} \inner{K^*Kv_i}{v_j}_X = \frac{\lambda_i}{\sigma_i\sigma_j}\inner{v_i}{v_j}_X =
        \begin{cases}
            1 & \text{if }i=j,\\ 0 & \text{else.}
        \end{cases}
    \end{equation*}
    Furthermore, we have for all $n\in\N$ that
    \begin{equation*}
        K^*u_n = \sigma_n^{-1} K^*K v_n = \sigma_n^{-1}\lambda_n v_n = \sigma_n v_n.
    \end{equation*}

    \Cref{thm:spektral} also yields that $\{v_n\}_{n\in\N}$ is an orthonormal basis of $\overline{\calR(K^*K)}$. In addition, $\overline{\calR(K^*K)} = \overline{\calR(K^*)}$, since for any $x\in \overline{\calR(K^*)}$, there exists a sequence $\{y_n\}_{n\in\N}\subset Y$ with $K^*y_n \to x$; in particular, we can take $y_n\in \calN(K^*)^\bot=\overline{\calR(K)}$, and a diagonal argument shows $x\in\overline{\calR(K^*K)}$. (The other direction is obvious.) Hence, $\{v_n\}_{n\in\N}$ is an orthonormal basis of $\overline{\calR(K^*)}=\calN(K)^\bot$, and therefore
    \begin{equation*}
        Kx = K P_{\calN^\bot} x = K \left(\sum_{n\in\N} \inner{x}{v_n}_X v_n\right) \qquad\text{for all }x\in X.
    \end{equation*}
    From this, we obtain the singular value decomposition \eqref{eq:svd} by \enquote{pushing} $K$ through the series representation. Since we will repeatedly apply such arguments in the following, we justify this step in detail.
    First, we set $x_N := \sum_{n=1}^N \inner{x}{v_n}_X v_n$ for any $x\in X$ and $N\in\N$.
    Then we clearly have $x_N \to P_{\calN^\bot} x$ as $N\to \infty$ and hence by continuity of $K$ also
    \begin{equation}\label{eq:svd:operator}
        \begin{aligned}[t]
            Kx &= K \left(P_{\calN^\bot} x\right) = K(\lim_{N\to\infty} x_N) = \lim_{N\to\infty} Kx_N \\
            &= \lim_{N\to\infty} \sum_{n=1}^N \inner{x}{v_n}_X Kv_n = \sum_{n\in\N} \inner{x}{v_n}_X Kv_n.
        \end{aligned}
    \end{equation}
    We thus have for all $x\in X$ that
    \begin{equation*}
        Kx = \sum_{n\in\N} \inner{x}{v_n}_X Kv_n = \sum_{n\in\N} \inner{x}{v_n}_X \sigma_n u_n= \sum_{n\in\N} \inner{x}{K^*u_n}_X u_n = \sum_{n\in\N} \inner{Kx}{u_n}_X u_n.
    \end{equation*}
    The second equation yields \eqref{eq:svd}, while the last implies that $\{u_n\}_{n\in\N}$ is an orthonormal basis of $\overline{\calR(K)}$.
\end{proof}
Since the eigenvalues of $K^*K$ with eigenvector $v_n$ are exactly the eigenvalues of $KK^*$ with eigenvector $u_n$, this also yields by \eqref{eq:sing_vec} a singular system $\{(\sigma_n,v_n,u_n)\}_{n\in\N}$ of $K^*$ such that
\begin{equation}
    \label{eq:svd_adj}
    K^* y = \sum_{n\in\N} \sigma_n \inner{y}{u_n}_Y v_n \qquad\text{for all }y\in Y.
\end{equation}

\bigskip

We now use the singular value decomposition of $K$ to characterize the domain $\calD(K^\dag) = \calR(K)\oplus\calR(K)^\bot$ of the Moore--Penrose inverse $K^\dag$.
As was already discussed before \cref{thm:inverse:cont}, this reduces to the question whether $y\in \overline{\calR(K)}$ is in fact an element of $\calR(K)$.
\begin{theorem}\label{thm:inverse:picard}
    Let $K\in\calK(X,Y)$ with singular system $\{(\sigma_n,u_n,v_n)\}_{n\in\N}$ and $y\in \overline{\calR(K)}$. Then $y\in \calR(K)$ if and only if the \term[condition!Picard]{Picard condition}
    \begin{equation}
        \label{eq:inverse:picard}
        \sum_{n\in\N} \sigma_n^{-2} |\inner{y}{u_n}_Y|^2 <\infty
    \end{equation}
    is satisfied. In this case,
    \begin{equation}\label{eq:inverse:picard_pseudo}
        K^\dag y = \sum_{n\in\N} \sigma_n^{-1} \inner{y}{u_n}_Y v_n.
    \end{equation}
\end{theorem}
\begin{proof}
    Let $y\in\calR(K)$, i.e., there exists $x\in X$ with $Kx=y$. Then
    \begin{equation*}
        \inner{y}{u_n}_Y = \inner{x}{K^*u_n}_X = \sigma_n \inner{x}{v_n}_X \qquad\text{for all }n\in\N,
    \end{equation*}
    and the Bessel inequality \eqref{eq:funktan:bessel} yields
    \begin{equation*}
        \sum_{n\in\N} \sigma_n^{-2} |\inner{y}{u_n}_Y|^2 = \sum_{n\in\N} |\inner{x}{v_n}_X|^2 \leq \norm{x}_X^2 < \infty.
    \end{equation*}

    Conversely, let $y\in \overline{\calR(K)}$ satisfy \eqref{eq:inverse:picard}, which
    implies that $\{\sum_{n=1}^N \sigma_n^{-2}|\inner{y}{u_n}_Y|^2\}_{N\in\N}$ is a Cauchy sequence. Then $\{x_N\}_{N\in\N}$ defined by
    \begin{equation*}
        x_N:= \sum_{n=1}^N \sigma_n^{-1} \inner{y}{u_n}_Yv_n
    \end{equation*}
    is a Cauchy sequence as well, since $\{v_n\}_{n\in\N}$ forms an orthonormal system and thus
    \begin{equation*}
        \norm{x_N-x_M}_X^2 = \norm{{\textstyle\sum\nolimits_{n=N+1}^M} \sigma_n^{-1} \inner{y}{u_n}_Yv_n}_X^2 = \sum_{n=N+1}^M |\sigma_n^{-1} \inner{y}{u_n}_Y|^2 \to 0\quad\text{as }N,M\to\infty.
    \end{equation*}
    Furthermore, $\{v_n\}_{n\in\N}\subset \overline{\calR(K^*)}$. Hence, $\{x_N\}_{N\in\N}\subset \overline{\calR(K^*)}$ converges to some
    \begin{equation*}
        x:= \sum_{n\in\N} \sigma_n^{-1} \inner{y}{u_n}_Yv_n \in \overline{\calR(K^*)}= \calN(K)^\bot
    \end{equation*}
    by the closedness of $\overline{\calR(K^*)}$.
    Now we have as in \eqref{eq:svd:operator} that
    \begin{equation*}
        Kx = \sum_{n\in\N} \sigma_n^{-1} \inner{y}{u_n}_YKv_n = \sum_{n\in\N} \inner{y}{u_n}_Yu_n = P_{\overline{\calR}} y = y,
    \end{equation*}
    which implies that $y\in\calR(K)$.

    Finally, $Kx = P_{\overline{\calR}} y$ for $x\in\calN(K)^\bot$ is equivalent to $x=K^\dag y$ by \cref{thm:inverse:minnorm}, which also shows \eqref{eq:inverse:picard_pseudo}.
\end{proof}

The Picard condition states that a minimum norm solution can only exist if the \enquote{Fourier coefficients}
$\inner{y}{u_n}_Y$ of $y$ decay fast enough compared to the singular values $\sigma_n$.
The representation \eqref{eq:inverse:picard_pseudo} also shows how perturbations of $y$ relate to perturbations of $x^\dag$: If $y^\delta = y+\delta u_n$ for some $\delta>0$ and $n\in\N$, then
\begin{equation*}
    \norm{K^\dag y^\delta -K^\dag y}_X = \delta \norm{K^\dag u_n}_X = \sigma_n^{-1}\delta \to \infty \quad\text{as }n\to \infty,
\end{equation*}
and the faster the singular values decay, the more the error is amplified for given $n$.
Hence one distinguishes
\begin{itemize}
    \item \term[ill-posed!moderately]{moderately ill-posed} problems, for which there exist $c,r>0$ with $\sigma_n \geq c n^{-r}$ for all $n\in\N$ (i.e., $\sigma_n$ decays at most polynomially), and
    \item \term[ill-posed!severely]{severely ill-posed} problems, for which this is not the case. If $\sigma_n \leq c e^{-n^r}$ for all $n\in\N$ and some $c,r>0$ (i.e., $\sigma_n$ decays at least exponentially), the problem is called \term[ill-posed!exponentially]{exponentially ill-posed}.
\end{itemize}
For exponentially ill-posed problems, one can in general not expect to obtain a solution that is more than a very rough approximation. On the other hand, if $\calR(K)$ is finite-dimensional, then the sequence $\{\sigma_n\}_{n\in\N}$ is finite and the error stays bounded; in this case, $K^\dag$ is continuous as expected.

The singular value decomposition is a valuable analytical tool, but its explicit computation for a concrete operator is in general quite involved. We again consider differentiation as an elementary example.
\begin{example}\label{ex:integration_svd}
    Let $X=L^2(\Omega)$ for $\Omega=(0,1)$ and let $K\in \calK(X,X)$ be an integral operator defined via
    \begin{equation*}
        [Kx](t) = \int_0^1 k(s,t) x(s)\,ds \qquad\text{with}\qquad
        k(s,t) =\begin{cases} 1 & \text{if }s\leq t,\\0 & \text{else}.\end{cases}
    \end{equation*}
    If $x=y'$ for some $y \in C^1([0,1])$ with $y(0)=0$, then
    \begin{equation*}
        [Kx](t) = \int_0^t x(s)\,ds = y(t) - y(0) = y(t),
    \end{equation*}
    i.e., the derivative $y'$ of $y\in C^1([0,1])$ with $y(0)=0$ is a solution of the operator equation $Kx=y$ (which is also meaningful for $y\in L^2(\Omega)$ but may not admit a solution then).

    The corresponding adjoint operator is given by
    \begin{equation*}
        [K^*y](t) = \int_0^1 k(t,s) y(s)\,ds = \int_t^1 y(s)\,ds,
    \end{equation*}
    since $k(t,s)=1$ for $s\geq t$ and $0$ else.
    We now compute the eigenvalues and eigenvectors of $K^*K$, i.e., any $\lambda>0$ and $v\in L^2(\Omega)$ with
    \begin{equation}
        \label{eq:svd_diff_ev}
        \lambda v(t) = [K^*Kv](t) = \int_t^1\int_0^s v(r)\,dr\,ds.
    \end{equation}
    We first proceed formally. Inserting $t=1$ yields $\lambda v(1) = 0$ and therefore $v(1) = 0$. Differentiating \eqref{eq:svd_diff_ev} yields
    \begin{equation*}
        \lambda v'(t) = \frac{d}{dt}\left( -\int_1^t\int_0^s v(r)\,dr\,ds\right) = -\int_0^t v(r)\,dr,
    \end{equation*}
    which for $t=0$ implies that $v'(0)=0$. Differentiating again now leads to the ordinary differential equation
    \begin{equation*}
        \lambda v''(t) + v(t) = 0
    \end{equation*}
    which has the general solution
    \begin{equation*}
        v(t) = c_1 \sin(\sigma^{-1} t) + c_2 \cos(\sigma^{-1} t)
    \end{equation*}
    for $\sigma := \sqrt{\lambda}$ and constants $c_1,c_2$ that have yet to be determined. For this, we insert the boundary conditions $v'(0)=0$ and $v(1)=0$, which yields $c_1 = 0$ and $c_2 \cos (\sigma^{-1}) = 0$, respectively. Since $c_2 = 0$ leads to the trivial solution $v=0$ and eigenvectors are by definition not trivial, we must have $\cos(\sigma^{-1})=0$; the only candidates for the singular values $\sigma_n$ are therefore the reciprocal roots of the cosine, i.e.,
    \begin{equation*}
        \sigma_n = \frac{2}{(2n-1)\pi}, \qquad n\in\N.
    \end{equation*}
    From this, we obtain the eigenvectors
    \begin{equation*}
        v_n(t) = \sqrt{2}\cos\left((n-\tfrac12)\pi\,t\right),\qquad n\in\N,
    \end{equation*}
    where the constant $c_2=\sqrt{2}$ is chosen such that $\norm{v_n}_{L^2(\Omega)}=1$.
    We further compute
    \begin{equation*}
        u_n := \sigma_n^{-1} Kv_n = (n-\tfrac12)\pi\int_0^t \sqrt{2}\cos\left((n-\tfrac12)\pi\,s\right)ds = \sqrt{2} \sin\left((n-\tfrac12)\pi\,t\right),\quad n\in\N.
    \end{equation*}
    Now we have $v_n,u_n\in L^2(\Omega)$, and it is straightforward to verify that $\sigma_n^2$ and $v_n$ satisfy the eigenvalue relation \eqref{eq:svd_diff_ev}. As in the proof of \cref{thm:svd}, this yields a singular value decomposition of $K$ and thus a singular system $\{(\sigma_n,u_n,v_n)\}_{n\in\N}$.

    Since $\sigma_n =\mathcal{O}(\frac1n)$, this implies that differentiation (in this formulation) is a moderately ill-posed problem.
    It is now possible to show that $\{u_n\}_{n\in\N}$ is an orthonormal basis of $L^2(\Omega)$ (which is not unique).
    Furthermore, the Picard condition \eqref{eq:inverse:picard} for $y\in L^2(\Omega)$ is given by
    \begin{equation*}
        \sum_{n\in\N} \frac14{(2n-1)^2\pi^2} |\inner{y}{u_n}_{L^2}|^2 <\infty.
    \end{equation*}
    In addition, \eqref{eq:funktan:ons_reihe} implies that
    \begin{equation*}
        y = \sum_{n\in\N} \inner{y}{u_n}_{L^2} u_n,
    \end{equation*}
    and hence formally differentiating the Fourier series term by term yields
    \begin{equation*}
        z:= \sum_{n\in\N} \inner{y}{u_n}_{L^2} u_n' =
        \sum_{n\in\N} \left(n-\tfrac12\right)\pi \inner{y}{u_n}_{L^2} v_n.
    \end{equation*}
    The Picard condition is thus equivalent to the condition that $\norm{z}_{L^2(\Omega)}^2 <\infty$ and hence that the formally differentiated series converges (in $L^2(\Omega)$); in this case $K^\dag y = z$. If $y$ is continuously differentiable, this convergence is even uniform and we obtain that $y'=z = K^\dag y$.
\end{example}

\bigskip

The singular value decomposition allows defining functions of compact operators, which will be a fundamental tool in the following chapters. Let $\phi:[0,\infty)\to\R$ be a piecewise continuous and locally bounded function. We then define for $K\in \calK(X,Y)$ with singular system $\{(\sigma_n,u_n,v_n)\}_{n\in\N}$ the operator $\phi(K^*K):X\to X$ by
\begin{equation}
    \label{eq:functional}
    \phi(K^*K)x = \sum_{n\in\N} \phi(\sigma_n^2) \inner{x}{v_n}_X v_n+ \phi(0)P_{\calN(K)} x \qquad \text{for all }x\in X.
\end{equation}
This series converges in $X$ since $\phi$ is only evaluated on the closed and bounded interval $[0,\sigma_1^2]=[0,\norm{K}_{\linop(X,Y)}^2]$. Furthermore, the Bessel inequality implies that
\begin{equation}\label{eq:functional_continuous}
    \norm{\phi(K^*K)}_{\linop(X,X)} \leq \sup_{n\in\N} |\phi(\sigma_n^2)| + |\phi(0)| \leq 2\sup_{\lambda\in [0,\norm{K}_{\linop(X,Y)}^2]} |\phi(\lambda)|<\infty,
\end{equation}
i.e., $\phi(K^*K)\in \linop(X,X)$.

In particular, we consider here power functions $\phi(t) = t^r$ for $r\geq 0$ and especially the following examples.
\begin{example}\label{ex:spectralcalculus}
    Let $K\in\calK(X,Y)$.
    \begin{enumerate}
        \item For $\phi(t) = 1$ we have $\phi(K^*K) = \Id$ since for all $x\in X$,
            \begin{equation*}
                \phi(K^*K)x = \sum_{n\in\N} \inner{x}{v_n}_X v_n+ P_{\calN(K)} x = P_{\overline{\calR(K^*)}} x + P_{\calN(K)} x = x
            \end{equation*}
            due to $\overline{\calR(K^*)} = \calN(K)^\bot$.
        \item For $\phi(t) = t$ we have $\phi(K^*K)=K^*K$ due to $\phi(0)=0$ and the spectral theorem.
        \item For $\phi(t) = \sqrt{t}$ we call $|K|:= \phi(K^*K)$ the \term[value, absolute (of an operator)]{absolute value} of $K$; since $\sigma_n>0$, we have
            \begin{equation*}
                |K|x = \sum_{n\in\N} \sigma_n \inner{x}{v_n}_X v_n \qquad\text{for all }x\in X.
            \end{equation*}
    \end{enumerate}
\end{example}

Comparing \cref{ex:spectralcalculus}\,(iii) with the singular value decomposition \eqref{eq:svd} shows that $|K|$ essentially has the same behavior as $K$, the only difference being that the former maps to $X$ instead of $Y$.
This is illustrated by the following properties, which will be used later.
\begin{lemma}\label{lem:functional_range}
    Let $K\in \calK(X,Y)$. Then
    \begin{enumerate}
        \item $|K|^{r+s} = |K|^r \circ |K|^s$ for all $r,s\geq 0$;
        \item $|K|^r$ is selfadjoint for all $r\geq 0$;
        \item $\norm{|K|x}_X = \norm{Kx}_Y$ for all $x\in X$;
        \item $\calR(|K|) = \calR(K^*)$.
    \end{enumerate}
\end{lemma}
\begin{proof}
    Ad (i): This follows directly from
    \begin{equation*}
        \begin{aligned}
            |K|^{r+s}x &= \sum_{n\in\N} \sigma_n^{r+s} \inner{x}{v_n}_X v_n =
            \sum_{n\in\N} \sigma_n^{r} \left(\sigma_n^s\inner{x}{v_n}_X\right) v_n \\
            &=
            \sum_{n\in\N} \sigma_n^{r} \inner{\sum_{m\in\N}\sigma_m^s\inner{x}{v_m}_X v_m}{v_n}_Xv_n \\
            &= \sum_{n\in\N} \sigma_n^{r} \inner{|K|^s x}{v_n}_X v_n
            = |K|^r ( |K|^s x)
        \end{aligned}
    \end{equation*}
    since $\{v_n\}_{n\in\N}$ is an orthonormal system.

    Ad (ii): For any $x,z\in X$ and $r\geq 0$, the bilinearity and symmetry of the inner product imply that
    \begin{equation*}
        \inner{|K|^r x}{z }_X = \sum_{n\in\N} \sigma_n^r \inner{x}{v_n}_X\inner{v_n}{z}_X = \inner{x}{|K|^rz}_X.
    \end{equation*}

    Ad (iii):  This follows from (i), (ii), and
    \begin{equation*}
        \norm{|K|x}_X^2 = \inner{|K|x}{|K|x}_X = \inner{|K|^2 x }{x}_X = \inner{K^*Kx}{x}_X = \inner{Kx}{Kx}_Y = \norm{Kx}_Y^2.
    \end{equation*}

    Ad (iv):
    Let $\{(\sigma_n,u_n,v_n)\}_{n\in\N}$ be a singular system of $K$. Then $\{(\sigma_n,v_n,u_n)\}_{n\in\N}$ is a singular system of $K^*$, and -- by definition -- $\{(\sigma_n,v_n,v_n)\}_{n\in\N}$ is a singular system of $|K|$.
    Now $x\in\calR(K^*)$ if and only if $Kx\in\calR(KK^*)$ and $x\in\calN(K)^\bot$. The Picard condition for $Kx\in\calR(KK^*)$ is
    \begin{equation*}
        \infty> \sum_{n\in\N} \sigma_n^{-4} |\inner{Kx}{u_n}_Y|^2 = \sum_{n\in\N} \sigma_n^{-4} |\inner{x}{K^*u_n}_X|^2 = \sum_{n\in\N} \sigma_n^{-2} |\inner{x}{v_n}_X|^2.
    \end{equation*}
    But this is also the Picard condition for $x\in\calR(|K|)$ (compare the proof of \cref{thm:inverse:picard}), which for $x\in\calN(K)^\bot$ is even a necessary condition.
\end{proof}

The proof of \cref{lem:functional_range}\,(iv) already indicates that we can use $|K|$ to formulate a variant of the Picard condition for $x\in \overline{\calR(K^*)}$ (instead of $y\in \overline {\calR(K)}$); we will use this in a following chapter to characterize minimum norm solutions that can be particularly well approximated.

We finally need the following inequality.
\begin{lemma}\label{lem:interpolation}
    Let $K\in\calK(X,Y)$. Then any $r>s\geq0$ and $x\in X$ satisfy the \term[inequality, interpolation]{interpolation inequality}
    \begin{equation}
        \label{eq:interpolation}
        \norm{|K|^s x}_X \leq \norm{|K|^{r}x}_X^{\frac{s}{r}} \ \norm{x}_X^{1-\frac{s}{r}}.
    \end{equation}
\end{lemma}
\begin{proof}
    By definition of $|K|^s$,
    \begin{equation}\label{eq:interpolation1}
        \norm{|K|^s x}_X^2 = \sum_{n\in\N} \sigma_n^{2s} |\inner{x}{v_n}_X|^2,
    \end{equation}
    which together with the Bessel inequality immediately yields the claim for $s=0$.

    For $s>0$, we apply the Hölder inequality
    \begin{equation*}
        \sum_{n\in\N} a_n b_n \leq \left(\sum_{n\in\N} a_n^p\right)^{\frac1p}\left(\sum_{n\in\N} b_n^q\right)^{\frac1q}
        \qquad\text{for}\quad\frac1p+\frac1q = 1
    \end{equation*}
    to
    \begin{equation*}
        a_n := \sigma_n^{2s} |\inner{x}{v_n}_X|^{2\frac{s}r},\qquad
        b_n := |\inner{x}{v_n}_X|^{2-2\frac{s}r}, \qquad p = \frac{r}{s},\qquad q = \frac{r}{r-s}.
    \end{equation*}
    Then, \eqref{eq:interpolation1} and the Bessel inequality yield
    \begin{equation*}
        \begin{aligned}
            \norm{|K|^s x}_X^2 &\leq
            \left(\sum_{n\in\N} \sigma_n^{2r} |\inner{x}{v_n}_X|^{2}\right)^{\frac{s}{r}}
            \left(\sum_{n\in\N} |\inner{x}{v_n}_X|^{2}\right)^{1-\frac{s}{r}}\\
            &\leq \norm{|K|^rx}_X^{2\frac{s}{r}}\ \norm{x}_X^{2(1-\frac{s}{r})},
        \end{aligned}
    \end{equation*}
    and the claim follows after taking the square root.
\end{proof}

\chapter{Regularization methods}\label{chap:regularization}

As shown in the last chapter, the ill-posed operator equation $Tx=y$ admits for any $y\in\calD(T^\dag)$ a unique minimum norm solution $x^\dag = T^\dag y$. In practice, one however usually does not have access to the \enquote{exact data} $y$ but only to a \enquote{noisy measurement} $y^\delta\in B_\delta(y)$, i.e., satisfying
\begin{equation*}
    \norm{y-y^\delta}_Y \leq \delta,
\end{equation*}
where $\delta>0$ is the \term[noise!level]{noise level}. Since $T^\dag$ is not continuous in general, $T^\dag y^\delta$ is not guaranteed to be a good approximation of $x^\dag$ even for $y^\delta\in\calD(T^\dag)$.
The goal is therefore to construct an approximation $x^\delta_\alpha$ that on the one hand depends continuously on $y^\delta$ -- and thus on $\delta$ -- and on the other hand can through the choice of a \term[parameter, regularization]{regularization parameter} $\alpha>0$ be brought as close to $x^\dag$ as the noise level $\delta$ allows. In particular, for $\delta \to 0$ and an appropriate choice of $\alpha(\delta)$, we want to ensure that $x^\delta_{\alpha(\delta)}\to x^\dag$. A method which constructs such an approximation is called a \term{regularization method}.

\section{Regularization and parameter choice}

For linear operators between Hilbert spaces, such constructions can be defined through \term[operator!regularization]{regularization operators}, which can be considered as a continuous replacement for the unbounded pseudoinverse $T^\dag$. This leads to the following definition.
\begin{definition}\label{def:regularisierung}
    Let $T\in\linop(X,Y)$ be a bounded linear operator between the Hilbert spaces $X$ and $Y$.
    A family $\{R_\alpha\}_{\alpha>0}$ of linear operators $R_\alpha:Y\to X$ is called a \term{regularization} (of $T^\dag$) if
    \begin{enumerate}
        \item $R_\alpha\in\linop(Y,X)$ for all $\alpha>0$;
        \item $R_\alpha y\to T^\dag y$ as $\alpha\to 0$ and all $y\in\calD(T^\dag)$.
    \end{enumerate}
\end{definition}
A regularization is therefore a pointwise approximation of the Moore--Penrose inverse by continuous operators. However, the Banach--Steinhaus theorem implies that the convergence cannot be uniform if $T^\dag$ is not continuous.
\begin{theorem}\label{thm:regularisierung:glm}
    Let $T\in\linop(X,Y)$ and $\{R_\alpha\}_{\alpha>0}\subset\linop(Y,X)$ be a regularization. If $T^\dag$ is not continuous, then $\{R_\alpha\}_{\alpha>0}$ is not uniformly bounded. In particular, then there exists a $y\in Y$ and a null sequence $\{\alpha_n\}_{n\in\N}$ with $\norm{R_{\alpha_n} y}_X\to \infty$.
\end{theorem}
\begin{proof}
    Assume to the contrary that no such $y\in Y$ exists. Then the family $\{R_\alpha\}_{\alpha>0}\subset\linop(Y,X)$ is bounded pointwise and hence uniformly by the Banach--Steinhaus \cref{thm:banach-steinhaus}. Thus there exists an $M>0$ with $\|R_\alpha\|_{\linop(Y,X)} \leq M$ for all $\alpha>0$. Together with the pointwise convergence $R_\alpha\to T^\dag$ on the dense subset $\calD(T^\dag)\subset Y$, \cref{cor:banach-steinhaus1} yields convergence on all of $\overline{\calD(T^\dag)} = Y$. By \cref{cor:banach-steinhaus2}, $T^\dag$ is then continuous, and the claim follows by contraposition.
\end{proof}
In fact, under an additional assumption, $R_{\alpha_n}y$ has to diverge for \emph{all} $y\notin\calD(T^\dag)$.
\begin{theorem}\label{thm:regularisierung:div}
    Let $T\in\linop(X,Y)$ be such that $T^\dag$ is not continuous, and let $\{R_\alpha\}_{\alpha>0}\subset\linop(Y,X)$ be a regularization of $T^\dag$. If
    \begin{equation}\label{eq:regularisierung:bounded}
        \sup_{\alpha>0} \norm{TR_\alpha}_{\linop(Y,Y)}<\infty,
    \end{equation}
    then $\norm{R_\alpha y}_X\to \infty$ as $\alpha\to0$ and all $y\notin\calD(T^\dag)$.
\end{theorem}
\begin{proof}
    Let $y\in Y\setminus\calD(T^\dag)=\overline{\calR(T)}\setminus\calR(T)$ be arbitrary and assume that there exists a null sequence $\{\alpha_n\}_{n\in\N}$ for which $\{R_{\alpha_n}y\}_{n\in\N}$ is bounded. Then there exists a subsequence $\{x_k\}_{k\in\N}$, $x_k:= R_{\alpha_{n_k}}y$, with $x_k\wkto x\in X$. Since bounded linear operators are weakly continuous, this also yields that $Tx_k\wkto Tx$.

    On the other hand, the continuity of $T$ and the pointwise convergence $R_\alpha\to T^\dag$ on $\calD(T^\dag)$ imply together with \cref{lem:inverse}\,(iv) that $TR_\alpha \tilde y \to TT^\dag \tilde y = P_{\overline{\calR}}\tilde y$ for all $\tilde y\in\calD(T^\dag)$.
    The assumption \eqref{eq:regularisierung:bounded} and \cref{cor:banach-steinhaus1} then yield the pointwise convergence of $TR_{\alpha_n}\to P_{\overline \calR}$ on all of $Y$. From $Tx_k =T R_{\alpha_{n_k}}y \to  P_{\overline{\calR}}y$ and $Tx_k\wkto Tx$, it now follows by the uniqueness of the limit that $Tx =  P_{\overline{\calR}}y$. Hence $P_{\overline{\calR}}y\in\calR(T)$ and therefore $y\in\calR(T)$, in contradiction to the assumption that $y\notin \calR(T)$. Hence $\{R_{\alpha_n}y\}_{n\in\N}$ cannot be bounded.
\end{proof}

\bigskip

However, we can in general not assume that a given noisy measurement $y^\delta\in B_\delta(y)$ is an element of $\calD(T^\dag)$.
We therefore have to consider the \term[error!regularization]{regularization error}
\begin{equation}
    \label{eq:regularisierung:splitting}
    \begin{aligned}[t]
        \norm{R_{\alpha}y^\delta - T^\dag y}_X &\leq  \norm{R_{\alpha}y^\delta - R_{\alpha}y}_X +  \norm{R_{\alpha}y - T^\dag y}_X\\
        &\leq \delta \norm{R_{\alpha}}_{\linop(Y,X)} + \norm{R_{\alpha}y - T^\dag y}_X.
    \end{aligned}
\end{equation}
This decomposition is a fundamental tool of regularization theory, and we will meet it repeatedly throughout the following.
Here the first term describes the \term[error!data]{(propagated) data error}, which by \cref{thm:regularisierung:glm} cannot be bounded for $\alpha\to 0$ as long as $\delta>0$. The second term describes the \term[error!approximation]{approximation error}, which due to the assumed pointwise convergence for $\alpha\to 0$ does tend to zero. To obtain a reasonable approximation, we thus have to choose $\alpha$ in a suitable dependence on $\delta$ such that the total regularization error vanishes as $\delta\to 0$.
\begin{definition}\label{def:parameterwahl}
    A function $\alpha:\R^+ \times Y\to\R^+$, $(\delta,y^\delta)\mapsto \alpha(\delta,y^\delta)$,
    is called a \term{parameter choice rule}.
    We distinguish
    \begin{enumerate}
        \item \term[parameter choice rule!a priori]{a priori choice rules} that only depend on $\delta$;
        \item \term[parameter choice rule!a posteriori]{a posteriori choice rules} that depend on $\delta$ and $y^\delta$;
        \item \term[parameter choice rule!heuristic]{heuristic choice rules} that only depend on $y^\delta$.
    \end{enumerate}

    If $\{R_\alpha\}_{\alpha>0}$ is a regularization of $T^\dag$ and $\alpha$ is a parameter choice rule, the pair $(R_\alpha,\alpha)$ is called a \term[regularization method!convergent]{(convergent) regularization method} if
    \begin{equation}\label{eq:regularisierung:paramwahl}
        \lim_{\delta\to 0}\sup_{y^\delta\in B_\delta(y)}\norm{R_{\alpha(\delta,y^\delta)}y^\delta - T^\dag y}_X = 0 \qquad\text{for all }y\in \calD(T^\dag).
    \end{equation}
\end{definition}
We thus demand that the regularization error vanishes for \emph{all} noisy measurements $y^\delta$ that are compatible with the noise level $\delta\to 0$.

For a posteriori choice rules, it will be convenient to also allow the value $\alpha(\delta,y^\delta)=\infty$, together with the convention $R_\infty := 0$; this will be used for the discrepancy principle below.

\subsection*{A priori choice rules}

We first characterize a priori choice rules that lead to convergent regularization methods.
\begin{theorem}\label{thm:regularisierung:apriori_char}
    Let $\{R_\alpha\}_{\alpha>0}$ be a regularization and $\alpha:\R^+\to\R^+$ an a priori choice rule. If
    \begin{enumerate}
        \item $\displaystyle\lim_{\delta\to 0} \alpha(\delta) = 0$,
        \item $\displaystyle \lim_{\delta\to 0} \delta\norm{R_{\alpha(\delta)}}_{\linop(Y,X)} = 0$,
    \end{enumerate}
    then $(R_\alpha,\alpha)$ is a regularization method. Conversely, if $(R_\alpha,\alpha)$ is a regularization method satisfying (i), then (ii) holds.
\end{theorem}
\begin{proof}
    The decomposition \eqref{eq:regularisierung:splitting} of the regularization error immediately implies that
    \begin{equation*}
        \norm{R_{\alpha(\delta)}y^\delta - T^\dag y}_X \leq \delta \norm{R_{\alpha(\delta)}}_{\linop(Y,X)} + \norm{R_{\alpha(\delta)}y - T^\dag y}_X\to 0\qquad\text{for }\delta\to 0
    \end{equation*}
    since the first term vanishes by assumption (ii), while the second vanishes due to the pointwise convergence of regularization operators together with assumption (i).

    Conversely, assume that (i) holds but (ii) is violated. Then there exists a null sequence $\{\delta_n\}_{n\in\N}$ with $\delta_n\norm{R_{\alpha(\delta_n)}}_{\linop(Y,X)}\geq \eps$ for some $\eps>0$.
    We can therefore find a sequence $\{z_n\}_{n\in\N}\subset Y$ with $\norm{z_n}_Y=1$ and  $\delta_n\norm{R_{\alpha(\delta_n)}z_n}_{X}\geq \eps$.
    Let now $y\in\calD(T^\dag)$ be arbitrary and set $y_n:= y+\delta_n z_n$. Then $y_n\in B_{\delta_n}(y)$, but
    \begin{equation*}
        R_{\alpha(\delta_n)} y_n - T^\dag y =  (R_{\alpha(\delta_n)} y - T^\dag y) + \delta_n R_{\alpha(\delta_n)} z_n \not\to 0
    \end{equation*}
    since the first term on the right-hand side is a null sequence by (i) and the pointwise convergence of $R_\alpha$, but the second term is not a null sequence by construction.
    Hence, \eqref{eq:regularisierung:paramwahl} is violated and $(R_\alpha,\alpha)$ therefore not a regularization method. The claim now follows by contraposition.
\end{proof}
Since $\norm{R_\alpha}_{\linop(Y,X)}\to\infty$ as $\alpha\to 0$, assumption (ii) states that $\alpha$ cannot tend to zero too fast compared to $\delta$. An a priori choice rule thus usually has the form $\alpha(\delta) = \delta^r$ for some $r\in(0,1)$ (with $r$ depending on, among others, the specific regularization $\{R_\alpha\}_{\alpha>0}$).
In particular, every regularization admits an a priori choice rule and hence leads to a convergent regularization method.
\begin{corollary}\label{thm:regularisierung:apriori}
    Let $\{R_\alpha\}_{\alpha>0}$ be a regularization of $T^\dag$. Then there exists an a priori choice rule $\alpha$ such that $(R_\alpha,\alpha)$ is a regularization method.
\end{corollary}
\begin{proof}
    Set $\alpha_n:=\frac1n$ and choose a strictly decreasing null sequence $\{\delta_n\}_{n\in\N}\subset\R^+$ with $\delta_n\norm{R_{\alpha_n}}_{\linop(Y,X)}\leq \frac1n$ for all $n\in\N$, e.g., $\delta_1 := (1+\norm{R_{\alpha_1}}_{\linop(Y,X)})^{-1}$ and
    \begin{equation*}
        \delta_n := \min\left\{\tfrac12\delta_{n-1},\left(n\left(1+\norm{R_{\alpha_n}}_{\linop(Y,X)}\right)\right)^{-1}\right\}\qquad\text{for }n\geq 2,
    \end{equation*}
    which is possible since $R_{\alpha_n}\in\linop(Y,X)$. We now define the a priori choice rule
    \begin{equation*}
        \alpha:\R^+\to\R^+,\qquad \alpha(\delta) := \begin{cases} \alpha_n & \text{if }\delta\in[\delta_{n+1},\delta_n) \text{ for some }n\in\N,\\ \alpha_1 & \text{if }\delta\geq \delta_1.\end{cases}
    \end{equation*}
    Since $\{\delta_n\}_{n\in\N}$ is a strictly decreasing null sequence, the intervals $[\delta_{n+1},\delta_n)$ partition $(0,\delta_1)$, i.e., every $\delta\in(0,\delta_1)$ lies in exactly one such interval with index $n=:n(\delta)$, where $n(\delta)\geq N$ whenever $\delta<\delta_N$. Hence as $\delta \to 0$, we have $n(\delta)\to\infty$ as $\delta\to 0$ and therefore
    \begin{equation*}
        \alpha(\delta) = \frac1{n(\delta)}\to 0
        \qquad\text{and}\qquad
        \delta\norm{R_{\alpha(\delta)}}_{\linop(Y,X)} \leq \delta_{n(\delta)}\norm{R_{\alpha_{n(\delta)}}}_{\linop(Y,X)}\leq\frac1{n(\delta)}\to 0.
    \end{equation*}
    The claim now follows from \cref{thm:regularisierung:apriori_char}.
\end{proof}

\subsection*{A posteriori choice rules}

As we will see later, the optimal choice of $\alpha(\delta)$ requires information about the exact (minimum norm) solution $x^\dag$ that is not easily accessible. Such information is not required for a posteriori choice rules. The main idea behind these is the following: Let again $y\in\calD(T^\dag)$ and $y^\delta\in B_\delta(y)$ and consider the \term{residual}
\begin{equation*}
    \norm{TR_\alpha y^\delta-y^\delta}_Y.
\end{equation*}
If now $y\in\calR(T)$ and $\norm{y-y^\delta}_Y=\delta$, even the (desired) minimum norm solution $x^\dag$ satisfies due to $Tx^\dag = y$ only
\begin{equation*}
    \norm{Tx^\dag-y^\delta}_Y = \norm{y-y^\delta}_Y = \delta.
\end{equation*}
It is therefore not reasonable to try to obtain a smaller residual for the regularization $R_\alpha y^\delta$ either. This motivates the \term[discrepancy principle]{Morozov discrepancy principle}\index{parameter choice rule!discrepancy principle}: For given $\delta>0$ and $y^\delta\in B_\delta(y)$ choose $\alpha = \alpha(\delta,y^\delta)$ (as large as possible) such that
\begin{equation}
    \label{eq:parameter:morozov}
    \norm{TR_\alpha y^\delta-y^\delta}_Y\leq \tau \delta \qquad\text{for some } \tau>1 \text{ independent of $\delta$ and $y^\delta$}.
\end{equation}
If $\norm{y^\delta}_Y\leq\tau\delta$, we set
\begin{equation}\label{eq:parameter:morozov_infty}
    \alpha(\delta,y^\delta) := \infty\qquad\text{and}\qquad R_{\alpha(\delta,y^\delta)}y^\delta = R_\infty y^\delta := 0
\end{equation}
(i.e., the data is indistinguishable from noise, and the best reconstruction is no reconstruction).

However, this principle may not be satisfiable: If $y\in \calR(T)^\bot\setminus\{0\}$, then even the exact data $y^\delta = y$ and the minimum norm solution $x^\dag$ only satisfy
\begin{equation*}
    \norm{Tx^\dag - y}_Y =\norm{TT^\dag y - y}_Y = \norm{P_{\overline{\calR}} y - y}_Y = \norm{y}_Y > \tau\delta
\end{equation*}
for some fixed $\tau>1$ and $\delta$ small enough. We therefore have to assume that this situation cannot occur; for this it is sufficient that $\calR(T)$ is dense in $Y$ (since in this case $\calR(T)^\bot = \overline{\calR(T)}{}^\bot=\{0\}$).

The practical realization usually consists in first checking whether $\norm{y^\delta}_Y\leq \tau\delta$ (in which case we take $R_\infty y^\delta =0$) and otherwise choosing a null sequence $\{\alpha_n\}_{n\in \N}$, computing  successively $R_{\alpha_n}y^\delta$ for $n=1,\dots$, and stopping as soon as the discrepancy principle \eqref{eq:parameter:morozov} is satisfied for an $\alpha_{n^*}$. The following theorem justifies this procedure.
\begin{theorem}\label{thm:morozov}
    Let $\{R_\alpha\}_{\alpha>0}$ be a regularization of $T^\dag$ with $\calR(T)$ dense in $Y$, $\{\alpha_n\}_{n\in \N}$ be a strictly decreasing null sequence, and $\tau>1$. If the family $\{TR_\alpha\}_{\alpha>0}$ is uniformly bounded, then for all $y\in \calD(T^\dag)$, $\delta>0$ and $y^\delta \in B_\delta(y)$ there exists an $n^*\in \N$ such that
    \begin{equation}
        \label{eq:parameter:morozov_seq}
        \norm{TR_{\alpha_{n^*}}y^\delta-y^\delta}_Y\leq \tau \delta < \norm{TR_{\alpha_n}y^\delta-y^\delta}_Y \qquad\text{for all } n<n^*.
    \end{equation}
\end{theorem}
\begin{proof}
    We proceed as in the proof of \cref{thm:regularisierung:div}. The family $\{TR_\alpha\}_{\alpha>0}$ converges pointwise to $TT^\dag = P_{\overline{\calR}}$ on $\calD(T^\dag)$ and hence, due to the uniform boundedness, on all of $Y=\overline{\calD(T^\dag)}$. This implies that for all $y\in\calD(T^\dag)=\calR(T)$ and $y^\delta \in B_\delta(y)$,
    \begin{equation*}
        \lim_{n\to\infty} \norm{TR_{\alpha_{n}}y^\delta-y^\delta}_Y =
        \norm{P_{\overline{\calR}}y^\delta-y^\delta}_Y =0
    \end{equation*}
    since $\overline{\calR(T)}=Y$. From this, the claim follows.
\end{proof}

To show that the discrepancy principle indeed leads to a regularization method, it has to be considered in combination with a concrete regularization. We will do so in the following chapters.

\subsection*{Heuristic choice rules}

Heuristic choice rules do not need knowledge of the noise level $\delta$, which is often relevant in practice where this knowledge is not available (sufficiently exactly).
However, the following pivotal result -- known in the literature as the \term[Bakushinskii veto@Bakushinski\u{\i} veto]{Bakushinski\u{\i} veto}, see \cite{Bakushinskii} -- states that this is not possible in general.
\begin{theorem}\label{thm:parameter:bakushinskii}
    Let $\{R_\alpha\}_{\alpha>0}$ be a regularization of $T^\dag$. If there exists a heuristic choice rule $\alpha$ such that $(R_\alpha,\alpha)$ is a regularization method, then $T^\dag$ is continuous.
\end{theorem}
\begin{proof}
    Assuming to the contrary that such a parameter choice rule $\alpha:Y\to\R^+$ exists, we can define the (possibly nonlinear) mapping
    \begin{equation*}
        R:Y\to X,\qquad y\mapsto R_{\alpha(y)} y.
    \end{equation*}
    Let now $y\in\calD(T^\dag)$ be arbitrary and consider any sequence $\{y_n\}_{n\in\N}\subset \calD(T^\dag)$ with $y_n\to y$.
    On the one hand, then naturally $y_n\in B_{\delta}(y_n)$ for all $\delta>0$ and $n\in \N$, and the assumption \eqref{eq:regularisierung:paramwahl} for fixed $y^\delta = y=y_n$ and $\delta\to 0$ yields that $Ry_n = T^\dag y_n$ for all $n\in \N$ (and hence that $R$ is in fact linear on $\calD(T^\dag)$). On the other hand, for $\delta_n:= \norm{y_n-y}_Y$ we also have $y_n\in B_{\delta_n}(y)$, and in this case passing to the limit $n\to\infty$ in \eqref{eq:regularisierung:paramwahl} shows that
    \begin{equation*}
        T^\dag y_n =  Ry_n = R_{\alpha(y_n)}y_n \to T^\dag y,
    \end{equation*}
    i.e., $T^\dag$ is continuous on $\calD(T^\dag)$.
\end{proof}
In particular for compact operators with infinite-dimensional range, \emph{no} heuristic choice rule can lead to a regularization method.
Of course, this does not mean that such methods cannot be used in practice. First, the veto does not rule out choice rules for finite-dimensional ill-posed problems (such as very ill-conditioned linear systems); however, these rules are then by necessity dimension-dependent.
Second, a sharp look at the proof shows that the crucial step consists in applying the choice rule to data $y^\delta\in\calD(T^\dag)$. The worst case for the noisy data is therefore $y^\delta\in\calR(T)$ (since only this subspace of $\calD(T^\dag)$ plays a role due to $\calR(T)^\bot = \calN(T^\dag)$), and in this case convergence cannot be guaranteed.
In many interesting cases, however, $T$ is a compact (i.e., smoothing) operator, while errors have a more random character and therefore do not typically lie in $\calR(T)$. Heuristic choice rules can therefore indeed work in \enquote{usual} situations. In fact, it is possible to show under the additional assumption $y^\delta \notin\calD(T^\dag)$ that a whole class of popular heuristic choice rules leads to a regularization method. Here, too, we need to consider the combination with a concrete regularization operator but already give some examples.
\begin{enumerate}
    \item The \term[parameter choice rule!quasi-optimality principle]{quasi-optimality principle} picks a finite strictly decreasing sequence $\{\alpha_n\}_{n\in\{1,\dots,N\}}$ and chooses $\alpha(y^\delta)=\alpha_{n^*}$ as the one satisfying
        \begin{equation*}
            {n^*}  \in \arg\min_{1\leq n< N}\norm{R_{\alpha_{n+1}}y^\delta - R_{\alpha_{n}}y^\delta}_X.
        \end{equation*}
    \item The \term[parameter choice rule!Hanke--Raus rule]{Hanke--Raus rule} chooses
        \begin{equation*}
            \alpha(y^\delta)  \in \arg\min_{\alpha>0} \frac1{\sqrt\alpha} \norm{TR_\alpha y^\delta - y^\delta}_Y.
        \end{equation*}
    \item The \term[parameter choice rule!L-curve criterion]{L-curve criterion}\footnote{The name is due to the practical realization: If one plots the curve $\alpha\mapsto (\norm{TR_\alpha y^\delta - y^\delta}_Y,\norm{R_\alpha y^\delta}_X)$ (or, rather, a finite set of points on it) in a doubly logarithmic scale, it often has -- more or less -- the form of an ``L''; the chosen parameter is then the one lying closest to the \enquote{knee} of the L.}
        chooses
        \begin{equation*}
            \alpha(y^\delta)  \in \arg\min_{\alpha>0} \norm{R_\alpha y^\delta}_X \norm{TR_\alpha y^\delta - y^\delta}_Y.
        \end{equation*}
\end{enumerate}

All of these methods in one way or another work by using the residual to obtain a reasonably close approximation of the noise level that is then used similarly as in a priori or a posteriori choice rules. An extensive numerical comparison of these and other choice rules can be found in \cite{Bauer}.

\section{Convergence rates}

A central goal in the regularization of inverse problems is to obtain error estimates of the form
\begin{equation*}
    \norm{R_{\alpha(\delta,y^\delta)} y^\delta - T^\dag y}_X \leq \psi(\delta)
\end{equation*}
for an increasing function $\psi:\R^+\to\R^+$ with $\lim_{t\to 0}\psi(t) =0$. In particular, we are interested in the \term[error!worst-case]{worst-case error}
\begin{equation}
    \label{eq:rates:worstcase}
    \calE(y,\delta):= \sup_{y^\delta\in B_\delta(y)} \norm{R_{\alpha(\delta,y^\delta)} y^\delta - T^\dag y}_X
\end{equation}
(which for regularization methods converges to zero as $\delta\to 0$ and any $y\in \calD(T^\dag)$ by \eqref{eq:regularisierung:paramwahl}). Here, $\psi$ has to depend in some form on $y$ since otherwise it would be possible to give regularization error estimates independently of $y$ and $y^\delta$ -- but since the convergence of $R_\alpha\to T^\dag$ is merely pointwise but not uniform, such estimates cannot be expected.
\begin{theorem}\label{thm:rates:counter}
    Let $(R_\alpha,\alpha)$ be a regularization method. If there exists a $\psi:\R^+\to\R^+$ with $\lim_{t\to 0}\psi(t) =0$ and
    \begin{equation}\label{eq:rates:bounded}
        \sup_{y\in \calD(T^\dag)\cap B_Y} \calE(y,\delta) \leq \psi(\delta),
    \end{equation}
    then $T^\dag$ is continuous.
\end{theorem}
\begin{proof}
    Let $y\in \calD(T^\dag)\cap B_Y$ and $\{y_n\}_{n\in\N}\subset \calD(T^\dag) \cap B_Y$ be a sequence with $y_n\to y$. Setting $\delta_n:= \norm{y-y_n}_Y\to 0$, we then have for $n\to\infty$ that
    \begin{equation*}
        \begin{aligned}
            \norm{T^\dag y_n - T^\dag y}_X &\leq \norm{T^\dag y_n - R_{\alpha(\delta_n,y_n)} y_n}_X + \norm{R_{\alpha(\delta_n,y_n)} y_n - T^\dag y}_X\\
            &\leq \calE(y_n,\delta_n) + \calE(y,\delta_n) \\
            &\leq 2\psi(\delta_n)\to 0.
        \end{aligned}
    \end{equation*}
    Hence $T^\dag$ is continuous on $\calD(T^\dag)\cap B_Y$ and thus, by linearity of $T^\dag$, on all of  $\calD(T^\dag)$.
\end{proof}
This implies that the convergence can be arbitrarily slow; knowledge of $\delta$ alone is therefore not sufficient to give error estimates -- we thus need additional assumptions on the exact data $y$ or, equivalently, the wanted minimum norm solution $x^\dag = T^\dag y$. As the proof of \cref{thm:rates:counter} shows, the existence of convergence rates is closely tied to the continuity of $T^\dag$ on closed subsets. We therefore consider for $\calM\subset X$ and $\delta>0$ the quantity
\begin{equation*}
    \eps(\calM,\delta) :=  \sup\setof{\norm{x}_X}{x\in\calM,\ \norm{Tx}_Y\leq \delta},
\end{equation*}
which can be interpreted as a \term{modulus of conditional continuity} of $T^\dag:\calR(T)\cap \delta B_Y\to \calM$.
This modulus is in fact a lower bound for the worst-case error. Since both $\eps(\calM,\delta)$ and  $\calE(y,\delta)$ are not finite if $\calM\cap\calN(T)\neq \{0\}$ and $\calM$ is unbounded, we will only consider the more interesting case that $\calM\subset\calN(T)^\bot$.
\begin{theorem}
    Let $(R_\alpha,\alpha)$ be a regularization method. Then for all $\delta>0$ and $\calM\subset \calN(T)^\bot$,
    \begin{equation*}
        \sup_{y\in\calD(T^\dag), T^\dag y\in \calM} \calE(y,\delta) \geq \eps(\calM,\delta).
    \end{equation*}
\end{theorem}
\begin{proof}
    Let $x\in\calM$ with $\norm{Tx}_Y \leq \delta$. For $y^\delta = 0$, we then deduce from $x\in\calN(T)^\bot$ that
    \begin{equation*}
        \norm{x}_X = \norm{T^\dag Tx - R_{\alpha(\delta,0)} 0}_X \leq \calE(Tx,\delta)
    \end{equation*}
    and hence
    \begin{equation*}
        \eps(\calM,\delta) = \sup_{x\in\calM,  \norm{Tx}_Y\leq \delta} \norm{x}_X \leq  \sup_{x\in\calM,  \norm{Tx}_Y\leq \delta} \calE(Tx,\delta) \leq \sup_{T^\dag y\in\calM,  y\in\calD(T^\dag)} \calE(y,\delta)
    \end{equation*}
    since $\calD(T^\dag)=\calR(T)\oplus\calR(T)^\bot$ and $\calR(T)^\bot = \calN(T^\dag)$.
\end{proof}

For an appropriate choice of $\calM$, we can now derive sharp bounds on $\eps(\calM,\delta)$. We consider here for compact operators $K\in\calK(X,Y)$ subsets of the form
\begin{equation*}
    X_{\nu,\rho} = \setof{|K|^\nu w \in X}{\norm{w}_X\leq \rho} \subset\calR(|K|^\nu).
\end{equation*}
The definition of $|K|^\nu w$ via the spectral decomposition of $K$ implies in particular that $X_{\nu,\rho}\subset \overline{\calR(K^*)} = \calN(K)^\bot$.
\begin{theorem}\label{thm:rates:limit}
    Let $K\in\calK(X,Y)$ and $\nu,\rho>0$. Then for all $\delta>0$,
    \begin{equation*}
        \eps(X_{\nu,\rho},\delta)\leq \delta^{\frac\nu{\nu+1}} \rho^{\frac1{\nu+1}}.
    \end{equation*}
\end{theorem}
\begin{proof}
    Let $x\in X_{\nu,\rho}$ and $\norm{Kx}_Y\leq \delta$. Then there exists a $w\in X$ with $x=|K|^\nu w$ and $\norm{w}_X\leq \rho$. The interpolation inequality from \cref{lem:interpolation} for $s = \nu$ and $r =\nu + 1$ together with the properties from \cref{lem:functional_range} then imply that
    \begin{equation*}
        \begin{aligned}
            \norm{x}_X &= \norm{|K|^\nu w}_X \leq \norm{|K|^{\nu+1}w}_X^{\frac{\nu}{\nu+1}} \norm{w}_X^{\frac{1}{\nu+1}} = \norm{K|K|^{\nu}w}_Y^{\frac{\nu}{\nu+1}} \norm{w}_X^{\frac{1}{\nu+1}} \\
            &= \norm{Kx}_Y^{\frac{\nu}{\nu+1}} \norm{w}_X^{\frac{1}{\nu+1}}
            \leq \delta^{\frac{\nu}{\nu+1}} \rho^{\frac{1}{\nu+1}}.
        \end{aligned}
    \end{equation*}
    Taking the supremum over all $x\in X_{\nu,\rho}$ with $\norm{Kx}_Y\leq \delta$ yields the claim.
\end{proof}
So far this is only an upper bound, but there always exists at least one sequence for which it is attained.
\begin{theorem}
    Let $K\in\calK(X,Y)$ and $\nu,\rho>0$. Then there exists a null sequence $\{\delta_n\}_{n\in\N}$ with
    \begin{equation*}
        \eps(X_{\nu,\rho},\delta_n) = \delta_n^{\frac\nu{\nu+1}} \rho^{\frac1{\nu+1}}.
    \end{equation*}
\end{theorem}
\begin{proof}
    Let $\{(\sigma_n,u_n,v_n)\}_{n\in\N}$ be a singular system for $K$ and set $\delta_n :=  \rho \sigma_n^{\nu+1}$ as well as $x_n:= |K|^\nu (\rho v_n)$. Since singular values form a null sequence, we have $\delta_n\to 0$. Furthermore, by construction $x_n\in X_{\nu,\rho}$.
    It now follows from $\sigma_n = (\rho^{-1}\delta_n)^{\frac1{\nu+1}}$ that
    \begin{equation*}
        x_n = \rho |K|^\nu v_n  = \rho \sigma_n^\nu v_n =  \delta_n^{\frac\nu{\nu+1}}\rho^{\frac{1}{\nu+1}}v_n
    \end{equation*}
    since $\sigma^\nu_n$ is an eigenvalue of $|K|^\nu$ corresponding to the eigenvector $v_n$. Hence, $\norm{x_n}_X =  \delta_n^{\frac\nu{\nu+1}}\rho^{\frac{1}{\nu+1}}$.
    Analogously, we obtain that
    \begin{equation*}
        K^*K x_n = \delta_n^{\frac\nu{\nu+1}}\rho^{\frac{1}{\nu+1}} \sigma_n^2 v_n = \delta_n^{\frac{\nu+2}{\nu+1}} \rho^{-\frac{1}{\nu+1}} v_n
    \end{equation*}
    and thus that
    \begin{equation*}
        \norm{Kx_n}_Y^2 = \inner{Kx_n}{Kx_n}_Y = \inner{K^*Kx_n}{x_n}_X = \delta_n^{2}.
    \end{equation*}
    For all $n\in\N$, we therefore have that
    \begin{equation*}
        \eps(X_{\nu,\rho},\delta_n) = \sup_{x\in X_{\nu,\rho},\ \norm{Kx}_Y\leq \delta_n}\norm{x}_X \geq \norm{x_n}_X = \delta_n^{\frac\nu{\nu+1}} \rho^{\frac1{\nu+1}},
    \end{equation*}
    which together with \cref{thm:rates:limit} yields the claimed equality.
\end{proof}

This theorem implies that for a compact operator $K$ with infinite-dimensional range, there can be no regularization method for which the worst-case error can go to zero faster than $\delta^{\frac\nu{\nu+1}} \rho^{\frac1{\nu+1}}$ as $\delta\to 0$ -- and even this is only possible under the additional assumption that $x^\dag\in  X_{\nu,\rho}$.
In particular, the regularization error always tends to zero more slowly than the data error.

We thus call a regularization method \term[regularization method!optimal]{optimal} (for $\nu$ and $\rho$) if
\begin{equation*}
    \calE(Kx^\dag,\delta) = \delta^{\frac\nu{\nu+1}} \rho^{\frac1{\nu+1}}\qquad\text{for all }x^\dag \in X_{\nu,\rho}
\end{equation*}
and \term{order optimal} (for $\nu$ and $\rho$) if there exists a constant $c\geq 1$ independent of $\delta$ and $x^\dag$ such that
\begin{equation}\label{eq:regularisierung:ordnung}
    \calE(Kx^\dag,\delta) \leq c \delta^{\frac\nu{\nu+1}} \rho^{\frac1{\nu+1}}\qquad\text{for all }\delta>0 \text{ and } x^\dag \in X_{\nu,\rho}.
\end{equation}
The constant $c$ will usually depend on $\nu$ and on the regularization method (e.g., on constants in the definition of the parameter choice rule). In particular, an a priori choice rule that is order optimal for $\nu$ and $\rho$ will typically explicitly depend on $\rho$; see \cref{thm:spektral:apriori} below.
If we allow this constant to depend on $x^\dag$ -- i.e., we are only interested in \term[rate!convergence]{convergence rates} -- then we set
\begin{equation*}
    X_\nu :=  \bigcup_{\rho>0} X_{\nu,\rho} = \calR(|K|^\nu)
\end{equation*}
and call a regularization method order optimal for $\nu$ if there exists a $c=c(x^\dag)\geq 1$ such that
\begin{equation*}
    \calE(Kx^\dag,\delta) \leq c \delta^{\frac\nu{\nu+1}} \qquad\text{for all }x^\dag \in X_{\nu}.
\end{equation*}

The assumption $x^\dag \in X_{\nu,\rho}$ is called a \term{source condition}, and the element $w\in X$ with $|K|^\nu w = x^\dag$ is sometimes referred to as a \term[source condition!source representer]{source representer}. Since $K$ is a compact (i.e., smoothing) operator, source conditions are abstract smoothness conditions; e.g., for the integral operator  $K$ from \cref{ex:integration_svd}, the condition $x\in X_{2,\rho}$ implies that $x=K^*Kw = \int_t^1 \int_0^s w(r)\,dr\,ds$ has a second (weak) derivative $w$ whose $L^2$ norm is bounded by $\rho$.

Using the singular value decomposition of $K$, it is not hard to show that in general the condition $x^\dag \in X_\nu$ corresponds to a strengthened Picard condition, i.e., that the decay of the Fourier coefficients of $y$ in relation to the singular values of $K$ is faster the larger $\nu$ is.
\begin{lemma}\label{lem:regularisierung:quell}
    Let $K\in\calK(X,Y)$ have the singular system $\{(\sigma_n,u_n,v_n)\}_{n\in\N}$ and let $y\in\calR(K)$. Then $x^\dag = K^\dag y\in X_\nu$ if and only if
    \begin{equation}\label{eq:regularisierung:quell}
        \sum_{n\in\N} \sigma_n^{-2-2\nu} |\inner{y}{u_n}_Y|^2 < \infty.
    \end{equation}
\end{lemma}
\begin{proof}
    From the definition and the representation \eqref{eq:inverse:picard_pseudo}, it follows that  $K^\dag y\in X_\nu$ if and only if there exists a $w\in X$ with
    \begin{equation*}
        \sum_{n\in\N} \sigma_n^{-1} \inner{y}{u_n}_Y v_n = K^\dag y = |K|^\nu w = \sum_{n\in\N} \sigma_n^{\nu} \inner{w}{v_n}_X v_n.
    \end{equation*}
    Since the $v_n$ form an orthonormal system, we can equate the corresponding coefficients to obtain that
    \begin{equation}\label{eq:regularisierung:quell1}
        \sigma_n^{-1} \inner{y}{u_n}_Y =  \sigma_n^{\nu} \inner{w}{v_n}_X \qquad\text{for all }n\in\N.
    \end{equation}
    As in the proof of \cref{thm:inverse:picard}, we have that $w\in X$ if and only if $\sum_{n\in\N} |\inner{w}{v_n}_X|^2$ is finite. Inserting \eqref{eq:regularisierung:quell1} now yields \eqref{eq:regularisierung:quell}.
\end{proof}

In fact, order optimality already implies the convergence of a regularization method. This is useful since it can be easier to show optimality of a method than its regularization property
(in particular for the discrepancy principle, which motivates the slightly complicated statement of the following theorem).
\begin{theorem}\label{thm:regularisierung:ordnung}
    Let $K\in\calK(X,Y)$ with $\calR(K)$ infinite-dimensional and dense in $Y$, $\{R_\alpha\}_{\alpha>0}$ be a regularization, and $\alpha(\delta,y^\delta)$ be a parameter choice rule. If there exist $\tau_0\geq 1$ and $\nu>0$ such that for all $\tau>\tau_0$, $R_\alpha$ together with $\alpha_\tau := \alpha(\tau\delta,y^\delta)$ satisfies the condition \eqref{eq:regularisierung:ordnung} for this $\nu$ and all $\rho>0$ with a constant $c=c_\tau$ (that may depend on $\tau$ and $\nu$ but \emph{not} on $\rho$), then $(R_\alpha,\alpha_\tau)$ is a regularization method for all $\tau>\tau_0$.
\end{theorem}
\begin{proof}
    We have to show that the \emph{uniform} convergence of the worst-case error for all $x^\dag \in X_{\nu,\rho}$ implies the \emph{pointwise} convergence for all $x^\dag \in \calR(K^\dag)$. For this, we construct a suitable $x_\eps\in X_{\nu,\rho}$, insert it into the error estimate, and apply the order optimality. The constant $\tau$ will be needed to adjust the noise level -- and hence be able to apply the parameter choice rule -- for $Kx^\dag$ to $Kx_\eps$.

    Let therefore $y\in \calD(K^\dag) = \calR(K)$ by the assumed density, and $x^\dag = K^\dag y$ (so that $Kx^\dag = y$). Furthermore, let $\{(\sigma_n,u_n,v_n)\}_{n\in\N}$ be a singular system of $K$.
    For given $\eps>0$, we now choose an $N_\eps\in \N$ such that $\sigma_{N_\eps} \geq \eps > \sigma_{N_\eps+1}$ and set
    \begin{equation*}
        x_\eps := \sum_{n=1}^{N_\eps} \inner{x^\dag}{v_n}_X v_n
    \end{equation*}
    as well as
    \begin{equation*}
        \begin{aligned}[t]
            y_\eps :=  Kx_\eps &= \sum_{n=1}^{N_\eps} \inner{x^\dag}{v_n}_X K v_n
            = \sum_{n=1}^{N_\eps} \inner{x^\dag}{v_n}_X \sigma_n u_n\\
            &= \sum_{n=1}^{N_\eps} \inner{x^\dag}{K^*u_n}_X  u_n
            = \sum_{n=1}^{N_\eps} \inner{y}{u_n}_Y u_n.
        \end{aligned}
    \end{equation*}
    Since $\{u_n\}_{n\in\N}$ is an orthonormal basis of $\overline{\calR(K)}$ and $\{v_n\}_{n\in\N}$ is an orthonormal basis of $\overline{\calR(K^*)} = \calN(K)^\bot$,
    we can represent $x^\dag = K^\dag y \in \calN(K)^\bot$ and $y = K x^\dag \in\calR(K)$ as
    \begin{equation*}
        x^\dag= \sum_{n\in\N} \inner{x^\dag}{v_n}_X v_n,\qquad
        y = \sum_{n\in\N} \inner{y}{u_n}_Y u_n.
    \end{equation*}
    Hence
    \begin{equation*}
        \norm{x^\dag-x_\eps}_X^2 = \sum_{n=N_\eps+1}^\infty \abs{\inner{x^\dag}{v_n}_X}^2
    \end{equation*}
    and
    \begin{equation}\label{eq:regularisierung:ordnung1}
        \begin{aligned}[t]
            \norm{y-y_\eps}_Y^2 &= \sum_{n=N_\eps+1}^\infty \abs{\inner{y}{u_n}_Y}^2 =  \sum_{n=N_\eps+1}^\infty \sigma_n^2\abs{\inner{x^\dag}{v_n}_X}^2\\
            &\leq \eps^2\sum_{n={N_\eps}+1}^\infty \abs{\inner{x^\dag}{v_n}_X}^2
            = \eps^2 \norm{x^\dag-x_\eps}_X^2
        \end{aligned}
    \end{equation}
    by the choice of $N_\eps$. In particular, $x_\eps\to x^\dag$ and $y_\eps\to y$ as $\eps\to 0$ (monotonically).

    By construction, $y_\eps\in \calR(K)$ and $x_\eps\in\calN(K)^\bot$, and therefore $x_\eps = K^\dag y_\eps$. From \cref{lem:regularisierung:quell} we thus deduce that $x_\eps\in X_\nu$ for all $\nu>0$, since it follows from $\inner{y_\eps}{u_n}_Y = 0$ for $n>N_\eps$ that the series in \eqref{eq:regularisierung:quell} is finite. Hence there exists a $w_\eps\in X$ with $x_\eps = |K|^\nu w_\eps$, i.e.,
    \begin{equation*}
        \sum_{n=1}^{N_\eps} \inner{x^\dag}{v_n}_X v_n = x_\eps = |K|^\nu w_\eps = \sum_{n\in\N} \sigma_n^{\nu} \inner{w_\eps}{v_n}_X v_n.
    \end{equation*}
    As $\calR(K)$ is not finite-dimensional, $\sigma_n>0$ for all $n\in\N$.
    Since the $v_n$ form an orthonormal system, we thus obtain that
    \begin{equation*}
        \inner{w_\eps}{v_n}_X =
        \begin{cases}
            \sigma_n^{-\nu} \inner{x^\dag}{v_n}_X & n\leq N_\eps,\\
            0 & n>N_\eps,
        \end{cases}
    \end{equation*}
    and hence that
    \begin{equation*}
        \begin{aligned}[t]
            \norm{w_\eps}_X^2 &= \sum_{n=1}^{N_\eps} \abs{\inner{w_\eps}{v_n}_X}^2 =  \sum_{n=1}^{N_\eps} \sigma_n^{-2\nu}\abs{\inner{x^\dag}{v_n}_X}^2 \\
            &\leq \eps^{-2\nu} \sum_{n\in\N} \abs{\inner{x^\dag}{v_n}_X}^2 = \eps^{-2\nu} \norm{x^\dag}_X^2,
        \end{aligned}
    \end{equation*}
    again by the choice of $N_\eps$.
    This implies that $x_\eps\in X_{\nu,\rho}$ with $\rho = \eps^{-\nu} \norm{x^\dag}_X$.

    Let now $\tau>\tau_0\geq 1$. If $x^\dag = 0$, then $x^\dag\in X_{\nu,\rho}$ for every $\rho>0$, and the order optimality \eqref{eq:regularisierung:ordnung} of $(R_\alpha,\alpha_\tau)$ for $\rho=1$ directly yields that $\calE(y,\delta) \leq c_\tau \delta^{\frac\nu{\nu+1}} \to 0$ as $\delta\to 0$. We can thus assume that $x^\dag\neq 0$ and hence that $y = Kx^\dag\neq 0$ since $x^\dag\in\calN(K)^\bot$ (and that $\rho>0$ above).

    Set $q:= \frac{\tau -\tau_0}{\tau + \tau_0}\in (0,1)$ and let $0<\delta<q^{-1}\norm{y}_Y$ and $y^\delta \in B_\delta(y)$. We now choose
    \begin{equation*}
        \eps(\delta) := \inf \setof{\eps>0}{\norm{y-y_\eps}_Y \geq q \delta}.
    \end{equation*}
    This set is nonempty since it contains all $\eps>\sigma_1$ (for which $y_\eps = 0$), and $\eps(\delta)>0$ since $\norm{y-y_\eps}_Y\to 0$ as $\eps\to 0$.
    By definition and by the left-continuity (by the choice of $N_\eps$) and monotonicity of $\eps \mapsto \norm{y-y_\eps}_Y$, we then have in particular that
    \begin{equation}
        \label{eq:regularisierung:ordnung3}
        \norm{y-y_{\eps(\delta)}}_Y \leq q \delta \leq \norm{y-y_{2\eps(\delta)}}_Y
    \end{equation}
    and hence that
    \begin{equation*}
        \norm{y^\delta -y_{\eps(\delta)}}_Y \leq \norm{y^\delta - y}_Y + \norm{y-y_{\eps(\delta)}}_Y \leq \left(1+q\right)\delta =: \tilde\delta.
    \end{equation*}
    This implies that if $y^\delta$ is a noisy measurement for the exact data $y$ with noise level $\delta$, then $y^\delta$ is also a noisy measurement for $y_{\eps(\delta)}$ with noise level $\tilde\delta$.
    Setting $\tilde\tau := \frac12(\tau+\tau_0)>\tau_0$, we thus have $\tilde\tau\tilde\delta = \tau\delta$ and therefore
    \begin{equation*}
        \alpha_{\tilde\tau}(\tilde\delta,y^\delta) = \alpha(\tilde\tau\tilde\delta,y^\delta) = \alpha(\tau\delta,y^\delta ) = \alpha_\tau(\delta,y^\delta),
    \end{equation*}
    i.e., the parameter choice rules $\alpha_\tau$ for $y$ and $\alpha_{\tilde\tau}$ for $y_{\eps(\delta)}$ coincide for given $y^\delta$.
    The order optimality \eqref{eq:regularisierung:ordnung} of $(R_\alpha,\alpha_{\tilde\tau})$ for $x_\eps\in X_{\nu,\rho}$ with $\rho = \eps^{-\nu}\norm{x^\dag}_X$ and $\eps = \eps(\delta)$ then implies that
    \begin{equation*}
        \begin{aligned}
            \norm{R_{\alpha_{\tau}(\delta, y^\delta)}y^\delta - x_\eps}_X  =
            \norm{R_{\alpha_{\tilde \tau}(\tilde \delta, y^\delta)}y^\delta - K^\dag y_\eps}_X
            \leq \calE(y_\eps,\tilde \delta)
            &\leq c_{\tilde\tau}\tilde\delta^\frac{\nu}{\nu+1} \left(\eps^{-\nu} \norm{x^\dag}_X\right)^{\frac{1}{\nu+1}}\\
            &=:  c_{\tau,\nu} \left(\frac{\delta}{\eps}\right)^\frac{\nu}{\nu+1}  \norm{x^\dag}_X^{\frac{1}{\nu+1}}.
        \end{aligned}
    \end{equation*}
    (Here it is crucial that $c_{\tilde\tau}$ does not depend on $\rho$, since $\rho = \eps(\delta)^{-\nu}\norm{x^\dag}_X$ will in general become unbounded as $\delta\to 0$.)

    We can thus estimate
    \begin{equation*}
        \begin{aligned}
            \norm{R_{\alpha_{\tau}(\delta, y^\delta)}y^\delta - x^\dag}_X &\leq \norm{R_{\alpha_{\tau}(\delta, y^\delta)}y^\delta - x_{\eps(\delta)}}_X + \norm{x_{\eps(\delta)} - x^\dag}_X\\
            &\leq c_{\tau,\nu} \left(\frac{\delta}{{\eps(\delta)}}\right)^\frac{\nu}{\nu+1} \left( \norm{x^\dag}_X\right)^{\frac{1}{\nu+1}}+ \norm{x_{\eps(\delta)} - x^\dag}_X,
        \end{aligned}
    \end{equation*}
    and it remains to show that both $\delta{\eps(\delta)}^{-1} \to 0$ and $x_{\eps(\delta)}\to x^\dag$ as $\delta\to 0$.
    Since the set in the definition of $\eps(\delta)$ grows as $\delta$ decreases, $\eps(\delta)>0$ is monotonically decreasing as $\delta\to 0$ and therefore converges to some $\eps_0\geq 0$. We now distinguish two cases:
    \begin{enumerate}
        \item $\eps_0 >0$. In this case, we obviously have that $\delta{\eps(\delta)}^{-1} \leq \delta {\eps_0}^{-1} \to 0$.
            Furthermore, $N_{\eps(\delta)}$ is increasing as $\delta\to 0$ and bounded by $N_{\eps_0}<\infty$, and hence constant for all $\delta$ sufficiently small. The same then holds for $y_{\eps(\delta)}$, and \eqref{eq:regularisierung:ordnung3} implies that this constant satisfies $\norm{y-y_{\eps(\delta)}}_Y\leq q\delta\to 0$, i.e., $y_{\eps(\delta)} = y$. Hence $x_{\eps(\delta)} = K^\dag y_{\eps(\delta)} = K^\dag y = x^\dag$ for all $\delta$ sufficiently small.
        \item $\eps_0 = 0$. This immediately implies that $x_{\eps(\delta)}\to x^\dag$. It then follows from \eqref{eq:regularisierung:ordnung3} and \eqref{eq:regularisierung:ordnung1} that
            \begin{equation*}
                q \delta \leq \norm{y-y_{2\eps(\delta)}}_Y
                \leq 2\eps(\delta) \norm{x^\dag - x_{2\eps(\delta)}}_X
            \end{equation*}
            and hence that
            \begin{equation*}
                \frac{\delta}{{\eps(\delta)}} \leq \frac{2}{q} \norm{x^\dag - x_{2\eps(\delta)}}_X \to 0.
            \end{equation*}
    \end{enumerate}
    Since the bound on $\norm{R_{\alpha_{\tau}(\delta, y^\delta)}y^\delta - x^\dag}_X$ does not depend on $y^\delta\in B_\delta(y)$, this shows that $\calE(y,\delta)\to 0$ as $\delta\to 0$ for all $y\in\calD(K^\dag)$, and thus $(R_\alpha,\alpha_\tau)$ is a regularization method.
\end{proof}
Note that the assumption of \cref{thm:regularisierung:ordnung} is quite strong: Since $0\in X_{\nu,\rho}$ for all $\rho>0$, taking $x^\dag = 0$ and $\rho\to 0$ in \eqref{eq:regularisierung:ordnung} shows that $R_{\alpha_\tau(\delta,y^\delta)}y^\delta = 0$ for all $y^\delta\in B_\delta(0)$. In particular, the theorem cannot be applied to a priori choice rules: By linearity, this would imply $R_{\alpha_\tau(\delta)} = 0$ and hence for $x^\dag := \sigma_1^\nu v_1\in X_{\nu,1}$ that $\sigma_1^\nu = \norm{R_{\alpha_\tau(\delta)}Kx^\dag - x^\dag}_X\leq c_\tau\delta^{\frac\nu{\nu+1}}$ for all $\delta>0$, which is impossible. Furthermore, a posteriori choice rules are forced to take $x_\alpha^\delta = 0$ for small data as in \eqref{eq:parameter:morozov_infty}.

Finally, we remark that it is possible to formulate weaker source conditions using more general \term[function!index]{index functions} $\psi$ than powers. One example is \term[source condition!logarithmic]{logarithmic source conditions} of the form $x^\dag \in\calR(-\ln |K|)$ that are appropriate for exponentially ill-posed problems; see, e.g., \cite{Hohage:2000}. In fact, it is possible to show that for every $x^\dag \in X$ there exists an index function $\psi$ with $x^\dag \in \calR(\psi(|K|))$ for which the worst-case error can be bounded in terms of $\psi$; see \cite{MatheHofmann:2008}.

\chapter{Spectral regularization}\label{chap:spektral}

As we have seen, regularizing an ill-posed operator equation $Tx=y$ consists in replacing the (unbounded) Moore--Penrose inverse $T^\dag$ by a family $\{R_\alpha\}_{\alpha>0}$ of operators that for $\alpha>0$ are continuous on $Y$ and for $\alpha\to 0$ converge pointwise on $\calD(T^\dag)$ to $T^\dag$.
For a compact operator $K\in\calK(X,Y)$, such regularizations can be constructed using the singular value decomposition together with the fact that by \cref{thm:inverse:normalen} we have for $y\in\calD(K^\dag)$ that
\begin{equation*}
    K^\dag y = (K^*K)^\dag K^* y.
\end{equation*}
Let therefore $\{(\sigma_n,u_n,v_n)\}_{n\in\N}$ be a singular system of $K$. By construction, $\{(\sigma_n^2,v_n,v_n)\}_{n\in\N}$ is then a singular system of $K^*K$, and \cref{thm:inverse:picard} yields that
\begin{equation*}
    \begin{aligned}
        (K^*K)^\dag K^* y &= \sum_{n\in\N} \sigma_n^{-2} \inner{K^*y}{ v_n}_X v_n = \sum_{n\in\N} \sigma_n^{-2} \sigma_n \inner{y}{ u_n}_Y v_n \\
        &= \sum_{n\in\N} \phi(\sigma_n^2) \sigma_n \inner{y}{ u_n}_Y v_n
    \end{aligned}
\end{equation*}
for $\phi(\lambda) = \lambda^{-1}$. The unboundedness of $K^\dag$ is thus due to the fact that $\phi$ is unbounded on $(0,\norm{K^*K}_{\linop(X,X)}]$ and that $\{\sigma_n\}_{n\in\N}$ is a null sequence.
To obtain a regularization, we therefore replace $\phi$ by a family $\{\phi_\alpha\}_{\alpha>0}$ of \emph{bounded} functions that converge pointwise to $\phi$. Here and throughout the following, we set $\kappa:= \norm{K}_{\linop(X,Y)}^2 = \norm{K^*K}_{\linop(X,X)}$ for brevity.
\begin{definition}
    Let $\{\phi_\alpha\}_{\alpha>0}$ be a family of piecewise continuous and bounded functions $\phi_\alpha:[0,\kappa] \to \R$. If
    \begin{enumerate}
        \item $\displaystyle \lim_{\alpha\to 0}\phi_\alpha(\lambda) = \frac1\lambda$  for all $\lambda\in (0,\kappa]$ and
        \item $\displaystyle \lambda|\phi_\alpha(\lambda)|\leq C_\phi$ for some $C_\phi>0$ and all $\lambda\in (0,\kappa]$ and $\alpha>0$,
    \end{enumerate}
    then $\{\phi_\alpha\}_{\alpha>0}$ is called a \term[filter]{(regularizing) filter}.
\end{definition}
Note that the definition of the filter depends on $K$ only via its norm. In particular, if conditions (i) and (ii) hold for all $\lambda>0$, then $\{\phi_\alpha\}_{\alpha>0}$ is a regularizing filter for \emph{any} compact operator.

The idea is now to take $R_\alpha:= \phi_\alpha(K^*K)K^*$ as a regularization operator, i.e., to set for $y\in Y$
\begin{equation*}
    \begin{aligned}[t]
        R_\alpha y= \phi_\alpha(K^*K)K^*y &=  \sum_{n\in\N} \phi_\alpha(\sigma_n^2) \inner{K^*y}{v_n}_X v_n+ \phi_\alpha(0)P_{\calN} K^*y\\
        &= \sum_{n\in\N} \phi_\alpha(\sigma_n^2)\sigma_n \inner{y}{u_n}_Y v_n
    \end{aligned}
\end{equation*}
since $K^* y \in \overline{\calR(K^*)} = \calN(K)^\bot$.
(In contrast to the filter, the corresponding regularization \emph{does} depend on the concrete $K$ through its singular system.)

This approach covers several prototypical regularizations.
\begin{example}\label{ex:spektral}
    \begin{enumerate}
        \item The \term[singular value decomposition!truncated]{truncated singular value decomposition}\index{regularization method!truncated singular value decomposition} corresponds to the choice
            \begin{equation}
                \label{eq:spektral:cutoff}
                \phi_\alpha(\lambda) = \begin{cases}\frac1\lambda & \text{if }\lambda\geq \alpha,\\ 0 & \text{else.}
                \end{cases}
            \end{equation}
            Obviously, $\phi_\alpha$ is bounded (by $\frac1{\alpha}$) and piecewise continuous, converges for $\lambda>0$ to $\frac1\lambda$ as $\alpha\to 0$, and satisfies the boundedness condition for $C_\phi = 1$. The corresponding regularization operator is given by
            \begin{equation} \label{eq:spektral:cutoff_R}
                R_\alpha y = \sum_{n\in\N}\phi_\alpha(\sigma_n^2)\sigma_n \inner{y}{u_n}_Y v_n = \sum_{\sigma_n\geq \sqrt\alpha}\frac1{\sigma_n} \inner{y}{u_n}_Y v_n,
            \end{equation}
            which also explains the name. We will revisit this example throughout this chapter.

        \item The \term{Tikhonov regularization}\index{regularization method!Tikhonov} corresponds to the choice
            \begin{equation*}
                \phi_\alpha(\lambda) = \frac{1}{\lambda+\alpha}.
            \end{equation*}
            Again, $\phi_\alpha$ is bounded (by $\frac1{\alpha}$) and continuous, converges for $\lambda>0$ to $\frac1\lambda$ as $\alpha\to 0$, and satisfies the boundedness condition for $C_\phi = 1$. The corresponding regularization operator is given by
            \begin{equation*}
                R_\alpha y = \sum_{n\in\N}\frac{\sigma_n}{\sigma_n^2+\alpha} \inner{y}{u_n}_Y v_n.
            \end{equation*}
            However, the regularization $\phi_\alpha(K^*K)K^*y$ can be computed without the aid of a singular value decomposition; we will treat this in detail in \cref{chap:tikhonov}.

        \item The \term[Landweber iteration]{Landweber regularization}\index{regularization method!Landweber} corresponds to the choice
            \begin{equation*}
                \phi_\alpha(\lambda) = \frac{1-(1-\omega\lambda)^{1/\alpha}}{\lambda}
            \end{equation*}
            for a suitable $\omega>0$. If $\omega$ is small enough, one can show that this choice satisfies the definition of a regularizing filter. But here as well we can give a (more intuitive) characterization of the corresponding regularization operator without singular value decompositions; we therefore postpone its discussion to \cref{chap:landweber}.
    \end{enumerate}
\end{example}

\section{Regularization}\label{sec:spektral:regularisierung}

We first show that if $\{\phi_\alpha\}_{\alpha>0}$ is a regularizing filter, then $R_\alpha:=\phi_\alpha(K^*K)K^*$ defines indeed a regularization $\{R_\alpha\}_{\alpha>0}$ of $K^\dag$.
For this we will need the following three fundamental lemmas, which will be used throughout this chapter.

\begin{lemma}\label{lem:spektral:beschraenkt_K}
    Let $\{\phi_\alpha\}_{\alpha>0}$ be a regularizing filter. Then
    \begin{equation*}
        \norm{KR_\alpha}_{\linop(Y,Y)} \leq \sup_{n\in\N} |\phi_\alpha(\sigma_n^2)|\sigma_n^2 \leq C_\phi \qquad\text{for all }\alpha>0.
    \end{equation*}
\end{lemma}
\begin{proof}
    For all $y\in Y$ and $\alpha>0$, we have that (compare \eqref{eq:svd:operator})
    \begin{equation}\label{eq:spektral:beschraenkt_K}
        \begin{aligned}[t]
            KR_\alpha y = K\phi_\alpha(K^*K)K^* y &= \sum_{n\in\N} \phi_\alpha(\sigma_n^2)\sigma_n \inner{y}{u_n}_Y Kv_n \\
            &=  \sum_{n\in\N} \phi_\alpha(\sigma_n^2)\sigma_n^2 \inner{y}{u_n}_Y u_n.
        \end{aligned}
    \end{equation}
    Together with the Bessel inequality \eqref{eq:funktan:bessel}, this implies that
    \begin{equation*}
        \begin{aligned}
            \norm{KR_\alpha y }_Y^2 = \sum_{n\in\N} |\phi_\alpha(\sigma_n^2)\sigma_n^2 \inner{y}{u_n}_Y|^2 &\leq
            \sup_{n\in\N} |\phi_\alpha(\sigma_n^2)\sigma_n^2|^2 \sum_{n\in\N} |\inner{y}{u_n}_Y|^2 \\
            &\leq  \sup_{n\in\N} |\phi_\alpha(\sigma_n^2)\sigma_n^2|^2 \norm{y}_Y^2.
        \end{aligned}
    \end{equation*}
    The second inequality now follows from the fact that $0<\sigma_n^2 \leq \sigma_1^2 = \norm{K^*K}_{\linop(X,X)} = \kappa$ together with the boundedness condition (ii) of regularizing filters.
\end{proof}
\begin{lemma}\label{lem:spektral:beschraenkt}
    Let $\{\phi_\alpha\}_{\alpha>0}$ be a regularizing filter. Then
    \begin{equation*}
        \norm{R_\alpha}_{\linop(Y,X)} \leq \sqrt{C_\phi} \sup_{\lambda\in(0,\kappa]} \sqrt{|\phi_\alpha(\lambda)|}  \qquad\text{for all }\alpha>0.
    \end{equation*}
    In particular, $R_\alpha:Y\to X$ is continuous for all $\alpha>0$.
\end{lemma}
\begin{proof}
    For all $y\in Y$ and $\alpha>0$, it follows from \cref{lem:spektral:beschraenkt_K} and $\sigma_n v_n = K^* u_n$ that
    \begin{equation*}
        \begin{aligned}[t]
            \norm{R_\alpha y}_{X}^2 = \inner{R_\alpha y}{R_\alpha y}_X
            &= \sum_{n\in\N} \phi_\alpha(\sigma_n^2)\sigma_n \inner{y}{u_n}_Y \inner{R_\alpha y}{v_n}_X \\
            &= \sum_{n\in\N} \phi_\alpha(\sigma_n^2) \inner{y}{u_n}_Y \inner{KR_\alpha y}{u_n}_Y \\
            &\leq \sup_{n\in\N} |\phi_\alpha(\sigma_n^2)|  \inner{KR_\alpha y}{{\textstyle\sum_{n\in\N}}\inner{y}{u_n}_Y u_n}_Y\\
            &\leq \sup_{n\in\N} |\phi_\alpha(\sigma_n^2)| \ \norm{KR_\alpha y}_Y \norm{P_{\overline{\calR(K^*)}}y}_Y\\
            &\leq \sup_{n\in\N} |\phi_\alpha(\sigma_n^2)| \ C_\phi \norm{y}_Y^2.
        \end{aligned}
    \end{equation*}
    Taking the supremum over all $y\in Y$ and using the boundedness of $\phi_\alpha$ now yields the claim.
\end{proof}
Finally, the third \enquote{fundamental lemma of spectral regularization} gives a spectral representation of the approximation error.
\begin{lemma}\label{lem:spektral:verfahrensfehler}
    Let $\{\phi_\alpha\}_{\alpha>0}$ be a regularizing filter, $y\in \calD(K^\dag)$, and $x^\dag := K^\dag y$. Then
    \begin{equation}\label{eq:spektral:fehler}
        K^\dag y - R_\alpha y = \sum_{n\in\N} r_\alpha(\sigma_n^2) \inner{x^\dag}{v_n}_X v_n,
    \end{equation}
    where $r_\alpha(\lambda) :=  1-\lambda\phi_\alpha(\lambda)$ satisfies
    \begin{align*}
        &\lim_{\alpha\to 0}r_\alpha(\lambda) = 0 \qquad\text{for all }\lambda\in(0,\kappa],\\
        &|r_\alpha(\lambda)|\leq 1+C_\phi  \quad\text{ for all }\lambda\in(0,\kappa] \text{ and }\alpha>0.
    \end{align*}
\end{lemma}
\begin{proof}
    Since $K^*Kx^\dag = K^*y$ by \cref{thm:inverse:normalen}, we can write
    \begin{equation*}
        R_\alpha y = \phi_\alpha(K^*K)K^* y =  \phi_\alpha(K^*K)K^*K x^\dag,
    \end{equation*}
    and the definition of $r_\alpha$ together with \eqref{eq:functional} for $x^\dag \in \calN(K)^\bot$ immediately yields that
    \begin{equation*}
        K^\dag y - R_\alpha y = (\Id - \phi_\alpha(K^*K)K^*K)x^\dag = r_\alpha(K^*K)x^\dag = \sum_{n\in\N} r_\alpha(\sigma_n^2) \inner{x^\dag}{v_n}_X v_n.
    \end{equation*}

    The remaining claims follow from the corresponding properties of regularizing filters.
\end{proof}

We now have everything at hand to show the pointwise convergence and thus the regularization property of $\{R_\alpha\}_{\alpha>0}$.
\begin{theorem}\label{thm:spektral:konvergenz}
    Let $\{\phi_\alpha\}_{\alpha>0}$ be a regularizing filter. Then
    \begin{equation*}
        \lim_{\alpha\to 0} R_\alpha y = K^\dag y\qquad\text{for all } y\in\calD(K^\dag),
    \end{equation*}
    i.e., $\{R_\alpha\}_{\alpha>0}$ is a regularization.

    Furthermore, if $K^\dag$ is not continuous, then $\lim_{\alpha\to 0} \norm{R_\alpha y}_X = \infty$ for all $y\notin\calD(K^\dag)$.
\end{theorem}
\begin{proof}
    Let $y\in \calD(K^\dag)$ and $x^\dag= K^\dag y$. \Cref{lem:spektral:verfahrensfehler} then yields that
    \begin{equation*}
        \norm{K^\dag y-R_\alpha y}_X^2 =\sum_{n\in\N} |r_\alpha(\sigma_n^2)|^2 \abs{\inner{x^\dag}{v_n}_X}^2.
    \end{equation*}
    To show that the right-hand side tends to zero as $\alpha\to 0$, we split the series into a finite sum, for which we can use the convergence of $r_\alpha$ and the boundedness of the Fourier coefficients, and a remainder term, for which we argue vice versa.

    Let therefore $\eps>0$ be arbitrary. Then we first obtain from the Bessel inequality an $N\in \N$ with
    \begin{equation*}
        \sum_{n=N+1}^\infty \abs{\inner{x^\dag}{v_n}_X}^2 < \frac{\eps^2}{2(1+C_\phi)^2}.
    \end{equation*}
    Furthermore, the pointwise convergence of $\{r_\alpha\}_{\alpha>0}$ -- which is uniform on the finite set $\{\sigma_1^2,\dots,\sigma_N^2\}$ -- yields an $\alpha_0>0$ with
    \begin{equation*}
        |r_\alpha(\sigma_n^2)|^2 < \frac{\eps^2}{2\norm{x^\dag}_X^2} \qquad\text{for all } n\leq N \text{ and }\alpha<\alpha_0.
    \end{equation*}
    We thus have for all $\alpha<\alpha_0$ that
    \begin{equation*}
        \begin{aligned}
            \norm{K^\dag y-R_\alpha y}_X^2 &= \sum_{n=1}^N |r_\alpha(\sigma_n^2)|^2 \abs{\inner{x^\dag }{v_n}_X}^2 +
            \sum_{n=N+1}^\infty |r_\alpha(\sigma_n^2)|^2 \abs{\inner{x^\dag }{v_n}_X}^2\\
            &\leq  \frac{\eps^2}{2\norm{x^\dag}_X^2}  \sum_{n=1}^N  \abs{\inner{x^\dag}{v_n}_X}^2 +  (1+C_\phi)^2 \frac{\eps^2}{2(1+C_\phi)^2} \\
            &\leq \frac{\eps^2}2+\frac{\eps^2}2 = \eps^2,
        \end{aligned}
    \end{equation*}
    i.e., $\norm{K^\dag y-R_\alpha y}_X \to 0$ as $\alpha\to 0$.
    Together with the continuity of $R_\alpha $ for $\alpha>0$ from \cref{lem:spektral:beschraenkt}, this implies by \cref{def:regularisierung} that $\{R_\alpha\}_{\alpha>0}$ is a regularization.

    Finally, the divergence for $y\notin\calD(K^\dag)$ follows from \cref{thm:regularisierung:div,lem:spektral:beschraenkt_K}.
\end{proof}

In particular, the truncated singular value decomposition, the Tikhonov regularization, and (after verifying the filter properties) the Landweber regularization from \cref{ex:spektral} all define regularizations for any compact operator.

\section{Parameter choice and convergence rates}\label{sec:spektral:parameter}

We now investigate which parameter choice rules $\alpha$ will for a given filter $\phi_\alpha$ lead to a convergent (and order optimal) regularization method $(R_\alpha,\alpha)$. To keep the notation concise, we will in the following write $x^\dag:= K^\dag y$, $x_\alpha:=R_\alpha y$ for $y\in \calD(K^\dag)$, and $x_\alpha^\delta:=R_\alpha y^\delta$ for $y^\delta \in B_\delta(y)$.

\subsection*{A priori choice rules}

By \cref{thm:regularisierung:apriori_char}, every a priori choice rule that satisfies $\alpha(\delta)\to 0$ and $\delta\norm{R_\alpha}_{\linop(Y,X)}\to 0$ as $\delta\to 0$ leads to a regularization method $(R_\alpha,\alpha)$. Together with \cref{lem:spektral:beschraenkt}, this leads to a condition on $\phi_\alpha$ and thus on $\alpha$.
\begin{example}[truncated singular value decomposition]
    Let $K\in\calK(X,Y)$ have the singular system $\{(\sigma_n,u_n,v_n)\}_{n\in\N}$. Then we have for $\phi_\alpha$ as in \eqref{eq:spektral:cutoff} that
    \begin{equation*}
        \norm{R_\alpha}_{\linop(Y,X)} \leq \sqrt{C_\phi} \sup_{n\in\N} \sqrt{|\phi_\alpha(\sigma_n^2)|} = \frac1{\sqrt\alpha}.
    \end{equation*}
    This yields a condition on the minimal singular value that we can include in \eqref{eq:spektral:cutoff_R} for given $\delta>0$:
    Choosing $n(\delta)$ with
    \begin{equation*}
        n(\delta)\to \infty,\qquad \frac\delta{\sigma_{n(\delta)}} \to 0 \qquad\text{as }\delta \to 0,
    \end{equation*}
    the truncated singular value decomposition together with the parameter choice rule $\alpha(\delta):= \sigma_{n(\delta)}^2$ becomes a regularization method.

    In particular, this holds for the choice $\alpha(\delta) :=  \sigma_{n(\delta)}^2\geq {\delta}> \sigma_{n(\delta)+1}^2$, which satisfies
    \begin{equation*}
        x_{\alpha(\delta)}^\delta = \sum_{\sigma_n\geq \sqrt\delta}\frac1{\sigma_n} \inner{y^\delta}{u_n}_Y v_n \to \sum_{n\in\N}\frac1{\sigma_n} \inner{y}{u_n}_Y v_n = x^\dag \qquad\text{as }\delta\to 0.
    \end{equation*}
\end{example}

We now consider convergence rates under the source condition $x^\dag \in X_{\nu,\rho}$ for $\nu,\rho>0$.
For this, we proceed as in the proof of \cref{thm:spektral:konvergenz} and first show that
\begin{equation*}
    \omega_\nu(\alpha) :=  \sup_{\lambda\in(0,\kappa]} \lambda^{\nu/2}|r_\alpha(\lambda)|
\end{equation*}
is an upper bound for the approximation error.
\begin{lemma}\label{lem:spektral:fehler}
    Let $y\in\calD(K^\dag)$ and $x^\dag \in X_{\nu,\rho}$ for some $\nu,\rho>0$. Then we have for all $\alpha>0$ that
    \begin{align*}
        \norm{x_\alpha-x^\dag}_X &\leq  \omega_{\nu}(\alpha)\rho,\\
        \norm{Kx_\alpha-Kx^\dag}_Y &\leq  \omega_{\nu+1}(\alpha)\rho.
    \end{align*}
\end{lemma}
\begin{proof}
    By definition, for $x^\dag\in X_{\nu,\rho}$ there exists a $w\in X$ with $x^\dag = |K|^\nu w = (K^*K)^{\nu/2}w$ and $\norm{w}_X \leq \rho$. It then follows from \cref{lem:spektral:verfahrensfehler} that
    \begin{equation*}
        \begin{aligned}
            x^\dag-x_\alpha  &=  r_\alpha(K^*K)x^\dag = r_\alpha(K^*K)(K^*K)^{\nu/2} w \\
            &= \sum_{n\in\N} r_\alpha(\sigma_n^2) \sigma_n^{\nu} \inner{w}{v_n}_X v_n
        \end{aligned}
    \end{equation*}
    and hence that
    \begin{equation*}
        \begin{aligned}
            \norm{x_\alpha -x^\dag}_X^2 &=  \sum_{n\in\N} |r_\alpha(\sigma_n^2)|^2 \sigma_n^{2\nu} |\inner{w}{v_n}_X|^2 \\
            &\leq \omega_\nu(\alpha)^2 \sum_{n\in\N}|\inner{w}{v_n}_X|^2 \leq \omega_\nu(\alpha)^2\norm{w}_X^2 \leq \omega_\nu(\alpha)^2\rho^2.
        \end{aligned}
    \end{equation*}

    Furthermore, \cref{lem:functional_range}\,(iii) yields
    \begin{equation*}
        \norm{Kx_\alpha-Kx^\dag}_Y = \norm{K(x_\alpha-x^\dag)}_Y = \norm{|K|(x_\alpha-x^\dag)}_X.
    \end{equation*}
    From this together with
    \begin{equation*}
        \begin{aligned}
            |K|(x^\dag-x_\alpha)  &=  (K^*K)^{1/2} r_\alpha(K^*K)(K^*K)^{\nu/2} w \\
            &= \sum_{n\in\N} \sigma_n r_\alpha(\sigma_n^2) \sigma_n^{\nu} \inner{w}{v_n}_X v_n
        \end{aligned}
    \end{equation*}
    and $|r_{\alpha}(\sigma_n^2)\sigma_n^{\nu+1}|^2\leq \omega_{\nu+1}(\alpha)^2$, we similarly obtain the second estimate.
\end{proof}

We now have everything at hand to show conditions for convergence and order optimality of a priori choice rules.
\begin{theorem}\label{thm:spektral:apriori}
    Let $y\in\calD(K^\dag)$ and $x^\dag = K^\dag y\in X_{\nu,\rho}$ for some $\nu,\rho>0$. If $\alpha(\delta)$ is an a priori choice rule with
    \begin{equation}\label{eq:spektral:apriori}
        c \left(\frac\delta\rho\right)^{\frac2{\nu+1}} \leq \alpha(\delta) \leq C  \left(\frac\delta\rho\right)^{\frac2{\nu+1}} \qquad\text{for } C>c>0
    \end{equation}
    and the filter $\{\phi_\alpha\}_{\alpha>0}$ satisfies for some $C_\nu>0$ the conditions
    \begin{align}
        \sup_{\lambda\in(0,\kappa]} |\phi_\alpha(\lambda)| &\leq C_\phi \alpha^{-1},\label{eq:spektral:ordnung_phi}\\
        \omega_\nu (\alpha) &\leq C_\nu \alpha^{\nu/2},\label{eq:spektral:ordnung_omega}
    \end{align}
    then $(R_\alpha,\alpha)$ is a regularization method that is order optimal for this $\nu$ and $\rho$.
\end{theorem}
\begin{proof}
    We first show order optimality.
    We again use the decomposition \eqref{eq:regularisierung:splitting} into data error and approximation error: For given $\delta>0$ and $y^\delta\in B_\delta(y)$,
    \begin{equation*}
        \norm{x_{\alpha(\delta)}^\delta - x^\dag }_X \leq \delta \norm{R_{\alpha(\delta)}}_{\linop(Y,X)} + \norm{x_{\alpha(\delta)} - x^\dag}_X.
    \end{equation*}
    By \cref{lem:spektral:beschraenkt} and the assumption \eqref{eq:spektral:ordnung_phi}, we have that
    \begin{equation}\label{eq:spektral:ordnung:regest}
        \norm{R_{\alpha(\delta)}}_{\linop(Y,X)}\leq \sqrt{C_\phi}  \sqrt{C_\phi\alpha(\delta)^{-1}} \leq  C_\phi \alpha(\delta)^{-1/2}.
    \end{equation}
    Similarly, it follows from \cref{lem:spektral:fehler} and the assumption \eqref{eq:spektral:ordnung_omega} that
    \begin{equation*}
        \norm{x_{\alpha(\delta)} - x^\dag}_X \leq \omega_\nu(\alpha(\delta))\rho \leq C_\nu \alpha(\delta)^{\nu/2} \rho.
    \end{equation*}
    Together with the parameter choice rule \eqref{eq:spektral:apriori}, we obtain
    \begin{equation}\label{eq:spektral:ordnung:splitting}
        \begin{aligned}[t]
            \norm{x_{\alpha(\delta)}^\delta - x^\dag}_X &\leq  C_\phi \alpha(\delta)^{-1/2}\delta + C_\nu \alpha(\delta)^{\nu/2} \rho\\
            &\leq C_\phi c^{-1/2} \delta^{-\frac1{\nu+1}}\rho^{\frac1{\nu+1}} \delta + C_\nu C^{\nu/2} \delta^{\frac\nu{\nu+1}}\rho^{-\frac{\nu}{\nu+1}}\rho\\
            &= (C_\phi  c^{-1/2}+ C_\nu C^{\nu/2}) \delta^{\frac\nu{\nu+1}} \rho^{\frac1{\nu+1}}
        \end{aligned}
    \end{equation}
    and thus the order optimality.

    It remains to show that $(R_\alpha,\alpha)$ is a regularization method, for which we apply \cref{thm:regularisierung:apriori_char}. By \cref{thm:spektral:konvergenz}, $\{R_\alpha\}_{\alpha>0}$ is a regularization. The upper bound in \eqref{eq:spektral:apriori} implies that $\alpha(\delta)\to 0$ as $\delta\to 0$, while the lower bound together with \eqref{eq:spektral:ordnung:regest} yields that
    \begin{equation*}
        \delta\norm{R_{\alpha(\delta)}}_{\linop(Y,X)}\leq C_\phi c^{-1/2} \delta^{\frac\nu{\nu+1}}\rho^{\frac1{\nu+1}} \to 0 \qquad\text{as }\delta\to 0.
    \end{equation*}
    Hence $(R_\alpha,\alpha)$ is a regularization method.
\end{proof}
Note that \eqref{eq:spektral:apriori} in general gives a parameter choice rule that will depend on $\rho$ and therefore be order optimal only for this $\rho$ (but with a constant that is independent of $\rho$).

Hence, to show for a given filter $\phi_\alpha$ the order optimality for some $\nu>0$, it suffices to verify for this $\nu$ the condition \eqref{eq:spektral:ordnung_omega} (as well as for $\phi_\alpha$ the condition \eqref{eq:spektral:ordnung_phi}).
The maximal $\nu_0>0$, for which all $\nu\in(0,\nu_0]$ satisfy the condition \eqref{eq:spektral:ordnung_omega}, is called the \term{qualification} of the filter.
\begin{example}[truncated singular value decomposition]
    \label{ex:spectral:tsvd}
    It follows from \eqref{eq:spektral:cutoff} that
    \begin{equation*}
        \sup_{\lambda\in(0,\kappa]} |\phi_\alpha(\lambda)| \leq \alpha^{-1},
    \end{equation*}
    and hence this filter satisfies \eqref{eq:spektral:ordnung_phi} with $C_\phi = 1$.

    Furthermore, for all $\nu>0$ and $\lambda\in(0,\kappa]$,
    \begin{equation*}
        \lambda^{\nu/2} |r_\alpha(\lambda)| = \lambda^{\nu/2} |1-\lambda \phi_\alpha(\lambda)| =
        \begin{cases}
            0 & \text{if }\lambda\geq \alpha,\\
            \lambda^{\nu/2} & \text{if }\lambda<\alpha.
        \end{cases}
    \end{equation*}
    Hence for all $\alpha\in(0,\kappa]$,
    \begin{equation*}
        \omega_\nu(\alpha) = \sup_{\lambda\in(0,\kappa]} \lambda^{\nu/2} |r_\alpha(\lambda)| \leq \max\{0,\alpha^{\nu/2}\} = \alpha^{\nu/2},
    \end{equation*}
    and the condition \eqref{eq:spektral:ordnung_omega} is therefore satisfied for all $\nu>0$ with $C_\nu = 1$.
    (For $\alpha>\kappa$, \emph{all} singular values are truncated, i.e., $R_\alpha =0$.)
    This shows that the truncated singular value decomposition is order optimal for all $\nu>0$ and thus has \term[qualification!infinite]{infinite qualification}.
\end{example}

\subsection*{A posteriori choice rules}

We again consider the discrepancy principle\index{discrepancy principle!for spectral regularization}\index{parameter choice rule!discrepancy principle!for spectral regularization}: Fix $\tau > 1$. If $\norm{y^\delta}_Y\leq \tau\delta$, set $\alpha(\delta,y^\delta) := \infty$, i.e., $x^\delta_{\alpha(\delta,y^\delta)} := R_\infty y^\delta = 0$. Otherwise, choose $\alpha(\delta,y^\delta)$ such that
\begin{equation}\label{eq:spektral:diskrepanz}
    \norm{Kx^\delta_{\alpha(\delta,y^\delta)}-y^\delta}_Y\leq \tau \delta < \norm{Kx^\delta_{\alpha}-y^\delta}_Y \qquad\text{for all } \alpha>\alpha(\delta,y^\delta).
\end{equation}

As before, we assume that $\calR(K)$ is dense in $Y$.
If in addition $\alpha\mapsto \phi_\alpha(\lambda)$ is continuous from the left and $\phi_\alpha(\lambda)\to 0$ as $\alpha\to\infty$ for all $\lambda\in (0,\kappa]$ (which is the case for all filters in \cref{ex:spektral}), one can show that $f(\alpha):= \norm{Kx^\delta_\alpha-y^\delta}_Y$ is continuous from the left as well and $f(\alpha)\to \norm{y^\delta}_Y$ as $\alpha\to\infty$. Similarly to \cref{thm:morozov} using \cref{lem:spektral:beschraenkt_K}, one can show that $f(\alpha)\to 0$ as $\alpha\to 0$. If $\norm{y^\delta}_Y>\tau\delta$, the supremum
\begin{equation*}
    \alpha(\delta,y^\delta) := \sup\setof{\alpha>0}{\norm{Kx^\delta_\alpha-y^\delta}_Y\leq\tau\delta}
\end{equation*}
is therefore finite and attained, and hence satisfies \eqref{eq:spektral:diskrepanz}.
To show that the discrepancy principle leads to an order optimal regularization method, we again apply \cref{thm:regularisierung:ordnung}. For this, we note that the discrepancy principle with parameter $\tau$ is of the form $\alpha_\tau(\delta,y^\delta) = \alpha(\tau\delta,y^\delta)$, where $\alpha$ denotes the discrepancy principle with $\tau=1$ (i.e., with $\tau\delta$ replaced by $\delta$ above).
\begin{theorem}\label{thm:spektral:morozov}
    Let $\{\phi_\alpha\}_{\alpha>0}$ be a filter with qualification $\nu_0>1$ (i.e., satisfying \eqref{eq:spektral:ordnung_phi} and \eqref{eq:spektral:ordnung_omega} for all $\nu\in (0,\nu_0]$), and let
    \begin{equation}
        \label{eq:spektral:ordnung_r}
        \tau > \sup_{\alpha>0,\, \lambda\in[0,\kappa]}|r_\alpha(\lambda)|=: C_r.
    \end{equation}
    Then the discrepancy principle defines for all $\nu\in (0,\nu_0-1]$ an order optimal regularization method $(R_\alpha,\alpha_\tau)$.
\end{theorem}
\begin{proof}
    We first observe that due to $|r_\alpha(\lambda)| \leq 1 + C_\phi$ for all $\alpha>0$ and $\lambda\in(0,\kappa]$, the constant $C_r$ is finite. Furthermore, $r_\alpha(0) = 1$ implies that $C_r\geq 1$, and hence \eqref{eq:spektral:ordnung_r} implies that $\tau>1$.

    Let now $y\in \calR(K)$, $x^\dag = K^\dag y\in X_{\nu,\rho}$ for some
    $\nu\in (0,\nu_0-1]$ and $\rho>0$, and $y^\delta \in B_\delta(y)$.

    We first consider the case $\norm{y^\delta}_Y\leq \tau\delta$, where $\alpha(\delta,y^\delta)=\infty$ and hence $x_{\alpha(\delta,y^\delta)}^\delta = 0$. Since then
    \begin{equation*}
        \norm{Kx^\dag}_Y = \norm{y}_Y \leq \norm{y^\delta}_Y + \norm{y-y^\delta}_Y \leq (1+\tau)\delta,
    \end{equation*}
    \cref{thm:rates:limit} (with $(1+\tau)\delta$ in place of $\delta$) yields that
    \begin{equation*}
        \norm{x_{\alpha(\delta,y^\delta)}^\delta - x^\dag}_X = \norm{x^\dag}_X \leq \eps(X_{\nu,\rho},(1+\tau)\delta) \leq (1+\tau)^{\frac\nu{\nu+1}}\delta^{\frac\nu{\nu+1}}\rho^{\frac1{\nu+1}}.
    \end{equation*}

    We can thus assume that $\norm{y^\delta}_Y> \tau\delta$ and hence that $\alpha(\delta,y^\delta)<\infty$ satisfies \eqref{eq:spektral:diskrepanz}. We again use for $x_\alpha^\delta := x_{\alpha(\delta,y^\delta)}^\delta$ and $x_\alpha := x_{\alpha(\delta,y^\delta)}$ the decomposition
    \begin{equation}\label{eq:spektral:morozov_split}
        \norm{x_\alpha^\delta  - x^\dag}_X \leq \norm{x_\alpha - x^\dag}_X + \norm{x_\alpha - x_\alpha^\delta}_X
    \end{equation}
    and estimate the terms on the right-hand side separately.

    For the first term, we again use the representation of the approximation error from \cref{lem:spektral:verfahrensfehler} as well as the source condition $x^\dag = |K|^\nu w$ to obtain
    \begin{equation*}
        \begin{aligned}
            x^\dag - x_\alpha &= \sum_{n\in\N} r_\alpha(\sigma_n^2) \sigma_n^\nu\inner{w}{v_n}_X v_n\\
            &=\sum_{n\in\N} r_\alpha(\sigma_n^2) \inner{w}{v_n}_X |K|^\nu v_n\\
            &= |K|^\nu \sum_{n\in\N} r_\alpha(\sigma_n^2) \inner{w}{v_n}_X v_n
            =:  |K|^\nu \xi.
        \end{aligned}
    \end{equation*}
    The interpolation inequality \eqref{eq:interpolation} for $s=\nu$ and $r=\nu+1$ then yields that
    \begin{equation*}
        \norm{x_\alpha-x^\dag}_X = \norm{|K|^\nu \xi}_X \leq \norm{|K|^{\nu+1} \xi}_X^{\frac{\nu}{\nu+1}} \,\norm{\xi}_X^{\frac1{\nu+1}}.
    \end{equation*}
    Again we estimate the terms separately: For the second factor, we obtain from the definition of $\xi$, the Bessel inequality, the boundedness of $r_\alpha$, and the source condition that
    \begin{equation*}
        \norm{\xi}_X^2 = \sum_{n\in\N} |r_\alpha(\sigma_n^2)|^2 |\inner{w}{v_n}_X|^2\leq C_r^2 \norm{w}_X^2 \leq C_r^2 \rho^2.
    \end{equation*}
    For the first factor, we use \cref{lem:functional_range}\,(i), (iii), $Kx^\dag = y$ since $y\in\calR(K)$, and the productive zero to obtain
    \begin{equation*}
        \begin{aligned}
            \norm{|K|^{\nu+1} \xi}_X &= \norm{|K|(|K|^\nu \xi)}_X = \norm{|K|(x_\alpha-x^\dag)}_X = \norm{K(x_\alpha-x^\dag)}_Y = \norm{Kx_\alpha - y}_Y \\
            &\leq \norm{Kx_\alpha^\delta - y^\delta}_Y + \norm{y-y^\delta - K(x_\alpha - x_\alpha^\delta)}_Y.
        \end{aligned}
    \end{equation*}
    Yet again we estimate the terms separately: First, by the choice $\alpha(\delta,y^\delta)$ according to the discrepancy principle we have that $\norm{Kx_\alpha^\delta - y^\delta}_Y \leq \tau\delta$. For the second term, we write
    \begin{equation*}
        y - Kx_\alpha = y - KR_\alpha y = (\Id - K\phi_\alpha(K^*K)K^*)y
    \end{equation*}
    and analogously for $y^\delta - Kx_\alpha^\delta$.
    Hence,
    \begin{equation}\label{eq:spektral:morozov:fortgepflanzter_verfahrensfehler}
        \begin{aligned}[t]
            \norm{y-y^\delta - K(x_\alpha - x_\alpha^\delta)}_Y^2 &= \norm{(\Id-K\phi_\alpha (K^*K)K^*)(y-y^\delta)}_Y^2\\
            &= \sum_{n\in \N} \abs{r_\alpha(\sigma_n^2)}^2 \left|\inner{y-y^\delta}{u_n}_Y\right|^2\\
            & \leq C_r^2 \delta^2,
        \end{aligned}
    \end{equation}
    where we have used for the second equality that (compare \eqref{eq:spektral:beschraenkt_K})
    \begin{equation*}
        K\phi_\alpha (K^*K)K^*(y-y^\delta) = \sum_{n\in\N} \phi_\alpha(\sigma_n^2)\sigma_n^2 \inner{y-y^\delta}{u_n}_Y u_n .
    \end{equation*}
    Together, we obtain for the first term in \eqref{eq:spektral:morozov_split} that
    \begin{equation*}
        \norm{x_\alpha - x^\dag}_X \leq (\tau + C_r)^{\frac{\nu}{\nu+1}}\delta^{\frac{\nu}{\nu+1}} C_r^{\frac{1}{\nu+1}}\rho^{\frac{1}{\nu+1}} =:  C_1 \delta^{\frac{\nu}{\nu+1}}\rho^{\frac{1}{\nu+1}}.
    \end{equation*}

    It remains to estimate the second term in \eqref{eq:spektral:morozov_split}. For this, we use \cref{lem:spektral:beschraenkt} and the condition \eqref{eq:spektral:ordnung_phi} to obtain
    \begin{equation}\label{eq:spektral:morozov2}
        \begin{aligned}[t]
            \norm{x_\alpha^\delta - x_\alpha}_X
            \leq \norm{R_\alpha}_{\linop(Y,X)}\delta
            &\leq \sqrt{C_\phi} \sup_{\lambda\in(0,\kappa]}\sqrt{|\phi_\alpha(\lambda)|} \delta\\
            &\leq {C_\phi}\alpha(\delta,y^\delta)^{-1/2}\delta.
        \end{aligned}
    \end{equation}
    To show that the right-hand side is of the optimal order, we need to bound $\alpha(\delta,y^\delta)$ in terms of $\delta$ appropriately. First, its choice according to the discrepancy principle implies in particular that
    \begin{equation*}
        \tau\delta < \norm{Kx_{2\alpha}^\delta -y^\delta}_Y \leq  \norm{Kx_{2\alpha} - y}_Y + \norm{y-y^\delta - K(x_{2\alpha} - x_{2\alpha}^\delta)}_Y
    \end{equation*}
    (where the choice $2\alpha>\alpha$ was arbitrary and for the sake of simplicity).
    Since the estimate \eqref{eq:spektral:morozov:fortgepflanzter_verfahrensfehler} is uniform in $\alpha>0$, we also have that
    \begin{equation*}
        \norm{y-y^\delta - K(x_{2\alpha} - x_{2\alpha}^\delta)}_Y \leq C_r\delta
    \end{equation*}
    and thus that
    \begin{equation*}
        \norm{Kx_{2\alpha} - y}_Y > \tau\delta - \norm{y-y^\delta - K(x_{2\alpha} - x_{2\alpha}^\delta)}_Y \geq (\tau - C_r) \delta.
    \end{equation*}
    Conversely, we obtain from \cref{lem:spektral:fehler} and condition \eqref{eq:spektral:ordnung_omega} for $\nu+1\leq \nu_0$ the estimate
    \begin{equation*}
        \norm{Kx_{2\alpha} -y}_Y \leq \omega_{\nu+1}(2\alpha(\delta,y^\delta))\rho \leq C_{\nu+1} (2\alpha(\delta,y^\delta))^{\frac{\nu+1}2}\rho.
    \end{equation*}
    Since $\tau >C_r$ by assumption, this implies that
    \begin{equation*}
        \delta  < (\tau - C_r)^{-1} C_{\nu+1} 2^{\frac{\nu+1}2} \alpha(\delta,y^\delta)^{\frac{\nu+1}2}\rho =:  C_\tau \alpha(\delta,y^\delta)^{\frac{\nu+1}2}\rho,
    \end{equation*}
    i.e.,
    \begin{equation}\label{eq:spektral:morozov_alpha}
        \alpha(\delta,y^\delta)^{-1/2}\leq C_\tau^{\frac{1}{\nu+1}}\delta^{-\frac{1}{\nu+1}}\rho^{\frac{1}{\nu+1}}.
    \end{equation}
    Inserting this into \eqref{eq:spektral:morozov2} now yields
    \begin{equation*}
        \norm{x_\alpha^\delta - x_\alpha}_X \leq {C_\phi}C_\tau^{\frac{1}{\nu+1}} \delta \delta^{-\frac{1}{\nu+1}}\rho^{\frac{1}{\nu+1}} =:  C_2 \delta^{\frac{\nu}{\nu+1}}\rho^{\frac{1}{\nu+1}}.
    \end{equation*}

    Combining the estimates for the two terms in \eqref{eq:spektral:morozov_split}, we obtain that
    \begin{equation*}
        \norm{x_\alpha^\delta - x^\dag}_X \leq (C_1 + C_2) \delta^{\frac{\nu}{\nu+1}}\rho^{\frac{1}{\nu+1}}.
    \end{equation*}

    Together with the first case $\alpha(\delta,y^\delta)=\infty$, this shows the order optimality \eqref{eq:regularisierung:ordnung} with the constant $c_\tau := \max\{C_1+C_2,(1+\tau)^{\frac\nu{\nu+1}}\}$, which depends on $\nu$, $\tau$, $C_r$, $C_\phi$, and $C_{\nu+1}$ but not on $\rho$. Since this holds for every $\tau>C_r$, \cref{thm:regularisierung:ordnung} for $\nu = \nu_0-1$ and $\tau_0 = C_r\geq 1$ then shows that $R_\alpha$ together with the discrepancy principle as parameter choice rule $\alpha_\tau=\alpha(\tau\delta,y^\delta)$ is a regularization method for all $\tau>\tau_0$.
\end{proof}
\begin{example}[truncated singular value decomposition]
    We have
    \begin{equation*}
        |r_\alpha(\lambda)| = \begin{cases} 1-\lambda\frac1\lambda = 0 & \lambda\geq \alpha\\
            1 & \lambda <\alpha
        \end{cases}
    \end{equation*}
    and hence $C_r = 1$. Since the truncated singular value decomposition has infinite qualification, it is also an order optimal regularization method for any $\nu>0$ when combined with the discrepancy principle for arbitrary $\tau > 1$.
\end{example}

If a filter only has finite qualification, the Morozov discrepancy principle only leads to an order optimal regularization method for $\nu\leq\nu_0-1$; this is the price to pay for the indirect control of $\alpha(\delta,y^\delta)$ through the residual (cf.~\eqref{eq:spektral:fehler}).
However, there are improved discrepancy principles that measure the residual in adapted norms and thus lead to order optimal regularization methods also for $\nu\in(\nu_0-1,\nu_0]$; see, e.g., \cite[Chapter~4.4]{Engl}.

\subsection*{Heuristic choice rules}

We consider as an example the Hanke--Raus rule\index{parameter choice rule!Hanke--Raus rule}: Define for $y^\delta \in Y$ the function
\begin{equation*}
    \Psi:(0,\kappa]\to \R,\qquad \Psi(\alpha) = \frac{\norm{Kx_{\alpha}^\delta - y^\delta}_Y}{\sqrt{\alpha}},
\end{equation*}
and choose
\begin{equation}\label{eq:spektral:hankeraus}
    \alpha(y^\delta) \in \arg\min_{\alpha\in(0,\kappa]} \Psi(\alpha).
\end{equation}
We assume in the following that $y\in\calR(K)$ and that $\delta = \norm{y^\delta-y}_Y<\norm{y}_Y$ (since the noise level is not used by the rule, we can take it to be exact).
First, we show a conditional error estimate.
\begin{theorem}\label{thm:spektral:hankeraus:fehler}
    Let $\{\phi_\alpha\}_{\alpha>0}$ be a filter with qualification $\nu_0>0$, i.e., satisfying \eqref{eq:spektral:ordnung_phi} as well as \eqref{eq:spektral:ordnung_omega} for all $\nu\in (0,\nu_0]$.
    Furthermore, assume there exists a minimizer $\alpha^*:= \alpha(y^\delta)\in (0,\kappa]$ of $\Psi$ with
    \begin{equation}\label{eq:spektral:heuristic_noise}
        \delta^* := \norm{Kx_{\alpha^*}^\delta - y^\delta}_Y>0.
    \end{equation}
    Then there exists a $c>0$ such that for all $x^\dag\in X_{\nu,\rho}$ with $\nu\in(0,\nu_0-1]$ and $\rho>0$,
    \begin{equation*}
        \norm{x_{\alpha^*}^\delta - x^\dag}_X \leq c \left(1+\frac{\delta}{\delta^*}\right)\max\{\delta,\delta^*\}^{\frac{\nu}{\nu+1}}\rho^{\frac{1}{\nu+1}}.
    \end{equation*}
\end{theorem}
\begin{proof}
    First, since $K\phi_\alpha(K^*K)K^*(\Id-P_{\overline\calR}) = 0$, we have as for \eqref{eq:spektral:morozov:fortgepflanzter_verfahrensfehler} (but without assuming that $\calR(K)$ is dense) that
    \begin{equation}\label{eq:spektral:heuristic_prop}
        \begin{aligned}[t]
            \norm{y-y^\delta - K(x_\alpha-x_\alpha^\delta)}_Y^2
            &= \sum_{n\in\N}\abs{r_\alpha(\sigma_n^2)}^2\abs{\inner{y-y^\delta}{u_n}_Y}^2 + \norm{(\Id-P_{\overline\calR})(y-y^\delta)}_Y^2\\
            &\leq C_r^2\norm{P_{\overline\calR}(y-y^\delta)}_Y^2 + \norm{(\Id-P_{\overline\calR})(y-y^\delta)}_Y^2 \leq C_r^2\delta^2
        \end{aligned}
    \end{equation}
    for all $\alpha>0$, where we have used $C_r\geq 1$ in the second line.

    We now once more start from the error decomposition
    \begin{equation*}
        \norm{x_{\alpha^*}^\delta - x^\dag}_X \leq \norm{x_{\alpha^*} -x^\dag}_X  + \norm{x_{\alpha^*}^\delta - x_{\alpha^*}}_X.
    \end{equation*}
    For the first term, we argue as in the proof of \cref{thm:spektral:morozov} using \eqref{eq:spektral:heuristic_noise} in place of the discrepancy principle and \eqref{eq:spektral:heuristic_prop} in place of \eqref{eq:spektral:morozov:fortgepflanzter_verfahrensfehler} to show that
    \begin{equation}\label{eq:spektral:heuristic1}
        \norm{x_{\alpha^*} - x^\dag}_X \leq C_r^{\frac1{\nu+1}}(\delta^* + C_r\delta)^{\frac{\nu}{\nu+1}} \rho^{\frac{1}{\nu+1}}\leq C_1\max\{\delta,\delta^*\}^{\frac{\nu}{\nu+1}} \rho^{\frac{1}{\nu+1}}
    \end{equation}
    for some constant $C_1>0$.

    For the second term, we obtain similarly as for \eqref{eq:spektral:morozov2} using \eqref{eq:spektral:heuristic_noise} (in the form of the productive $1=\delta^*/\delta^*$) that
    \begin{equation*}
        \norm{x_{\alpha^*}^\delta - x_{\alpha^*}}_X\leq {C_\phi}\frac{1}{\sqrt{\alpha^*}} \delta
        = {C_\phi}\frac\delta{\delta^*}   \frac{\norm{Kx_{\alpha^*}^\delta - y^\delta}_Y}{\sqrt{\alpha^*}} = {C_\phi}\frac\delta{\delta^*}\Psi(\alpha^*).
    \end{equation*}
    Again, we need to bound the last factor by the correct power of $\delta$, for which we use the choice rule. In this case, \eqref{eq:spektral:hankeraus} states that $\Psi(\alpha^*)\leq \Psi(\alpha)$ for all $\alpha\in(0,\kappa]$. We compare with the a priori choice $\hat\alpha := (\delta/\rho)^{\frac2{\nu+1}}$ from \cref{thm:spektral:apriori}, which is feasible since
    \begin{equation*}
        \delta < \norm{y}_Y = \norm{K x^\dag}_Y = \norm{K|K|^\nu w}_Y \leq \norm{K}_{\linop(X,Y)}^{\nu+1}\rho = \kappa^{\frac{\nu+1}2}\rho
    \end{equation*}
    implies that $\hat\alpha\leq\kappa$. \Cref{lem:spektral:fehler} and condition \eqref{eq:spektral:ordnung_omega} for $\nu+1\leq\nu_0$ together with \eqref{eq:spektral:heuristic_prop} yield that
    \begin{equation*}
        \norm{Kx_{\hat\alpha}^\delta-y^\delta}_Y \leq \norm{Kx_{\hat\alpha}-y}_Y + \norm{y-y^\delta - K(x_{\hat\alpha}-x_{\hat\alpha}^\delta)}_Y \leq C_{\nu+1}\hat\alpha^{\frac{\nu+1}2}\rho + C_r\delta = (C_{\nu+1}+C_r)\delta
    \end{equation*}
    and hence that
    \begin{equation}\label{eq:spektral:heuristic2}
        \Psi(\alpha^*)\leq \Psi(\hat \alpha) \leq (C_{\nu+1}+C_r)\delta\hat\alpha^{-\frac12} = (C_{\nu+1}+C_r)\delta^{\frac{\nu}{\nu+1}} \rho^{\frac{1}{\nu+1}}.
    \end{equation}
    We thus obtain that
    \begin{equation*}
        \norm{x_{\alpha^*}^\delta - x_{\alpha^*}}_X\leq C_\phi(C_{\nu+1}+C_r) \frac\delta{\delta^*}\delta^{\frac{\nu}{\nu+1}} \rho^{\frac{1}{\nu+1}}.
    \end{equation*}
    Together with \eqref{eq:spektral:heuristic1}, this shows the claimed estimate.
\end{proof}

Hence the Hanke--Raus rule would be order optimal if $\delta^* \approx \delta$.
Conversely, the rule would fail if $\alpha^*=0$ or $\delta^*=0$ occurred.
In the latter case, $y^\delta\in \calR(K)$, and the unboundedness of $K^\dag$ would imply that $\norm{K^\dag y^\delta - K^\dag y}_X$ could be arbitrarily large. We thus need to exclude this case in order to show error estimates. For example, we can assume that there exists an $\eps>0$ such that
\begin{equation}
    \label{eq:spektral:heuristic:noise_ass}
    y^\delta \in \calN_\eps :=  \setof{y+\eta \in Y}{\norm{(\Id-P_{\overline{\calR}})\eta}_Y \geq \eps \norm{\eta}_Y} ,
\end{equation}
where $P_{\overline{\calR}}$ denotes the orthogonal projection onto $\overline{\calR(K)}$.
Intuitively, this means that the noisy data $y^\delta$ cannot be arbitrarily close to $\overline{\calR(K)}$ (which therefore cannot be dense for the following result to apply).
Restricted to such data, the Hanke--Raus rule indeed leads to a convergent regularization method.
\begin{theorem}
    Let $\{\phi_\alpha\}_{\alpha>0}$ be a filter with qualification $\nu_0>0$ satisfying \eqref{eq:spektral:ordnung_phi} as well as \eqref{eq:spektral:ordnung_omega} for all $\nu\in(0,\nu_0]$. Assume further that \eqref{eq:spektral:heuristic:noise_ass} holds and that $\Psi$ has a minimizer $\alpha^*\in(0,\kappa]$. Then for every $x^\dag\in X_{\nu,\rho}\setminus\{0\}$ with $\nu\in(0,\nu_0-1]$ and $\rho>0$,
    \begin{equation*}
        \lim_{\delta\to 0}\sup_{y^\delta\in \calN_\eps,\ \norm{y^\delta-Kx^\dag}_Y=\delta} \norm{x_{\alpha^*}^\delta - x^\dag}_X = 0.
    \end{equation*}
\end{theorem}
\begin{proof}
    Let $y\in\calR(K)$ and $y^\delta\in\calN_\eps$ with $\norm{y^\delta-y}_Y=\delta$. Since $\Id-P_{\overline{\calR}}$ is an orthogonal projection and therefore has operator norm $1$, we have for all $\alpha>0$ that
    \begin{equation}\label{eq:spektral:heuristic:conv1}
        \begin{aligned}[t]
            \norm{Kx_{\alpha}^\delta - y^\delta}_Y
            &\geq \norm{(\Id-P_{\overline{\calR}})(Kx_{\alpha}^\delta - y^\delta)}_Y = \norm{(\Id-P_{\overline{\calR}})y^\delta}_Y
            \\
            &= \norm{(\Id-P_{\overline{\calR}})(y^\delta-y)}_Y \geq \eps \norm{y^\delta -y}_Y \\
            &= \eps \delta >0.
        \end{aligned}
    \end{equation}
    In particular, it follows from \eqref{eq:spektral:heuristic:conv1} that
    \begin{equation*}
        \delta^* = \norm{Kx_{\alpha^*}^\delta - y^\delta}_Y \geq \eps \delta >0.
    \end{equation*}
    We thus obtain from \cref{thm:spektral:hankeraus:fehler} and the estimate $\delta\leq \eps^{-1}\delta^*$ that
    \begin{equation*}
        \norm{x_{\alpha^*}^\delta - x^\dag}_X \leq C_\eps (\delta^*)^{\frac{\nu}{\nu+1}}\rho^{\frac{1}{\nu+1}}
    \end{equation*}
    for some constant $C_\eps>0$.
    It thus suffices to show that $\delta\to 0$ implies that $\delta^*\to 0$ as well.
    But this follows from $\alpha^*\leq\kappa$ and \eqref{eq:spektral:heuristic2} (which holds for $\delta<\norm{y}_Y$), since as $\delta \to 0$, we have that
    \begin{equation*}
        \delta^* = \norm{Kx_{\alpha^*}^\delta - y^\delta}_Y = \sqrt{\alpha^*}\Psi(\alpha^*)
        \leq \sqrt{\kappa}\Psi(\alpha^*) \leq \sqrt{\kappa}(C_{\nu+1}+C_r)\delta^{\frac{\nu}{\nu+1}}\rho^{\frac1{\nu+1}}
        \to 0 .
        \qedhere
    \end{equation*}
\end{proof}

Under similar assumptions (and with more effort), it is also possible to show order optimality of the Hanke--Raus rule as well as of related minimization-based heuristic choice rules; see \cite{Kindermann}.

\chapter{Tikhonov regularization}\label{chap:tikhonov}

Due to its central role in the theory and practice of inverse problems, we again consider in more detail Tikhonov regularization, which corresponds to the filter
\begin{equation*}
    \phi_\alpha(\lambda) = \frac{1}{\lambda+\alpha}.
\end{equation*}
We get to the point quickly since we are well prepared. As we have already noted in \cref{ex:spektral}\,(ii), the filter $\phi_\alpha$ is continuous, converges to $\frac1\lambda$ as $\alpha\to0$, is uniformly bounded by $\alpha^{-1}$, and satisfies
\begin{equation*}
    \lambda\phi_\alpha(\lambda) = \frac{\lambda}{\lambda+\alpha} < 1=: C_\phi \qquad \text{for all }\alpha >0.
\end{equation*}
By \cref{thm:spektral:konvergenz}, the operator $R_\alpha = \phi_\alpha(K^*K)K^*$ is therefore a regularization, satisfies by \cref{lem:spektral:beschraenkt}
\begin{equation*}
    \norm{R_\alpha}_{\linop(Y,X)} \leq \frac1{\sqrt{\alpha}},
\end{equation*}
and by \cref{thm:regularisierung:apriori_char} leads together with the a priori choice rule $\alpha(\delta) = \delta$ to a convergent regularization method.

To show convergence rates, we apply \cref{thm:spektral:apriori} (for a priori choice rules) and \cref{thm:spektral:morozov} (for the Morozov discrepancy principle). First, since $\phi_\alpha(\lambda)\leq \alpha^{-1} = C_\phi\alpha^{-1}$ for all $\alpha>0$, the condition \eqref{eq:spektral:ordnung_phi} is satisfied.
Furthermore,
\begin{equation*}
    r_\alpha(\lambda) = 1-\lambda \phi_\alpha(\lambda) = \frac{\alpha}{\lambda + \alpha}\leq 1=: C_r \qquad\text{for all }\alpha >0, \lambda\in(0,\kappa].
\end{equation*}
To show the second condition \eqref{eq:spektral:ordnung_omega}, we have to estimate
\begin{equation*}
    \omega_\nu(\alpha) = \sup_{\lambda\in(0,\kappa]}\lambda^{\nu/2} |r_\alpha(\lambda)| = \sup_{\lambda\in(0,\kappa]} \frac{\lambda^{\nu/2}\alpha}{\lambda+\alpha} =: \sup_{\lambda\in(0,\kappa]} h_\alpha(\lambda)
\end{equation*}
by $C_\nu\alpha^{\nu/2}$ for a constant $C_\nu>0$.
To do this, we consider $h_\alpha(\lambda)$ for fixed $\alpha>0$ as a function of $\lambda$ and compute
\begin{equation*}
    h_\alpha'(\lambda) = \frac{\alpha\frac\nu2\lambda^{\nu/2-1}(\lambda + \alpha) -  \alpha\lambda^{\nu/2}}{(\lambda+\alpha)^2} =  \frac{\alpha\lambda^{\nu/2-1}}{(\lambda+\alpha)^2}\left(\frac\nu2 \alpha + \left(\frac\nu2 -1 \right)\lambda\right).
\end{equation*}
For $\nu\geq2$, the function $h_\alpha(\lambda)$ is therefore increasing, and the maximum over all $\lambda\in(0,\kappa]$ is attained at $\lambda^* := \kappa$. In this case,
\begin{equation*}
    \omega_\nu(\alpha) = h_\alpha(\kappa) = \frac{\alpha\kappa^{\nu/2}}{\kappa + \alpha} \leq \kappa^{\nu/2-1} \alpha.
\end{equation*}
We thus obtain the desired estimate (only) for $\nu=2$.

For $\nu\in(0,2)$, we can compute the root of $h_\alpha'(\lambda)$ as $\lambda^* := \frac{\alpha\frac\nu2}{1-\frac\nu2}$. There, $h_\alpha''(\lambda^*)<0$, which yields for all $\alpha>0$ that
\begin{equation*}
    \omega_\nu(\alpha) = h_\alpha(\lambda^*) = \frac{\alpha\left(\alpha\frac\nu2\left(1-\frac\nu2\right)^{-1}\right)^{\nu/2}}{\alpha+\alpha\frac\nu2\left(1-\frac\nu2\right)^{-1}} \leq \left(\frac\nu2\left(1-\frac\nu2\right)^{-1}\right)^{\nu/2} \alpha^{\nu/2}
\end{equation*}
and hence the desired estimate.

Tikhonov regularization thus has at least (and, as we will show, at most) qualification $\nu_0=2$. The corresponding order optimality for a priori and a posteriori choice rules now follows easily from \cref{thm:spektral:apriori} and \cref{thm:spektral:morozov}, respectively.
\begin{corollary}\label{cor:tikhonov:rate}
    For all $\nu\in(0,2]$, Tikhonov regularization together with the parameter choice rule
    \begin{equation*}
        c \left(\tfrac\delta\rho\right)^{\frac2{\nu+1}} \leq \alpha(\delta) \leq C  \left(\tfrac\delta\rho\right)^{\frac2{\nu+1}} \qquad\text{for } C>c>0
    \end{equation*}
    is an order optimal regularization method. In particular, for $\alpha \sim \delta^{2/3}$,
    \begin{equation*}
        \norm{x_{\alpha(\delta)}^\delta-x^\dag}_X \leq c \delta^{\frac{2}{3}} \qquad \text{for all }x^\dag\in \calR(K^*K)\text{ and } y^\delta\in B_\delta(Kx^\dag).
    \end{equation*}
\end{corollary}
\begin{corollary}
    For all $\nu\in(0,1]$ and $\tau>1$, Tikhonov regularization together with the parameter choice rule $\alpha(\delta,y^\delta):=\infty$ (i.e., $x^\delta_{\alpha(\delta,y^\delta)} = 0$) if $\norm{y^\delta}_Y\leq \tau\delta$ and
    \begin{equation*}
        \norm{Kx^\delta_{\alpha(\delta,y^\delta)}-y^\delta}_Y\leq \tau \delta < \norm{Kx^\delta_{\alpha}-y^\delta}_Y \qquad\text{for all } \alpha>\alpha(\delta,y^\delta)
    \end{equation*}
    otherwise is an order optimal regularization method. In particular,
    \begin{equation*}
        \norm{x_{\alpha(\delta,y^\delta)}^\delta-x^\dag}_X \leq c \delta^{\frac12} \qquad \text{for all }x^\dag\in  \calR(K^*)\text{ and } y^\delta\in B_\delta(Kx^\dag).
    \end{equation*}
\end{corollary}
In fact, the qualification cannot be larger than $2$; Tikhonov regularization thus \term[saturation]{saturates} in contrast to, e.g., the truncated singular value decomposition.
To show this, we first derive the alternative characterization that was promised in \cref{ex:spektral}\,(ii).
\begin{lemma}\label{lem:tikhonov:normalen}
    Let $y\in Y$ and $\alpha>0$. Then $x = x_\alpha :=  R_\alpha y$ if and only if
    \begin{equation}
        \label{eq:tikhonov:normalen}
        (K^*K + \alpha \Id)x_\alpha = K^*y.
    \end{equation}
\end{lemma}
\begin{proof}
    We use the singular system $\{(\sigma_n,u_n,v_n)\}_{n\in\N}$ of $K$ to obtain
    \begin{align*}
        \alpha x_\alpha  &= \sum_{n\in\N} \alpha \frac{\sigma_n}{\sigma_n^2+\alpha} \inner{y}{u_n}_Y v_n
        \intertext{as well as}
        K^*Kx_\alpha &= \sum_{n\in\N} \frac{\sigma_n}{\sigma_n^2+\alpha} \inner{y}{u_n}_Y K^*Kv_n\\
        &=\sum_{n\in\N} \sigma_n^2 \frac{\sigma_n}{\sigma_n^2+\alpha} \inner{y}{u_n}_Y v_n.
    \end{align*}
    This implies that
    \begin{equation*}
        (K^*K + \alpha \Id)x_\alpha =  \sum_{n\in\N} \sigma_n  \inner{y}{u_n}_Y v_n = K^* y.
    \end{equation*}

    Conversely, let $x\in X$ be a solution of \eqref{eq:tikhonov:normalen}. Inserting the representation
    \begin{equation}\label{eq:tikhonov:normal1}
        x = \sum_{n\in\N} \inner{x}{v_n}_Xv_n + P_{\calN}x
    \end{equation}
    into \eqref{eq:tikhonov:normalen} then yields
    \begin{equation*}
        \sum_{n\in\N} (\sigma_n^2+\alpha)  \inner{x}{v_n}_X v_n + \alpha  P_{\calN}x=  (K^*K + \alpha \Id)x = K^* y =  \sum_{n\in\N} \sigma_n\inner{y}{u_n}_Y v_n.
    \end{equation*}
    Since $\{v_n\}_{n\in\N}$ is an orthonormal basis of $\overline{\calR(K^*)}=\calN(K)^\bot$, we must have $P_{\calN}x = 0$.
    Equating coefficients then shows that
    \begin{equation*}
        \inner{x}{v_n}_X = \frac{\sigma_n}{\sigma_n^2+\alpha} \inner{y}{u_n}_Y \qquad\text{for all }n\in \N.
    \end{equation*}
    Inserting this into \eqref{eq:tikhonov:normal1} in turn yields
    \begin{equation*}
        x = \sum_{n\in\N} \inner{x}{v_n}_Xv_n = \sum_{n\in\N} \frac{\sigma_n}{\sigma_n^2+\alpha} \inner{y}{u_n}_Y v_n = x_\alpha,
    \end{equation*}
    i.e., $x_\alpha$ is the unique solution of \eqref{eq:tikhonov:normalen}.
\end{proof}

The practical value of the characterization \eqref{eq:tikhonov:normalen} cannot be emphasized enough: Instead of a singular value decomposition, it suffices to compute the solution of a \emph{well-posed} linear equation (for a selfadjoint positive definite operator), which can be done using standard methods.

We now show that in general there cannot be an a priori choice rule for which the regularization error
$\norm{x_{\alpha(\delta)}^\delta-x^\dag}_X$ tends to zero faster than $\delta^{2/3}$.
\begin{theorem}
    Let $K\in\calK(X,Y)$ have infinite-dimensional range and let $y\in\calR(K)$. If there exists an a priori parameter choice rule $\alpha$ with $\lim_{\delta\to 0}\alpha(\delta) = 0$ such that
    \begin{equation}\label{eq:tikhonov:sat_ass}
        \lim_{\delta\to 0}\sup_{y^\delta\in B_\delta(y)}\norm{x_{\alpha(\delta)}^\delta-x^\dag}_X\delta^{-\frac23} = 0,
    \end{equation}
    then $x^\dag =0$.
\end{theorem}
\begin{proof}
    Assume to the contrary that $x^\dag\neq 0$.
    We first show that the given assumptions imply that $\alpha(\delta)\delta^{-2/3}\to 0$. For this, we use the characterization \eqref{eq:tikhonov:normalen} for $x_\alpha^\delta$ and $y^\delta$ to write
    \begin{equation*}
        \left(K^*K+\alpha(\delta) \Id\right)\left(x_{\alpha(\delta)}^\delta-x^\dag\right) = K^*y^\delta - K^* y -\alpha(\delta) x^\dag.
    \end{equation*}
    Together with $\kappa = \norm{K^*K}_{\linop(X,X)} = \norm{K^*}_{\linop(Y,X)}^2$, this implies that
    \begin{equation*}
        |\alpha(\delta)| \norm{x^\dag}_X \leq \sqrt{\kappa}\delta + (\alpha(\delta) + \kappa)\norm{x_{\alpha(\delta)}^\delta-x^\dag}_X.
    \end{equation*}
    Multiplying this with $\delta^{-2/3}$ and using the assumption \eqref{eq:tikhonov:sat_ass} as well as $x^\dag \neq 0$ then yields that
    \begin{equation*}
        |\alpha(\delta)|\delta^{-2/3}\leq \norm{x^\dag}_X^{-1} \left(\sqrt{\kappa}\delta^{\frac13} + (\alpha(\delta)+\kappa)\norm{x_{\alpha(\delta)}^\delta-x^\dag}_X\delta^{-\frac23}\right)\to 0 .
    \end{equation*}

    We now construct a contradiction.
    Let $\{(\sigma_n,u_n,v_n)\}_{n\in\N}$ be a singular system of $K$ and define
    \begin{equation*}
        \delta_n :=  \sigma_n^3\quad\text{and}\quad y_n :=  y + \delta_n u_n,\qquad n\in\N,
    \end{equation*}
    such that $\norm{y_n-y}_Y = \delta_n\to 0$ as $n\to\infty$. Furthermore, setting $\alpha_n :=  \alpha(\delta_n)$, we have that
    \begin{equation*}
        \begin{aligned}
            x_{\alpha_n}^{\delta_n} - x^\dag &= (x_{\alpha_n}^{\delta_n} - x_{\alpha_n}) + (x_{\alpha_n} - x^\dag)\\
            &= R_{\alpha_n} (y_n - y) +  (x_{\alpha_n} - x^\dag)\\
            &= \sum_{m\in\N}\frac{\sigma_m}{\sigma_m^2+\alpha_n}\inner{\delta_n u_n}{u_m}_Y v_m + (x_{\alpha_n} - x^\dag)\\
            &= \frac{\delta_n\sigma_n}{\sigma_n^2+\alpha_n} v_n + (x_{\alpha_n} - x^\dag).
        \end{aligned}
    \end{equation*}
    Together with the assumption \eqref{eq:tikhonov:sat_ass} for $y^\delta = y_n$ as well as for $y^\delta = y$, this implies that
    \begin{equation*}
        \frac{\sigma_n\delta_n^{1/3}}{\sigma_n^2+\alpha_n} \leq \norm{x_{\alpha_n}^{\delta_n} - x^\dag}_X\delta_n^{-2/3} + \norm{x_{\alpha_n} - x^\dag}_X\delta_n^{-2/3}\to 0 \qquad\text{as }n\to\infty.
    \end{equation*}
    On the other hand, $\sigma_n = \delta_n^{1/3}$ and $\alpha_n\delta_n^{-2/3}\to 0$ imply that
    \begin{equation*}
        \frac{\sigma_n\delta_n^{1/3}}{\sigma_n^2+\alpha_n} = \frac{\delta_n^{2/3}}{\delta_n^{2/3} + \alpha_n} = \frac{1}{1+\alpha_n\delta_n^{-2/3}}\to 1 \qquad\text{as }n\to \infty
    \end{equation*}
    and hence the desired contradiction.
\end{proof}

\bigskip

Comparing the characterization \eqref{eq:tikhonov:normalen} with the normal equations \eqref{eq:inverse:normal} suggests that Tikhonov regularization also has a minimization property. This is indeed the case.
\begin{theorem}\label{thm:tikhonov:funktional}
    Let $y\in Y$ and $\alpha>0$. Then $x_\alpha := R_\alpha y$ is the unique minimizer of the \term{Tikhonov functional}
    \begin{equation}\label{eq:tikhonov:funktional}
        J_\alpha(x):= \frac12\norm{Kx-y}_Y^2 + \frac\alpha2 \norm{x}_X^2.
    \end{equation}
\end{theorem}
\begin{proof}
    A minimizer $\bar x\in X$ of $J_\alpha$ is defined as satisfying $J_\alpha(\bar x)\leq J_\alpha(x)$ for all $x\in X$. We therefore take the difference of functional values for arbitrary $x\in X$ and for the solution $x_\alpha$ of \eqref{eq:tikhonov:normalen} and rearrange the inner products to obtain
    \begin{equation*}
        \begin{aligned}
            J_\alpha(x)-J_\alpha(x_\alpha) &= \frac12\inner{Kx-y}{Kx-y}_Y+ \frac\alpha2\inner{x}{x}_X  \\
            \MoveEqLeft[-1]
            - \frac12\inner{Kx_\alpha-y}{Kx_\alpha-y}_Y- \frac\alpha2\inner{x_\alpha}{x_\alpha}_X\\
            &= \frac12\norm{Kx-Kx_\alpha}_Y^2 + \frac\alpha2 \norm{x-x_\alpha}_X^2
            + \inner{K^*(Kx_\alpha -y)+\alpha x_\alpha}{x-x_\alpha}_X\\
            &= \frac12\norm{Kx-Kx_\alpha}_Y^2 + \frac\alpha2 \norm{x-x_\alpha}_X^2 \\
            &\geq 0,
        \end{aligned}
    \end{equation*}
    where we have used \eqref{eq:tikhonov:normalen} in the last equality.
    Hence, $x_\alpha$ is a minimizer of $J_\alpha$.

    Conversely, if $J_\alpha(x)-J_\alpha(\bar x)\geq 0$ for all $x\in X$, we in particular have for $x = \bar x + tz$ with arbitrary $t>0$ and $z\in X$ that
    \begin{equation*}
        0\leq J_\alpha(\bar x + tz)-J_\alpha(\bar x) = \frac{t^2}{2}\norm{Kz}_Y^2 + \frac{t^2\alpha}{2}\norm{z}_X^2 + t\inner{K^*(K\bar x -y)  + \alpha\bar x}{z}_X.
    \end{equation*}
    Dividing by $t>0$ and passing to the limit $t\to 0$ then yields
    \begin{equation*}
        \inner{K^*(K\bar x -y)  + \alpha\bar x}{z}_X \geq 0.
    \end{equation*}
    Since $z\in X$ was arbitrary, this can only hold if $K^*K \bar x+\alpha \bar x =K^*y$. As $x_\alpha$ is the unique solution of \eqref{eq:tikhonov:normalen}, we obtain $\bar x = x_\alpha$. Hence, $x_\alpha$ is in fact the unique minimizer of \eqref{eq:tikhonov:funktional}.
\end{proof}

The characterization of Tikhonov regularization as minimization of the functional \eqref{eq:tikhonov:funktional} furthermore yields another connection to the minimum norm solution $x^\dag$: Instead of insisting on a least squares solution, whose norm need not be bounded for $y\notin\calD(K^\dag)$, we look for an approximation that minimizes the (squared) residual norm $\norm{Kx-y}_Y^2$ \emph{together} with the (squared) norm $\norm{x}_X^2$.%
\footnote{This is also the form in which this regularization was introduced by \href{http://www-history.mcs.st-andrews.ac.uk/Biographies/Tikhonov.html}{Andre\u{\i} Nikolaevich Tikhonov}, a prominent Russian mathematician of the 20th century; see \cite{Tikhonov2,Tikhonov1}.}
Here the regularization parameter $\alpha$ determines the trade-off: the smaller the noise level $\delta$, the more importance one can put on the minimization of the residual (i.e., the smaller $\alpha$ can be chosen).
Conversely, a larger noise level requires putting more weight on minimizing the \term[Tikhonov functional!penalty term]{penalty term} $\norm{x}_X^2$ (and hence choosing a larger $\alpha$) in order to obtain a stable approximation.

In addition, this characterization can be used to derive monotonicity properties of the
\term[function!value]{value functions}
\begin{align*}
    f(\alpha) &:= \frac12\norm{Kx_{\alpha}^\delta - y^\delta}_Y^2,\qquad g(\alpha) := \frac12\norm{x_\alpha^\delta}_X^2,
    \shortintertext{and}
    j(\alpha) &:= J_\alpha(x_\alpha^\delta)= f(\alpha)+\alpha g(\alpha).
\end{align*}
\begin{lemma}\label{lem:tikhonov:monoton}
    The value functions $f$ and $g$ are monotone in the sense that for all $\alpha_1,\alpha_2>0$,
    \begin{align}
        \left(f(\alpha_1) - f(\alpha_2)\right)(\alpha_1-\alpha_2) &\geq 0,\label{eq:tikhonov:monoton_f}\\
        \left(g(\alpha_1) - g(\alpha_2)\right)(\alpha_1-\alpha_2) &\leq 0.\label{eq:tikhonov:monoton_g}
    \end{align}
\end{lemma}
\begin{proof}
    The minimization property of $x_{\alpha_1}^\delta$ for $J_{\alpha_1}$ and of $x_{\alpha_2}^\delta$ for $J_{\alpha_2}$ implies that
    \begin{align*}
        f(\alpha_1) + \alpha_1 g(\alpha_1)  &\leq f(\alpha_2) + \alpha_1 g(\alpha_2),\\
        f(\alpha_2) + \alpha_2 g(\alpha_2)  &\leq f(\alpha_1) + \alpha_2 g(\alpha_1).
    \end{align*}
    Adding these inequalities and rearranging immediately yields \eqref{eq:tikhonov:monoton_g}.
    Dividing the first inequality by $\alpha_1>0$, the second by $\alpha_2>0$, and adding both yields
    \begin{equation*}
        \frac1{\alpha_1}\left(f(\alpha_1)-f(\alpha_2)\right)\leq \frac1{\alpha_2}\left(f(\alpha_1)-f(\alpha_2)\right).
    \end{equation*}
    Multiplying by $\alpha_1\alpha_2>0$ and rearranging then yields \eqref{eq:tikhonov:monoton_f}.
\end{proof}
As expected, the residual norm is decreasing and the norm of $x_\alpha^\delta$ is increasing as $\alpha\to 0$.
We next consider for the value function $j$ the one-sided difference quotients
\begin{align*}
    D^+ j(\alpha) &:=  \lim_{t\to 0^+} \frac{j(\alpha+t)-j(\alpha)}{t},\\
    D^- j(\alpha) &:=  \lim_{t\to 0^-} \frac{j(\alpha+t)-j(\alpha)}{t}.
\end{align*}
\begin{lemma}
    For all $\alpha>0$,
    \begin{equation*}
        \begin{array}{rcccl}
            D^+ j(\alpha) &\leq & g(\alpha) & \leq& D^- j(\alpha),\\
            j(\alpha) - \alpha D^- j(\alpha) &\leq & f(\alpha)& \leq& j(\alpha) - \alpha D^+j(\alpha).
        \end{array}
    \end{equation*}
\end{lemma}
\begin{proof}
    For any $\alpha,\tilde\alpha>0$, the minimization property for $j$ yields that
    \begin{equation*}
        j(\tilde\alpha) = f(\tilde\alpha) + \tilde\alpha g(\tilde\alpha)\leq f(\alpha) + \tilde\alpha g(\alpha).
    \end{equation*}
    Hence,
    \begin{equation*}
        \begin{aligned}
            j(\alpha) - j(\tilde\alpha)
            &=  f(\alpha) + \alpha g(\alpha) -  f(\tilde\alpha) - \tilde\alpha g(\tilde\alpha)\\
            &\geq  f(\alpha) + \alpha g(\alpha) -  f(\alpha) - \tilde\alpha g(\alpha)\\
            &=(\alpha - \tilde \alpha) g(\alpha),
        \end{aligned}
    \end{equation*}
    which implies for $\tilde\alpha:= \alpha+t>\alpha$ with $t>0$ that
    \begin{equation*}
        \frac{j(\alpha+t) - j(\alpha)}{t} \leq g(\alpha).
    \end{equation*}
    Passing to the limit $t\to 0$ thus shows that $D^+j(\alpha)\leq g(\alpha)$. The corresponding inequality for $D^-j(\alpha)$ follows analogously with $t<0$.

    The remaining inequalities follow from this together with the definition of $j$; for example, using
    \begin{equation*}
        j(\alpha) = f(\alpha) + \alpha g(\alpha) \leq f(\alpha) + \alpha D^-j(\alpha),
    \end{equation*}
    and rearranging.
\end{proof}

By one of Lebesgue's theorems (see \cite[Theorem~V.17.12]{HewittStromberg}), a monotone function is differentiable almost everywhere (i.e., $D^-f \neq D^+f$ on at most a set of Lebesgue measure zero).
Hence, $f$ and $g$ and therefore also $j = f+\alpha g$ are differentiable almost everywhere, and we obtain the following expression for the derivative of the latter.
\begin{corollary}
    For almost all $\alpha>0$, the value function $j$ is differentiable with
    \begin{equation*}
        j'(\alpha) = g(\alpha).
    \end{equation*}
\end{corollary}
This characterization can be useful for example when implementing minimization-based heuristic parameter choice rules.

Furthermore, \cref{thm:tikhonov:funktional} suggests a new interpretation of the simplest source condition $x^\dag\in X_1=\calR(K^*)$. Since the minimizer of \eqref{eq:tikhonov:funktional} does not change when dividing the Tikhonov functional by $\alpha>0$, the minimizer $x_\alpha^\delta$ is also a minimizer of
\begin{equation}\label{eq:tikhonov:funkt_ref}
    \min_{x\in X} \frac{1}{2\alpha}\norm{Kx-y^\delta}_Y^2 + \frac12\norm{x}_X^2.
\end{equation}
Now we want $x_\alpha^\delta\to x^\dag$ as $\delta\to 0$ and $\alpha\to 0$. Formally passing to the limits in \eqref{eq:tikhonov:funkt_ref}, i.e., first replacing $y^\delta$ with $y\in\calR(K)$ and then letting $\alpha\to0$, we see that the limit functional can only have a finite minimum in some $\bar x$ if $K\bar x=y$. The limit functional is therefore given by
\begin{equation}
    \label{eq:tikhonov:funkt_const}
    \min_{x\in X,\ Kx=y}  \frac12\norm{x}_X^2.
\end{equation}
We again proceed formally. Introducing the Lagrange multiplier $p\in Y$, we can write \eqref{eq:tikhonov:funkt_const} as the unconstrained saddle-point problem
\begin{equation*}
    \min_{x\in X}\max_{p\in Y} \frac12\norm{x}_X^2 - \inner{p}{Kx-y}_Y.
\end{equation*}
For $(\bar x,\bar p)\in X\times Y$ to be a saddle point, the partial derivatives with respect to both $x$ and $p$ have to vanish, leading to the conditions
\begin{equation*}
    \left\{
        \begin{aligned}
            \bar x &= K^*\bar p,\\
            K\bar x &= y.
        \end{aligned}
    \right.
\end{equation*}
But for $y\in \calR(K)$, the solution of \eqref{eq:tikhonov:funkt_const} describes exactly the minimum norm solution $x^\dag$, i.e., $\bar x = x^\dag$.
The existence of a Lagrange multiplier $\bar p$ with $x^\dag = K^*\bar p$ is therefore equivalent to the source condition $x^\dag \in \calR(K^*)$.
(Since $K^*$ need not be surjective, this is a non-trivial assumption.)
Intuitively, this makes sense: If we want to approximate $x^\dag$ by a sequence of minimizers $x_\alpha^\delta$, the limit $x^\dag$ should itself be a minimizer (of an appropriate limit problem).

Finally, the interpretation of Tikhonov regularization as minimizing a functional can -- in contrast to the construction via the singular value decomposition -- be extended to \emph{nonlinear} operator equations as well as to equations in Banach spaces. It can further be generalized by replacing the squared norms by other \term[Tikhonov functional!discrepancy term]{discrepancy} and \term[Tikhonov functional!penalty term]{penalty functionals}. Of course, this also entails generalized source conditions. We will return to this in \cref{chap:tikhonov-nl}.

\chapter{Landweber regularization}\label{chap:landweber}

The usual starting point for deriving Landweber regularization is the characterization from \cref{thm:inverse:normalen} of the minimum norm solution as the solution $x\in\calN(K)^\bot$ of the normal equations \eqref{eq:inverse:normal}. These can be written equivalently for any $\omega>0$ as the fixed-point equation
\begin{equation*}
    x = x - \omega(K^*Kx-K^*y) = x + \omega K^*(y-Kx).
\end{equation*}
The corresponding fixed-point iteration -- also known as \term{Richardson iteration}%
\footnote{This method for the solution of linear systems of equations traces back to \href{http://www-history.mcs.st-andrews.ac.uk/Biographies/Richardson.html}{Lewis Fry Richardson}. He also proposed in 1922 the modern method of weather prediction by numerical simulation. (His own first attempt in 1910 -- by hand! -- was correct in principle but gave wrong results due to noisy input data. Weather prediction is an ill-posed problem!)}%
-- is
\begin{equation}\label{eq:landweber:iteration}
    x_n = x_{n-1} + \omega K^*(y-Kx_{n-1}), \qquad n\in \N,
\end{equation}
for some $x_0\in X$. Here we only consider $x_0=0$ for the sake of simplicity.
The Banach fixed-point theorem ensures that this iteration converges to a solution of the normal equations if $y\in \calR(K)$ and $\norm{\Id - \omega K^*K}_{\linop(X,X)}<1$.
Since $x_0 = 0\in\calR(K^*)$, an induction argument shows that $x_n\in\calR(K^*)\subset\calN(K)^\bot$ for all $n\in\N$, and therefore $x_n\to x^\dag$. If $y^\delta\notin\calR(K)$, however, no convergence can be expected.
The idea is therefore to stop the iteration early, i.e., take $x_m$ for an appropriate $m\in\N$ as the regularized approximation. The \term[index, stopping]{stopping index} $m\in\N$ thus plays the role of the regularization parameter here, which fits into the framework of \cref{chap:spektral} if we set $\alpha = \frac1m>0$.%
\footnote{This method was first proposed for the solution of ill-posed operator equations by Lawrence Landweber. In \cite{Landweber}, he shows the convergence for $y\in\calR(K)$; otherwise, he then writes,
\enquote{such a sequence may give useful successive approximations}.}

Performing $m$ steps of the iteration \eqref{eq:landweber:iteration} can be formulated as a spectral regularization. For this, we first derive a recursion-free characterization of the final iterate $x_m$.
\begin{lemma}
    If $x_0 =0$, then
    \begin{equation*}
        x_m = \omega \sum_{n=0}^{m-1}(\Id - \omega K^*K)^n K^*y\qquad\text{for all }m\in\N.
    \end{equation*}
\end{lemma}
\begin{proof}
    We proceed by induction. For $m=1$,
    \begin{equation*}
        x_1 = \omega K^*y = \omega (\Id - \omega K^*K)^0 K^*y.
    \end{equation*}
    Let now $m\in\N$ be arbitrary, and let the claim hold for $x_m$. Then
    \begin{equation*}
        \begin{split}
            x_{m+1} & =  x_{m} + \omega K^*(y-Kx_m)\\
            & = (\Id - \omega K^*K)x_m + \omega K^*y \\
            & =  (\Id - \omega K^*K)\left(\omega \sum_{n=0}^{m-1}(\Id - \omega K^*K)^n K^*y\right) + \omega K^*y\\
            & = \omega \sum_{n=0}^{m-1}(\Id - \omega K^*K)^{n+1} K^*y +  \omega (\Id - \omega K^*K)^0 K^*y\\
            & =  \omega \sum_{n=0}^{m}(\Id - \omega K^*K)^{n} K^*y.
            \qedhere
        \end{split}
    \end{equation*}
\end{proof}
Performing $m$ steps of the \term{Landweber iteration}\index{regularization method!Landweber} \eqref{eq:landweber:iteration} is thus equivalent to applying a linear operator, i.e.,
\begin{equation*}
    x_m = \phi_m(K^*K)K^*y
\end{equation*}
for
\begin{equation*}
    \phi_m(\lambda) = \omega \sum_{n=0}^{m-1}(1 - \omega \lambda)^n = \omega \frac{1-(1-\omega\lambda)^{m}}{1-(1-\omega\lambda)} = \frac{1-(1-\omega\lambda)^{m}}{\lambda}.
\end{equation*}
Apart from the notation $\phi_m$ instead of $\phi_\alpha$ for $\alpha=\frac1m$ (i.e., considering $m\to\infty$ instead of $\alpha\to 0$), this is exactly the filter from \cref{ex:spektral}\,(iii).
\begin{theorem}
    For any $\omega\in (0,\kappa^{-1})$, the family $\{\phi_m\}_{m\in\N}$ defines a regularization $\{R_m\}_{m\in\N}$ with $R_m :=  \phi_m(K^*K)K^*$.
\end{theorem}
\begin{proof}
    We only have to show that $\phi_m(\lambda)\to \frac1\lambda$ as $m\to\infty$ and that $\lambda\phi_m(\lambda)$ is uniformly bounded for all $m\in\N$.
    By the assumption on $\omega$, we have
    $0< 1-\omega\lambda<1$ for all $\lambda\in(0,\kappa]$, which yields $(1-\omega\lambda)^m\to 0$ as $m\to\infty$ as well as
    \begin{equation*}
        \lambda |\phi_m(\lambda)| = |1-(1-\omega\lambda)^m| \leq 1 =: C_\phi \qquad\text{for all } m\in\N \text{ and } \lambda\in(0,\kappa].
    \end{equation*}
    Hence $\{\phi_m\}_{m\in\N}$ is a regularizing filter, and the claim follows from \cref{thm:spektral:konvergenz}.
\end{proof}
Hence the Landweber iteration converges to a minimum norm solution $x^\dag$ as $m\to\infty$ if and only if  $y\in\calD(K^\dag)$; otherwise it diverges. It therefore suggests itself to choose the stopping index by the discrepancy principle\index{discrepancy principle!for Landweber iteration}\index{parameter choice rule!discrepancy principle!for Landweber iteration}: Pick $\tau>1$ and take $m(\delta,y^\delta)\in\N_0$ such that $x_m^\delta := R_my^\delta$ satisfies
\begin{equation}
    \label{eq:landweber:diskrepanz}
    \norm{Kx_{m(\delta,y^\delta)}^\delta - y^\delta}_Y \leq \tau \delta < \norm{Kx_{m}^\delta - y^\delta}_Y \qquad\text{for all }m<m(\delta,y^\delta).
\end{equation}
(This does not require any additional effort since the residual $y^\delta - Kx_m^\delta$ is computed as part of the iteration \eqref{eq:landweber:iteration}.)
Again, if $\norm{y^\delta}_Y \leq \tau\delta$, we take $m(\delta,y^\delta)=0$ and hence $x^\delta_{m(\delta,y^\delta)} = x_0 = 0$, which corresponds to \eqref{eq:parameter:morozov_infty} for $\alpha = \frac1m = \infty$; otherwise, the existence of such an $m(\delta,y^\delta)$ is guaranteed by \cref{thm:morozov} together with \cref{lem:spektral:beschraenkt_K}.

We now address convergence rates, where from now on we assume that $\omega\in(0,\kappa^{-1})$.
\begin{theorem}
    For all $\nu>0$ and $\tau>1$, the Landweber iteration \eqref{eq:landweber:iteration} together with the discrepancy principle \eqref{eq:landweber:diskrepanz} is an order optimal regularization method.
\end{theorem}
\begin{proof}
    We apply \cref{thm:spektral:morozov}, for which we verify the necessary conditions (following the convention $\alpha:=\frac1m$).
    First, due to $\omega\lambda<1$ Bernoulli's inequality yields that
    \begin{equation*}
        |\phi_m(\lambda)| = \frac{|1-(1-\omega\lambda)^m|}{\lambda} \leq \frac{|1-1+ m \omega\lambda|}{\lambda} = \omega m = \omega \alpha^{-1}\qquad \text{for all }\lambda \in(0,\kappa]
    \end{equation*}
    and hence that \eqref{eq:spektral:ordnung_phi} holds. (Clearly for $\omega \leq 1$; otherwise we can follow the proof of \cref{thm:spektral:morozov} and see that the additional constant $\omega$ only leads to a larger constant $C_2$.)

    Bernoulli's inequality further implies that $(1+x) \leq e^x$ and hence that
    \begin{equation*}
        r_m(\lambda) = 1-\lambda\phi_m(\lambda) = (1-\omega\lambda)^m \leq e^{-\omega\lambda m}\leq 1 =:  C_r \qquad\text{for all }m\in\N, \lambda\in(0,\kappa].
    \end{equation*}
    We now consider for fixed $\nu>0$ and $m\in\N$ the function $h_m(\lambda) :=  \lambda^{\nu/2}e^{-\omega\lambda m}$ and compute
    \begin{equation*}
        h_m'(\lambda) = \frac\nu2 \lambda^{\nu/2-1}e^{-\omega\lambda m} - \omega m  \lambda^{\nu/2}e^{-\omega\lambda m} = \lambda^{\nu/2-1}e^{-\omega\lambda m}\omega m\left(\frac{\nu}{2\omega m} - \lambda\right).
    \end{equation*}
    The root $\lambda^* = \frac{\nu}{2\omega m}$ of this derivative satisfies $h_m''(\lambda^*)<0$, and hence
    \begin{equation*}
        \sup_{\lambda\in(0,\kappa]} \lambda^{\nu/2} r_m(\lambda) \leq\sup_{\lambda\in(0,\infty)} h_m(\lambda) = h_m\left( \frac{\nu}{2\omega m}\right) = e^{-\nu/2} \left(\frac{\nu}{2\omega}\right)^{\nu/2}m^{-\nu/2} =:  C_\nu \alpha^{\nu/2}.
    \end{equation*}
    This shows that \eqref{eq:spektral:ordnung_omega} holds for all $\nu>0$. Landweber regularization thus has infinite qualification, and the claim follows for $\tau>C_r =1$ from \cref{thm:spektral:morozov}.
\end{proof}

We next study the monotonicity of the Landweber iteration.
\begin{theorem}\label{thm:landweber:res_monoton}
    Let $m\in \N$. If $Kx_m^\delta-y^\delta\neq 0$, then
    \begin{equation*}
        \norm{Kx_{m+1}^\delta -y^\delta}_Y < \norm{Kx_{m}^\delta -y^\delta}_Y.
    \end{equation*}
\end{theorem}
\begin{proof}
    The iteration \eqref{eq:landweber:iteration} implies that
    \begin{equation*}
        \begin{aligned}
            Kx_{m+1}^\delta - y^\delta &= K\left((\Id-\omega K^*K)x_m^\delta + \omega K^*y^\delta\right) - y^\delta \\
            &= (\Id -\omega KK^*)Kx_m^\delta - (\Id-\omega KK^*)y^\delta\\
            &= (\Id -\omega KK^*)(Kx_m^\delta-y^\delta)
        \end{aligned}
    \end{equation*}
    and hence due to $\omega < \kappa^{-1} = \sigma_1^{-2}\leq \sigma_n^{-2}$ for all $n\in\N$ that
    \begin{equation*}
        \begin{split}
            \norm{Kx_{m+1}^\delta - y^\delta}_Y^2 &= \sum_{n\in\N} (1-\omega\sigma_n^2)^2 \left|\inner{Kx_m^\delta-y^\delta}{u_n}_Y\right|^2 \\
            &< \sum_{n\in\N}  \left|\inner{Kx_m^\delta-y^\delta}{u_n}_Y\right|^2 \leq \norm{Kx_{m}^\delta - y^\delta}_Y^2.
            \qedhere
        \end{split}
    \end{equation*}
\end{proof}
The residual therefore always decreases as $m\to\infty$ (even though a least squares solution minimizing the residual will not exist for $y\notin\calD(K^\dag)$). For the error, this can be guaranteed only up to a certain step.
\begin{theorem}\label{thm:landweber:semikonvergenz}
    Let $m\in\N$. If
    \begin{equation*}
        \norm{Kx_m^\delta -y^\delta}_Y > 2\delta,
    \end{equation*}
    then
    \begin{equation*}
        \norm{x_{m+1}^{\delta}-x^\dag}_X < \norm{x_{m}^{\delta}-x^\dag}_X.
    \end{equation*}
\end{theorem}
\begin{proof}
    We again use the iteration to write with $\rho_m^\delta = y^\delta - Kx_m^\delta$ and $y=Kx^\dag$
    \begin{equation*}
        \begin{aligned}
            \norm{x_{m+1}^\delta - x^\dag}_X^2 &= \norm{x_m^\delta - x^\dag + \omega K^*(y^\delta - Kx_m^\delta)}_X^2 \\
            &= \norm{x_m^\delta - x^\dag}_X^2 - 2 \omega\inner{Kx^\dag - Kx_m^\delta}{ \rho_m^\delta}_Y + \omega^2\norm{K^*\rho_m^\delta}_X^2\\
            &= \norm{x_m^\delta - x^\dag}_X^2  + \omega\inner{\rho_m^\delta - 2 y + 2Kx_m^\delta}{\rho_m^\delta}_Y+ \omega\left(\omega\norm{K^*\rho_m^\delta}_X^2-\norm{\rho_m^\delta}_Y^2\right).
        \end{aligned}
    \end{equation*}
    We now have to show that the last two terms are negative.
    For the first term, we use the definition of $\rho_m^\delta$ and obtain by inserting $\rho_m^\delta= 2\rho_m^\delta-\rho_m^\delta = 2y^\delta -2Kx_m^\delta - \rho_m^\delta$ that
    \begin{equation*}
        \begin{aligned}
            \inner{\rho_m^\delta - 2 y + 2Kx_m^\delta}{\rho_m^\delta}_Y &= 2\inner{y^\delta -y}{ \rho_m^\delta}_Y - \norm{\rho_m^\delta}_Y^2\\
            &\leq 2\delta \norm{\rho_m^\delta}_Y - \norm{\rho_m^\delta}_Y^2\\
            &= \left(2\delta -\norm{Kx_m^\delta -y^\delta}_Y\right)\norm{\rho_m^\delta}_Y < 0
        \end{aligned}
    \end{equation*}
    since the term in parentheses is negative by assumption and $\norm{\rho_m^\delta}_Y>2\delta>0$.

    For the second term, we use $\omega < \kappa^{-1}$ and therefore that
    \begin{equation*}
        \omega\norm{K^*\rho_m^\delta}_X^2 \leq \omega\norm{K^*}_{\linop(Y,X)}^2\norm{\rho_m^\delta}^2_Y =\omega\kappa \norm{\rho_m^\delta}_Y^2 < \norm{\rho_m^\delta}_Y^2
    \end{equation*}
    and hence
    \begin{equation}\label{eq:landweber:res_spd}
        \omega\left(\omega\norm{K^*\rho_m^\delta}_X^2-\norm{\rho_m^\delta}_Y^2\right) < 0.
    \end{equation}
    Hence both terms are negative, and the claim follows.
\end{proof}

Hence the Landweber iteration reduces the error until the residual norm drops below twice the noise level. (This implies that for the discrepancy principle, $\tau$ should always be chosen less than $2$ since otherwise the iteration is guaranteed to terminate \emph{too} early.) From this point on, the error will start to increase again for $y^\delta\notin\calR(K)$ by \cref{thm:spektral:konvergenz}. This behavior is called \term{semiconvergence} and is typical for iterative methods when applied to ill-posed problems.
The discrepancy principle then prevents the error from increasing arbitrarily. (A slight increase is accepted -- how much depends on the choice of $\tau \in (1,2)$.)

An important question relating to the efficiency of the Landweber method is the number of steps required for the discrepancy principle to terminate the iteration. The following theorem gives an upper bound.
\begin{theorem}
    Let $\tau>1$ and $y^\delta\in B_\delta(Kx^\dagger)$. Then the discrepancy principle \eqref{eq:landweber:diskrepanz} terminates the Landweber iteration \eqref{eq:landweber:iteration} in step
    \begin{equation*}
        m(\delta,y^\delta) \leq C \delta^{-2} \qquad\text{for some }C>0.
    \end{equation*}
\end{theorem}
\begin{proof}
    We first derive a convergence rate for the residual norm in terms of $m$. For this, we consider for $n \geq 0$ the iterate $x_n$ produced by the Landweber iteration applied to the exact data $y:=Kx^\dag\in \calR(K)$ and denote the corresponding residual by $\rho_n := y - Kx_n$.
    We now proceed similarly to the proof of \cref{thm:landweber:semikonvergenz}. Using the iteration \eqref{eq:landweber:iteration} and \eqref{eq:landweber:res_spd} shows that
    \begin{equation*}
        \begin{aligned}
            \norm{x^\dag-x_n}_X^2 - \norm{x^\dag-x_{n+1}}_X^{2}
            &= \norm{x^\dag-x_n}_X^2 - \norm{x^\dag - x_n - \omega K^*\rho_n}_X^2\\
            &= 2\omega \inner{Kx^\dag-Kx_n}{\rho_n}_Y - \omega^2\norm{K^*\rho_n}_X^2\\
            &= \omega \left(\norm{\rho_n}_Y^2 - \omega\norm{K^*\rho_n}_X^2\right) + \omega\norm{\rho_n}_Y^2\\
            &>  \omega\norm{\rho_n}_Y^2.
        \end{aligned}
    \end{equation*}
    Summing over all $n=0,\dots,m-1$ and using the monotonicity of the residual from \cref{thm:landweber:res_monoton} then yields
    \begin{equation*}
        \begin{aligned}
            \norm{x^\dag-x_0}_X^2 - \norm{x^\dag-x_{m}}_X^{2}
            &= \sum_{n=0}^{m-1}\left(\norm{x^\dag-x_n}_X^2 - \norm{x^\dag-x_{n+1}}_X^{2}\right)\\
            &> \omega \sum_{n=0}^{m-1}\norm{\rho_n}_Y^2 > \omega m \norm{\rho_m}_Y^2.
        \end{aligned}
    \end{equation*}
    In particular,
    \begin{equation*}
        \norm{y-Kx_m}_Y^2 < (\omega m)^{-1} \norm{x^\dag-x_0}_X^2.
    \end{equation*}

    As in the proof of \cref{thm:landweber:res_monoton}, we now have due to $x_0 = 0$ that
    \begin{equation*}
        \rho_m^\delta = y^\delta - Kx_{m}^\delta = (\Id - \omega KK^*)(y^\delta - Kx_{m-1}^\delta) = \dots = (\Id-\omega KK^*)^m y^\delta
    \end{equation*}
    and similarly for $\rho_m = (\Id-\omega KK^*)^m y$. This yields using $\omega < \kappa^{-1} \leq \sigma_n^{-2}$ the estimate
    \begin{equation*}
        \norm{(\Id-\omega KK^*)^m (y^\delta-y)}_Y^2 = \sum_{n\in\N} (1-\omega \sigma_n^2)^{2m} \abs{\inner{y^\delta-y}{u_n}_Y}^2 \leq \norm{y^\delta-y}_Y^2
    \end{equation*}
    and hence that
    \begin{equation*}
        \begin{aligned}
            \norm{Kx_m^\delta - y^\delta}_Y &= \norm{(\Id-\omega KK^*)^m y^\delta}_Y \\
            &\leq \norm{(\Id-\omega KK^*)^m y}_Y + \norm{(\Id-\omega KK^*)^m(y^\delta-y)}_Y\\
            &\leq \norm{y-Kx_m}_Y + \norm{y^\delta-y}_Y\\
            &\leq (\omega m)^{-1/2} \norm{x^\dag-x_0}_X + \delta.
        \end{aligned}
    \end{equation*}

    The discrepancy principle now chooses the stopping index $m(\delta,y^\delta)$ as the first index for which $\norm{Kx_{m(\delta,y^\delta)}^\delta - y^\delta}_Y\leq \tau\delta$. Due to the monotonicity of the residual norm, this is the case at the latest for the first $\bar m\in \N$ with
    \begin{equation*}
        (\omega \bar m)^{-1/2} \norm{x^\dag-x_0}_X + \delta \leq \tau\delta;
    \end{equation*}
    in other words, for which
    \begin{equation*}
        \bar m \geq \frac{\norm{x^\dag-x_0}_X^2}{\omega(\tau-1)^2} \delta^{-2} > \bar m - 1.
    \end{equation*}
    This implies that
    \begin{equation*}
        m(\delta,y^\delta) \leq \bar m < \frac{\norm{x^\dag-x_0}_X^2}{\omega(\tau-1)^2} \delta^{-2}+1.
    \end{equation*}
    It remains to absorb the constant $1$ into the first term, for which we distinguish two cases with respect to $\delta_0 := (\tau-1)^{-1}\norm{y}_Y$.
    If $\delta \geq \delta_0$, then $\norm{y^\delta}_Y\leq \norm{y}_Y+\delta\leq \tau\delta$, and hence the discrepancy principle \eqref{eq:landweber:diskrepanz} already stops with $m(\delta,y^\delta) = 0$.
    If $\delta<\delta_0$, then $1< \delta_0^2\delta^{-2}$, and hence
    \begin{equation*}
        m(\delta,y^\delta) < \frac{\norm{x^\dag-x_0}_X^2}{\omega(\tau-1)^2} \delta^{-2}+\frac{\norm{y}_Y^2}{(\tau-1)^2}\delta^{-2}.
    \end{equation*}
    In both cases, we thus obtain the claim with $C:= (\tau-1)^{-2}\left(\omega^{-1}\norm{x^\dag-x_0}_X^2+\norm{y}_Y^2\right)$.
\end{proof}

It is not surprising that this estimate can be improved under the usual source condition $x^\dag \in X_\nu$.
\begin{corollary}
   If $x^\dag\in X_\nu$ for some $\nu>0$, then the discrepancy principle for $\tau>1$ and $y^\delta\in B_\delta(Kx^\dagger)$ terminates in step
    \begin{equation*}
        m(\delta,y^\delta) \leq C \delta^{-\frac2{\nu+1}} \qquad\text{for some }C>0.
    \end{equation*}
\end{corollary}
\begin{proof}
    Inserting $\alpha = \frac1m$ in the estimate \eqref{eq:spektral:morozov_alpha} in the proof of \cref{thm:spektral:morozov} yields
    \begin{equation*}
        m(\delta,y^\delta)^{1/2} = \alpha(\delta,y^\delta)^{-1/2} \leq C \delta^{-\frac1{\nu+1}}.
    \end{equation*}
    (The derivation of \eqref{eq:spektral:morozov_alpha} compares with the residual for $2\alpha$, which here corresponds to the non-integer step $\frac{m}2$. However, it only requires a $\beta\in(\alpha,2\alpha]$ for which the residual exceeds $\tau\delta$; for $m\geq 2$, we can take $\beta = \frac1{m-1}$, while the claim is trivial for $m\leq 1$.)
\end{proof}
Still, Landweber regularization in practice often requires too many iterations, which motivates accelerated variants such as the one described in \cite[Chapter~6.2, 6.3]{Engl}.
Furthermore, regularization by early stopping can be applied to other iterative methods for solving the normal equation; a particularly popular choice is the conjugate gradient (CG) method; see \cite[Chapter~7]{Engl}.

\chapter{Discretization as regularization}

And now for something completely different.
We have seen that the fundamental difficulty in inverse problems is due to the unboundedness of the pseudoinverse for compact operators $K:X\to Y$ with infinite-dimensional range.
It thus suggests itself to construct a sequence $\{K_n\}_{n\in\N}$ of operators with \emph{finite-dimensional} ranges and approximate the wanted minimum norm solution $K^\dag y$ using the (now continuous) pseudoinverses $(K_n)^\dag$. This is indeed possible -- up to a point.
Such finite-dimensional operators can be constructed by either of the following approaches:
\begin{enumerate}
    \item We restrict the domain of $K$ to a finite-dimensional subspace $X_n\subset X$ and define $K_n:X_n \to Y$, which has finite-dimensional range because if $\{x_1,\dots,x_n\}$ is a basis of $X_n$, then $\calR(K_n) = \spann\{Kx_1,\dots,Kx_n\}$. This approach is referred to as \term[projection!least squares]{least squares projection}.
    \item We directly restrict the range of $K$ to a finite-dimensional subspace $Y_n\subset Y$ and define $K_n:X \to Y_n$. This approach is referred to as \term[projection!dual least squares]{dual least squares projection}.
\end{enumerate}
(Of course, we could also restrict domain \emph{and} range and define $K_n:X_n\to Y_n$, but this will not add anything useful from the point of view of regularization theory.)
In this chapter, we will study both approaches, where the second will be seen to have advantages. Since we do not require any spectral theory for this, we will consider again an arbitrary bounded operator $T\in\linop(X,Y)$.

\section{Least squares projection}

Let $\{X_n\}_{n\in\N}$ be a sequence of nested subspaces, i.e.,
\begin{equation*}
    X_1 \subset X_2 \subset \dots \subset X,
\end{equation*}
with $\dim X_n = n$ and $\overline{\bigcup_{n\in\N} X_n} = X$. Furthermore, let $P_n:= P_{X_n}$ denote the orthogonal projection onto $X_n$ and set $T_n := TP_n\in\linop(X,Y)$. Since $T_n$ has finite-dimensional range, $T_n^\dag:= (T_n)^\dag$ is continuous. We thus define for $y\in Y$ the regularization $x_n:= T_n^\dag y$, i.e., the minimum norm solution of $TP_n x =y$. By \cref{lem:inverse}, we then have
\begin{equation*}
    x_n \in \calR(T_n^\dag) = \calN(T_n)^\bot = \overline{\calR(T_n^*)} = \overline{\calR(P_n T^*)}\subset X_n
\end{equation*}
since $X_n$ is finite-dimensional and therefore closed and $P_n$ is selfadjoint. (We are thus only looking for a minimum norm solution in $X_n$ instead of in all of $X$.)
To show that $T_n^\dag$ is a regularization in the sense of \cref{def:regularisierung}, we have to show that $y\in \calD(T^\dag)$ implies that $T_n^\dag y\to T^\dag y$ as $n\to\infty$.
This requires an additional assumption.%
\footnote{Here we follow \cite{Kindermann:2016}; the proof in \cite{Engl} using a similar equivalence for weak convergence requires an additional assumption, as was pointed out in \cite{Du:2008}.}
\begin{lemma}\label{lem:discret:conv}
    Let $y\in \calD(T^\dag)$. Then $x_n\to x^\dag$ if and only if $\limsup_{n\to\infty}\norm{x_n}_X\leq \norm{x^\dag}_X$.
\end{lemma}
\begin{proof}
    If $x_n\to x^\dag$, the triangle inequality directly yields that
    \begin{equation*}
        \norm{x_n}_X \leq \norm{x_n-x^\dag}_X + \norm{x^\dag}_X \to \norm{x^\dag}_X.
    \end{equation*}

    Conversely, if the $\limsup$ assumption holds, the sequence $\{\norm{x_n}_X\}_{n\in\N}$ and thus also $\{x_n\}_{n\in\N}$ is bounded in $X$. Hence there exists a subsequence $\{x_{n_k}\}_{k\in\N}$ and a $\bar x\in X$ with $x_k:= x_{n_k}\rightharpoonup \bar x$ and $Tx_k \rightharpoonup T\bar x$. By the definition of $x_k$ as a least squares solution of $T_k x = y$ (of minimal norm) and by $Tx^\dag = P_{\overline\calR} y$ due to \cref{lem:inverse}\,(iv) and $\overline{\calR(T)}=\calN(T)^\bot$, we now have
    \begin{equation*}
        \begin{aligned}
            \norm{T_kx_k -Tx^\dag}_Y^2 + \norm{(\Id-P_{\overline\calR})y}_Y^2 &=
            \norm{T_kx_k - y}_Y^2
            \leq \norm{T_k x -y}_Y^2 \\
            &=
            \norm{T_kx -Tx^\dag}_Y^2 + \norm{(\Id-P_{\overline\calR})y}_Y^2  \qquad\text{for all }x\in X.
        \end{aligned}
    \end{equation*}
    Since $x_k\in X_k$, we have $x_k = P_k x_k$, and thus $x = P_kx^\dag$ satisfies
    \begin{equation}\label{eq:discret:wkconv}
        \begin{aligned}[t]
            \norm{Tx_k - Tx^\dag}_Y = \norm{T_k x_k - Tx^\dag}_Y
            &\leq \norm{T_k P_k x^\dag - Tx^\dag}_Y = \norm{T P_k x^\dag - Tx^\dag}_Y \\
            &\leq \norm{T}_{\linop(X,Y)}\norm{(\Id-P_k)x^\dag}_X.
        \end{aligned}
    \end{equation}
    By the assumptions on $\{X_n\}_{n\in\N}$, the last term converges to zero as $k\to \infty$, and hence we have that $Tx_k \to Tx^\dag$. This implies that $\bar x - x^\dag \in \calN(T)$.
    Now we always have that $x^\dag \in \calN(T)^\bot$, and hence the weak lower semicontinuity of the norm together with the $\limsup$ assumption yields that
    \begin{equation*}
        \norm{\bar x - x^\dag}_X^2 + \norm{x^\dag}_X^2 = \norm{\bar x}_X^2 \leq \liminf_{k\to\infty} \norm{x_k}_X^2 \leq \limsup_{k\to\infty} \norm{x_k}_X^2 \leq \norm{x^\dag}_X^2,
    \end{equation*}
    which implies that $\bar x = x^\dag$. This shows that every weakly convergent subsequence has the same limit $x^\dag$, and therefore the full sequence has to converge weakly to $x^\dag$.
    Combining the weak lower semicontinuity of the norm with the $\limsup$ assumption finally yields that  $\norm{x_n}_X\to \norm{x^\dag}_X$ as well, and hence the sequence even converges strongly in the Hilbert space $X$.
\end{proof}

Unfortunately, it is possible to construct examples where $\{\norm{x_n}_X\}_{n\in\N}$ is not bounded; see, e.g., \cite[Example~3.19]{Engl}.
A sufficient condition for convergence is given in the following theorem.
\begin{theorem}\label{thm:discret:lsp_conv}
    Let $y\in \calD(T^\dag)$. If
    \begin{equation}
        \label{eq:discret:lsp_cond}
        \limsup_{n\to\infty}\norm{(T_n^*)^\dag x_n}_Y = \limsup_{n\to\infty}\norm{(T_n^\dag)^*x_n}_Y <\infty,
    \end{equation}
    then $x_n\to x^\dag$ as $n\to \infty$.
\end{theorem}
\begin{proof}
    Since
    \begin{equation*}
        \norm{x_n}_X^2 = \inner{x_n-x^\dag}{x_n}_X + \inner{x^\dag}{x_n}_X \leq \inner{x_n-x^\dag}{x_n}_X +\norm{x^\dag}_X\norm{x_n}_X,
    \end{equation*}
    it suffices to show that the first term on the right-hand side tends to zero as $n\to\infty$.
    For this, we set $w_n:= (T_n^\dag)^*x_n$ and use that $T_n^*w_n = x_n$ since $\calR(T_n^*)\subset X_n$ and therefore $x_n \in \calR(T_n^\dag) = \calR(T_n^*)$. This allows us to estimate
    \begin{equation}
        \label{eq:discret:conv}
        \begin{aligned}[t]
            \inner{x_n-x^\dag}{x_n}_X &= \inner{x_n-x^\dag}{T_n^*w_n}_X
            = \inner{T_nx_n-T_nx^\dag}{w_n}_Y\\
            &= \inner{T_n x_n-Tx^\dag}{w_n}_Y + \inner{Tx^\dag-T_n x^\dag}{ w_n}_Y\\
            &\leq \left(\norm{T_nx_n-Tx^\dag}_Y + \norm{T(\Id-P_n)x^\dag}_Y\right)\norm{w_n}_Y\\
            &\leq 2\norm{T}_{\linop(X,Y)} \norm{(\Id-P_n)x^\dag}_{X} \norm{w_n}_Y,
        \end{aligned}
    \end{equation}
    where in the last step we have again used \eqref{eq:discret:wkconv}.
    The last term is now bounded by the assumption \eqref{eq:discret:lsp_cond}, while the second term and thus the whole right-hand side tend to zero. We can therefore apply \cref{lem:discret:conv} to obtain the claim.
\end{proof}

This shows that the least squares projection only defines a convergent regularization if the subspaces $X_n$ are chosen appropriately for the operator $T$.
Before moving on to the dual least squares projection (which does not require such a condition), we consider the special case of compact operators.
\begin{theorem}\label{thm:discret:source}
    If $K\in\calK(X,Y)$ and $x^\dag \in X$ satisfy the condition \eqref{eq:discret:lsp_cond}, then $x^\dag\in \calR(K^*)$.
\end{theorem}
\begin{proof}
    Setting again $w_n:= (K_n^\dag)^*x_n$, the condition \eqref{eq:discret:lsp_cond} implies that $\{w_n\}_{n\in\N}$ is bounded and therefore contains a weakly convergent subsequence $w_k\rightharpoonup \bar w\in Y$. Since $K$ and therefore also $K^*$ is compact, $K^*w_k\to K^* \bar w$. On the other hand, it follows from $(K_n^\dag)^* = (K_n^*)^\dag = (P_nK^*)^\dag$ that
    \begin{equation*}
        K^*w_k = P_k K^* w_k + (\Id - P_k)K^* w_k = x_k + (\Id-P_k)K^* w_k.
    \end{equation*}
    Passing to the limit on both sides of the equation, we have $K^*w_k\to K^*\bar w$ on the left and $x_k\to x^\dag$ by \cref{thm:discret:lsp_conv} on the right, while
    \begin{equation*}
        \norm{(\Id-P_k)K^*w_k}_X \leq \norm{(\Id-P_k)K^*\bar w}_X + \norm{K^*w_k-K^*\bar w}_X \to 0
    \end{equation*}
    since $P_k x\to x$ for all $x\in X$ and $\norm{\Id-P_k}_{\linop(X,X)}\leq 1$. Hence $K^*\bar w = x^\dag$, i.e., $x^\dag\in\calR(K^*)$.
\end{proof}

Hence the condition \eqref{eq:discret:lsp_cond} already implies a source condition.
It is therefore not surprising that we can give an estimate for the convergence $x_n\to x^\dag$.
\begin{theorem}\label{thm:discret:lsp_conv_rate}
    If $K\in\calK(X,Y)$ and $x^\dag \in X$ satisfy the condition \eqref{eq:discret:lsp_cond} and $y\in\calD(K^\dag)$, then there exists a constant $C>0$ such that
    \begin{equation*}
        \norm{x_n-x^\dag}_X \leq C \norm{(\Id-P_n)K^*}_{\linop(Y,X)}\qquad\text{for all }n\in\N.
    \end{equation*}
\end{theorem}
\begin{proof}
    By \cref{thm:discret:source} there exists a $w\in Y$ with $x^\dag = K^*w$. Hence, \eqref{eq:discret:wkconv} implies that
    \begin{equation*}
        \inner{x_n-x^\dag}{x^\dag}_X \leq \norm{Kx_n-Kx^\dag}_Y\norm{w}_Y \leq \norm{K(P_n-\Id)x^\dag}_Y\norm{w}_Y.
    \end{equation*}
    It follows from this together with \eqref{eq:discret:conv} and the boundedness of the $w_n:= (K_n^\dag)^*x_n$ that
    \begin{equation*}
        \begin{aligned}
            \norm{x_n-x^\dag}_X^2 &= \inner{x_n-x^\dag}{x_n}_X - \inner{x_n-x^\dag}{x^\dag}_X\\
            &\leq 2\norm{K(\Id-P_n)x^\dag}_Y\norm{w_n}_Y + \norm{K(\Id-P_n)x^\dag}_Y\norm{w}_Y\\
            &\leq C\norm{K(\Id-P_n)x^\dag}_Y = C\norm{K(\Id-P_n)(\Id-P_n)K^*w}_Y\\
            &\leq C \norm{(\Id-P_n)K^*}_{\linop(Y,X)}^2\norm{w}_Y,
        \end{aligned}
    \end{equation*}
    where we have used in the last step that orthogonal projections are selfadjoint and thus that $(K(\Id-P_n))^* = (\Id-P_n)K^*$.
\end{proof}

\section{Dual least squares projection}

Here we directly discretize the range of $T$. We thus consider a sequence $\{Y_n\}_{n\in\N}$ of nested subspaces, i.e.,
\begin{equation*}
    Y_1 \subset Y_2 \subset \dots \subset \overline{\calR(T)} = \calN(T^*)^\bot \subset Y,
\end{equation*}
with $\dim Y_n = n$ and $\overline{\bigcup_{n\in\N} Y_n} = \calN(T^*)^\bot$.
Let now $Q_n:= P_{Y_n}$ denote the orthogonal projection onto $Y_n$ and set $T_n := Q_nT \in \linop(X,Y_n)$.
Again, $T_n^\dag$ and hence also $T_n^\dag Q_n$ are continuous, and we can take $x_n := T_n^\dag Q_n y$ -- i.e., the minimum norm solution of $Q_nTx=Q_ny$ -- as a candidate for our regularization.
To show that this indeed defines a regularization, we introduce the orthogonal projection $P_n:= P_{X_n}$ onto
\begin{equation*}
    X_n:= T^*Y_n:=\setof{T^*y}{y\in Y_n}.
\end{equation*}
We then have the following useful characterization.
\begin{lemma}\label{lem:discret:dual_lsq_char}
    Let $y\in \calD(T^\dag)$. Then $x_n = P_n x^\dag$.
\end{lemma}
\begin{proof}
    We first note that by definition of the pseudoinverse and of $X_n$, we have that
    \begin{equation*}
        \calR(T_n^\dag) = \calN(T_n)^\bot = {\calR(T_n^*)} = {\calR(T^*Q_n)} = T^*Y_n = X_n
    \end{equation*}
    (where the second equation follows from the fact that $\calR(T_n^*) = X_n$ is finite-dimensional)
    and hence that $x_n\in X_n$ as well as $X_n^\bot = \calN(T_n)$. This also implies that
    \begin{equation*}
        T_n (\Id-P_n)x = 0 \qquad\text{for all } x\in X,
    \end{equation*}
    i.e., that $T_nP_n = T_n$.
    Furthermore, it follows from the fact that $Y_n\subset \calN(T^*)^\bot=\overline{\calR(T)}$ (and hence that $\calR(T)^\bot\subset \calN(Q_n)$) together with \cref{lem:inverse}\,(iv) that
    \begin{equation*}
        Q_n y = Q_n P_{\overline{\calR(T)}} y = Q_n TT^\dag y = Q_n Tx^\dag = T_n x^\dag.
    \end{equation*}
    We thus obtain for any $x\in X$ that
    \begin{equation*}
        \norm{T_n x - Q_n y}_Y = \norm{T_n x- T_n x^\dag}_Y = \norm{T_n x-T_nP_nx^\dag}_Y = \norm{T_n(x-P_n x^\dag)}_Y.
    \end{equation*}
    Now $x_n$ is defined as the minimum norm solution of $T_nx = Q_n y$, i.e., as the one $x\in \calN(T_n)^\bot = X_n$ minimizing $\norm{T_n x - Q_n y}_Y$ -- which is obviously minimal for $x=P_nx^\dag\in X_n$. Since the minimum norm solution is unique, we have that $x_n = P_n x^\dag$.
\end{proof}
\begin{theorem}
    Let $y\in \calD(T^\dag)$. Then $x_n\to x^\dag$.
\end{theorem}
\begin{proof}
    The construction of $Y_n$ implies that $X_n\subset X_{n+1}$ and hence that
    \begin{equation*}
        \overline{\bigcup_{n\in\N} X_n} = \overline{\bigcup_{n\in\N} T^*Y_n} = \overline{T^*\bigcup_{n\in\N} Y_n} = \overline{T^*\calN(T^*)^\bot} = \overline{\calR(T^*)} =\calN(T)^\bot.
    \end{equation*}
    Using $x^\dag \in \calR(T^\dag) = \calN(T)^\bot$, we deduce that $x_n\to x^\dag$.
\end{proof}
Under a source condition, we can show a similar error estimate as in \cref{thm:discret:lsp_conv_rate}.
\begin{theorem}\label{thm:discret:dlsp_conv_rate}
    Let $T\in\linop(X,Y)$ and $y\in\calD(T^\dag)$. If $x^\dag = T^\dag y\in \calR(T^*)$, then there exists a constant $C>0$ such that
    \begin{equation*}
        \norm{x_n-x^\dag}_X \leq C \norm{(\Id-P_n)T^*}_{\linop(Y,X)}\qquad\text{for all }n\in \N.
    \end{equation*}
\end{theorem}
\begin{proof}
    The source condition $x^\dag = T^*w$ for some $w\in Y$ and \cref{lem:discret:dual_lsq_char} immediately yield that
    \begin{equation*}
        \norm{x_n - x^\dag}_X = \norm{P_n x^\dag - x^\dag}_X = \norm{(\Id-P_n)T^*w}_X \leq \norm{(\Id-P_n)T^*}_{\linop(Y,X)}\norm{w}_Y.
        \qedhere
    \end{equation*}
\end{proof}

The dual least squares projection thus defines a regularization operator as well. By \cref{thm:regularisierung:apriori}, there thus exists (at least for compact operators) an a priori choice rule that turns the dual least squares projection into a convergent regularization method. Characterizing this choice rule requires estimating the norm of $T_n^\dag$, for which we can use that $T_n$ has finite-dimensional range and is therefore compact. Hence there exists a (finite) singular system $\{(\mu_k,\tilde u_k,\tilde v_k)\}_{k\in\{1,\dots,n\}}$; in particular, we can use that $\mu_n$ is the smallest (by magnitude) singular value of $T_n$.
\begin{theorem}\label{thm:discret:apriori}
    Let $y\in \calD(T^\dag)$ and for $y^\delta \in B_\delta(y)$ set $x_n^\delta := T_n^\dag Q_n y^\delta$. If $n(\delta)$ is chosen such that
    \begin{equation*}
        n(\delta)\to \infty, \qquad\frac{\delta}{\mu_{n(\delta)}}\to 0\qquad \text{for }\delta \to 0,
    \end{equation*}
    then $x_{n(\delta)}^\delta \to x^\dag$ as $\delta\to 0$.
\end{theorem}
\begin{proof}
    We proceed as in the proof of \cref{thm:regularisierung:apriori_char} and use the standard error decomposition
    \begin{equation*}
        \norm{x_{n(\delta)}^\delta -x^\dag}_X
        \leq \norm{x_{n(\delta)}-x^\dag}_X + \norm{x_{n(\delta)}^\delta -x_{n(\delta)}}_X.
    \end{equation*}
    By \cref{thm:discret:dlsp_conv_rate}, the first term tends to zero as $n\to\infty$.

    For the second term, we use the singular value decomposition of $T_n$ and \eqref{eq:inverse:picard_pseudo} to obtain for any $n\in \N$ that
    \begin{equation*}
        \norm{T_n^\dag y}_X^2 = \sum_{k=1}^n \mu_k^{-2} |\inner{y}{\tilde u_k}_Y|^2\leq \mu_n^{-2} \norm{y}_Y^2 \qquad \text{for all }y\in Y,
    \end{equation*}
    with equality for $y=\tilde u_n\in Y$. This implies that $\norm{T_n^\dag}_{\linop(Y,X)} = \mu_n^{-1}$. Since $Q_n$ is an orthogonal projection, we have that
    \begin{equation*}
        \norm{x_n^{\delta}-x_n}_X = \norm{T_n^\dag Q_n (y^\delta-y)}_X
        \leq \norm{T_n^\dag}_{\linop(Y,X)}\norm{y^\delta-y}_Y\leq \frac\delta{\mu_n}.
    \end{equation*}
    The claim now follows from the assumptions on $n(\delta)$.
\end{proof}
Under the source condition from \cref{thm:discret:dlsp_conv_rate}, we can in this way also obtain convergence rates as in \cref{thm:spektral:apriori}. (Similar results also hold for the least squares projection under the additional assumption \eqref{eq:discret:lsp_cond}.)

We can now ask how to choose $Y_n$ for given $n\in\N$ in order to minimize the regularization error, which by \cref{thm:discret:apriori} entails maximizing $\mu_n$. This question can be answered explicitly for compact operators.
\begin{theorem}
    Let $K\in\calK(X,Y)$ have the singular system $\{(\sigma_n,u_n,v_n)\}_{n\in\N}$. If $Y_n\subset Y$ with $\dim Y_n = n$, then $\mu_n\leq \sigma_n$.
\end{theorem}
\begin{proof}
    If $\mu_n$ is a singular value of $K_n$, then $\mu_n^2$ is an eigenvalue of $K_nK_n^* = Q_nKK^*Q_n$; similarly, $\sigma_n^2$ is an eigenvalue of $KK^*$. Set now $U_{k} := \spann\{u_1,\dots,u_k\}\subset\overline{\calR(K)}$ for all $k\in\N$.
    Since $\dim Y_n = n$, there exists $\bar y\in U_{n-1}^\bot\cap Y_n$ with $\norm{\bar y}_Y = 1$ (otherwise $U_{n-1}^\bot \subset Y_n^\bot$, but this is impossible since the codimension of $U_{n-1}$ is too small). The Courant--Fischer min--max principle \eqref{eq:courant} thus implies that
    \begin{equation*}
        \begin{aligned}
            \mu_n^2 &= \max_{V}\min_{y}\setof{\inner{Q_nKK^*Q_n y}{y}_Y}{\norm{y}_Y= 1,\ y\in V,\ \dim V = n}\\
            &= \min_{y}\setof{\inner{KK^* y}{y}_Y}{\norm{y}_Y= 1,\ y\in Y_n}
            \leq \inner{KK^* \bar y }{\bar y}_Y\\
            &\leq \max_{y} \setof{\inner{KK^* y}{y}_Y}{\norm{y}_Y= 1,\ y\in U_{n-1}^\bot} = \sigma_n^2
        \end{aligned}
    \end{equation*}
    since the maximum is attained for $y=u_n\in U_{n-1}^\bot$.
\end{proof}
The proof also shows that equality of the singular values holds for $Y_n=U_n$, because then $\bar y=u_n$ is the only vector that is a candidate for minimization or maximization. But this choice corresponds exactly to the truncated singular value decomposition from \cref{ex:spektral}\,(i). In fact, the choice $Y_n = U_n$ is optimal with respect to the approximation error as well.
\begin{theorem}
    Let $K\in\calK(X,Y)$ have the singular system $\{(\sigma_n,u_n,v_n)\}_{n\in\N}$. If $Y_n\subset Y$ with $\dim Y_n = n$, then
    \begin{equation*}
        \norm{(\Id-P_n)K^*}_{\linop(Y,X)}\geq \sigma_{n+1},
    \end{equation*}
    with equality for $Y_n = U_n$.
\end{theorem}
\begin{proof}
    We again use the Courant--Fischer min--max principle, this time for the eigenvalue $\sigma_{n+1}^2$ of $K^*K$. Setting $X_n:= K^*Y_n$ and $P_n:= P_{X_n}$, we then have that (if $\dim X_n<n$, we replace $X_n$ in the second step by an $n$-dimensional subspace $\tilde X_n\supset X_n$, which only decreases the maximum)
    \begin{equation*}
        \begin{aligned}
            \sigma_{n+1}^2
            &= \min_{V}\max_{x}\setof{\inner{K^*Kx}{x}_X}{\norm{x}_X= 1,\ x\in V\subset X,\ \dim V^\bot = n}\\
            &\leq \max_{x}\setof{\inner{K^*Kx}{x}_X}{\norm{x}_X= 1,\ x\in X_n^\bot}\\
            &=\max_{x}\setof{\inner{K^*K(\Id-P_n)x}{(\Id-P_n)x}_X}{\norm{x}_X= 1}\\
            &=\max_{x}\setof{\norm{K(\Id-P_n)x}_Y^2}{\norm{x}_X= 1}\\
            &= \norm{K(\Id-P_n)}_{\linop(X,Y)}^2= \norm{(\Id-P_n)K^*}_{\linop(Y,X)}^2.
        \end{aligned}
    \end{equation*}
    If $Y_n = U_n$, then $X_n= K^*U_n =\spann\{v_1,\dots,v_n\}$, and the minimum and maximum in the inequality are attained for $x= v_{n+1} \in X_n^\bot$.
\end{proof}
Hence the best possible convergence rate (under the source condition from \cref{thm:discret:dlsp_conv_rate}) for the dual least squares projection is
\begin{equation*}
    \norm{x_n^\delta-x^\dag}_X \leq C \left(\sigma_{n+1}+\frac{\delta}{\sigma_n}\right),
\end{equation*}
and this rate is attained for the truncated singular value decomposition.

Without knowledge of a singular system, however, it is necessary in practice to choose $n$ very small in order to ensure the condition on $\mu_n$. But this leads to a very coarse discretization that does not sufficiently capture the behavior of the infinite-dimensional operator.
The usual approach is therefore to combine a much finer discretization with one of the regularization methods discussed in the previous chapters. To obtain an optimal convergence rate and to avoid needless computational effort, one should then appropriately choose the regularization parameter in dependence of $\delta$ as well as of $n$ (or, vice versa, choose $n$ in dependence of $\alpha(\delta)$).

\part{Nonlinear inverse problems}

\chapter{Nonlinear ill-posed problems}

We now consider \emph{nonlinear} operators $F:U\to Y$ for $U\subset X$ and Hilbert spaces $X$ and $Y$.
The corresponding nonlinear inverse problem then consists in solving the operator equation $F(x)=y$. Such problems occur in many areas; in particular, trying to reconstruct the coefficients of a partial differential equation from a solution for given data (right-hand sides, initial or boundary conditions), e.g., in electrical impedance tomography, is a nonlinear ill-posed problem.
Here we will characterize this ill-posedness in an abstract setting; concrete examples would require results on partial differential equations that would go far beyond the scope of these notes.

A fundamental difference between linear and nonlinear operators is that the latter can act very differently on different subsets of $X$. The global characterization of well- or ill-posedness in the sense of Hadamard is hence too restrictive. We therefore call the operator $F:U\to Y$ \term[well-posed!locally]{locally well-posed} in $x\in U$ if there exists an $r>0$ such that for all sequences $\{x_n\}_{n\in\N}\subset B_r(x)\cap U$ with $F(x_n)\to F(x)$, we also have that $x_n\to x$. Otherwise the operator is called \term[ill-posed!locally]{locally ill-posed} (in $x$).
In this case, there exists for all $r>0$ a sequence $\{x_n\}_{n\in\N}\subset B_r(x)\cap U$ with $F(x_n)\to F(x)$ such that $x_n$ does not converge to $x$. A linear operator $T:X\to Y$ is either locally well-posed for all $x\in X$ or locally ill-posed for all $x\in X$.
The latter holds if and only if $T$ is not injective or $\calR(T)$ is not closed (e.g., for compact operators with infinite-dimensional range).%
\footnote{The local ill-posedness thus generalizes the (global) ill-posedness in the sense of Nashed, not of Hadamard.}
For nonlinear operators, the situation is a bit more involved. As in the linear case, we call $F:U\to Y$ \term[operator!compact!nonlinear]{compact}, if every bounded sequence $\{x_n\}_{n\in\N}\subset U$ admits a convergent subsequence of $\{F(x_n)\}_{n\in\N}\subset Y$. However, nonlinear compact operators need \emph{not} be continuous and hence completely continuous (consider, e.g., an arbitrary bounded operator with finite-dimensional range); the latter is therefore an additional assumption. In fact, a weaker assumption suffices: an operator $F:U\to Y$ is called \term{weakly closed}, if $x_n\wkto x\in U$ and $F(x_n)\wkto y$ imply that $F(x)=y$.
\begin{lemma}\label{lem:nichtlin:vollcontinuous}
    Let $F:U\to Y$ be compact and weakly closed. Then $F$ is \term[continuity, complete]{completely continuous}, i.e., maps weakly convergent sequences in $X$ to strongly convergent sequences in $Y$.
\end{lemma}
\begin{proof}
    Let $\{x_n\}_{n\in\N}\subset U$ be a weakly converging sequence with $x_n\wkto x\in U$. Then $\{x_n\}_{n\in\N}$ is bounded, and hence $\{F(x_n)\}_{n\in\N}$ contains a convergent subsequence $\{F(x_{n_k})\}_{k\in\N}$ with $F(x_{n_k})\to y\in Y$. Since strongly convergent sequences also converge weakly (to the same limit), the weak closedness yields that $y=F(x)$. Hence the limit is independent of the subsequence, which implies that the whole sequence converges.
\end{proof}
For such operators, we can show an analogous result to \cref{thm:inverse:discont}.
\begin{theorem}
    Let $X$ be an infinite-dimensional separable Hilbert space and $U\subset X$. If $F:U\to Y$ is completely continuous, then $F$ is locally ill-posed in all interior points of $U$.
\end{theorem}
\begin{proof}
    Since $X$ is separable, there exists an (infinite) orthonormal basis $\{u_n\}_{n\in\N}$. Let now $x\in U$ be an interior point and define for $r>0$ with $B_r(x)\subset U$ the points $x_n:= x+\frac{r}{2} u_n\in B_r(x)$. Then $\norm{x_n-x}_X=\frac{r}2$, but the fact that $u_n\wkto 0$ for any orthonormal basis implies that $x_n\wkto x$ and hence that $F(x_n)\to F(x)$ due to the complete continuity of $F$.
\end{proof}

As in the linear case we now define minimum norm solutions and regularizations.
Since $0\in X$ can no longer be taken as a generic point, we denote for given $y\in \calR(F)$ and $x_0\in X$ any point $x^\dag\in U$ with $F(x^\dag)=y$ and
\begin{equation*}
    \norm{x^\dag -x_0}_X = \min\setof{\norm{x-x_0}_X}{F(x) = y}
\end{equation*}
as \term[solution!x0-minimum norm@$x_0$-minimum norm]{$x_0$-minimum norm solution}.
For nonlinear inverse problems, these need not be unique in contrast to the linear case. Their existence also requires that $F(x)=y$ actually admits a solution.
A regularization\index{regularization!nonlinear} of $F(x)=y$ is now a family $\{R_\alpha\}_{\alpha>0}$ of continuous (possibly nonlinear) operators $R_\alpha: X\times Y\to X$ such that $R_\alpha(x_0,y)$ converges to an $x_0$-minimum norm solution as $\alpha\to 0$.
In combination with a parameter choice rule for $\alpha$, we define (convergent) regularization methods\index{regularization method!nonlinear} as before. For nonlinear inverse problems, these operators cannot in general be given explicitly; most regularizations are instead based on an (iterative) linearization of the problem.

\bigskip

This requires a suitable notion of derivatives for operators between normed vector spaces.
Let $X,Y$ be normed vector spaces, $F:U\to Y$ be an operator with $U\subset X$ and $x\in U$, and $h\in X$ be arbitrary.
\begin{itemize}
    \item If the one-sided limit
        \begin{equation*}
            F'(x;h) := \lim_{t\to 0^+} \frac{F(x+th)-F(x)}{t}\in Y,
        \end{equation*}
        exists, it is called the \term[derivative!directional]{directional derivative} of $F$ in $x$ in direction $h$.
    \item If $F'(x;h)$ exists for all $h\in X$ and
        \begin{equation*}
            DF(x):X\to Y, h\mapsto F'(x;h)
        \end{equation*}
        defines a bounded linear operator, we call $F$ \term[derivative!Gateaux@Gâteaux]{Gateaux differentiable} (in $x$) and $DF\in \linop(X,Y)$ its \term[derivative!Gateaux@Gâteaux]{Gateaux derivative}.
    \item If additionally
        \begin{equation*}
            \lim_{\norm{h}_X\to 0} \frac{\norm{F(x+h) - F(x) - DF(x)h}_Y}{\norm{h}_X} = 0,
        \end{equation*}
        then $F$ is called \term[derivative!Frechet@Fréchet]{Fréchet differentiable} (in $x$) and $F'(x):=DF(x)\in \linop(X,Y)$ its \term[derivative!Frechet@Fréchet]{Fréchet derivative}.
    \item If the mapping $F':U\to \linop(X,Y)$, $x\mapsto F'(x)$, is (Lipschitz) continuous, we call $F$ \term[derivative!continuously differentiable]{(Lipschitz) continuously differentiable}.
\end{itemize}
The difference between Gateaux and Fréchet differentiability lies in the approximation error of $F$ near $x$ by $F(x)+DF(x)h$: While it only has to be bounded in $\norm{h}_X$ -- i.e., linear in $\norm{h}_X$ -- for a Gateaux differentiable function, it has to be superlinear in $\norm{h}_X$ if $F$ is Fréchet differentiable. (For a \emph{fixed} direction $h$, this is of course also the case for Gateaux differentiable functions; Fréchet differentiability thus additionally requires a uniformity in $h$.)

If $F$ is Gateaux differentiable, the Gateaux derivative can be computed via
\begin{equation*}
    DF(x) h = \left(\tfrac{d}{dt}F(x+th)\right)\Big|_{t=0}.
\end{equation*}
(However, the existence and linearity of this limit does \emph{not} show the Gateaux differentiability of $F$ since it doesn't imply that $DF(x)$ is continuous with respect to the right norms.)
Bounded linear operators $F\in \linop(X,Y)$ are obviously Fréchet differentiable with derivative $F'(x) = F \in \linop(X,Y)$ for all $x\in X$.
Note that the Gateaux derivative of a functional $F:X\to\R$ is an element of the \term[space!dual]{dual space}
$X^* = \linop(X,\R)$ and thus cannot be added to elements in $X$. However, in Hilbert spaces (and in particular in $\R^n$), we can use the \nameref{thm:frechetriesz} \cref{thm:frechetriesz} to identify $DF(x)\in X^*$ with an element $\nabla F(x) \in X$, called the \term{gradient} of $F$, in a canonical way via
\begin{equation*}
    DF(x)h = \inner{\nabla F(x)}{h}_X \qquad\text{for all } h\in X.
\end{equation*}
As an example, let us consider the functional $F(x) = \frac12\norm{x}_X^2$, where the norm is induced by the inner product. Then we have for all $x,h\in X$ that
\begin{equation*}
    F'(x;h) = \lim_{t\to 0^+} \frac{\frac12\inner{x+th}{x+th}_X - \frac12\inner{x}{x}_X}{t} = \inner{x}{h}_X = DF(x)h,
\end{equation*}
since the inner product is linear in $h$ for fixed $x$.
Hence, the squared norm is Gateaux differentiable in $x$ with derivative $DF(x) = h\mapsto \inner{x}{h}_X\in X^*$ and gradient $\nabla F(x) = x\in X$; it is even Fréchet differentiable since
\begin{equation*}
    \lim_{\norm{h}_X\to 0} \frac{\left|\frac12\norm{x+h}_X^2 - \frac12\norm{x}_X^2 - \inner{x}{h}_X\right|}{\norm{h}_X} = \lim_{\norm{h}_X\to 0}\frac12 \norm{h}_X = 0.
\end{equation*}
If the same mapping is now considered on a smaller Hilbert space $X'\hookrightarrow X$ (e.g., $X=L^2(\Omega)$ and $X'=H^1(\Omega)$), then the derivative $DF(x)\in (X')^*$ is still given by  $DF(x)h=\inner{x}{h}_X$ (now only for all $h\in X'$), but the gradient $\nabla F\in X'$ is now characterized by
\begin{equation*}
    DF(x)h = \inner{\nabla F(x)}{h}_{X'} \qquad\text{for all } h\in X'.
\end{equation*}
Different inner products thus lead to different gradients.

Further derivatives can be obtained through the usual calculus, whose proof in Banach spaces is exactly as in $\R^n$. As an example, we prove a chain rule.
\begin{theorem}\label{thm:frechet_chain}
    Let $X$, $Y$, and $Z$ be Banach spaces, and let $F:X\to Y$ be Fréchet differentiable in $x\in X$ and $G:Y\to Z$ be Fréchet differentiable in $y:=F(x)\in Y$. Then, $G\circ F$ is Fréchet differentiable in $x$ and
    \begin{equation*}
        (G\circ F)'(x) = G'(F(x))\circ F'(x).
    \end{equation*}
\end{theorem}
\begin{proof}
    For $h\in X$ with $x+h\in\dom F$ we have
    \begin{equation*}
        (G\circ F)(x+h ) - (G\circ F)(x) = G(F(x+h))-G(F(x)) = G(y+g)  - G(y)
    \end{equation*}
    with $g := F(x+h)-F(x)$. The Fréchet differentiability of $G$ thus implies that
    \begin{equation*}
        \norm{(G\circ F)(x+h ) - (G\circ F)(x) - G'(y)g }_Z =  r_1(\norm{g}_Y)
    \end{equation*}
    with $r_1(t)/t \to 0$ for $t\to 0$. The Fréchet differentiability of $F$ further implies
    \begin{equation*}
        \norm{g - F'(x)h}_Y = r_2(\norm{h}_X)
    \end{equation*}
    with $r_2(t)/t \to 0$ for $t\to 0$. In particular,
    \begin{equation}\label{eq:frechet_chain:est}
        \norm{g}_Y \leq \norm{F'(x)h}_Y + r_2(\norm{h}_X).
    \end{equation}
    Hence, with $c:=\norm{G'(F(x))}_{\linop(Y,Z)}$ we have
    \begin{equation*}
        \norm{(G\circ F)(x+h) - (G\circ F)(x) -  G'(F(x)) F'(x)h}_Z \leq  r_1(\norm{g}_Y) +  c\, r_2(\norm{h}_X).
    \end{equation*}
    If $\norm{h}_X\to 0$, we obtain from \eqref{eq:frechet_chain:est} and $F'(x)\in \linop(X,Y)$ that $\norm{g}_Y\to 0$ as well, and the claim follows.
\end{proof}
A similar rule for Gateaux derivatives does not hold, however.

We will also need the following variant of the mean value theorem.
Let $[a,b]\subset \R$ be a bounded interval and $f:[a,b]\to X$ be continuous. We then define  the \term{Bochner integral} $\int_a^b f(t)\,dt\in X$ using the \nameref{thm:frechetriesz} \cref{thm:frechetriesz} via
\begin{equation} \label{eq:bochner}
    \inner{\int_a^b f(t)\,dt}{z}_X = \int_a^b\inner{f(t)}{z}_X\,dt \qquad\text{for all }z\in X,
\end{equation}
since by the continuity of $t\mapsto \norm{f(t)}_X$ on the compact interval $[a,b]$, the right-hand side defines a continuous linear functional on $X$. The construction then directly implies that
\begin{equation}
    \label{eq:bochner_est}
    \left\|\int_a^b f(t)\,dt\right\|_X \leq \int_a^b \norm{f(t)}_X\,dt.
\end{equation}
\begin{theorem}\label{thm:frechet:mittelwert}
    Let $F:U\to Y$ be Fréchet differentiable, and let $x\in U$ and $h\in X$ be given with $x+th\in U$ for all $t\in[0,1]$. Then
    \begin{equation*}
        F(x+h) - F(x) =\int_0^1 F'(x+th)h\,dt.
    \end{equation*}
\end{theorem}
\begin{proof}
    Consider for arbitrary $y\in Y$ the function
    \begin{equation*}
        f:[0,1]\to \R,\qquad t\mapsto\inner{F(x+th)}{y}_Y.
    \end{equation*}
    From \cref{thm:frechet_chain} we obtain that $f$ (as a composition of operators between normed vector spaces) is differentiable with
    \begin{equation*}
        f'(t) = \inner{F'(x+th)h}{y}_Y,
    \end{equation*}
    and the fundamental theorem of calculus in $\R$ yields that
    \begin{equation*}
        \inner{F(x+h)-F(x)}{y}_Y = f(1)-f(0) = \int_0^1f'(t)\,dt = \inner{\int_0^1 F'(x+th)h\,dt}{y}_Y,
    \end{equation*}
    where the last equality follows from \eqref{eq:bochner}.
    Since $y\in Y$ was arbitrary, the claim follows.
\end{proof}

If the Fréchet derivative is locally Lipschitz continuous, i.e., if there exist $L>0$ and $\delta>0$ such that
\begin{equation}
    \label{eq:nichtlin:lipschitz}
    \norm{F'(x_1)-F'(x_2)}_{\linop(X,Y)} \leq L\norm{x_1-x_2}_X\qquad\text{for all }x_1,x_2\in B_\delta(x),
\end{equation}
the linearization error can even be estimated quadratically.
\begin{lemma}\label{lem:nichtlin:lipschitz}
    Let $F:U\to Y$ be Lipschitz continuously differentiable in a neighborhood $V\subset U$ of $x\in U$. Then for all $h\in X$ with $x+th\in V$ for $t\in[0,1]$,
    \begin{equation*}
        \norm{F(x+h)-F(x) - F'(x)h}_Y \leq \frac{L}2\norm{h}_X^2.
    \end{equation*}
\end{lemma}
\begin{proof}
    \Cref{thm:frechet:mittelwert} together with \eqref{eq:bochner_est} and \eqref{eq:nichtlin:lipschitz} directly yields that
    \begin{equation*}
        \begin{split}
            \norm{F(x+h)-F(x) - F'(x)h}_Y
            &= \left\|\int_0^1 F'(x+th)h -F'(x)h\,dt\right\|_Y\\
            &\leq \int_0^1 \norm{F'(x+th)h -F'(x)h}\,dt\\
            &\leq \int_0^1 Lt\norm{h}_X^2\,dt = \frac{L}2 \norm{h}_X^2.
            \qedhere
        \end{split}
    \end{equation*}
\end{proof}

\bigskip

A natural question is now about the relationship between the local ill-posedness of $F:U\to Y$ in $x$ and of its linearization $F'(x)\in \linop(X,Y)$. The following result suggests that at least for completely continuous operators, the latter inherits the ill-posedness of the former.
\begin{theorem}\label{thm:nichtlin:frechet_complete_cont}
    If $F:U\to Y$ is completely continuous and Fréchet differentiable in $x\in U$, then $F'(x)\in \linop(X,Y)$ is compact.
\end{theorem}
\begin{proof}
    Let $x\in U$ be arbitrary and assume to the contrary that $F'(x)$ is not compact and therefore not completely continuous. Then there exists a sequence $\{h_n\}_{n\in\N}$ with $h_n\wkto 0$ as well as an $\eps>0$ such that
    \begin{equation*}
        \norm{F'(x)h_n}_Y \geq \eps \qquad\text{for all }n\in \N.
    \end{equation*}
    Since weak convergence implies boundedness, we can assume without loss of generality (by proper scaling of $h_n$ and $\eps$) that $\norm{h_n}_X\leq 1$ for all $n\in\N$.
    By definition of the Fréchet derivative, there then exists a $\delta>0$ such that
    \begin{equation*}
        \norm{F(x+h)-F(x)-F'(x)h}_Y \leq \frac{\eps}2 \norm{h}_X \qquad\text{for all }\norm{h}_X\leq\delta.
    \end{equation*}
    Since $\{h_n\}_{n\in\N}$ is bounded, there exists a $\tau>0$ sufficiently small that $\norm{\tau h_n}_X \leq\delta$ and $x+\tau h_n\in U$ for all $n\in \N$ (otherwise $F$ would not be differentiable in $x$). Then we have that $x+\tau h_n\wkto x$; however, for all $n\in\N$,
    \begin{equation*}
        \begin{aligned}
            \norm{F(x+\tau h_n) - F(x)}_Y
            &= \norm{F'(x)(\tau h_n) + F(x+\tau h_n) - F(x)-F'(x)(\tau h_n)}_Y\\
            &\geq \norm{F'(x)(\tau h_n)}_Y - \norm{F(x+\tau h_n)-F(x)-F'(x)(\tau h_n)}_Y\\
            &\geq \tau \eps - \tau\norm{h_n}_X \frac\eps{2} \geq \tau \frac\eps{2}.
        \end{aligned}
    \end{equation*}
    Hence $F$ is not completely continuous.
\end{proof}
Note that this does not necessarily imply that $F'(x)h=y-F(x)$ is ill-posed, as $F'(x)$ may happen to have finite-dimensional range. Conversely, a locally well-posed problem may have an ill-posed linearization; see \cite[Example~\textsc{a}.1, \textsc{a}.2]{EKN}. This naturally has consequences for any regularization that relies on linearization.
The reason for this discrepancy is the fact that although the linearization error tends to zero superlinearly as $\norm{h}_X\to0$, for fixed $h\in X$ the error may be much larger than either the nonlinear residual $y-F(x)$ or the linear residual $y-F(x)-F'(x)h$. To obtain stronger results, we thus have to impose conditions on the nonlinearity of $F$.

One possibility is to require more smoothness of $F$, e.g., local Lipschitz continuity of the derivative around $x\in U$. Under this assumption, the linearization indeed inherits the local ill-posedness.
\begin{theorem}
    Let $F:U\to Y$ be Fréchet differentiable with locally Lipschitz continuous derivative. If $F$ is locally ill-posed in $x\in U$, then $F'(x)$ is locally ill-posed in all $h\in\calN(F'(x))$.
\end{theorem}
\begin{proof}
    Assume to the contrary that the nonlinear operator is locally ill-posed but its linearization is locally well-posed. The latter is equivalent to $F'(x)$ being injective and having closed range. Hence by \cref{thm:inverse:cont} there exists a continuous pseudoinverse $F'(x)^\dag\in \linop(Y,X)$.
    Now if $F'(x)^\dag$ is continuous, so is $(F'(x)^*)^\dag = (F'(x)^\dag)^*$, and we can thus find for all $h\in X$ a $w:= (F'(x)^*)^\dag h\in Y$ with $\norm{w}_Y \leq C\norm{h}_X$.
    Letting $\mu\in(0,1)$ and setting $\delta:= \frac{2\mu}{CL}$, we then have in particular that $\norm{w}_Y\leq \frac{2\mu}{L}$ for all $\norm{h}_X\leq \delta$.
    Furthermore, \cref{lem:inverse}\,(iv) together with $\calR(F'(x)^*) = \overline{\calR(F'(x)^*)}=\calN(F'(x))^\bot = X$ (since if $(F'(x)^*)^\dag$ is continuous, $F'(x)^*$ has closed range as well) implies that
    \begin{equation*}
        F'(x)^*w = F'(x)^*(F'(x)^*)^\dag h = h.
    \end{equation*}

    We now bound the linearization error with the help of this \enquote{linearized source condition} and \cref{lem:nichtlin:lipschitz}: For all $h\in X$ with $\norm{h}_X\leq \delta$, we have that
    \begin{equation*}
        \begin{aligned}
            \norm{F(x+h)-F(x) - F'(x)h}_Y
            &\leq \frac{L}2\norm{h}_X^2 = \frac{L}2\norm{F'(x)^*w}_X^2 = \frac{L}{2} \inner{F'(x)F'(x)^*w}{w}_Y\\
            &\leq \frac{L}{2}\norm{F'(x)F'(x)^*w}_Y\norm{w}_Y\\
            &\leq \mu \norm{F'(x)h}_Y.
        \end{aligned}
    \end{equation*}
    The triangle inequality then yields that
    \begin{equation*}
        \begin{aligned}
            \norm{F'(x)h}_Y &= \norm{F(x+h)-F(x) - F'(x)h - F(x+h)+F(x)}_Y\\
            &\leq \mu \norm{F'(x)h}_Y + \norm{F(x+h)-F(x)}_Y
        \end{aligned}
    \end{equation*}
    and hence that
    \begin{equation}
        \label{eq:nichtlin:bound1}
        \norm{F'(x)h}_Y \leq \frac{1}{1-\mu}\norm{F(x+h)-F(x)}_Y \qquad\text{for all }\norm{h}_X\leq \delta.
    \end{equation}

    Since we have assumed that $F$ is locally ill-posed, there has to exist a sequence $\{h_n\}_{n\in\N}$ with $\norm{x+h_n-x}_X = \norm{h_n}_X=\frac\delta2$ but $F(x+h_n)\to F(x)$. But from \eqref{eq:nichtlin:bound1}, we then obtain that $F'(x)(x+h_n-x) = F'(x)h_n \to 0$, in contradiction to the assumed local well-posedness of the linearization.
\end{proof}

An alternative to \eqref{eq:nichtlin:lipschitz} is the so-called \term[condition!tangential cone]{tangential cone condition}: For given $x\in U$, there exist $\eta<1$ and $\delta > 0$ such that
\begin{equation}
    \label{eq:nichtlin:tangential}
    \norm{F(x+h)-F(x) - F'(x)h}_Y \leq \eta \norm{F(x+h)-F(x)}_Y \qquad\text{for all }\norm{h}_X\leq \delta.
\end{equation}
In other words, the linearization error should be \emph{uniformly} bounded by the nonlinear residual. Here we can even show equivalence.
\begin{theorem}
    Let $F:U\to Y$ be Fréchet differentiable and satisfy the tangential cone condition \eqref{eq:nichtlin:tangential} in $x\in U$. Then $F$ is locally ill-posed in $x\in U$ if and only if $F'(x)$ is locally ill-posed (in any $h\in X$).
\end{theorem}
\begin{proof}
    From the tangential cone condition together with the (standard and reverse) triangle inequalities, we obtain that
    \begin{equation}\label{eq:nichtlin:tangential_est}
        (1-\eta)\norm{F(x+h)-F(x)}_Y\leq \norm{F'(x)h}_Y \leq (1+\eta) \norm{F(x+h)-F(x)}_Y
    \end{equation}
    for all $\norm{h}_X\leq \delta$.
    The second inequality coincides with \eqref{eq:nichtlin:bound1}, which we have already shown to imply the local ill-posedness of the linearization of a locally ill-posed nonlinear operator.
    We can argue similarly for the first inequality: Assume that $F'(x)$ is locally ill-posed. Then there exists a sequence $\{h_n\}_{n\in\N}$ with $\norm{x+h_n-x}_X = \norm{h_n}=\frac\delta2$ but $F'(x)h_n\to 0$, which together with \eqref{eq:nichtlin:tangential_est} implies that $F(x+h_n)\to F(x)$ as well. Hence, $F$ is also ill-posed.
\end{proof}
In combination with a weak source condition, the tangential cone condition even implies local uniqueness of the $x_0$-minimum norm solution.
\begin{theorem}
    Let $F:U\to Y$ be Fréchet differentiable and $y\in Y$ and $x_0\in X$ be given. If the tangential cone condition \eqref{eq:nichtlin:tangential} holds in $x^\dag \in U$ with $F(x^\dag)=y$ and $x^\dag -x_0\in\calN(F'(x^\dag))^\bot$, then $x^\dag$ is the unique $x_0$-minimum norm solution in $B_\delta(x^\dag)$ for the $\delta>0$ from \eqref{eq:nichtlin:tangential}.
\end{theorem}
\begin{proof}
    Let $x\in B_\delta(x^\dag)\setminus\{x^\dag\}$ with $F(x) = y$ be arbitrary. Then \eqref{eq:nichtlin:tangential} for $h:= x- x^\dag$ implies that $F'(x^\dag)(x-x^\dag)=0$, i.e., that $x-x^\dag\in \calN(F'(x^\dag))\setminus\{0\}$. It follows that
    \begin{equation*}
        \begin{aligned}
            \norm{x-x_0}_X^2 &= \norm{x^\dag - x_0 + x - x^\dag}_X^2 \\
            &= \norm{x^\dag-x_0}^2_X + 2\inner{x^\dag - x_0}{x-x^\dag}_X + \norm{x-x^\dag}^2_X \\
            &> \norm{x^\dag-x_0}^2_X
        \end{aligned}
    \end{equation*}
    since the inner product vanishes due to orthogonality and we have assumed that $x\neq x^\dag$. Hence $x^\dag$ is the (locally) unique $x_0$-minimum norm solution.
\end{proof}

It should be admitted that it is often very difficult to verify these abstract conditions for concrete nonlinear inverse problems; there are even examples where these can be shown \emph{not} to hold.
Thus one often uses strongly problem-specific approaches instead of an abstract theory for nonlinear problems.%
\footnote{\enquote{Linear inverse problems are all alike; every nonlinear inverse problem is nonlinear in its own way.}}
Still, the abstract perspective can be useful by showing limits and possibilities.

\chapter{Tikhonov regularization}\label{chap:tikhonov-nl}

The starting point of Tikhonov regularization of nonlinear inverse problems $F(x)=y$ is \cref{thm:tikhonov:funktional}: For given $\alpha>0$, $x_0\in X$, and $y\in Y$, we choose $x_\alpha$ as minimizer of the  Tikhonov functional\index{Tikhonov functional!nonlinear}
\begin{equation}\label{eq:tikh_nl:funktional}
    J_\alpha(x):= \frac12\norm{F(x)-y}_Y^2 + \frac\alpha2 \norm{x-x_0}_X^2.
\end{equation}
If $F$ is not linear, we cannot express this choice through an explicit regularization operator $R_\alpha$. We thus have to proceed differently to show existence of a solution, continuous dependence of $x_\alpha$ on $y$, and convergence to an $x_0$-minimum norm solution as $\alpha\to 0$.
On the other hand, this is possible under weaker assumptions: It suffices to require that $F$ is weakly closed with nonempty and weakly closed domain $\dom F=:U$ (which we always assume from here on).
These assumptions also ensure for $y\in \calR(F)$ the existence of a (not necessarily unique) $x_0$-minimum norm solution $x^\dag\in U$.

\section{Well-posedness}\label{sec:tikh_nl:wellposed}

We first show existence of a minimizer. The proof is a classical application of \href{http://www-history.mcs.st-and.ac.uk/Biographies/Tonelli.html}{Tonelli's} \term[method!direct]{direct method} of the calculus of variations, which generalizes the Weierstraß theorem (every continuous function attains its minimum and maximum on a finite-dimensional compact set) to infinite-dimensional vector spaces.
\begin{theorem}\label{thm:tikh_nl:existenz}
    Let $F:U\to Y$ be weakly closed, $\alpha>0$, $x_0\in X$, and $y\in Y$. Then there exists a minimizer $x_\alpha\in U$ of $J_\alpha$.
\end{theorem}
\begin{proof}
    We first note that $J_\alpha(x)\geq 0$ for all $x\in U$. Hence the set $\setof{J_\alpha(x)}{x\in U}\subset \R$ is bounded from below and thus has a finite infimum. This implies that there exists a sequence $\{x_n\}_{n\in \N}\subset U$ such that
    \begin{equation*}
        J_\alpha(x_n)\to m:= \inf\setof{J_\alpha(x)}{x\in U}.
    \end{equation*}
    Such a sequence is called a \term[sequence, minimizing]{minimizing sequence}. Note that the convergence of $\{J_\alpha(x_n)\}_{n\in\N}$ does not imply the convergence of $\{x_n\}_{n\in\N}$.

    However, since convergent sequences are bounded, there exists an $M>0$ such that
    \begin{equation}
        \label{eq:tikh_nl:existenz1}
        \frac12\norm{F(x_n)-y}_Y^2 + \frac\alpha2 \norm{x_n-x_0}_X^2 = J_\alpha(x_n) \leq M\qquad \text{for all }n\in \N.
    \end{equation}
    It follows that
    \begin{equation*}
        \frac\alpha2\left(\norm{x_n}_X -  \norm{x_0}_X\right)^2 \leq  \frac\alpha2 \norm{x_n-x_0}_X^2 \leq J_\alpha(x_n) \leq M,
    \end{equation*}
    i.e., the sequence $\{x_n\}_{n\in\N}$ is bounded and thus contains a weakly convergent subsequence -- which we again denote by $\{x_k\}_{k\in\N}$ for simplicity -- with limit $\bar x\in U$ (since $U$ is assumed to be weakly closed). This limit is a candidate for a minimizer.

    Similarly, \eqref{eq:tikh_nl:existenz1} implies that $\{F(x_k)\}_{k\in\N}$ is bounded in $Y$.
    By passing to a further subsequence (which we still denote by $\{x_k\}_{k\in\N}$), we thus obtain that $F(x_k)\wkto \bar y\in Y$, and the weak closedness of $F$ yields that $\bar y = F(\bar x)$.
    Together with the weak lower semicontinuity of norms, we obtain that
    \begin{equation*}
        \begin{aligned}
            \frac12\norm{F(\bar x)-y}_Y^2 + \frac\alpha2 \norm{\bar x-x_0}_X^2 &\leq \liminf_{k\to\infty}\frac12\norm{F(x_k)-y}_Y^2 +
            \liminf_{k\to\infty} \frac\alpha2 \norm{x_k-x_0}_X^2\\
            &\leq  \limsup_{k\to\infty}\left(\frac12\norm{F(x_k)-y}_Y^2 + \frac\alpha2 \norm{x_k-x_0}_X^2\right).
        \end{aligned}
    \end{equation*}
    By definition of the minimizing sequence, $J_\alpha(x_k)\to m$ for any subsequence as well, and hence
    \begin{equation*}
        \inf_{x\in U} J_\alpha(x) \leq J_\alpha(\bar x) \leq \limsup_{k\to\infty} J_\alpha(x_k) = m = \inf_{x\in U} J_\alpha (x).
    \end{equation*}
    The infimum is thus attained in $\bar x$, i.e., $J_\alpha(\bar x) = \min_{x\in U} J_\alpha(x)$.
\end{proof}
Due to the nonlinearity of $F$, we cannot in general expect the minimizer to be unique, so that we cannot introduce a well-defined mapping $y\mapsto x_\alpha$ as a regularization operator. In place of the continuity of $R_\alpha$, we can therefore only show the following weaker stability result.
\begin{theorem}\label{thm:tikh_nl:stabil}
    Let $F:U\to Y$ be  weakly closed, $\alpha>0$, $x_0\in X$, and $y\in Y$. Let $\{y_n\}_{n\in\N}$ be a  sequence with $y_n\to y$ and let $\{x_n\}_{n\in\N}$ be a sequence of minimizers of $J_{\alpha}$ for $y_n$ in place of $y$. Then the sequence $\{x_n\}_{n\in\N}$ contains a weakly convergent subsequence, and every weak accumulation point of $\{x_n\}_{n\in\N}$ is a minimizer of $J_\alpha$.

    If $J_\alpha$ has for $y$ a unique minimizer, then the whole sequence converges strongly.
\end{theorem}
\begin{proof}
    First, \cref{thm:tikh_nl:existenz} ensures that for every $y_n\in Y$ there exists a minimizer $x_n\in U$. The minimizing property of $x_n$ then implies for all $n\in \N$ and any $x\in U$ that
    \begin{equation*}
        \frac12\norm{F(x_n)-y_n}_Y^2 + \frac\alpha2\norm{x_n-x_0}_X^2 \leq  \frac12\norm{F(x)-y_n}_Y^2 +  \frac\alpha2\norm{x-x_0}_X^2.
    \end{equation*}
    Since $y_n\to y$, the right-hand side is bounded in $n\in\N$, and hence both $\{x_n\}_{n\in\N}$ and $\{F(x_n)-y_n\}_{n\in\N}$ are bounded as well.
    We can thus find a weakly convergent subsequence $\{x_k\}_{k\in\N}$ and a $\bar x\in U$ such that (possibly after passing to a further subsequence)
    \begin{equation*}
        x_k\wkto \bar x,\qquad F(x_k)-y_k\wkto \bar y.
    \end{equation*}
    The convergence of $y_k\to y$ and the weak closedness of $F$ then imply that $F(x_k)\wkto F(\bar x)$.

    From the weak lower semicontinuity of norms, we obtain from this that
    \begin{align}
        \frac\alpha2\norm{\bar x - x_0}_X^2 &\leq \liminf_{k\to\infty}\frac\alpha2\norm{ x_k - x_0}_X^2,\\
        \frac12\norm{F(\bar x)-y}_Y^2 &\leq \liminf_{k\to\infty}\frac12\norm{F(x_k)-y_k}_Y^2.\label{eq:tikh_nl:stabil1}
    \end{align}
    Using again the minimization property of the $x_n$, this implies that for any $x\in U$,
    \begin{equation}\label{eq:tikh_nl:stabil2}
        \begin{aligned}[t]
            J_\alpha(\bar x) &= \frac12\norm{F(\bar x)-y}_Y^2 + \frac\alpha2 \norm{\bar x-x_0}_X^2 \\
            &\leq \liminf_{k\to\infty}\left(\frac12\norm{F(x_k)-y_k}_Y^2 + \frac\alpha2 \norm{x_k-x_0}_X^2\right)\\
            &\leq \limsup_{k\to\infty}\left(\frac12\norm{F(x_k)-y_k}_Y^2 + \frac\alpha2 \norm{x_k-x_0}_X^2\right)\\
            &\leq \limsup_{k\to\infty}\left(\frac12\norm{F(x)-y_k}_Y^2 + \frac\alpha2 \norm{x-x_0}_X^2\right)\\
            &= \lim_{k\to\infty}\frac12\norm{F(x)-y_k}_Y^2 + \frac\alpha2 \norm{x-x_0}_X^2\\
            &= \frac12\norm{F(x)-y}_Y^2 + \frac\alpha2 \norm{x -x_0}_X^2 = J_\alpha(x).
        \end{aligned}
    \end{equation}
    Hence $\bar x$ is a minimizer of $J_\alpha$. Since this argument can be applied to any weakly convergent subsequence of $\{x_n\}_{n\in\N}$, we also obtain the second claim.

    If now the minimizer $x_\alpha$ of $J_\alpha$ is unique, then every weakly convergent subsequence has the same limit, and hence the whole sequence must converge weakly to $x_\alpha$. To show that this convergence is in fact strong, it suffices by \eqref{eq:funktan:lsc} to show that $\limsup_{n\to\infty}\norm{x_n}_X\leq \norm{x_\alpha}_X$.
    Assume to the contrary that this inequality does not hold. Then there must exist a subsequence $\{x_k\}_{k\in\N}$ with $x_k\wkto x_\alpha$ and $F(x_k)\wkto F(x_\alpha)$ but
    \begin{equation*}
        \lim_{k\to\infty} \norm{x_k-x_0}_X =:  M > \norm{x_\alpha-x_0}_X.
    \end{equation*}
    But \eqref{eq:tikh_nl:stabil2} for $x=\bar x= x_\alpha$ implies that
    \begin{equation*}
        \lim_{k\to\infty} \left(\frac12\norm{F(x_k)-y_k}_Y^2 + \frac\alpha2 \norm{x_k-x_0}_X^2\right) =  \frac12\norm{F(x_\alpha)-y}_Y^2 + \frac\alpha2 \norm{x_\alpha-x_0}_X^2.
    \end{equation*}
    Together with the calculus for convergent sequences, this shows that
    \begin{equation*}
        \begin{aligned}
            \lim_{k\to\infty} \frac12\norm{F(x_k)-y_k}_Y^2 &=
            \lim_{k\to\infty} \left(\frac12\norm{F(x_k)-y_k}_Y^2 +\frac\alpha2\norm{x_k-x_0}_X^2\right) -\lim_{k\to\infty} \frac\alpha2\norm{x_k-x_0}_X^2 \\
            &=
            \frac12\norm{F(x_\alpha)-y}_Y^2 + \frac\alpha2 \norm{x_\alpha-x_0}_X^2-  \frac\alpha2 M^2 \\
            &< \frac12\norm{F(x_\alpha)-y}_Y^2
        \end{aligned}
    \end{equation*}
    in contradiction to \eqref{eq:tikh_nl:stabil1} and $\bar x = x_\alpha$.
\end{proof}

\section{Convergence}\label{sec:tikh_nl:conv}

It remains to show that $x_\alpha$ converges to an $x_0$-minimum norm solution as $\alpha\to 0$. In contrast to the linear case, we do this already in combination with an a priori choice rule, i.e., we prove that this combination leads to a convergent regularization method. In analogy to \cref{sec:spektral:parameter}, we denote by $x_\alpha^\delta$ a minimizer of $J_\alpha$ for fixed $\alpha>0$ and noisy data $y^\delta\in Y$.
\begin{theorem}\label{thm:tikh_nl:conv}
    Let $F:U\to Y$ be weakly closed, $y\in \calR(F)$, and $y^\delta\in B_\delta(y)$. If $\alpha(\delta)$ is a  parameter choice rule such that
    \begin{equation*}
        \alpha(\delta) \to 0\quad \text{and}\quad \frac{\delta^2}{\alpha(\delta)}\to 0 \qquad\text{for }\delta\to 0,
    \end{equation*}
    then every sequence $\{x_{\alpha(\delta_n)}^{\delta_n}\}_{n\in\N}$ with $\delta_n\to 0$ contains a strongly convergent subsequence, and every accumulation point is an $x_0$-minimum norm solution of $F(x)=y$.
    If the $x_0$-minimum norm solution $x^\dag\in U$ is unique, then the whole sequence converges strongly to $x^\dag$.
\end{theorem}
\begin{proof}
    Set $\alpha_n:= \alpha(\delta_n)$ and $x_n:= x_{\alpha_n}^{\delta_n}$, and let $x^\dag$ be an $x_0$-minimum norm solution of $F(x)=y$. Then the minimization property of $x_n$ implies that for all $n\in\N$,
    \begin{equation}
        \label{eq:tikh_nl:conv1}
        \begin{aligned}[t]
            \frac12\norm{F(x_n)-y^{\delta_n}}_Y^2 + \frac{\alpha_n}2 \norm{x_n-x_0}_X^2
            &\leq  \frac12\norm{F(x^\dag)-y^{\delta_n}}_Y^2 + \frac{\alpha_n}2 \norm{x^\dag-x_0}_X^2\\
            &\leq \frac{\delta_n^2}2 +  \frac{\alpha_n}2 \norm{x^\dag-x_0}_X^2.
        \end{aligned}
    \end{equation}
    In particular,
    \begin{equation}
        \label{eq:tikh_nl:conv2}
        \norm{x_n-x_0}_X^2 \leq \frac{\delta_n^2}{\alpha_n}  + \norm{x^\dag-x_0}_X^2 \qquad\text{for all }n\in\N,
    \end{equation}
    and the right-hand side is bounded due to the convergence $\frac{\delta_n^2}{\alpha_n} \to 0$.
    Hence there exists a weakly convergent subsequence $\{x_{k}\}_{k\in\N}$ and an $\bar x\in U$ with $x_k\wkto \bar x$.
    Similarly, we obtain from \eqref{eq:tikh_nl:conv1} that
    \begin{equation}
        \label{eq:tikh_nl:conv3}
        \frac12\norm{F(x_k)-y^{\delta_k}}_Y^2 \leq \frac{\delta_k^2}{2}  +  \frac{\alpha_k}2 \norm{x^\dag-x_0}_X^2 \qquad\text{for all }k\in\N.
    \end{equation}
    This implies that $\{F(x_k)-y^{\delta_k}\}_{k\in\N}$ in turn contains a weakly convergent subsequence (which we do not further distinguish) with limit $\bar y\in Y$. The weak closedness of $F$ and the strong convergence $y^{\delta_n}\to y$ then again yield that $\bar y=F(\bar x)-y$, i.e., that $F(x_k)\wkto F(\bar x)$.

    We now obtain from the weak lower semicontinuity of the norm together with \eqref{eq:tikh_nl:conv2} that
    \begin{equation}
        \label{eq:tikh_nl:conv4}
        \begin{aligned}[t]
            \norm{\bar x-x_0}_X^2  &\leq \liminf_{k\to\infty} \norm{x_k-x_0}_X^2\leq \limsup_{k\to\infty} \norm{x_k-x_0}_X^2\\
            &\leq \lim_{k\to\infty}\frac{\delta_k^2}{\alpha_k} + \norm{x^\dag-x_0}_X^2
            = \norm{x^\dag-x_0}_X^2,
        \end{aligned}
    \end{equation}
    and similarly from \eqref{eq:tikh_nl:conv3} that
    \begin{equation*}
        \norm{F(\bar x)-y}_Y^2 \leq \liminf_{k\to\infty}\norm{F(x_k)-y^{\delta_k}}_Y^2\leq \lim_{k\to\infty} \left(\delta_k^2  +  {\alpha_k} \norm{x^\dag-x_0}_X^2\right)=0.
    \end{equation*}
    Hence, $F(\bar x) = y$ and
    \begin{equation*}
        \norm{\bar x - x_0}_X \leq  \norm{x^\dag - x_0}_X = \min\setof{\norm{x-x_0}_X}{F(x)=y}\leq  \norm{\bar x - x_0}_X,
    \end{equation*}
    i.e., $\bar x$ is an $x_0$-minimum norm solution.

    It remains to show that the subsequence $\{x_k\}_{k\in\N}$ converges strongly. We start from the Pythagoras identity
    \begin{equation*}
        \norm{x_k-\bar x}_X^2 = \norm{x_k - x_0}_X^2 - 2\inner{x_k-x_0}{\bar x-x_0}_X + \norm{\bar x-x_0}_X^2.
    \end{equation*}
    The weak convergence $x_k\wkto \bar x$ then implies that
    \begin{equation*}
        \lim_{k\to\infty} 2\inner{x_k-x_0}{\bar x-x_0}_X = 2\inner{\bar x -x_0}{\bar x-x_0}_X = 2\norm{\bar x-x_0}^2.
    \end{equation*}
    Furthermore, \eqref{eq:tikh_nl:conv4} and the fact that both $\bar x$ and $x^\dag$ are $x_0$-minimum norm solutions imply that
    \begin{equation*}
        \lim_{k\to\infty} \norm{x_k - x_0}_X = \norm{\bar x - x_0}_X=\norm{x^\dag - x_0}_X.
    \end{equation*}
    Together, we obtain that
    \begin{equation*}
        0\leq \limsup_{k\to\infty} \norm{x_k-\bar x}_X^2 \leq  \norm{\bar x - x_0}_X^2 -2\norm{\bar x-x_0}^2 + \norm{\bar x-x_0}_X^2 = 0,
    \end{equation*}
    i.e., that $x_k\to \bar x$. The claim for a unique $x_0$-minimum norm solution again follows from a subsequence-subsequence argument.
\end{proof}

\section{Convergence rates}\label{sec:tikh_nl:rates}

We now derive error estimates under a source condition, where we restrict ourselves to the simplest case that corresponds to the choice $\nu=1$ for linear inverse problems. As a motivation, we again consider the formal limit problem \eqref{eq:tikhonov:funkt_const} for $\alpha=0$, which in the nonlinear case becomes
\begin{equation*}
    \min_{x\in U,\,F(x)=y} \frac12 \norm{x-x_0}_X^2
\end{equation*}
and again characterizes the $x_0$-minimum norm solutions.
As before, we introduce a Lagrange multiplier $p\in Y$ for the equality constraint to obtain the saddle-point problem
\begin{equation*}
    \min_{x\in U}\max_{p\in Y} L(x,p),\qquad L(x,p):=  \frac12\norm{x-x_0}_X^2 - \inner{p}{F(x)-y}_Y.
\end{equation*}
Setting the partial Fréchet derivative $L'_p(\bar x,\bar p)$ of $L$ with respect to $p$ to zero again yields the necessary condition $F(\bar x)=y$ for a saddle point $(\bar x,\bar p)\in U\times Y$.
If we assume for simplicity that the $x_0$-minimum norm solution $x^\dag$ is an interior point of $U$,
then we can also set the Fréchet derivative $L'_x(x^\dag,p^\dag)$ of $L$ with respect to $x$ in the corresponding saddle point $(x^\dag,p^\dag)$ to zero; this implies for all $h\in X$ that
\begin{equation*}
    0=L'_x(x^\dag,p^\dag)h = \inner{x^\dag-x_0}{h}_X - \inner{p^\dag}{F'(x^\dag)h}_Y = \inner{x^\dag-x_0 - F'(x^\dag)^*p^\dag}{h}_X,
\end{equation*}
i.e., the existence of a $p^\dag\in Y$ with
\begin{equation*}
    x^\dag- x_0 = F'(x^\dag)^*p^\dag.
\end{equation*}
This is our source condition\index{source condition!nonlinear} in the nonlinear setting. However, as in the last chapter we require an additional nonlinearity condition for $F$ in the $x_0$-minimum norm solution; here we assume the Lipschitz condition \eqref{eq:nichtlin:lipschitz}.
\begin{theorem}\label{thm:tikh_nl:rate_apriori}
    Let $F:U\to Y$ be Fréchet differentiable with convex domain $\dom F =U$. Let further $y\in\calR(F)$ and $y^\delta\in B_\delta(y)$, and let $x^\dag$ be an $x_0$-minimum norm solution such that
    \begin{enumerate}
        \item $F'$ is Lipschitz continuous near $x^\dag$ with Lipschitz constant $L$;
        \item there exists a $w\in Y$ with $x^\dag -x_0 = F'(x^\dag)^*w$ and $L\norm{w}_Y <1$.
    \end{enumerate}
    If $\alpha(\delta)$ is a parameter choice rule with
    \begin{equation*}
        c\delta \leq \alpha(\delta) \leq C \delta  \qquad\text{for } c,C>0,
    \end{equation*}
    then there exist constants $c_1,c_2>0$ such that for all $\delta>0$ small enough,
    \begin{align*}
        \norm{x_{\alpha(\delta)}^\delta-x^\dag}_X &\leq c_1 \sqrt{\delta},\\
        \norm{F(x_{\alpha(\delta)}^\delta)-y^\delta}_Y &\leq c_2 {\delta}.
    \end{align*}
\end{theorem}
\begin{proof}
    First, the minimizing property of $x_\alpha^\delta$ for $\alpha:= \alpha(\delta)$ again implies that
    \begin{equation}\label{eq:tikh_nl:rate0}
        \frac12\norm{F(x_\alpha^\delta)-y^{\delta}}_Y^2 + \frac{\alpha}2 \norm{x_\alpha^\delta-x_0}_X^2 \leq  \frac{\delta^2}2 + \frac{\alpha}2 \norm{x^\dag-x_0}_X^2.
    \end{equation}
    To obtain from this an estimate of $x_\alpha^\delta-x^\dag$, we use the productive zero $x^\dag -x^\dag$ on the left-hand side and the Pythagoras identity, which yields the identity
    \begin{equation*}
        \norm{x_\alpha^\delta -x_0}_X^2 = \norm{x_\alpha^\delta - x^\dag}_X^2 + 2\inner{x_\alpha^\delta - x^\dag}{ x^\dag -x_0}_X + \norm{x^\dag -x_0}_X^2.
    \end{equation*}
    Inserting this into \eqref{eq:tikh_nl:rate0} and using the source condition (ii) then shows that
    \begin{equation}
        \label{eq:tikh_nl:rate1}
        \begin{aligned}[t]
            \frac12\norm{F(x_\alpha^\delta)-y^{\delta}}_Y^2 + \frac{\alpha}2 \norm{x_\alpha^\delta-x^\dag}_X^2
            &\leq \frac{\delta^2}2 + {\alpha} \inner{x^\dag-x_0}{x^\dag-x_\alpha^\delta}_X\\
            &= \frac{\delta^2}2 + \alpha\inner{w}{F'(x^\dag)(x^\dag-x_\alpha^\delta)}_Y\\
            &\leq \frac{\delta^2}2 + \alpha\norm{w}_Y\norm{F'(x^\dag)(x^\dag-x_\alpha^\delta)}_Y.
        \end{aligned}
    \end{equation}

    Since $x_\alpha^\delta,x^\dag\in U$ and $U$ is convex, the condition (i) allows us to apply \cref{lem:nichtlin:lipschitz} for $x=x^\dag$ and $h=x_\alpha^\delta-x^\dag$ to obtain
    \begin{equation*}
        \norm{F(x^\dag)-F(x_\alpha^\delta) - F'(x^\dag)(x^\dag-x_\alpha^\delta)}_Y \leq \frac{L}2\norm{x^\dag -x_\alpha^\delta}_X^2.
    \end{equation*}
    Together with the triangle inequalities, we arrive at
    \begin{equation}\label{eq:tikh_nl:rate2}
        \begin{aligned}[t]
            \norm{F'(x^\dag)(x^\dag-x_\alpha^\delta)}_Y &\leq \frac{L}2\norm{x^\dag -x_\alpha^\delta}_X^2 + \norm{F(x_\alpha^\delta) - F(x^\dag)}_Y\\
            &\leq \frac{L}2\norm{x^\dag -x_\alpha^\delta}_X^2 + \norm{F(x_\alpha^\delta) - y^\delta}_Y + \delta.
        \end{aligned}
    \end{equation}
    Inserting this into \eqref{eq:tikh_nl:rate1} then yields that
    \begin{equation*}
        \norm{F(x_\alpha^\delta)-y^{\delta}}_Y^2 + {\alpha} \norm{x_\alpha^\delta-x^\dag}_X^2 \leq \delta^2 + \alpha \norm{w}_Y \left(L\norm{x^\dag -x_\alpha^\delta}_X^2 + 2\norm{F(x_\alpha^\delta) - y^\delta}_Y + 2\delta\right).
    \end{equation*}
    We now add $\alpha^2\norm{w}_Y^2$ to both sides and rearrange to obtain the inequality
    \begin{equation*}
        \left(\norm{F(x_\alpha^\delta)-y^{\delta}}_Y - \alpha \norm{w}_Y\right)^2 + \alpha(1-L\norm{w}_Y) \norm{x_\alpha^\delta-x^\dag}_X^2 \leq  \left(\delta + \alpha\norm{w}_Y\right)^2.
    \end{equation*}
    Dropping one of the two terms on the left-hand side and applying the parameter choice rule $c\delta\leq \alpha\leq C\delta$ then yields
    \begin{align*}
        \norm{F(x_\alpha^\delta)-y^{\delta}}_Y &\leq \delta + 2\alpha \norm{w}_Y \leq (1+2C\norm{w}_Y) \delta
        \intertext{as well as (since $L\norm{w}_Y < 1$ by assumption)}
        \norm{x_\alpha^\delta-x^\dag}_X  &\leq \frac{\delta +\alpha\norm{w}_Y}{\sqrt{\alpha(1-L\norm{w}_Y)}} \leq \frac{1+C\norm{w}_Y}{\sqrt{c(1-L\norm{w}_Y)}}\sqrt{\delta},
    \end{align*}
    respectively, and hence the claim.
\end{proof}
Note that condition (ii) entails a smallness condition on $x^\dag -x_0$: To obtain the claimed convergence rate, $x_0$ already has to be a sufficiently good approximation of the desired solution $x^\dag$.
Conversely, the condition indicates \emph{which} $x_0$-minimum norm solution the minimizers converge to if $x^\dag$ is not unique.

With a bit more effort, one can show analogously to \cref{cor:tikhonov:rate} the higher rate $\delta^{\nu/(\nu+1)}$ under the stronger source condition $x^\dag -x_0 \in\calR((F'(x^\dag)^*F'(x^\dag))^{\nu/2})$ and the corresponding choice of $\alpha(\delta)$, up to the qualification $\nu_0 = 2$; see \cite[Theorem~10.7]{Engl}.

\bigskip

We next consider the a posteriori choice of $\alpha$ according to the discrepancy principle\index{discrepancy principle!for nonlinear Tikhonov regularization}\index{parameter choice rule!discrepancy principle!for nonlinear Tikhonov regularization}: Set $\tau>1$ and choose $\alpha=\alpha(\delta,y^\delta)$ such that
\begin{equation}
    \label{eq:tikh_nl:morozov}
    \delta < \norm{F(x_\alpha^\delta)-y^\delta}_Y \leq  \tau \delta.
\end{equation}
\begin{theorem}\label{thm:tikh_nl:rate_morozov}
    Let $F:U\to Y$ be Fréchet differentiable with convex domain $\dom F =U$. Let further $y\in\calR(F)$ and $y^\delta\in B_\delta(y)$, and let $x^\dag$ be an $x_0$-minimum norm solution such that conditions (i) and (ii) from \cref{thm:tikh_nl:rate_apriori} are satisfied.
    If $\alpha:=\alpha(\delta,y^\delta)$ is chosen according to \eqref{eq:tikh_nl:morozov}, then there exists a constant $c>0$ such that
    \begin{equation*}
        \norm{x_\alpha^\delta-x^\dag}_X \leq c \sqrt{\delta}.
    \end{equation*}
\end{theorem}
\begin{proof}
    From \eqref{eq:tikh_nl:morozov} and the minimizing property of $x_\alpha^\delta$, we directly obtain that
    \begin{equation*}
        \frac{\delta^2}{2} + \frac{\alpha}2 \norm{x_\alpha^\delta-x_0}_X^2
        < \frac12 \norm{F(x_\alpha^\delta)-y^\delta}_Y^2 + \frac{\alpha}2 \norm{x_\alpha^\delta-x_0}_X^2
        \leq \frac{\delta^2}2 + \frac{\alpha}2 \norm{x^\dag-x_0}_X^2
    \end{equation*}
    and hence that
    \begin{equation*}
        \frac{\alpha}2 \norm{x_\alpha^\delta-x_0}_X^2 \leq \frac{\alpha}2 \norm{x^\dag-x_0}_X^2.
    \end{equation*}
    As for \eqref{eq:tikh_nl:rate1} and \eqref{eq:tikh_nl:rate2}, we can then use the conditions (i) and (ii) together with the parameter choice \eqref{eq:tikh_nl:morozov} to show that
    \begin{equation*}
        \begin{aligned}
            \norm{x_\alpha^\delta-x^\dag}_X^2 &\leq \norm{w}_Y\left(L\norm{x_\alpha^\delta-x^\dag}_X^2 + 2\norm{F(x_\alpha^\delta)-y^\delta}_Y + 2\delta\right)\\
            &\leq  \norm{w}_Y\left(L\norm{x_\alpha^\delta-x^\dag}_X^2 + 2(1+ \tau)\delta\right).
        \end{aligned}
    \end{equation*}
    Since $L\norm{w}_Y<1$, we can again rearrange this to
    \begin{equation*}
        \norm{x_\alpha^\delta-x^\dag}_X^2 \leq \frac{2(1+\tau)\norm{w}_Y}{1-L\norm{w}_Y} \delta,
    \end{equation*}
    which yields the desired estimate.
\end{proof}
In contrast to Tikhonov regularization of linear problems, it is however not guaranteed that an $\alpha$ satisfying \eqref{eq:tikh_nl:morozov} exists; this requires (strong) assumptions on the nonlinearity of $F$. Another sufficient -- and more general -- assumption is the uniqueness of minimizers of $J_\alpha$ together with a condition on $x_0$.
\begin{theorem}\label{thm:tikh_nl:value}
    Assume that for fixed $y^\delta\in B_\delta(y)$ and arbitrary $\alpha>0$, the minimizer $x_\alpha^\delta$ of $J_\alpha$ is unique. If $x_0\in U$ and $\tau >1$ satisfy $\norm{F(x_0)-y^\delta}_Y> \tau\delta$, then there exists an $\alpha>0$ such that \eqref{eq:tikh_nl:morozov} holds.
\end{theorem}
\begin{proof}
    We first show the continuity of the value function $f(\alpha):= \norm{F(x_\alpha^\delta)-y^\delta}_Y$.
    Let $\alpha>0$ be arbitrary and $\{\alpha_n\}_{n\in\N}$ be a sequence with $\alpha_n\to \alpha$ as $n\to\infty$.
    Then there exist $\eps>0$ and $N\in\N$ such that $0<\alpha-\eps\leq \alpha_n\leq \alpha+\eps$ for all $n>N$.
    Let further $x_\alpha^\delta$ be the unique minimizer of $J_\alpha$ and $x_n := x_{\alpha_n}^\delta$ for $n\in\N$ be the minimizer of $J_{\alpha_n}$. The minimizing property of $x_n$ for $J_{\alpha_n}$ for all $n>N$ then yields that
    \begin{equation*}
        \begin{aligned}
            \frac12\norm{F(x_n)-y^\delta}_Y^2 + \frac{\alpha-\eps}2 \norm{x_n-x_0}_X^2
            &\leq \frac12\norm{F(x_n)-y^\delta}_Y^2 + \frac{\alpha_n}2 \norm{x_n-x_0}_X^2\\
            &\leq \frac12\norm{F(x_\alpha^\delta)-y^\delta}_Y^2 + \frac{\alpha_n}2 \norm{x_\alpha^\delta-x_0}_X^2\\
            &\leq \frac12\norm{F(x_\alpha^\delta)-y^\delta}_Y^2 + \frac{\alpha+\eps}2 \norm{x_\alpha^\delta-x_0}_X^2,
        \end{aligned}
    \end{equation*}
    which implies that both $\{x_n\}_{n>N}$ and $\{F(x_n)\}_{n>N}$ are bounded. As in the proof of \cref{thm:tikh_nl:stabil}, we obtain from this that
    \begin{equation}\label{eq:tikh_nl:disk1}
        \lim_{n\to\infty}\left( \frac12\norm{F(x_n)-y^\delta}_Y^2 + \frac{\alpha_n}2 \norm{x_n-x_0}_X^2\right) = \frac12\norm{F(x_\alpha^\delta)-y^\delta}_Y^2 + \frac{\alpha}2 \norm{x_\alpha^\delta-x_0}_X^2
    \end{equation}
    as well as that (using the uniqueness of the minimizers) $x_n\to x_\alpha^\delta$.
    Hence $\alpha\mapsto x_\alpha^\delta$ is continuous. Together with the continuity of the norm, this implies the continuity of $g:\alpha\mapsto\frac\alpha2\norm{x_\alpha^\delta-x_0}_X^2$ and thus by \eqref{eq:tikh_nl:disk1} also of $f$.

    As in \cref{lem:tikhonov:monoton}, we can now use the minimizing property of $x_\alpha^\delta$ to show the monotonicity of $f$, which implies that
    \begin{equation*}
        \begin{aligned}
            \lim_{\alpha\to\infty} \norm{F(x_\alpha^\delta)-y^\delta}_Y &= \norm{F(x_0)-y^\delta}_Y > \tau\delta,\\
            \lim_{\alpha\to 0} \norm{F(x_\alpha^\delta)-y^\delta}_Y &= \inf_{x\in U} \norm{F(x)-y^\delta}_Y \leq \norm{F(x^\dag)-y^\delta}_Y \leq \delta.
        \end{aligned}
    \end{equation*}
    Hence, the continuous function $f(\alpha)$ attains all values in $(\delta,\tau\delta]$; in particular, there exists an $\alpha$ such that \eqref{eq:tikh_nl:morozov} holds.
\end{proof}

Since under our assumptions $J_\alpha$ is a differentiable nonlinear functional, the minimizer $x_\alpha^\delta$ can be computed by standard methods from nonlinear optimization such as gradient or (quasi-)Newton methods. Here again the possible non-uniqueness of minimizers leads to practical difficulties. Note in particular that all claims have been about \emph{global} minimizers of the Tikhonov functional, while (gradient-based) numerical methods in general can only produce (approximations of) \emph{local} minimizers. This gap between theory and practice is still an open problem in inverse problems.

\section{Alternative source conditions}\label{sec:tikh_nl:alternative}

In the proof of \cref{thm:tikh_nl:rate_apriori}, we have used the source and nonlinearity conditions to bound the right-hand side of \eqref{eq:tikh_nl:rate1} by a suitable function of the terms on the left-hand side. It is possible to take this estimate directly as a source condition without introducing source representations or Lipschitz constants. In recent years, such \term[source condition!variational]{variational source conditions} have attracted increasing interest. In our context, they have the following form: There exist $\beta_1\in [0,1)$ and $\beta_2\geq 0$ such that
\begin{equation}
    \label{eq:tikh_nl:var_source}
    \inner{x^\dag - x_0}{ x^\dag - x}_X \leq \beta_1 \left(\frac12\norm{x-x^\dag}_X^2\right) + \beta_2 \norm{F(x)-F(x^\dag)}_Y \quad \text{for all }x\in U,
\end{equation}
where $U$ is a sufficiently large neighborhood of $x^\dag$ (in particular, containing all minimizers $x_\alpha^\delta$ of $J_\alpha$). Note the different powers in the two terms on the right-hand side, which are supposed to account for the different convergence speeds of error and residual.
\begin{theorem}\label{thm:tikh_nl:vsc_apriori}
    Let $y\in\calR(F)$, $y^\delta\in B_\delta(y)$, and $x^\dag$ be an $x_0$-minimum norm solution satisfying the variational source condition \eqref{eq:tikh_nl:var_source} for some $\beta_1<1$.
    If $\alpha(\delta)$ is a parameter choice rule with
    \begin{equation*}
        c\delta \leq \alpha(\delta) \leq C \delta  \qquad\text{for } c,C>0,
    \end{equation*}
    then there exist constants $c_1,c_2>0$ such that
    \begin{align*}
        \norm{x_{\alpha(\delta)}^\delta-x^\dag}_X &\leq c_1 \sqrt{\delta},\\
        \norm{F(x_{\alpha(\delta)}^\delta)-y^\delta}_Y &\leq c_2 {\delta}.
    \end{align*}
\end{theorem}
\begin{proof}
    From the minimizing property of $x_\alpha^\delta$, we again obtain the first inequality of \eqref{eq:tikh_nl:rate1}. We now estimate this further using the variational source condition, the triangle inequality, the generalized Young inequality $ab\leq \frac{1}{2\eps}a^2 + \frac{\eps}2 b^2$ for $\eps=\frac12$, and the parameter choice to obtain that
    \begin{equation*}
        \begin{aligned}[t]
            \frac12\norm{F(x_\alpha^\delta)-y^{\delta}}_Y^2 + \frac{\alpha}2 \norm{x_\alpha^\delta-x^\dag}_X^2 &\leq \frac{\delta^2}2 + {\alpha} \inner{x^\dag-x_0}{x^\dag-x_\alpha^\delta}_X\\
            &\leq  \frac{\delta^2}2 + \alpha\beta_1 \left(\tfrac12\norm{x_\alpha^\delta-x^\dag}_X^2\right) + \alpha \beta_2 \norm{F(x_\alpha^\delta)-F(x^\dag)}_Y\\
            &\leq \frac{\delta^2}2 + \frac{\alpha}2\beta_1\norm{x_\alpha^\delta-x^\dag}_X^2 + \alpha\beta_2\left(\norm{F(x_\alpha^\delta)-y^\delta}_Y + \delta\right)\\
            &\leq \frac{\delta^2}2+ \frac{\alpha}2\beta_1\norm{x_\alpha^\delta-x^\dag}_X^2 +
            {\alpha^2\beta_2^2} + \frac14\norm{F(x_\alpha^\delta)-y^\delta}_Y^2 \\
            \MoveEqLeft[-1] +\alpha\beta_2\delta\\
            &\leq \left(\frac12+C^2\beta_2^2 + C\beta_2\right)\delta^2 +
            \frac{\alpha}2\beta_1\norm{x_\alpha^\delta-x^\dag}_X^2\\
            \MoveEqLeft[-1] + \frac14\norm{F(x_\alpha^\delta)-y^\delta}_Y^2.
        \end{aligned}
    \end{equation*}
    Due to the assumption that $\beta_1<1$, we can absorb the last two terms on the right-hand side into the left-hand side, which yields
    \begin{align*}
        \norm{x_\alpha^\delta-x^\dag}_X &\leq \sqrt{\frac{1 + 2C\beta_2 + 2C^2\beta_2^2}{c(1-\beta_1)}}\,\sqrt{\delta}
        \intertext{as well as}
        \norm{F(x_\alpha^\delta)-y^\delta}_Y &\leq \sqrt{2 + 4C\beta_2 + 4C^2\beta_2^2}\,\delta.
        \qedhere
    \end{align*}
\end{proof}

A similar result holds for the discrepancy principle.
\begin{theorem}\label{thm:tikh_nl:vsc_morozov}
    Let $y\in\calR(F)$, $y^\delta\in B_\delta(y)$, and $x^\dag$ be an $x_0$-minimum norm solution satisfying the variational source condition \eqref{eq:tikh_nl:var_source} for some $\beta_1<1$.
    If $\alpha:=\alpha(\delta,y^\delta)$ is chosen according to \eqref{eq:tikh_nl:morozov}, then
    \begin{align*}
        \norm{x_{\alpha}^\delta-x^\dag}_X &\leq \sqrt{\frac{2(1+\tau)\beta_2}{1-\beta_1}}\,\sqrt{\delta},\\
        \norm{F(x_{\alpha}^\delta)-y^\delta}_Y &\leq \tau {\delta}.
    \end{align*}
\end{theorem}
\begin{proof}
    The second estimate holds by definition of the discrepancy principle. For the first estimate, we argue as in the proof of \cref{thm:tikh_nl:rate_morozov} to obtain from the minimizing property of $x_\alpha^\delta$ and the first inequality in \eqref{eq:tikh_nl:morozov} that $\norm{x_\alpha^\delta-x_0}_X^2 \leq \norm{x^\dag-x_0}_X^2$. Inserting
    \begin{equation*}
        \norm{x_\alpha^\delta -x_0}_X^2 = \norm{x_\alpha^\delta - x^\dag}_X^2 + 2\inner{x_\alpha^\delta - x^\dag}{ x^\dag -x_0}_X + \norm{x^\dag -x_0}_X^2
    \end{equation*}
    and using the variational source condition for $x=x_\alpha^\delta\in U$, the triangle inequality together with $F(x^\dag)=y$, and the second inequality in \eqref{eq:tikh_nl:morozov} then yields
    \begin{equation*}
        \begin{aligned}
            \frac12\norm{x_\alpha^\delta - x^\dag}_X^2 &\leq \inner{x^\dag-x_0}{x^\dag-x_\alpha^\delta}_X\\
            &\leq \frac{\beta_1}2\norm{x_\alpha^\delta-x^\dag}_X^2 + \beta_2\left(\norm{F(x_\alpha^\delta)-y^\delta}_Y + \norm{y^\delta-y}_Y\right)\\
            &\leq \frac{\beta_1}2\norm{x_\alpha^\delta-x^\dag}_X^2 + \beta_2(1+\tau)\delta.
        \end{aligned}
    \end{equation*}
    Since $\beta_1<1$, we can absorb the first term on the right-hand side into the left-hand side to obtain the first claim.
\end{proof}

We finally study the connection between variational and classical source conditions.
\begin{lemma}
    Let $F:U\to Y$ be Fréchet differentiable and $x^\dag$ be an $x_0$-minimum norm solution. If there exists a $w\in Y$ with $x^\dag -x_0 = F'(x^\dag)^*w$ and either
    \begin{enumerate}
        \item $F'$ is Lipschitz continuous with constant $L\norm{w}_Y<1$ or
        \item the tangential cone condition \eqref{eq:nichtlin:tangential} is satisfied,
    \end{enumerate}
    then the variational source condition \eqref{eq:tikh_nl:var_source} holds.
\end{lemma}
\begin{proof}
    We first apply the classical source condition to the left-hand side of \eqref{eq:tikh_nl:var_source} and estimate
    \begin{equation*}
        \begin{aligned}
            \inner{x^\dag - x_0}{ x^\dag - x}_X
            &= \inner{F'(x^\dag)^*w}{x^\dag-x}_X\\
            &= \inner{w}{F'(x^\dag)(x^\dag -x)}_Y\\
            &\leq \norm{w}_Y\norm{F'(x^\dag)(x^\dag -x)}_Y\\
            &\leq \norm{w}_Y\left(\norm{F(x)-F(x^\dag)-F'(x^\dag)(x^\dag -x)}_Y + \norm{F(x)-F(x^\dag)}_Y\right).
        \end{aligned}
    \end{equation*}
    If now assumption (i) holds, we can apply \cref{lem:nichtlin:lipschitz} to obtain the inequality
    \begin{equation*}
        \inner{x^\dag - x_0}{ x^\dag - x}_X \leq \norm{w}_Y\left(\frac{L}{2}\norm{x^\dag-x}_X^2 + \norm{F(x)-F(x^\dag)}_Y\right),
    \end{equation*}
    i.e., \eqref{eq:tikh_nl:var_source} with $\beta_1 = L\norm{w}_Y<1$ and $\beta_2 = \norm{w}_Y$.

    On the other hand, if assumption (ii) holds, we can directly estimate
    \begin{equation*}
        \inner{x^\dag - x_0}{ x^\dag - x}_X \leq \norm{w}_Y(\eta+1)\norm{F(x)-F(x^\dag)}_Y,
    \end{equation*}
    which implies \eqref{eq:tikh_nl:var_source} with $\beta_1=0$ and $\beta_2=(1+\eta)\norm{w}_Y>0$.
\end{proof}

For a linear operator $T\in\linop(X,Y)$, we of course do not need any nonlinearity condition; in this case the variational source condition \eqref{eq:tikh_nl:var_source} is equivalent to the classical source condition $x^\dag\in \calR(T^*)$, see \cite[Lemma~2]{Scherzer:2014}.
For nonlinear operators, however, it is a weaker (albeit even more abstract) condition.
The main advantage of this type of condition is that it does not involve the Fréchet derivative of $F$ and hence can also be applied for non-differentiable $F$; furthermore, it can be applied to generalized Tikhonov regularization, in particular in Banach spaces; see, e.g., \cite{Poeschl:2007a,Scherzer:2009,HKKS}.

\bigskip

Another alternative is to directly require the desired link, but only in a suitable neighborhood of $x^\dag$. This amounts to a \term[estimate, conditional stability]{conditional stability estimate} of the form
\begin{equation}\label{eq:tikh_nl:condstab}
    \norm{x_1-x_2}_X \leq R(\rho)\, \norm{F(x_1)-F(x_2)}_Y^\nu\qquad \text{for all }x_1,x_2 \in \dom F\cap B_\rho(x_0),
\end{equation}
where $\nu\in(0,1]$, $R:\R\to\R$ is monotonically increasing, and $B_\rho(x_0)\subset X$ is the closed ball of radius $\rho>0$ around $x_0$. (The case $\nu = 1$ would correspond to conditional Lipschitz stability, which is too much to expect in general.) The crucial point here is to ensure that $\rho$ can be chosen large enough to contain both $x^\dag$ and all $x^\delta_\alpha$ of interest.
\begin{theorem}\label{thm:tikh_nl:condstab}
    Let $y\in\calR(F)$, $y^\delta\in B_\delta(y)$, and $x^\dag\in B_\rho(x_0)$ for $\rho>0$ be an $x_0$-minimum norm solution.
    If \eqref{eq:tikh_nl:condstab} holds for some $\nu\in(0,1]$ and $\rho>0$, and
    \begin{equation*}
        c\delta^2 \leq \alpha(\delta) \leq C \delta^2  \qquad\text{for } c,C>0,
    \end{equation*}
    then there exist constants $c_1,c_2>0$ depending only on $c$, $C$, and $\rho$ such that
    \begin{align*}
        \norm{x_{\alpha(\delta)}^\delta-x^\dag}_X &\leq c_1 \delta^\nu,\\
        \norm{F(x_{\alpha(\delta)}^\delta)-y^\delta}_Y &\leq c_2 {\delta}.
    \end{align*}
\end{theorem}
\begin{proof}
    We again start from the minimizing properties of $x_\alpha^\delta$, which together with the parameter choice imply
    \begin{equation*}
        \begin{aligned}
            \norm{F(x_\alpha^\delta)-y^\delta}_Y^2 + \alpha\norm{x_\alpha^\delta-x_0}_X^2 &\leq
            \norm{F(x^\dag)-y^\delta}_Y^2 + \alpha\norm{x^\dag-x_0}_X^2 \\
            &\leq \delta^2 + \alpha \rho^2
            \leq (1+C\rho^2)\delta^2.
        \end{aligned}
    \end{equation*}
    Estimating the second nonnegative term from below by zero immediately yields the second claim for $c_2:=\sqrt{1+C\rho^2}$. Similarly estimating the first nonnegative term and using the parameter choice shows that
    \begin{equation*}
        \norm{x_\alpha^\delta-x_0}_X^2 \leq \frac{1+C\rho^2}{\alpha}\delta^2 \leq \frac{1+C\rho^2}{c}
    \end{equation*}
    for all $\delta>0$.

    Take now $M_\rho:=\max\{\rho, \sqrt{c^{-1}(1+C\rho^2)}\}$. Then $x^\dag,x_\alpha^\delta\in B_{M_\rho}(x_0)$ for all $\delta >0$, and we can thus apply the conditional stability estimate \eqref{eq:tikh_nl:condstab} to obtain
    \begin{equation*}
        \begin{aligned}
            \norm{x_\alpha^\delta-x^\dag}_X &\leq R(M_\rho) \norm{F(x_\alpha^\delta)-F(x^\dag)}_Y^\nu\\
            &\leq R(M_\rho) \left(\norm{F(x_\alpha^\delta)-y^\delta}_Y + \norm{y-y^\delta}_Y\right)^\nu\\
            &\leq R(M_\rho) (1+c_2)^\nu\delta^\nu.
        \end{aligned}
    \end{equation*}
    This yields the first claim with $c_1 := R(M_\rho)(1+c_2)^\nu$.
\end{proof}
For the discrepancy principle, the minimizers automatically stay in the ball $B_\rho(x_0)$ containing $x^\dag$, which yields a simpler result.
\begin{theorem}\label{thm:tikh_nl:condstab_morozov}
    Let $y\in\calR(F)$, $y^\delta\in B_\delta(y)$, and $x^\dag\in B_\rho(x_0)$ for $\rho>0$ be an $x_0$-minimum norm solution.
    If \eqref{eq:tikh_nl:condstab} holds for some $\nu\in(0,1]$ and $\rho>0$, and $\alpha:=\alpha(\delta,y^\delta)$ is chosen according to \eqref{eq:tikh_nl:morozov}, then
    \begin{equation*}
        \norm{x_{\alpha}^\delta-x^\dag}_X \leq R(\rho)(1+\tau)^\nu \delta^\nu.
    \end{equation*}
\end{theorem}
\begin{proof}
    As in the proof of \cref{thm:tikh_nl:rate_morozov}, the minimizing property of $x_\alpha^\delta$ together with the first inequality in \eqref{eq:tikh_nl:morozov} implies that $\norm{x_\alpha^\delta-x_0}_X\leq\norm{x^\dag-x_0}_X\leq\rho$. Hence $x_\alpha^\delta,x^\dag\in B_\rho(x_0)$, and the conditional stability estimate \eqref{eq:tikh_nl:condstab} together with the second inequality in \eqref{eq:tikh_nl:morozov} yields
    \begin{equation*}
        \norm{x_\alpha^\delta-x^\dag}_X \leq R(\rho) \left(\norm{F(x_\alpha^\delta)-y^\delta}_Y + \norm{y^\delta-y}_Y\right)^\nu \leq R(\rho) (\tau+1)^\nu\delta^\nu.
        \qedhere
    \end{equation*}
\end{proof}
\Cref{thm:tikh_nl:condstab} can also be extended to cover more general parameter choice rules and conditional stability estimates with more general functions than powers and in other, weaker, norms; see, e.g., \cite{Hofmann:2018}.

\chapter{Iterative regularization}

There also exist iterative methods for nonlinear inverse problems that, like the Landweber iteration, construct a sequence of approximations and can be combined with a suitable termination criterion to obtain a regularization method.
Specifically, a \term[regularization method!iterative]{(convergent) iterative regularization method} refers to a procedure that constructs for  given $y^\delta\in Y$ and $x_0\in U$ a sequence $\{x_n^\delta\}_{n\in\N}\subset U$ together with a stopping index $N(\delta,y^\delta)$, such that for all $y\in\calR(F)$ and all $x_0=x_0^\delta$ sufficiently close to an isolated solution $x^\dag\in U$ of $F(x)=y$, we have that%
\footnote{In contrast to the previous chapters, we denote here by $x^\dag$ not an ($x_0$-)minimum norm solution, but any solution of $F(x)=y$.}%
\begin{subequations}
    \begin{align}
        &N(0,y)<\infty,\quad x_{N(0,y)} = x^\dag\qquad\text{or}\qquad  N(0,y)=\infty,\quad x_n\to x^\dag \text{ for } n\to\infty,
        \label{eq:iter:stab}\\
        &\lim_{\delta\to 0}\sup_{y^\delta\in B_\delta(y)}\norm{x_{N(\delta,y^\delta)}^\delta-x^\dag}_X = 0.
        \label{eq:iter:conv}
    \end{align}
\end{subequations}
The first condition states that for exact data (i.e., $\delta=0$), the sequence either converges to a solution or reaches one after finitely many steps. The second condition corresponds to the definition of a convergent regularization method in the linear setting.

We again terminate by the Morozov discrepancy principle\index{discrepancy principle!for iterative regularization methods}\index{parameter choice rule!discrepancy principle!for iterative regularization methods}: Set $\tau>1$ and choose $N=N(\delta,y^\delta)$ such that
\begin{equation}
    \label{eq:iter:diskrepanz}
    \norm{F(x_{N}^\delta) - y^\delta}_Y \leq \tau \delta < \norm{F(x_{n}^\delta) - y^\delta}_Y \qquad\text{for all }n<N.
\end{equation}
In this case, a sufficient condition for \eqref{eq:iter:conv} is the monotonicity and stability of the method. Here and in the following, we again denote by $x_n$ the elements of the sequence generated for the exact data $y\in \calR(F)$ and by $x_n^\delta$ the elements for the noisy data $y^\delta\in B_\delta(y)$.
\begin{lemma}\label{lem:iter:conv}
    Let $N(\delta,y^\delta)$ be chosen by the discrepancy principle \eqref{eq:iter:diskrepanz}. If an iterative method for a continuous operator $F:U\to Y$ satisfies the condition \eqref{eq:iter:stab} as well as
    \begin{subequations}
        \begin{align}
            &\norm{x_n^\delta - x^\dag}_X \leq \norm{x_{n-1}^\delta -x^\dag}_X &&\text{for all } n\in\{1,\dots,N(\delta,y^\delta)\},
            \label{eq:iter:conv_mon}\\
            &\lim_{\delta\to 0} \norm{x_n^\delta-x_n}_X = 0   &&\text{for every fixed } n\in\N,
            \label{eq:iter:conv_stab}
        \end{align}
    \end{subequations}
    then the condition \eqref{eq:iter:conv} is also satisfied.
\end{lemma}
\begin{proof}
    Let $F:U\to Y$ be continuous, $\{y^{\delta_k}\}_{k\in\N}$ with $y^{\delta_k}\in B_{\delta_k}(y)$ and $\delta_k\to 0$ as $k\to\infty$, and set $N_k:= N(\delta_k,y^{\delta_k})$.
    We first consider the case that $\{N_k\}_{k\in\N}$ is bounded and hence that the set $\setof{N_k}{k\in\N}\subset\N$ is finite.
    After passing to a subsequence if necessary, we can therefore assume that $N_k =\bar N$ for all $k\in \N$. It then follows from \eqref{eq:iter:conv_stab} that $x^{\delta_k}_{\bar N}\to x_{\bar N}$ as $k\to\infty$. Since all $N_k$ are chosen according to the discrepancy principle \eqref{eq:iter:diskrepanz}, we have that
    \begin{equation*}
        \norm{F(x^{\delta_k}_{\bar N})-y^{\delta_k}}_Y\leq \tau\delta_k  \qquad\text{for all }k\in\N.
    \end{equation*}
    Passing to the limit on both sides and using the continuity of $F$ then yields that $F(x_{\bar N})=y$, i.e., $x^{\delta_k}_{\bar N}$ converges to a solution of $F(x)=y$ and the condition \eqref{eq:iter:conv} is thus satisfied.

    Otherwise, there exists a subsequence with $N_k\to\infty$. We can assume (possibly after passing to a further subsequence) that $N_k$ is increasing. Then \eqref{eq:iter:conv_mon} yields that for all $l\leq k$,
    \begin{equation*}
        \norm{x^{\delta_k}_{N_k}-x^\dag}_X \leq \norm{x^{\delta_k}_{N_l}-x^\dag}_X
        \leq
        \norm{x^{\delta_k}_{N_l}-x_{N_l}}_X + \norm{x_{N_l}-x^\dag}_X.
    \end{equation*}
    Let now $\eps>0$ be arbitrary. Since we have assumed that condition \eqref{eq:iter:stab} holds, there exists an $L>0$ such that $\norm{x_{N_L}-x^\dag}_X\leq\frac\eps2$. Similarly, \eqref{eq:iter:conv_stab} for $n={N_L}$ shows the existence of a $K>0$ such that $ \norm{x^{\delta_k}_{N_L}-x_{N_L}}_X \leq\frac\eps2$ for all $k\geq K$. Hence, the condition \eqref{eq:iter:conv} holds in this case as well.
\end{proof}
A sequence $\{x_n\}_{n\in\N}$ satisfying \eqref{eq:iter:conv_mon} is called \term[Fejer monotone@Fejér monotone]{Fejér monotone}; this property is fundamental for the convergence proof of many iterative methods.

In general, iterative methods for nonlinear inverse problems rely on a linearization of $F$, with different methods applying the linearization at different points in the iteration.

\section{Landweber regularization}

Analogously to the linear Landweber regularization, we start from the characterization of the wanted solution $x^\dag$ as a minimizer of the functional $J_0(x)=\frac12\norm{F(x)-y}_Y^2$. If $F$ is Fréchet differentiable, the chain rule yields the necessary optimality condition
\begin{equation*}
    0 = J_0'(x^\dag)h = \inner{F(x^\dag)-y}{F'(x^\dag)h}_Y  = \inner{F'(x^\dag)^*(F(x^\dag)-y)}{h}_X\qquad\text{for all }h\in X.
\end{equation*}
This is now a nonlinear equation for $x^\dag$, which as in the linear case can be written as a fixed-point equation. This leads to the nonlinear Richardson iteration\index{Richardson iteration!nonlinear}\index{Landweber iteration!nonlinear}
\begin{equation}\label{eq:iter:landweber}
    x_{n+1} = x_n - \omega_n F'(x_n)^*(F(x_n)-y),
\end{equation}
for which we can expect convergence if $\omega_n\norm{F'(x_n)^*}_{\linop(Y,X)}^2<1$. (Alternatively, \eqref{eq:iter:landweber} can be interpreted as a steepest descent method with step size $\omega_n$ for the minimization of $J_0$.)
For simplicity, we assume in the following that $\norm{F'(x)}_{\linop(X,Y)}< 1$ for all $x$ sufficiently close to $x^\dag$, so that we can take $\omega_n = 1$. (This is not a significant restriction since we can always scale $F$ and $y$ appropriately without changing the solution of $F(x)=y$.)
Furthermore, we assume that $F$ is continuously Fréchet differentiable and satisfies the tangential cone condition \eqref{eq:nichtlin:tangential} in a neighborhood of $x^\dag$.
Specifically, we make the following assumption:
\begin{assumption}\label{ass:iter-nl}
    Let $F:U\to Y$ be continuously differentiable and $x_0\in U$. Assume that there exists an $r>0$ such that
    \begin{enumerate}
        \item $B_{2r}(x_0)\subset U$;
        \item there exists a solution $x^\dag\in B_r(x_0)$;
        \item for all $x,\tilde x\in B_{2r}(x_0)$,
            \begin{align}
                \norm{F'(x)}_{\linop(X,Y)}&\leq 1,
                \label{eq:iter:landweber_ass:bnd}\\
                \norm{F(x)-F(\tilde x) - F'(x)(x-\tilde x)}_Y &\leq \eta \norm{F(x)-F(\tilde x)}_Y \qquad\text{with }  \eta<\tfrac12.
                \label{eq:iter:landweber_ass:tan}
            \end{align}
    \end{enumerate}
\end{assumption}
Under these assumptions, the nonlinear Landweber iteration \eqref{eq:iter:landweber} is well-posed and Fejér monotone even for noisy data $y^\delta\in B_\delta(y)$.
\begin{lemma}\label{lem:iter:landweber_mon}
    Let \cref{ass:iter-nl} hold.
    If $x_n^\delta\in B_r(x^\dag)$ for some $\delta\geq 0$ and satisfies
    \begin{align}
        \label{eq:iter:landweber:mon1}
        \norm{F(x_n^\delta)-y^\delta}_Y &\geq 2 \frac{1+\eta}{1-2\eta}\delta,
        \intertext{then}
        \norm{x_{n+1}^\delta - x^\dag}_X &\leq \norm{x_{n}^\delta - x^\dag}_X
    \end{align}
    and thus $x_{n+1}^\delta \in B_r(x^\dag)\subset B_{2r}(x_0)$.
\end{lemma}
\begin{proof}
    The iteration \eqref{eq:iter:landweber} together with \eqref{eq:iter:landweber_ass:bnd} for  $x_n^\delta\in B_r(x^\dag)\subset B_{2r}(x_0)$ leads to the estimate
    \begin{equation*}
        \begin{aligned}
            \norm{x_{n+1}^\delta-x^\dag}_X^2 - \norm{x_{n}^\delta-x^\dag}_X^2
            &= 2 \inner{x_{n+1}^\delta-x_n^\delta}{x_n^\delta - x^\dag}_X  + \norm{x_{n+1}^\delta- x_n^\delta}_X^2\\
            &= 2\inner{F'(x_n^\delta)^*(y^\delta - F(x_n^\delta))}{x_n^\delta - x^\dag}_X\\
            + \norm{F'(x_n^\delta)^*(y^\delta - F(x_n^\delta))}_X^2\span\omit\\
            &\leq 2 \inner{y^\delta-F(x_n^\delta)}{F'(x_n^\delta)(x_n^\delta - x^\dag)}_Y+ \norm{y^\delta - F(x_n^\delta)}_Y^2\\
            &= 2 \inner{y^\delta-F(x_n^\delta)}{y^\delta-F(x_n^\delta)+F'(x_n^\delta)(x_n^\delta - x^\dag)}_Y\\
            - \norm{y^\delta - F(x_n^\delta)}_Y^2\span\omit\\
            &\leq \norm{y ^\delta- F(x_n^\delta)}_Y\big(2\norm{y^\delta-F(x_n^\delta)+F'(x_n^\delta)(x_n^\delta - x^\dag)}_Y\\
            - \norm{y^\delta - F(x_n^\delta)}_Y\big).\span\omit
        \end{aligned}
    \end{equation*}
    Inserting the productive zero $F(x^\dag)-y$ in the first norm inside the parentheses and applying the triangle inequality as well as the tangential cone condition \eqref{eq:iter:landweber_ass:tan} then yields that
    \begin{equation*}
        \begin{aligned}
            \norm{y^\delta-F(x_n^\delta)+F'(x_n^\delta)(x_n^\delta - x^\dag)}_Y&\leq \delta + \norm{F(x_n^\delta) -F(x^\dag)  - F'(x_n^\delta)(x_n^\delta - x^\dag)}_Y\\
            &\leq \delta + \eta \norm{F(x_n^\delta)-F(x^\dag)}_Y\\
            &\leq (1+\eta)\delta + \eta \norm{F(x_n^\delta) - y^\delta}_Y
        \end{aligned}
    \end{equation*}
    and hence that
    \begin{equation}
        \label{eq:iter:landweber:mon2}
        \norm{x_{n+1}^\delta-x^\dag}_X^2 - \norm{x_{n}^\delta-x^\dag}_X^2 \leq \norm{y^\delta - F(x_n^\delta)}_Y\big(2(1+\eta)\delta - (1-2\eta)\norm{y^\delta - F(x_n^\delta)}_Y\big).
    \end{equation}
    By \eqref{eq:iter:landweber:mon1}, the term in parentheses is non-positive, from which the desired monotonicity follows.
\end{proof}
By induction, this shows that $x_n^\delta\in B_{2r}(x_0)\subset U$ as long as \eqref{eq:iter:landweber:mon1} holds. If we choose $\tau$ for the discrepancy principle \eqref{eq:iter:diskrepanz} such that
\begin{equation}
    \label{eq:iter:diskrepanz_landweber}
    \tau > 2 \frac{1+\eta}{1-2\eta} > 2,
\end{equation}
then this is the case for all $n\leq N(\delta,y^\delta)$. This choice also guarantees that the stopping index $N(\delta,y^\delta)$ is finite.
\begin{theorem}\label{thm:iter:landweber_stop}
    Let \cref{ass:iter-nl} hold.
    If $N(\delta,y^\delta)$ is chosen according to the discrepancy principle \eqref{eq:iter:diskrepanz} with $\tau$ satisfying \eqref{eq:iter:diskrepanz_landweber} then
    \begin{equation}
        \label{eq:iter:landweber_stop_noisy}
        N(\delta,y^\delta)< C \delta^{-2} \qquad\text{for some }C>0.
    \end{equation}
    For exact data (i.e., $\delta = 0$),
    \begin{equation}
        \label{eq:iter:landweber_stop_exact}
        \sum_{n=0}^{\infty} \norm{F(x_n)-y}_Y^2 < \infty.
    \end{equation}
\end{theorem}
\begin{proof}
    Since $x_0^\delta= x_0 \in B_{2r}(x_0)$ and by the choice of $\tau$, we can apply \cref{lem:iter:landweber_mon} for all $n< N=N(\delta,y^\delta)$. In particular, it follows from \eqref{eq:iter:landweber:mon2} and \eqref{eq:iter:diskrepanz_landweber} that
    \begin{equation*}
        \norm{x_{n+1}^\delta-x^\dag}_X^2 - \norm{x_{n}^\delta-x^\dag}_X^2 < \norm{y^\delta - F(x_n^\delta)}_Y^2\left(\frac2\tau (1+\eta) + 2\eta -1\right)\quad\text{for all }n<N.
    \end{equation*}
    Summing from $n=0$ to $N-1$ and telescoping thus yields
    \begin{equation*}
        \left(1-2\eta - \frac2\tau(1+\eta)\right) \sum_{n=0}^{N-1} \norm{F(x_n^\delta)-y^\delta}_Y^2 < \norm{x_0-x^\dag}_X^2 - \norm{x_{N}^\delta -x^\dag}_X^2\leq \norm{x_0-x^\dag}_X^2.
    \end{equation*}
    Since $N$ is chosen according to the discrepancy principle, we have that $\norm{F(x_n^\delta)-y^\delta}_Y>\tau\delta$ for all $n<N$. Together we thus obtain that
    \begin{equation*}
        N\tau^2\delta^2 <  \sum_{n=0}^{N-1} \norm{F(x_n^\delta)-y^\delta}_Y^2 < \left(1-2\eta-2\tau^{-1}(1+\eta)\right)^{-1} \norm{x_0-x^\dag}_X^2
    \end{equation*}
    and hence \eqref{eq:iter:landweber_stop_noisy} for $C:= \left((1-2\eta)\tau^2 - 2(1+\eta)\tau\right)^{-1}\norm{x_0-x^\dag}_X^2>0$.

    For $\delta = 0$, \eqref{eq:iter:landweber:mon1} is satisfied for all $n\in\N$, and we obtain directly from \eqref{eq:iter:landweber:mon2} by summing and telescoping that
    \begin{equation*}
        (1-2\eta) \sum_{n=0}^{N-1} \norm{F(x_n)-y}_Y^2 \leq  \norm{x_0-x^\dag}_X^2 \qquad\text{for all }N\in\N.
    \end{equation*}
    Passing to the limit $N\to\infty$ then yields \eqref{eq:iter:landweber_stop_exact}.
\end{proof}
Although \eqref{eq:iter:landweber_stop_exact} implies that $F(x_n)\to y$ for exact data $y\in\calR(F)$, we cannot yet conclude that the $x_n$ converge. This we show next.
\begin{theorem}\label{thm:iter:landweber:conv_exact}
    Let \cref{ass:iter-nl} hold.
    Then $x_n\to \bar x$ with $F(\bar x) = y$ as $n\to\infty$.
\end{theorem}
\begin{proof}
    We show that $\{e_n\}_{n\in\N}$ with $e_n:= x_n-x^\dag$ is a Cauchy sequence.
    Let $m,n\in\N$ with $m\geq n$ be given and choose $k\in \N$ with $m\geq k \geq n$ such that
    \begin{equation}
        \label{eq:iter:landweber:conv_exact1}
        \norm{y-F(x_k)}_Y \leq \norm{y-F(x_j)}_Y \qquad\text{for all } n\leq j\leq m.
    \end{equation}
    (I.e., we choose $k\in\{n,\dots,m\}$ such that the residual -- which need not be monotone in the nonlinear case -- is minimal in this range.)
    We now estimate
    \begin{equation*}
        \norm{e_m-e_n}_X \leq \norm{e_m-e_k}_X + \norm{e_k-e_n}_X
    \end{equation*}
    and consider each term separately.
    First,
    \begin{align*}
        \norm{e_m-e_k}_X^2 &= 2\inner{e_k-e_m}{e_k}_X + \norm{e_m}_X^2 - \norm{e_k}_X^2,\\
        \norm{e_k-e_n}_X^2 &= 2\inner{e_k-e_n}{e_k}_X + \norm{e_n}_X^2 - \norm{e_k}_X^2.
    \end{align*}
    It follows from \cref{lem:iter:landweber_mon} that $\norm{e_n}_X\geq0$ is decreasing and thus converges to some $\eps\geq 0$. Hence, both differences on the right-hand side converge to zero as $n\to\infty$, and it remains to look at the inner products. Here, inserting the definition of $e_n$, telescoping the sum, and using the iteration \eqref{eq:iter:landweber} yields that
    \begin{equation*}
        e_m-e_k = x_m - x_k = \sum_{j=k}^{m-1} x_{j+1} - x_j
        = \sum_{j=k}^{m-1} F'(x_j)^*(y-F(x_j)).
    \end{equation*}
    Inserting this into the inner product, generously adding productive zeros, and using the tangential cone condition \eqref{eq:iter:landweber_ass:tan} then leads to
    \begingroup
    \allowdisplaybreaks
    \begin{equation*}
        \begin{aligned}
            \inner{e_k-e_m}{e_k}_X &= \sum_{j=k}^{m-1}-\inner{y-F(x_j)}{F'(x_j)(x_k-x^\dag)}_Y\\
            &\leq  \sum_{j=k}^{m-1}\norm{y-F(x_j)}_Y\norm{F'(x_j)(x_k-x_j+x_j-x^\dag)}_Y\\
            &\leq   \sum_{j=k}^{m-1}\norm{y-F(x_j)}_Y\big(\norm{y-F(x_j)-F'(x_j)(x^\dag-x_j)}_Y + \norm{y-F(x_k)}_Y \\
            + \norm{F(x_j)-F(x_k)-F'(x_j)(x_j-x_k)}_Y\big)\span\omit\\
            &\leq (1+\eta)  \sum_{j=k}^{m-1}\norm{y-F(x_j)}_Y\norm{y-F(x_k)}_Y + 2\eta \sum_{j=k}^{m-1}\norm{y-F(x_j)}_Y^2\\
            &\leq (1+3\eta) \sum_{j=k}^{m-1}\norm{y-F(x_j)}_Y^2,
        \end{aligned}
    \end{equation*}
    \endgroup
    where we have used the definition \eqref{eq:iter:landweber:conv_exact1} of $k$ in the last estimate.
    Similarly we obtain that
    \begin{equation*}
        \inner{e_k-e_n}{e_k}_X \leq (1+3\eta) \sum_{j=n}^{k-1}\norm{y-F(x_j)}_Y^2.
    \end{equation*}
    Due to \cref{thm:iter:landweber_stop}, both remainder terms converge to zero as $n\to\infty$.
    Hence $\{e_n\}_{n\in\N}$ and therefore also $\{x_n\}_{n\in\N}$ are Cauchy sequences, which implies that $x_n\to \bar x$ with $F(\bar x)=y$ (due to \eqref{eq:iter:landweber_stop_exact}).
\end{proof}

It remains to show the convergence condition \eqref{eq:iter:conv} for noisy data.
\begin{theorem}\label{thm:iter:landweber:conv_noisy}
    Let \cref{ass:iter-nl} hold.
    Then $x_{N(\delta,y^\delta)}^\delta\to \bar x$ with $F(\bar x) = y$ as $\delta\to 0$.
\end{theorem}
\begin{proof}
    We apply \cref{lem:iter:conv}, for which we have already shown condition \eqref{eq:iter:stab} in \cref{thm:iter:landweber:conv_exact}.
    Since $F$ and $F'$ are by assumption continuous, the right-hand side of \eqref{eq:iter:landweber} for fixed  $n\in\N$ depends continuously on $x_n$.
    Hence for all $k\leq n$, the right-hand side of \eqref{eq:iter:landweber} for $x_{k+1}^\delta$ converges to that for $x_{k+1}$ as $\delta\to 0$, which implies the stability condition \eqref{eq:iter:conv_stab}. Finally, the monotonicity condition \eqref{eq:iter:conv_mon} follows from \cref{lem:iter:landweber_mon}, and hence \cref{lem:iter:conv} yields \eqref{eq:iter:conv}.
\end{proof}

Under the usual source condition $x^\dag-x_0\in\calR(F'(x^\dag)^*)$ -- together with additional, technical, assumptions on the nonlinearity of $F$ -- it is possible to show the expected convergence rate of $\mathcal{O}(\sqrt{\delta})$, see \cite[Theorem~3.2]{HankeNeubauerScherzer}, \cite[Theorem~2.13]{Kaltenbacher}.

\section{Levenberg--Marquardt method}

As in the linear case, one drawback of the Landweber iteration is that \eqref{eq:iter:landweber_stop_noisy} shows that $N(\delta,y^\delta)=\mathcal{O}(\delta^{-2})$ may be necessary to satisfy the discrepancy principle, which in practice can be too many. Faster iterations can be built on Newton-type methods.
For the original equation $F(x)=y$, one step of Newton's method consists in solving the linearized equation
\begin{equation}
    \label{eq:iter:newton}
    F'(x_n) h_n = -(F(x_n)-y)
\end{equation}
and setting $x_{n+1}:= x_n+h_n$. However, if $F$ is completely continuous, the Fréchet derivative $F'(x_n)$  is compact by \cref{thm:nichtlin:frechet_complete_cont}, and hence \eqref{eq:iter:newton} is in general ill-posed as well. The idea is now to apply Tikhonov regularization to the Newton step \eqref{eq:iter:newton}, i.e., to compute $h_n$ as the solution of the minimization problem
\begin{equation}
    \label{eq:iter:newton_tikh}
    \min_{h\in X} \frac12\norm{F'(x_n)h + F(x_n)-y}_Y^2 + \frac{\alpha_n}2 \norm{h}_X^2
\end{equation}
for suitable $\alpha_n>0$. Using \cref{lem:tikhonov:normalen} and $h_n = x_{n+1} -x_n$, this leads to an explicit scheme that is known as the \term{Levenberg--Marquardt method}\index{regularization method!Levenberg--Marquardt}:
\begin{equation}
    \label{eq:iter:lm}
    x_{n+1} = x_n + \left(F'(x_n)^*F'(x_n) + \alpha_n\Id\right)^{-1}F'(x_n)^* (y-F(x_n)).
\end{equation}
We now show similarly to the Landweber iteration that \eqref{eq:iter:lm} leads to an iterative regularization method even for noisy data $y^\delta\in B_\delta(y)$.
This requires choosing $\alpha_n$ appropriately; we do this such that the corresponding minimizer $h_{\alpha_n}$ satisfies for some $\sigma\in(0,1)$ the equation
\begin{equation}
    \label{eq:iter:lm:alpha}
    \norm{F'(x_n^\delta)h_{\alpha_n} + F(x_n^\delta)-y^\delta}_Y = \sigma \norm{F(x_n^\delta)-y^\delta}_Y.
\end{equation}
Note that this is a heuristic choice rule; we thus require additional assumptions.
\begin{assumption}\label{ass:iter:lm1}
    Let $F:U\to Y$ be continuously differentiable and $x_0\in U$. Assume that there exists an $r>0$ such that
    \begin{enumerate}
        \item $B_{2r}(x_0)\subset U$;
        \item there exists a solution $x^\dag\in B_r(x_0)$;
        \item there exists a $\gamma>1$ such that
            \begin{equation}
                \label{eq:iter:lm:alpha_bed}
                \norm{F'(x_n^\delta)(x^\dag-x_n^\delta) + F(x_n^\delta)-y^\delta}_Y \leq \frac\sigma\gamma\norm{F(x_n^\delta)-y^\delta}_Y \qquad\text{for all }n\in\N.
            \end{equation}
    \end{enumerate}
\end{assumption}
\begin{theorem}\label{thm:iter:lm:alpha}
    If \cref{ass:iter:lm1} holds, then there exists an $\alpha_n>0$ satisfying \eqref{eq:iter:lm:alpha}.
\end{theorem}
\begin{proof}
    Set $f_n(\alpha):= \norm{F'(x_n^\delta)h_{\alpha} + F(x_n^\delta)-y^\delta}_Y$. Since $F'(x_n^\delta)$ is linear, the minimizer $h_{\alpha}$ of \eqref{eq:iter:newton_tikh} is unique for all $\alpha>0$. As in the proof of \cref{thm:tikh_nl:value}, this implies the continuity of $f_n$ as well as that
    \begin{align*}
        \lim_{\alpha\to\infty} f_n(\alpha) &= \norm{F(x_n^\delta)-y^\delta}_Y,\\
        \lim_{\alpha\to 0} f_n(\alpha) &= \inf_{h\in X}\norm{F'(x_n^\delta)h + F(x_n^\delta)-y^\delta}_Y \leq \norm{F'(x_n^\delta)(x^\dag -x_n^\delta) + F(x_n^\delta)-y^\delta}_Y.
    \end{align*}
    By assumption, we now have that
    \begin{equation*}
        \lim_{\alpha\to 0} f_n(\alpha) \leq \frac\sigma\gamma  \norm{F(x_n^\delta)-y^\delta}_Y < \sigma\norm{F(x_n^\delta)-y^\delta}_Y < \norm{F(x_n^\delta)-y^\delta}_Y = \lim_{\alpha\to\infty} f_n(\alpha),
    \end{equation*}
    which together with the continuity of $f_n(\alpha)$ implies the existence of a solution $\alpha_n>0$ of $f_n(\alpha)= \sigma\norm{F(x_n^\delta)-y^\delta}_Y$.
\end{proof}
For this choice of $\alpha_n$, we can again show the Fejér monotonicity property \eqref{eq:iter:conv_mon}.
\begin{lemma}\label{lem:iter:lm_mon}
    Let \cref{ass:iter:lm1} hold.
    If $x_n^\delta\in B_{r}(x^\dag)$, then
    \begin{equation}
        \label{eq:iter:lm_mon}
        \norm{x_n^\delta-x^\dag}_X^2 - \norm{x_{n+1}^\delta-x^\dag}_X^2 \geq \norm{x_{n+1}^\delta-x_n^\delta}_X^2 + \frac{2(\gamma-1)\sigma^2}{\gamma \alpha_n} \norm{F(x_n^\delta)-y^\delta}_Y^2.
    \end{equation}
    In particular,
    \begin{equation}
        \label{eq:iter:lm_mon2}
        \norm{x_{n+1}^\delta-x^\dag}_X \leq \norm{x_{n}^\delta-x^\dag}_X
    \end{equation}
    and hence $x_{n+1}^\delta\in B_r(x^\dag)\subset B_{2r}(x_0)$.
\end{lemma}
\begin{proof}
    We proceed as for \cref{lem:iter:landweber_mon} by using the iteration \eqref{eq:iter:lm} to estimate the error difference, this time applying the parameter choice \eqref{eq:iter:lm:alpha} in place of the discrepancy principle.
    For the sake of legibility, we set $T_n:= F'(x_n^\delta)$, $h_n:= x_{n+1}^\delta-x_n^\delta$, and $\tilde y_n:= y^\delta-F(x_n^\delta)$.
    First, we rewrite \eqref{eq:iter:lm} as $\alpha_n h_n = T_n^*\tilde y_n - T_n^*T_n h_n$, which implies that
    \begin{align}
        \inner{x_{n+1}^\delta-x_n^\delta}{x_n^\delta - x^\dag}_X &= \alpha_n^{-1}\inner{\tilde y_n - T_n h_n}{ T_n(x_n^\delta-x^\dag)}_Y\label{eq:iter:lm_mon3}
        \intertext{and similarly that}
        \inner{x_{n+1}^\delta-x_n^\delta}{x_{n+1}^\delta- x_n^\delta}_X &= \alpha_n^{-1}\inner{\tilde y_n - T_n h_n}{ T_n h_n}_Y.
    \end{align}
    Together with the productive zero $\tilde y_n - \tilde y_n$, this shows that
    \begingroup
    \allowdisplaybreaks
    \begin{equation*}
        \begin{aligned}[t]
            \norm{x_{n+1}^\delta-x^\dag}_X^2 - \norm{x_{n}^\delta-x^\dag}_X^2
            &= 2 \inner{x_{n+1}^\delta-x_n^\delta}{x_n^\delta - x^\dag}_X  + \norm{x_{n+1}^\delta- x_n^\delta}_X^2\\
            &= 2\alpha_n^{-1} \inner{\tilde y_n - T_n h_n}{ \tilde y_n + T_n(x_n^\delta-x^\dag)}_Y\\
            \MoveEqLeft[-1] + 2\alpha_n^{-1} \inner{\tilde y_n - T_n h_n}{ T_n h_n-\tilde y_n}_Y - \norm{x_{n+1}^\delta- x_n^\delta}_X^2\\
            &=  2\alpha_n^{-1} \inner{\tilde y_n - T_n h_n}{ \tilde y_n - T_n(x^\dag-x_n^\delta)}_Y\\
            \MoveEqLeft[-1] - 2\alpha_n^{-1} \norm{\tilde y_n - T_n h_n}_Y^2 - \norm{x_{n+1}^\delta- x_n^\delta}_X^2\\
            &\leq  2\alpha_n^{-1} \norm{\tilde y_n - T_n h_n}_Y\norm{\tilde y_n - T_n(x^\dag-x_n^\delta)}_Y\\
            \MoveEqLeft[-1] - 2\alpha_n^{-1} \norm{\tilde y_n - T_n h_n}_Y^2 - \norm{x_{n+1}^\delta- x_n^\delta}_X^2.
        \end{aligned}
    \end{equation*}
    \endgroup
    For the terms with $h_n$, we can directly insert the parameter choice rule \eqref{eq:iter:lm:alpha}.
    For the terms with $x^\dag$, we apply the assumption \eqref{eq:iter:lm:alpha_bed} together with  \eqref{eq:iter:lm:alpha} to obtain that
    \begin{equation*}
        \norm{\tilde y_n - T_n (x^\dag-x_n^\delta)}_Y \leq \frac\sigma\gamma\norm{\tilde y_n}_Y = \frac1\gamma \norm{\tilde y_n - T_n h_n}_Y.
    \end{equation*}
    Inserting this, rearranging, and multiplying with $-1$ now yields \eqref{eq:iter:lm_mon}.
\end{proof}

We next show that for noisy data $y^\delta\in B_\delta(y)$, the discrepancy principle \eqref{eq:iter:diskrepanz} yields a finite stopping index $N(\delta,y^\delta)$.
This requires a stronger version of the tangential cone condition \eqref{eq:iter:landweber_ass:tan}.
\begin{assumption}\label{ass:iter:lm2}
    Let \cref{ass:iter:lm1} hold with (iii) replaced by
    \begin{enumerate}
        \item[(iii$'$)] there exist $M>0$ and $c>0$ such that for all $x,\tilde x\in B_{2r}(x_0)$,
            \begin{align}
                \norm{F'(x)}_{\linop(X,Y)}&\leq M,\\
                \norm{F(x)-F(\tilde x) - F'(x)(x-\tilde x)}_Y &\leq c\norm{x-\tilde x}_X \norm{F(x)-F(\tilde x)}_Y.
                \label{eq:iter:lm_tan}
            \end{align}
    \end{enumerate}
\end{assumption}
\begin{theorem}\label{thm:iter:lm_stop}
    Let \cref{ass:iter:lm2} hold.
    If $N(\delta,y^\delta)$ is chosen according to the discrepancy principle \eqref{eq:iter:diskrepanz} with $\tau>\sigma^{-1}$ and if $\norm{x_0-x^\dag}_X$ is sufficiently small, then
    \begin{equation*}
        N(\delta,y^\delta) < C (1+|\log \delta|)\qquad\text{for some } C>0.
    \end{equation*}
\end{theorem}
\begin{proof}
    We first show that under these assumptions, the error decreases up to the stopping index.
    Assume that $N:=N(\delta,y^\delta)\geq 1$ (otherwise there is nothing to show) and that
    \begin{equation}
        \label{eq:iter:lm_stop1}
        \norm{x_0-x^\dag}_X \leq \min \{r,\tilde r\}, \qquad \tilde r:=  \frac{\sigma\tau-1}{c(1+\tau)}.
    \end{equation}
    From \eqref{eq:iter:lm_tan} with $x=x_0$ and $\tilde x = x^\dag$, we then obtain by inserting $y-y$ that
    \begin{equation*}
        \begin{aligned}
            \norm{F'(x_0)(x^\dag - x_0)+F(x_0)-y^\delta}_Y &\leq \delta + \norm{F(x_0)-y-F'(x_0)(x_0-x^\dag)}_Y \\
            &\leq \delta + c\norm{x_0-x^\dag}_X\norm{F(x_0)-y}_Y\\
            &\leq (1+c \norm{x_0-x^\dag}_X)\delta + c\norm{x_0-x^\dag}_X\norm{F(x_0)-y^\delta}_Y.
        \end{aligned}
    \end{equation*}
    Since $x_0$ by assumption does not satisfy the discrepancy principle, $\delta < \tau^{-1}\norm{F(x_0)-y^\delta}_Y$. Inserting this thus yields \eqref{eq:iter:lm:alpha_bed} with $\gamma := \sigma\tau(1+c(1+\tau)\norm{x_0-x^\dag}_X)^{-1}>1$ for $\norm{x_0-x^\dag}_X$ sufficiently small. Hence \cref{lem:iter:lm_mon} implies that
    \begin{equation*}
        \norm{x_1^\delta-x^\dag}_X \leq \norm{x_0-x^\dag}_X \leq \min\{r,\tilde r\}
    \end{equation*}
    and therefore in particular that $x_1^\delta \in B_{2r}(x_0)\subset U$.
    If now $N>1$, we obtain as above that
    \begin{equation*}
        \begin{aligned}
            \norm{F'(x_1^\delta)(x^\dag - x_1^\delta)+F(x_1^\delta)-y^\delta}_Y &\leq  (1+c \norm{x_1^\delta-x^\dag}_X)\delta + c\norm{x_1^\delta-x^\dag}_X\norm{F(x_1^\delta)-y^\delta}_Y
            \\
            &\leq  (1+c \norm{x_0-x^\dag}_X)\delta + c\norm{x_0-x^\dag}_X\norm{F(x_1^\delta)-y^\delta}_Y.
        \end{aligned}
    \end{equation*}
    By induction, the iteration \eqref{eq:iter:lm} is thus well-defined for all $n<N$, and \eqref{eq:iter:lm_mon} holds.

    Proceeding as for the Landweber iteration by summing the residuals now requires a uniform bound on $\alpha_n$. For this, we use that with $T_n$, $h_n$ and $\tilde y_n$ as in the proof of \cref{lem:iter:lm_mon},
    \begin{equation*}
        (T_nT_n^*+\alpha_n\Id)\left(\tilde y_n -T_nh_n\right) = T_n\left(T_n^*\tilde y_n - T_n^*T_nh_n - \alpha_n h_n\right) + \alpha_n \tilde y_n = \alpha_n \tilde y_n,
    \end{equation*}
    where we have used the iteration \eqref{eq:iter:lm} in the last step.
    Using the assumption $\norm{T_n}_{\linop(X,Y)}\leq M$ and the parameter choice \eqref{eq:iter:lm:alpha} then implies that
    \begin{equation}\label{eq:iter:lm_stop2}
        \begin{aligned}[t]
            \alpha_n \norm{\tilde y_n}_Y &=\norm{(T_nT_n^*+\alpha_n\Id)(\tilde y_n-T_n h_n)}_Y\\
            &\leq (M^2+\alpha_n)\norm{\tilde y_n-T_n h_n}_Y\\
            &= (M^2+\alpha_n)\sigma \norm{\tilde y_n}_Y.
        \end{aligned}
    \end{equation}
    Solving \eqref{eq:iter:lm_stop2} for $\alpha_n$ now yields that
    $\alpha_n \leq \frac{\sigma M^2}{1-\sigma}$, which together with \eqref{eq:iter:lm_mon} leads to
    \begin{equation*}
        \norm{x_n^\delta-x^\dag}_X^2 - \norm{x_{n+1}^\delta-x^\dag}_X^2 \geq  \frac{2(\gamma-1)(1-\sigma)\sigma}{\gamma M^2} \norm{F(x_n^\delta)-y^\delta}_Y^2 \qquad\text{for all }n<N.
    \end{equation*}

    Since $N$ was chosen according to the discrepancy principle \eqref{eq:iter:diskrepanz}, we can sum this inequality from $n=0$ to $N-1$ to obtain the estimate
    \begin{equation*}
        N (\tau\delta)^2 \leq \sum_{n=0}^{N-1}\norm{F(x_n^\delta)-y^\delta}_Y^2
        \leq \frac{\gamma M^2}{2(\gamma-1)(1-\sigma)\sigma}\norm{x_0-x^\dag}_X^2.
    \end{equation*}
    This implies that $N$ is finite for all $\delta>0$.

    For the logarithmic estimate, we use the parameter choice \eqref{eq:iter:lm:alpha} together with the assumption \eqref{eq:iter:lm_tan} to show that for arbitrary $n<N$,
    \begin{equation*}
        \begin{aligned}
            \sigma \norm{F(x_n^\delta)-y^\delta}_Y &= \norm{F'(x_n^\delta)h_n + F(x_n^\delta) -y^\delta}_Y\\
            &\geq \norm{F(x_{n+1}^\delta)-y^\delta}_Y - \norm{F'(x_n^\delta)h_n+F(x_n^\delta)-F(x_{n+1}^\delta)}_Y\\
            &\geq \norm{F(x_{n+1}^\delta)-y^\delta}_Y - c\norm{h_n}_X\norm{F(x_{n+1}^\delta)-F(x_n^\delta)}_Y\\
            &\geq (1-c\norm{h_n}_X)\norm{F(x_{n+1}^\delta)-y^\delta}_Y - c\norm{h_n}_X\norm{F(x_n^\delta)-y^\delta}_Y.
        \end{aligned}
    \end{equation*}
    We now obtain from \eqref{eq:iter:lm_mon} that
    \begin{equation*}
        \norm{h_n}_X \leq \norm{x_n^\delta-x^\dag}_X \leq \norm{x_0-x^\dag}_X,
    \end{equation*}
    which together with the discrepancy principle yields for $n=N-2$ that
    \begin{equation*}
        \begin{aligned}
            \tau\delta \leq \norm{F(x_{N-1}^\delta)-y^\delta}_Y &\leq \frac{\sigma + c\norm{x_0-x^\dag}_X}{1-c\norm{x_0-x^\dag}_X} \norm{F(x_{N-2}^\delta)-y^\delta}_Y\\
            &\leq \left(\frac{\sigma + c\norm{x_0-x^\dag}_X}{1-c\norm{x_0-x^\dag}_X}\right)^{N-1} \norm{F(x_{0})-y^\delta}_Y.
        \end{aligned}
    \end{equation*}
    For $\norm{x_0-x^\dag}_X$ sufficiently small, the term in parentheses is strictly less than $1$, and taking the logarithm shows the desired bound on $N$.
\end{proof}

If the noise level $\delta$ is small, $\mathcal{O}(1+|\log \delta|)$ is a significantly smaller bound than $\mathcal{O}(\delta^{-2})$ (for comparable constants, which however cannot be assumed in general), and therefore the Levenberg--Marquardt method can be expected to terminate much earlier than the Landweber iteration. On the other hand, each step is more involved since it requires the solution of a linear system. Which of the two methods is faster in practice (as measured by actual time) depends on the individual inverse problem.

We now consider (local) convergence for exact data.
\begin{theorem}
    Let \cref{ass:iter:lm2} hold.
    If $\norm{x_0-x^\dag}_X$ is sufficiently small, then $x_n\to \bar x$ with $F(\bar x)=y$ as $n\to\infty$.
\end{theorem}
\begin{proof}
    From \eqref{eq:iter:lm_tan} for $x=x_0$ and $\tilde x = x^\dag$, we directly obtain that
    \begin{equation*}
        \norm{F(x_0)-y - F'(x_0)(x_0-x^\dag)}_Y \leq c\norm{x_0-x^\dag}_X \norm{F(x_0)-y}_Y.
    \end{equation*}
    For $\norm{x_0-x^\dag}_X$ sufficiently small we then have that $\gamma:= \sigma(c\norm{x_0-x^\dag}_X)^{-1}>1$ and thus that \eqref{eq:iter:lm:alpha_bed} holds. We can thus apply \cref{lem:iter:lm_mon} to deduce that $\norm{x_1-x^\dag}_X\leq \norm{x_0-x^\dag}_X$. Hence, $x_1\in B_{2r}(x_0)$ and thus $\norm{x_1-x^\dag}_X$ is sufficiently small as well.
    By induction, we then obtain the well-posedness of the iteration and the monotonicity of the error for all $n\in\N$.
    As in the proof of \cref{thm:iter:lm_stop}, rearranging and summing yields that
    \begin{equation*}
        \sum_{n=0}^{\infty}\norm{F(x_n)-y}_Y^2
        \leq \frac{\gamma M^2}{2(\gamma-1)(1-\sigma)\sigma}\norm{x_0-x^\dag}_X^2 <\infty
    \end{equation*}
    and hence that $F(x_n)\to y$ as $n\to\infty$.

    The remainder of the proof proceeds analogously to that of \cref{thm:iter:landweber:conv_exact}.
    We set $e_n :=  x_n-x^\dag$ and consider
    \begin{equation*}
        \norm{e_m-e_n}_X \leq \norm{e_m-e_k}_X + \norm{e_k-e_n}_X
    \end{equation*}
    for any $m\geq n$ and $k\in\{n,\dots,m\}$ chosen according to \eqref{eq:iter:landweber:conv_exact1}.
    The Fejér monotonicity from \cref{lem:iter:lm_mon} again shows that $\|e_n\|_X\to\eps$ for some $\eps\geq 0$ as $n\to\infty$, requiring us to only look at the mixed terms.
    Using \eqref{eq:iter:lm_mon3} and the parameter choice \eqref{eq:iter:lm:alpha}, we obtain that
    \begin{equation*}
        \begin{aligned}
            \inner{e_k-e_m}{e_k}_X &= \sum_{j=k}^{m-1} -\inner{x_{j+1}-x_j}{x_k-x^\dag}_X\\
            &= \sum_{j=k}^{m-1} -\alpha_j^{-1}\inner{y-F(x_j)-F'(x_j)(x_{j+1}-x_j)}{F'(x_j)(x_k-x^\dag)}_Y\\
            &\leq \sum_{j=k}^{m-1} \alpha_j^{-1}\norm{y-F(x_j)-F'(x_j)(x_{j+1}-x_j)}_Y\norm{F'(x_j)(x_k-x^\dag)}_Y\\
            &= \sum_{j=k}^{m-1} \sigma\alpha_j^{-1}\norm{F(x_j)-y}_Y\norm{F'(x_j)(x_k-x^\dag)}_Y.
        \end{aligned}
    \end{equation*}
    For the second term, we use \eqref{eq:iter:lm_tan} and set $\eta:= c\norm{x_0-x^\dag}_X\geq c\norm{x_j-x^\dag}_X$ for all $j\geq0$ to arrive at
    \begin{equation*}
        \begin{aligned}
            \norm{F'(x_j)(x_k-x^\dag)}_Y &\leq \norm{F(x_k)-y}_Y + \norm{y-F(x_j)-F'(x_j)(x^\dag-x_j)}_Y\\
            \MoveEqLeft[-7] + \norm{F(x_j)-F(x_k)-F'(x_j)(x_j-x_k)}_Y\\
            &\leq \norm{F(x_k)-y}_Y + c\norm{x_j-x^\dag}_X \norm{F(x_j)-y}_Y \\
            \MoveEqLeft[-7] + c\norm{x_j-x_k}_X \norm{F(x_j)-F(x_k)}_Y\\
            &\leq (1+5\eta)\norm{F(x_j)-y}_Y,
        \end{aligned}
    \end{equation*}
    where we have again used multiple productive zeros as well as \eqref{eq:iter:landweber:conv_exact1}.

    We can now apply \eqref{eq:iter:lm_mon} to obtain that
    \begin{equation*}
        \begin{aligned}
            \inner{e_k-e_m}{e_k}_X &\leq \sum_{j=k}^{m-1}(1+5\eta)\sigma\alpha_j^{-1}\norm{F(x_j)-y}_Y^2\\
            &\leq \sum_{j=k}^{m-1} \frac{\gamma(1+5\eta)}{2\sigma(\gamma-1)} \left(\norm{e_j}_X^2-\norm{e_{j+1}}_X^2\right)\\
            &= \frac{\gamma(1+5\eta)}{2\sigma(\gamma-1)} \left(\norm{e_k}_X^2-\norm{e_{m}}_X^2\right) \to 0
        \end{aligned}
    \end{equation*}
    as $n\to \infty$ due to the convergence of $\|e_n\|_X\to\eps$.
    We similarly deduce that
    \begin{equation*}
        \inner{e_k-e_n}{e_k}_X\leq  \frac{\gamma(1+5\eta)}{2\sigma(\gamma-1)} \left(\norm{e_n}_X^2-\norm{e_{k}}_X^2\right) \to 0
    \end{equation*}
    as $n\to\infty$,
    which again implies that $\{e_n\}_{n\in\N}$ and hence that $\{x_n\}_{n\in\N}$ is a Cauchy sequence. The claim now follows since $F(x_n)\to y$.
\end{proof}

We now have almost everything at hand to apply \cref{lem:iter:conv} and show the convergence of the Levenberg--Marquardt method for noisy data $y^\delta\in Y$.
\begin{theorem}
    Let \cref{ass:iter:lm2} hold. If $\norm{x_0-x^\dag}_X$ is sufficiently small, then $x_{N(\delta,y^\delta)}^\delta\to \bar x$ with $F(\bar x) = y$ as $\delta\to 0$.
\end{theorem}
\begin{proof}
    It remains to verify the continuity condition \eqref{eq:iter:conv_stab}. Since $F$ is assumed to be continuously differentiable, $F'(x^\dag)^*F'(x^\dag)+\alpha\Id$ is continuous.
    By the inverse function theorem (e.g., \cite[Theorem~10.4]{renardyrogers2004}), there thus exists a sufficiently small neighborhood of $x^\dag$ where $(F'(x)^*F'(x)+\alpha\Id)^{-1}$ is continuous as well. For fixed $n\in\N$, the right-hand side of \eqref{eq:iter:lm} is therefore continuous in $x_n$, which implies the condition~\eqref{eq:iter:conv_stab} and hence the claimed convergence.
\end{proof}

Under a source condition and with a suitable a priori choice of $\alpha_n$ and $N=N(\delta)$, it is possible to show (logarithmic) convergence rates as $\delta\to 0$; see \cite[Theorem~4.7]{Kaltenbacher}.

\section{Iteratively regularized Gauß--Newton method}

We finally consider the following version of the Levenberg--Marquardt method which was proposed in \cite{Bakushinskii:1992}: Set $x_{n+1}=x_n+h_n$ where now $h_n$ is the solution of the minimization problem
\begin{equation}
    \label{eq:iter:irgn_min}
    \min_{h\in X} \frac12\norm{F'(x_n)h + F(x_n)-y}_Y^2 + \frac{\alpha_n}2 \norm{h+x_n-x_0}_X^2.
\end{equation}
By \cref{lem:tikhonov:normalen}, this is equivalent to the explicit iteration known as the \term[iteratively regularized Gauss--Newton method@iteratively regularized Gauß--Newton method]{iteratively regularized Gauß--Newton method}\index{regularization method!iteratively regularized Gauss--Newton@iteratively regularized Gauß--Newton}:
\begin{equation}
    \label{eq:iter:irgn}
    x_{n+1} = x_n + \left(F'(x_n)^*F'(x_n) + \alpha_n\Id\right)^{-1}\left(F'(x_n)^* (y-F(x_n))+\alpha_n(x_0-x_n)\right).
\end{equation}
Note that the only difference to the Levenberg--Marquardt method is the additional term on the right-hand side.
Similarly, comparing \eqref{eq:iter:irgn_min} to \eqref{eq:iter:newton_tikh}, the former has $x_n+h_n-x_0 = x_{n+1}-x_0$ in the regularization term. The point is that this allows interpreting $x_{n+1}$ directly as the minimizer of the \emph{linearized} Tikhonov functional
\begin{equation*}
    \min_{x\in X} \frac12\norm{F'(x_n)(x-x_n) + F(x_n)-y}_Y^2 + \frac{\alpha_n}2 \norm{x-x_0}_X^2,
\end{equation*}
and hence using the properties of linear Tikhonov regularization for the analysis.
In practice, this method also shows better stability since the explicit regularization of $x_{n+1}$ prevents unchecked growth through the constant addition of (bounded) increments $h_n$.

As for the Levenberg--Marquardt method, one can now show (under some conditions on the nonlinearity) that this iteration is well-defined and converges for exact as well as noisy data; see \cite[Theorem~4.2]{Kaltenbacher}.
Instead, we will only show convergence rates for an a priori choice of $\alpha_n$ and $N(\delta)$.
To make use of the results for linear Tikhonov regularization from \cref{chap:tikhonov}, we assume that $F$ is Fréchet differentiable and completely continuous such that $F'(x)$ is compact for all $x$ by \cref{thm:nichtlin:frechet_complete_cont}. Specifically, we make the following assumptions.
\begin{assumption}\label{ass:iter:irgn}
    Let $F:U\to Y$ be continuously differentiable and completely continuous, and let $x^\dag$ be an $x_0$-minimum norm solution. Assume further that
    \begin{enumerate}
        \item $F'$ is Lipschitz continuous with constant $L$;
        \item there exists a $w\in X$ with $x^\dag -x_0 = |F'(x^\dag)|^\nu w$ and $\norm{w}_X \leq \rho$ for some $\nu\in[1,2]$ and $\rho>0$;
    \end{enumerate}
\end{assumption}
We first show that the regularization error satisfies a quadratic recursion.
\begin{lemma}\label{lem:iter:irgn_rek}
    Let \cref{ass:iter:irgn} hold.
    If the stopping index $N(\delta)$ and $\alpha_n$, $1\leq n\leq N(\delta)$, are chosen such that
    \begin{equation}\label{eq:iter:irgn:apriori}
        \alpha_{N(\delta)}^{(\nu+1)/2}\leq \tau \delta \leq \alpha_n^{(\nu+1)/2} \qquad\text{for all }n<N(\delta)
    \end{equation}
    and some $\tau>0$, then
    \begin{multline*}
        \norm{x_{n+1}^\delta-x^\dag}_X \leq \left(C_\nu\rho+\tau^{-1}\right)\alpha_n^{\nu/2} + L\rho \left(C_\nu\alpha_n^{(\nu-1)/2} + \norm{F'(x^\dag)}_{\linop(X,Y)}^{\nu-1}\right)\norm{x_{n}^\delta-x^\dag}_X\\
        + \frac{L}{2\alpha_n^{1/2}}\norm{x_{n}^\delta-x^\dag}_X^2 \qquad\text{for all } n<N(\delta).
    \end{multline*}
\end{lemma}
\begin{proof}
    Using the iteration and rearranging appropriately, we split the regularization error $x_{n+1}^\delta-x^\dag$ into three components that we then estimate separately. We set $K_n:= F'(x_n^\delta)$ as well as $K:= F'(x^\dag)$ and write
    \begin{equation*}
        \begin{aligned}
            x_{n+1}^\delta-x^\dag &= x_n^\delta - x^\dag + \left(K_n^*K_n + \alpha_n\Id\right)^{-1}\left(K_n^* (y^\delta-F(x_n^\delta))+\alpha_n(x_0-x_n^\delta)\right)\\
            &= \left(K_n^*K_n + \alpha_n\Id\right)^{-1}\left(\alpha_n(x_0-x^\dag) + K_n^*\left(y^\delta-F(x_n^\delta) + K_n(x_n^\delta - x^\dag)\right)\right)\\
            &= \left[\alpha_n\left(K^*K+\alpha_n\Id\right)^{-1}(x_0-x^\dag)\right]
            + \left[ \left(K_n^*K_n + \alpha_n\Id\right)^{-1}K_n^* (y^\delta-y)\right]\\
            \MoveEqLeft[-1]+ \Big[ \left(K_n^*K_n + \alpha_n\Id\right)^{-1}K_n^*\left (F(x^\dag)-F(x_n^\delta) + K_n(x_n^\delta-x^\dag)\right)\\
            \MoveEqLeft[-2]+ \alpha_n\left(K_n^*K_n + \alpha_n\Id\right)^{-1}(K_n^*K_n - K^*K)\left(K^*K +\alpha_n\Id\right)^{-1}(x_0-x^\dag)\Big]\\
            &=:  [e_1] + [e_2] + [e_{3a}+e_{3b}].
        \end{aligned}
    \end{equation*}
    We first estimate the \enquote{approximation error} $e_1$. Since $K$ is compact, we obtain from \cref{lem:tikhonov:normalen} the representation $(K^*K+\alpha\Id)^{-1} x = \phi_\alpha(K^*K)x$ for $\phi_\alpha(\lambda) = (\lambda+\alpha)^{-1}$. Together with the source condition, this implies for all $\nu\leq \nu_0 = 2$ that
    \begin{equation*}
        \begin{aligned}
            \norm{e_1}_X &= \norm{\alpha_n\left(K^*K+\alpha_n\Id\right)^{-1}(x_0-x^\dag)}_X\\
            &= \norm{\alpha_n\phi_{\alpha_n}(K^*K)(K^*K)^{\nu/2}w}_X\\
            &\leq \sup_{\lambda\in(0,\kappa]} \frac{\alpha_n \lambda^{\nu/2}}{\lambda+\alpha_n} \norm{w}_X = \sup_{\lambda\in(0,\kappa]} \omega_\nu(\alpha_n)\norm{w}_X\\
            &\leq C_\nu \alpha_n^{\nu/2}\rho
        \end{aligned}
    \end{equation*}
    as shown in \cref{chap:tikhonov}.

    For the \enquote{data error} $e_2$, we also use the estimates from \cref{chap:tikhonov} together with the a priori choice of $\alpha_n$ to obtain for all $n<N(\delta)$ that
    \begin{equation*}
        \begin{aligned}
            \norm{e_2}_X &= \norm{\left(K_n^*K_n + \alpha_n\Id\right)^{-1}K_n^* (y^\delta-y)}_X\\
            &\leq \norm{\phi_{\alpha_n}(K_n^*K_n) K_n^*}_{\linop(Y,X)} \norm{y^\delta -y}_Y\\
            &\leq \frac1{\sqrt{\alpha_n}}\,\delta \leq \tau^{-1}\alpha_n^{\nu/2}.
        \end{aligned}
    \end{equation*}

    The \enquote{nonlinearity error} $e_{3a}+e_{3b}$ is again estimated separately. For the first term, we use the Lipschitz condition and \cref{lem:nichtlin:lipschitz} to bound
    \begin{equation*}
        \begin{aligned}
            \norm{e_{3a}}_X &:=  \norm{\left(K_n^*K_n + \alpha_n\Id\right)^{-1}K_n^*\left (F(x^\dag)-F(x_n^\delta) + K_n(x_n^\delta-x^\dag)\right)}_X \\
            &\leq
            \norm{\phi_{\alpha_n}(K_n^*K_n)K^*_n}_{\linop(Y,X)}\norm{F(x^\dag)-F(x_n^\delta) - F'(x_n^\delta)(x^\dag - x_n^\delta)}_Y\\
            &\leq \frac{1}{\sqrt{\alpha_n}} \frac{L}{2} \norm{x_n^\delta -x^\dag}_X^2.
        \end{aligned}
    \end{equation*}
    For the second term, we use the identity
    \begin{equation*}
        K_n^*K_n - K^*K = K_n^*(K_n-K) + (K_n^*-K^*)K
    \end{equation*}
    as well as the Lipschitz continuity of $F'(x)$ and the source condition to estimate similarly as above
    \begin{equation*}
        \begin{aligned}
            \norm{e_{3b}}_X &:=  \norm{\alpha_n\left(K_n^*K_n + \alpha_n\Id\right)^{-1}(K_n^*K_n - K^*K)\left(K^*K +\alpha_n\Id\right)^{-1}(x_0-x^\dag)}_X\\
            &\leq \norm{\phi_{\alpha_n}(K_n^*K_n)K_n^*}_{\linop(Y,X)}\norm{K-K_n}_{\linop(X,Y)}\norm{\alpha_n\phi_{\alpha_n}(K^*K)(K^*K)^{\nu/2}w}_X\\
            \MoveEqLeft[-1] + \norm{\alpha_n\phi_{\alpha_n}(K_n^*K_n)}_{\linop(X,X)}\norm{K_n-K}_{\linop(X,Y)}\norm{K\phi_{\alpha_n}(K^*K)(K^*K)^{1/2}}_{\linop(X,Y)}\\
            \cdot\norm{(K^*K)^{(\nu-1)/2}w}_X\span\omit \\
            &\leq \frac{1}{\sqrt{\alpha_n}}\ L\norm{x^\dag-x_n^\delta}_X\  C_\nu \alpha_n^{\nu/2}\rho + \sup_{\lambda\in(0,\kappa]} \frac{\alpha_n}{\alpha_n+\lambda}\  L \norm{x_n^\delta-x^\dag}\  \norm{K}_{\linop(X,Y)}^{\nu-1}\rho
            \\
            &\leq {L\rho}\left(C_\nu\alpha_n^{(\nu-1)/2} + \norm{K}_{\linop(X,Y)}^{\nu-1}\right)\norm{x_n^\delta-x^\dag}_X,
        \end{aligned}
    \end{equation*}
    where we have used $\norm{K^*}_{\linop(Y,X)} =\norm{K}_{\linop(X,Y)}$ and -- applying \cref{lem:functional_range}\,(iii) -- the inequality
    \begin{equation*}
        \norm{K\phi_\alpha(K^*K)(K^*K)^{1/2}}_{\linop(X,Y)} = \norm{(K^*K)^{1/2}\phi_\alpha(K^*K)(K^*K)^{1/2}}_{\linop(X,X)} \leq \sup_{\lambda\in(0,\kappa]} \frac{\lambda}{\lambda+\alpha}\leq 1.
    \end{equation*}
    Combining the separate estimates yields the claim.
\end{proof}

If the initial error is small enough, we obtain from this the desired error estimate.
\begin{theorem}
    Let \cref{ass:iter:irgn} hold for $\rho>0$ sufficiently small and $\tau>0$ sufficiently large. Assume further that $\alpha_0\leq 1$ and
    \begin{equation*}
        1 < \frac{\alpha_n}{\alpha_{n+1}} \leq q \qquad\text{for some }q>1.
    \end{equation*}
    Then we have for exact data (i.e., $\delta=0$) that
    \begin{align}
        \norm{x_n-x^\dag}_X &\leq c_1 \alpha_n^{\nu/2}\qquad \text{for all }n\in\N
        \label{eq:iter:irgn:conv_exact}
        \intertext{and for noisy data that}
        \norm{x_{N(\delta)}^\delta - x^\dag}_X &\leq c_2 \delta^{\frac{\nu}{\nu+1}} \qquad\text{as }\delta\to 0.
        \label{eq:iter:irgn:conv_noisy}
    \end{align}
\end{theorem}
\begin{proof}
    \Cref{lem:iter:irgn_rek} shows that $\xi_n:= \alpha_n^{-\nu/2}\norm{x_n^\delta - x^\dag}_X$ satisfies the quadratic recursion
    \begin{equation*}
        \xi_{n+1} \leq a + b \xi_n + c \xi_n^2
    \end{equation*}
    with
    \begin{equation*}
        a:=  q^{\nu/2}(C_\nu\rho+\tau^{-1}),\qquad
        b:=  q^{\nu/2}L\rho\left(C_\nu + \norm{F'(x^\dag)}_{\linop(X,Y)}^{\nu-1}\right),\qquad
        c:=  q^{\nu/2}\frac{L}{2}\rho,
    \end{equation*}
    where we have used that $\nu\geq1$ and hence that $\alpha_n^{-1/2}\leq \alpha_n^{-\nu/2}$ and $\alpha_n^{\nu/2} < \alpha_0^{\nu/2}\leq 1$.
    Clearly we can make $a$, $b$ and $c$ arbitrarily small by choosing $\rho$ sufficiently small and $\tau$ sufficiently large.
    Let now $t_1,t_2$ be the solutions of the fixed-point equation $a+bt+ct^2 = t$, i.e.,
    \begin{equation*}
        t_1 = \frac{2a}{1-b+\sqrt{(1-b)^2-4ac}},\qquad t_2 = \frac{1-b+\sqrt{(1-b)^2-4ac}}{2c}.
    \end{equation*}
    Now the source condition yields $\norm{x_0-x^\dag}_X \leq \norm{F'(x^\dag)}_{\linop(X,Y)}^{\nu} \rho$, and hence we can guarantee that $x_0\in B_r(x^\dag)\subset U$ for some $r>0$ by choosing $\rho$ sufficiently small. In particular, we can assume that $t_2 \geq \xi_0$.

    We now show by induction that
    \begin{equation}
        \label{eq:iter:irgn:conv_ind1}
        \xi_n \leq \max\{t_1,\xi_0\} =:  C_\xi\qquad\text{for all }n\leq N(\delta).
    \end{equation}
    For $n=0$, this claim follows straight from the definition; we thus assume that \eqref{eq:iter:irgn:conv_ind1} holds for some fixed $n<N(\delta)$.
    Then we have in particular that $\xi_n\leq \xi_0$, and the definition of $\xi_n$ together with the assumptions that $\alpha_n\leq \alpha_0\leq 1$ and $\nu\geq 1$ implies that
    \begin{equation*}
        \norm{x_n^\delta-x^\dag}_X\leq \alpha_n^{\nu/2}\alpha_0^{-\nu/2}\norm{x_0-x^\dag}_X\leq r
    \end{equation*}
    and hence that $x_n^\delta \in B_r(x^\dag)\subset U$.
    This shows that the iteration \eqref{eq:iter:irgn} is well-defined and that we can apply \cref{lem:iter:irgn_rek}. We now distinguish two cases in \eqref{eq:iter:irgn:conv_ind1}:
    \begin{enumerate}
        \item $\xi_n\leq t_1$: Then we have by $a,b,c\geq 0$ and the definition of $t_1$ that
            \begin{equation*}
                \xi_{n+1}\leq a +b\xi_n + c\xi_n^2 \leq a+b t_1 +c t_1^2 =t_1.
            \end{equation*}
        \item $t_1 < \xi_n \leq \xi_0$: Since we have assumed that $t_2\geq \xi_0$, it follows that $\xi_n\in(t_1,t_2]$, and $a+(b-1)t+ct^2\leq 0$ for $t\in[t_1,t_2]$ due to $c\geq 0$ implies that
            \begin{equation*}
                \xi_{n+1}\leq a +b\xi_n + c\xi_n^2 \leq \xi_n \leq \xi_0.
            \end{equation*}
    \end{enumerate}
    In both cases, we have obtained \eqref{eq:iter:irgn:conv_ind1} for $n+1$.

    For $\delta=0$ we have $N(0)=\infty$, and \eqref{eq:iter:irgn:conv_ind1} implies that
    \begin{equation*}
        \norm{x_n-x^\dag}_X \leq \alpha_n^{\nu/2} C_\xi \qquad\text{for all }n\in\N,
    \end{equation*}
    yielding \eqref{eq:iter:irgn:conv_exact} with $c_1 :=  C_\xi$.
    For $\delta>0$, \eqref{eq:iter:irgn:conv_ind1} for $n=N(\delta)$ together with the parameter choice \eqref{eq:iter:irgn:apriori} implies that
    \begin{equation*}
        \norm{x_{N(\delta)}^\delta-x^\dag}_X \leq \alpha_{N(\delta)}^{\nu/2} C_\xi \leq (\tau\delta)^{\frac{\nu}{\nu+1}} C_\xi,
    \end{equation*}
    yielding \eqref{eq:iter:irgn:conv_noisy} with $c_2 :=  C_\xi \tau^{\frac{\nu}{\nu+1}}$.
\end{proof}
In a similar way (albeit with a bit more effort), it is also possible to derive convergence rates (up to the saturation $\nu_0-1=1$) if the stopping index is chosen according to the discrepancy principle, see \cite[Theorem~4.13]{Kaltenbacher}.

\part{Banach space inverse problems}

\chapter{Basics of functional and convex analysis}

In this chapter, we collect the concepts and results from functional analysis in Banach spaces, the calculus of variations, and convex analysis that are needed for the treatment of Tikhonov regularization in Banach spaces. For details and proofs, the reader is referred to \cite{Clason} for functional analysis and to \cite{ClasonValkonen} for calculus of variations and convex analysis.

\section{Banach spaces and weak convergence}

Let $X$ be a normed vector space. The space $X^*:=\linop(X,\R)$ of bounded linear functionals on $X$ is called the \term[space!dual]{dual space} of $X$. For $x^*\in X^*$ and $x\in X$, we write
\begin{equation*}
    \dual{x^*,x}_X := x^*(x) \in\R
\end{equation*}
for the \term[pairing, duality]{duality pairing}. The definition of the operator norm immediately implies that
\begin{equation}\label{eq:funktan:cs_banach}
    \dual{x^*,x}_X \leq \norm{x^*}_{X^*}\norm{x}_X\qquad\text{for all }x\in X,x^*\in X^*.
\end{equation}
If $X$ is a Hilbert space, the Fréchet--Riesz \cref{thm:frechetriesz} allows identifying $X^*$ with $X$, in which case the duality pairing coincides with the inner product.

A central result on normed vector spaces is the Hahn--Banach theorem, which has several important consequences for our purpose. First, elements of the dual space can be used to \enquote{separate} arbitrary disjoint \emph{convex} sets.
\begin{theorem}[Hahn--Banach, separation]\label{thm:hb_separation}\index{Hahn--Banach theorem}
    Let $X$ be a normed vector space and $A,B\subset X$ be convex, nonempty, and disjoint. If $A$ is closed and $B$ is compact, there exist an $x^*\in X^*$ and a $\lambda\in \R$ with
    \begin{equation*}
        \dual{x^*,x_1}_X \leq \lambda < \dual{x^*,x_2}_X \qquad\text{for all }x_1\in A, x_2\in B.
    \end{equation*}
\end{theorem}
Another consequence of the Hahn--Banach theorem is that the norm can be expressed through the dual space as
\begin{equation}\label{eq:funktan:norm_dual}
    \norm{x}_X = \sup_{\norm{x^*}_{X^*}\leq 1} \dual{x^*,x}_X\qquad\text{for all }x\in X.
\end{equation}
In particular, every $x\in X$ defines a bounded linear functional on $X^*$, i.e., an element of the \term{bidual} $X^{**}:=(X^*)^*$, via the \term[injection, canonical]{canonical injection}
\begin{equation*}
    J:X\to X^{**},\qquad \dual{Jx,x^*}_{X^{*}} := \dual{x^*,x}_X\quad\text{for all }x^*\in X^*.
\end{equation*}
If $J$ is surjective (so that we can identify $X^{**}$ with $X$), the space $X$ is called \term{reflexive}. All Hilbert spaces as well as $\ell^p$ and $L^p(\Omega)$ for $1<p<\infty$ are reflexive, but $\ell^1$, $\ell^\infty$, $L^1(\Omega)$, $L^\infty(\Omega)$, and $C(\overline\Omega)$ are not.

\bigskip

The duality pairing induces further notions of convergence:
\begin{enumerate}
    \item A sequence $\{x_n\}_{n\in\N}\subset X$ converges \term[convergence!weak]{weakly} (in $X$) to $x\in X$, denoted by $x_n\wkto x$, if
        \begin{equation*}
            \dual{x^*,x_n}_X\to  \dual{x^*,x}_X \qquad\text{for all }x^*\in X^*.
        \end{equation*}
    \item A sequence $\{x^*_n\}_{n\in\N}\subset X^*$ converges \term[convergence!weak-$*$]{weakly-$*$} (in $X^*$) to $x^*\in X^*$, denoted by $x^*_n\wkto^* x^*$, if
        \begin{equation*}
            \dual{x_n^*,x}_X\to  \dual{x^*,x}_X \qquad\text{for all }x\in X.
        \end{equation*}
\end{enumerate}
For Hilbert spaces, the weak convergence coincides with the one defined in \cref{chap:funktan_hilbert}. Strong convergence implies weak convergence; if $X$ is reflexive, weak and weak-$*$ convergence in $X=X^{**}$ coincide. As for strong convergence, we define weak(-$*$) continuity of mappings as well as weak(-$*$) closedness of sets.

The main reason for considering weak convergence is that bounded sequences in infinite-dimensional spaces in general do not contain strongly convergent subsequences, but weakly(-$*$) convergent ones.
\begin{theorem}[Eberlein--\u{S}mulyan, Banach--Alaoglu]\label{thm:banachal}\index{Eberlein--Smulyan theorem@Eberlein--\u{S}mulyan theorem}\index{Banach--Alaoglu theorem}
    Let $X$ be a Banach space and $\{x_n\}_{n\in\N}\subset X$ be bounded.
    \begin{enumerate}
        \item If $X$ is reflexive, there exists a weakly convergent subsequence $\{x_{n_k}\}_{k\in\N}$.
        \item If $X=Z^*$ for a separable Banach space $Z$ (i.e., $Z$ contains a countable dense subset), there exists a weakly-$*$ convergent subsequence $\{x_{n_k}\}_{k\in\N}$.
    \end{enumerate}
\end{theorem}
Since $L^1(\Omega)$ and $C(\overline\Omega)$ are separable, case (ii) applies in particular to $L^\infty(\Omega) = L^1(\Omega)^*$ and to the space $\mathcal{M}(\Omega)=C(\overline\Omega)^*$ of Radon measures.

Weak closedness is a stronger property than closedness, since it has to hold for more sequences. For convex sets, however, both concepts coincide.
\begin{lemma}\label{lem:convex_closed}
    A convex set $U\subset X$ is closed if and only if it is weakly closed.
\end{lemma}
\begin{proof}
    Weakly closed sets are always closed since strongly convergent sequences are also weakly convergent.
    Let now $U\subset X$ be convex, closed, and nonempty (otherwise there is nothing to show), and let $\{x_n\}_{n\in\N}\subset U$ with $x_n\wkto x\in X$. Assume that $x\notin U$. Then $U$ and $\{x\}$ satisfy the assumptions of \cref{thm:hb_separation}, and we obtain an $x^*\in X^*$ and a $\lambda\in\R$ with
    \begin{equation*}
        \dual{x^*,x_n}_X \leq \lambda < \dual{x^*,x}_X\qquad \text{for all } n\in\N.
    \end{equation*}
    Passing to the limit $n\to\infty$ in the first inequality yields the contradiction $\dual{x^*,x}_X < \dual{x^*,x}_X$.
\end{proof}

Finally, for any $T\in \linop(X,Y)$, the \term[operator!adjoint]{adjoint operator} $T^*\in \linop(Y^*,X^*)$ is defined via
\begin{equation*}
    \dual{T^*y^*,x}_X = \dual{y^*,Tx}_Y\qquad\text{for all } x\in X, y^*\in Y^*.
\end{equation*}
(If $X$ and $Y$ are Hilbert spaces, this coincides with the definition in \cref{chap:funktan_hilbert} after identifying $X^*$ and $Y^*$ with $X$ and $Y$, respectively.) It again holds that $\norm{T^*}_{\linop(Y^*,X^*)} = \norm{T}_{\linop(X,Y)}$. Furthermore, this definition directly implies that every $T\in\linop(X,Y)$ is weakly continuous, i.e., $x_n\wkto x$ implies $Tx_n\wkto Tx$.

\section{Calculus of variations}

We now consider the existence of minimizers of functionals $F:X\to\Rbar := \R\cup\{\infty\}$, where the value $\infty$ allows including constraints. We extend the usual arithmetic on $\R$ to $\Rbar$ by letting $t<\infty$ and $t+\infty = \infty$ for all $t\in\R$.
The set on which $F$ is finite is called the \term[domain!effective]{effective domain}
\begin{equation*}
    \dom F := \setof{x\in X}{F(x)<\infty}.
\end{equation*}
If $\dom F \neq \emptyset$, the functional $F$ is called \term[functional!proper]{proper}.
We call $F$
\begin{itemize}
    \item \term[semicontinuous!lower]{lower semicontinuous} if $F(x) \leq \liminf_{n\to\infty} F(x_n)$ for every sequence $\{x_n\}_{n\in\N}\subset X$ with $x_n\to x\in X$;
    \item \term[semicontinuous!weakly lower]{weakly(-$*$) lower semicontinuous} if the same holds for every weakly(-$*$) convergent sequence;
    \item \term{coercive} if $\norm{x_n}_X\to\infty$ implies $F(x_n)\to\infty$.
\end{itemize}
The following central existence result is proved by the so-called \term[method!direct]{direct method of the calculus of variations}. We have already used this method in \cref{thm:tikh_nl:existenz} for the special case of the Tikhonov functional \eqref{eq:tikh_nl:funktional} in Hilbert spaces; the following theorem shows that the same three steps -- choosing a minimizing sequence, extracting a weakly convergent subsequence, and passing to the limit using weak lower semicontinuity -- apply in a much more general setting.
\begin{theorem}[direct method]\label{thm:variation:existenz}
    Let $X$ be a reflexive Banach space and $F:X\to\Rbar$ be proper, coercive, and weakly lower semicontinuous. Then there exists a minimizer $\bar x\in \dom F$ of $F$.
\end{theorem}
\begin{proof}
    Since $F$ is proper, $M:=\inf_{x\in X} F(x)<\infty$ (where $M=-\infty$ is not excluded yet). Hence there exists a \term[sequence, minimizing]{minimizing sequence} $\{x_n\}_{n\in\N}\subset X$ with $F(x_n)\to M$. This sequence is bounded, since otherwise the coercivity of $F$ would imply that $F(x_n)\to\infty$ for a subsequence, in contradiction to $M<\infty$. By \cref{thm:banachal}, there thus exists a subsequence $\{x_{n_k}\}_{k\in\N}$ with $x_{n_k}\wkto \bar x\in X$. The weak lower semicontinuity of $F$ now implies that
    \begin{equation*}
        \inf_{x\in X} F(x) \leq F(\bar x) \leq \liminf_{k\to\infty} F(x_{n_k}) = M = \inf_{x\in X} F(x)<\infty,
    \end{equation*}
    i.e., $\bar x\in \dom F$ is a minimizer of $F$.
\end{proof}
If $X$ is not reflexive but the dual of a separable Banach space, the same argument yields a minimizer if $F$ is weakly-$*$ lower semicontinuous.

Weak lower semicontinuity is preserved under the following operations.
\begin{lemma}\label{lem:variation:wlsc}
    Let $X$ and $Y$ be Banach spaces and $F:X\to\Rbar$ be weakly lower semicontinuous. Then the following functionals are weakly lower semicontinuous as well:
    \begin{enumerate}
        \item $\alpha F$ for all $\alpha \geq 0$;
        \item $F+G$ for $G:X\to\Rbar$ weakly lower semicontinuous;
        \item $\phi\circ F$ for $\phi:\Rbar\to\Rbar$ lower semicontinuous and increasing;
        \item $F\circ \Phi$ for $\Phi:Y\to X$ weakly continuous;
        \item $x\mapsto \sup_{i\in I} F_i(x)$ with $F_i:X\to\Rbar$ weakly lower semicontinuous for an arbitrary set $I$.
    \end{enumerate}
\end{lemma}
Since $x\mapsto \dual{x^*,x}_X$ is weakly continuous for every $x^*\in X^*$ by definition, \eqref{eq:funktan:norm_dual} together with \cref{lem:variation:wlsc}\,(v) implies that the norm $\norm{\cdot}_X$ is weakly lower semicontinuous as well. By \cref{lem:variation:wlsc}\,(iii) and (iv), the same holds for $x\mapsto \norm{Tx-y}_Y^p$ for any $T\in \linop(X,Y)$, $y\in Y$, and $1\leq p<\infty$.

Constraints $x\in U$ for a set $U\subset X$ can be included via the \term[function!indicator]{indicator function}
\begin{equation*}
    \delta_U:X\to\Rbar,\qquad\delta_U(x) = \begin{cases} 0 & x\in U,\\ \infty& x \in X\setminus U.\end{cases}
\end{equation*}
Clearly, $\delta_U$ is proper if $U$ is nonempty. If $U$ is convex and closed and hence weakly closed by \cref{lem:convex_closed}, then $\delta_U$ is weakly lower semicontinuous: If $x_n\wkto x$ and $x\notin U$, then $x_n\notin U$ for all $n$ large enough (since otherwise a subsequence would converge weakly to $x\in U$), and hence $\delta_U(x)=\infty = \lim_{n\to\infty}\delta_U(x_n)$; for $x\in U$ the claimed inequality holds trivially.

\section{Convex analysis}

A proper functional $F:X\to\Rbar$ is called \term[convex!functional]{convex} if
\begin{equation*}
    F(\lambda x + (1-\lambda)y)\leq \lambda F(x) + (1-\lambda)F(y)\qquad\text{for all }x,y\in X \text{ and } \lambda\in [0,1]
\end{equation*}
(where the value $\infty$ is allowed on both sides). If for $x\neq y$ and $\lambda\in(0,1)$ the inequality is even strict, $F$ is called \term[convex!strictly]{strictly convex}.
Directly from the definition, affine functionals $x\mapsto \dual{x^*,x}_X-\alpha$, the norm $\norm{\cdot}_X$, and the indicator function $\delta_U$ of a convex set $U$ are convex. If $X$ is a Hilbert space, $x\mapsto\norm{x}_X^2$ is even strictly convex.
Further convex functionals can be obtained through the following operations.
\begin{lemma}\label{lem:konvex:fkt}
    Let $X$ and $Y$ be normed vector spaces and let $F:X\to\Rbar$ be convex. Then the following functionals are convex as well:
    \begin{enumerate}
        \item $\alpha F$ for all $\alpha \geq 0$;
        \item $F+G$ for $G:X\to\Rbar$ convex (strictly if $F$ \emph{or} $G$ is strictly convex);
        \item $\phi\circ F$ for $\phi:\Rbar\to\Rbar$ convex and increasing;
        \item $F\circ T$ for $T:Y\to X$ linear;
        \item $x\mapsto \sup_{i\in I} F_i(x)$ with $F_i:X\to\Rbar$ convex for an arbitrary set $I$.
    \end{enumerate}
\end{lemma}

For convex functionals, lower semicontinuity and weak lower semicontinuity coincide.
\begin{lemma}\label{lem:konvex:lsc}
    Let $F:X\to\Rbar$ be convex. Then $F$ is weakly lower semicontinuous if and only if $F$ is lower semicontinuous.
\end{lemma}
\begin{proof}
    It follows directly from the definitions that $F$ is (weakly) lower semicontinuous if and only if all sublevel sets $F_t:=\setof{x\in X}{F(x)\leq t}$, $t\in \R$, are (weakly) closed. Since $F$ is convex, all $F_t$ are convex, and the claim follows from \cref{lem:convex_closed}.
\end{proof}
Together with \cref{thm:variation:existenz}, we thus obtain the following existence result.
\begin{theorem}\label{thm:konvex:existenz}
    Let $X$ be a reflexive Banach space, $U\subset X$ be nonempty, convex, and closed, and $F:X\to\Rbar$ be proper, convex, lower semicontinuous, and coercive with $\dom F\cap U\neq \emptyset$. Then there exists a minimizer $\bar x\in U\cap\dom F$ of $F$ over $U$. If $F$ is strictly convex, the minimizer is unique.
\end{theorem}
\begin{proof}
    By assumption and \cref{lem:konvex:lsc}, the functional $F+\delta_U$ is proper, coercive (since $\delta_U\geq 0$), and weakly lower semicontinuous; hence \cref{thm:variation:existenz} yields a minimizer $\bar x\in\dom (F+\delta_U) = U\cap \dom F$.

    If $\bar x\neq \bar x'$ are two minimizers, then $\frac12(\bar x+\bar x')\in U$ by convexity of $U$, and the strict convexity of $F$ implies that
    \begin{equation*}
        F(\tfrac12(\bar x+\bar x')) < \tfrac12F(\bar x) + \tfrac12F(\bar x') = F(\bar x),
    \end{equation*}
    in contradiction to the minimizing property of $\bar x$.
\end{proof}

\bigskip

For a convex functional $F:X\to\Rbar$ and $x\in \dom F$, the \term[derivative!subdifferential]{(convex) subdifferential} is defined as
\begin{equation*}
    \partial F(x) :=  \setof{x^*\in X^*}{\dual{x^*,\tilde x - x}_X \leq F(\tilde x) - F(x)\quad\text{for all }\tilde x\in X};
\end{equation*}
for $x\notin\dom F$, we set $\partial F(x) = \emptyset$. An element $\xi\in \partial F(x)$ is called a \term[derivative!subderivative]{subderivative}. Note that $\partial F(x)$ can be empty even for $x\in\dom F$.

The subdifferential allows characterizing minimizers of (nonsmooth) convex functionals.
\begin{theorem}[Fermat principle]\label{thm:konvex:fermat}\index{Fermat principle}
    Let $F:X\to\Rbar$ be convex and $\bar x\in\dom F$. Then $\bar x$ is a minimizer of $F$ if and only if $0\in\partial F(\bar x)$.
\end{theorem}
\begin{proof}
    By definition of the subdifferential, $0\in\partial F(\bar x)$ holds if and only if
    \begin{equation*}
        0 = \dual{0,\tilde x-\bar x}_X \leq F(\tilde x)-F(\bar x)\qquad\text{for all }\tilde x\in X,
    \end{equation*}
    i.e., if and only if $F(\bar x)\leq F(\tilde x)$ for all $\tilde x\in X$.
\end{proof}

The subdifferential generalizes the classical derivative.
\begin{theorem}
    Let $F:X\to\Rbar$ be convex and Gateaux differentiable in $x$. Then $\partial F(x) = \{DF(x)\}$.
\end{theorem}

We collect some canonical examples of subdifferentials of convex (in particular nonsmooth) functionals.
\begin{example}\label{bsp:konvex:subdiff}~
    \begin{enumerate}
        \item If $X$ is a Hilbert space and $F(x)=\frac12\norm{x}_X^2$, then $\partial F(x) = \{x\}$ (after identifying $X^*$ with $X$).
        \item For the norm in a normed vector space $X$, we have
            \begin{equation*}
                \partial(\norm{\cdot}_X)(x) =
                \begin{cases}
                    \setof{x^*\in X^*}{\dual{x^*,x}_X = \norm{x}_X \text{ and } \norm{x^*}_{X^*} = 1} &\text{if } x\neq 0,\\
                    B_{X^*} &\text{if } x = 0.
                \end{cases}
            \end{equation*}
            In particular, for $X=\R$ we obtain
            \begin{equation*}
                \partial(|\cdot|)(t) = \sign(t) :=
                \begin{cases}
                    \{1\} & \text{if }t>0,\\
                    \{{-}1\} & \text{if }t<0,\\
                    [-1,1] & \text{if }t=0.
                \end{cases}
            \end{equation*}
        \item For the indicator function of a convex set $C\subset X$ and $x\in C$, we have the \term[cone, normal]{normal cone}
            \begin{equation*}
                \partial\delta_C(x) = \setof{x^*\in X^*}{\dual{x^*,\tilde x - x}_X \leq 0 \quad\text{for all }\tilde x \in C}.
            \end{equation*}
            In particular, for $X=\R$ and $C=[-1,1]$ we obtain
            \begin{equation*}
                \partial\delta_{[-1,1]}(t) = \begin{cases}
                    [0,\infty) & \text{if }t=1,\\
                    (-\infty,0] & \text{if }t=-1,\\
                    \{0\} & \text{if }t\in (-1,1),\\
                    \emptyset & \text{if }t \in \R\setminus[-1,1].
                \end{cases}
            \end{equation*}
        \item Let $\Omega\subset\R^d$ be open and bounded, $f:\R\to\Rbar$ be proper, convex, and lower semicontinuous, and $1\leq p<\infty$. Then the \term[functional!integral]{integral functional}
            \begin{equation*}
                F:L^p(\Omega)\to\Rbar,\qquad u\mapsto
                \begin{cases}
                    \int_\Omega f(u(x))\,dx &\text{if }f\circ u \in L^1(\Omega),\\
                    \infty &\text{else,}
                \end{cases}
            \end{equation*}
            is proper, convex, and lower semicontinuous, and we have for all $u\in\dom F$ and $q:=\frac{p}{p-1}$ that
            \begin{equation*}
                \partial F(u)  = \setof{u^*\in L^q(\Omega)}{u^*(x) \in\partial f(u(x))\quad\text{for almost every } x\in\Omega}.
            \end{equation*}
    \end{enumerate}
\end{example}

\bigskip

For a proper functional $F:X\to\Rbar$, the \term{Fenchel conjugate} is defined as
\begin{equation*}
    F^*:X^*\to \Rbar,\qquad F^*(x^*) = \sup_{x\in X}\ \dual{x^*,x}_X - F(x).
\end{equation*}
By \cref{lem:konvex:fkt}\,(v) and \cref{lem:variation:wlsc}\,(v), $F^*$ is always convex and lower semicontinuous. Similarly, we define the conjugate of $F:X^*\to\Rbar$ as $F^*:X\to\Rbar$, $F^*(x)=\sup_{x^*\in X^*} \dual{x^*,x}_X - F(x^*)$, so that the \term{biconjugate} $F^{**}:=(F^*)^*$ is again defined on $X$.
The definition directly yields the \term{Fenchel--Young inequality}
\begin{equation}\label{eq:konvex:fenchel-young}
    \dual{x^*,x}_X \leq F(x) + F^*(x^*)\qquad\text{for all }x\in X, x^*\in X^*.
\end{equation}
The following theorem shows that convex lower semicontinuous functionals are fully characterized by their conjugates.
\begin{theorem}[Fenchel--Moreau]\label{thm:konvex:moreau}\index{Fenchel--Moreau theorem}
    Let $F:X\to\Rbar$ be proper. Then $F^{**}= F$ if and only if $F$ is convex and lower semicontinuous.
\end{theorem}
\begin{example}\label{bsp:konvex:fenchel}
    \begin{enumerate}
        \item Let $X$ be a Hilbert space and $F(x) = \frac12\norm{x}_{X}^2$. After identifying $X^*$ with $X$, the supremum in the definition of $F^*(x^*)$ is attained in $x=x^*$, which can be seen by expanding $0\leq \frac12\norm{x-x^*}_X^2$. Hence, $F^*(x^*) = \frac12\norm{x^*}_{X}^2$.
        \item Let $X$ be a normed vector space. Then
            \begin{equation*}
                (\delta_{B_X})^* = \norm{\cdot}_{X^*}\qquad\text{and}\qquad (\norm{\cdot}_X)^* = \delta_{B_{X^*}},
            \end{equation*}
            and similarly $(\delta_{B_{X^*}})^* = \norm{\cdot}_X$ and $(\norm{\cdot}_{X^*})^* = \delta_{B_X}$.
        \item Let $F:L^p(\Omega)\to\Rbar$ with $1\leq p <\infty$ be an integral functional as in \cref{bsp:konvex:subdiff}\,(iv). Then we have for $q=\frac{p}{p-1}$ that
            \begin{equation*}
                F^*:L^q(\Omega) \to\Rbar, \qquad F^*(u^*) =
                \int_\Omega f^*(u^*(x))\,dx .
            \end{equation*}
    \end{enumerate}
\end{example}
The Fenchel conjugate is related to the subdifferential through the following characterization.
\begin{lemma}[Fenchel--Young]\label{lem:konvex:fenchel-young}\index{Fenchel--Young lemma}
    Let $F:X\to\Rbar$ be proper, convex, and lower semicontinuous. Then the following statements are equivalent for any $x\in X$ and $x^*\in X^*$:
    \begin{enumerate}
        \item $\dual{x^*,x}_X = F(x) + F^*(x^*)$;
        \item $x^*\in\partial F(x)$;
        \item $x\in\partial F^*(x^*)$.
    \end{enumerate}
\end{lemma}
\begin{proof}
    By definition of $F^*$, statement (i) is equivalent to
    \begin{equation*}
        \dual{x^*,x}_X - F(x) \geq  \dual{x^*,\tilde x}_X - F(\tilde x) \qquad\text{for all }\tilde x\in X
    \end{equation*}
    (since the reverse inequality holds by \eqref{eq:konvex:fenchel-young}), which in turn is equivalent to (ii) by definition of the subdifferential.
    Similarly, since $F=F^{**}$ by \cref{thm:konvex:moreau}, statement (i) is equivalent to
    \begin{equation*}
        \dual{x^*,x}_X - F^*(x^*) \geq  \dual{\tilde x^*,x}_X - F^*(\tilde x^*) \qquad\text{for all }\tilde x^*\in X^*,
    \end{equation*}
    which in turn is equivalent to (iii).
\end{proof}
(If $X$ is not reflexive, (iii) has to be understood as $Jx\in \partial F^*(x^*)\subset X^{**}$.)

\chapter{Tikhonov regularization}\label{chap:tikhonov_convex}

We now consider linear inverse problems $Tx=y$ for $T\in\linop(X,Y)$ between Banach spaces $X$ and $Y$. In this setting, Tikhonov regularization\index{Tikhonov functional!in Banach spaces} chooses $x_\alpha$ as a minimizer of
\begin{equation}\label{eq:tikh_cvx:funktional}
    J_{\alpha}(x;y) := \frac1p\norm{Tx-y}_Y^p + \alpha G(x)
\end{equation}
for $\alpha>0$, $1\leq p< \infty$, and a convex \term[functional!regularization]{regularization functional} $G:X\to\Rbar$. This generalizes the Tikhonov functional for linear and nonlinear problems in Hilbert spaces (for $p=2$ and $G(x) = \frac12\norm{x-x_0}_X^2$) and allows incorporating different a priori knowledge on $x$ (e.g., sparsity for $G(x) = \norm{x}_{\ell^1}$). Throughout this chapter, we assume that
\begin{enumerate}
    \item $X$ is reflexive;
    \item $G$ is proper, convex, lower semicontinuous, and coercive.
\end{enumerate}
Since the \term[Tikhonov functional!discrepancy term]{discrepancy term} $x\mapsto \frac1p\norm{Tx-y}_Y^p$ is convex and continuous by \cref{lem:konvex:fkt,lem:variation:wlsc}, the functional $J_\alpha(\cdot;y)$ is proper, convex, and lower semicontinuous for every $y\in Y$ and $\alpha>0$. Furthermore, $J_\alpha(\cdot;y)$ is coercive since $G$ is coercive and the discrepancy term is nonnegative. By \cref{thm:konvex:existenz}, it thus admits a (not necessarily unique) minimizer $x_\alpha\in \dom G$. Instead of (i), we can also assume that $X$ is the dual of a separable Banach space and $G$ is weakly-$*$ lower semicontinuous; in this case, \enquote{weakly} has to be replaced by \enquote{weakly-$*$} in the following.

\section{Stability and convergence}

We first consider stability with respect to $y$. For this, we use the following \enquote{$p$-triangle inequality}.
\begin{lemma}\label{lem:tikh_cvx:p-triangle}
    Let $X$ be a normed vector space and $1\leq p<\infty$. Then
    \begin{equation*}
        \norm{x+z}_X^p \leq 2^{p-1}\left(\norm{x}_X^p+\norm{z}_X^p\right)\qquad\text{for all }x,z\in X.
    \end{equation*}
\end{lemma}
\begin{proof}
    Since $t\mapsto t^p$ is convex and increasing on $[0,\infty)$, \cref{lem:konvex:fkt} implies that $x\mapsto \norm{x}_X^p$ is convex as well. In particular,
    \begin{equation*}
        \frac1{2^p} \norm{x+z}_X^p = \norm{\tfrac12 x+ \tfrac12 z}_X^p \leq \frac12\left(\norm{x}_X^p+\norm{z}_X^p\right),
    \end{equation*}
    and multiplying by $2^p$ yields the claim.
\end{proof}

\begin{theorem}\label{thm:tikh_cvx:stabil}
    Let $\alpha>0$, $\{y_n\}_{n\in\N}\subset Y$ with $y_n\to y\in Y$, and let $x_n$ be a minimizer of $J_\alpha(\cdot\,; y_n)$ for all $n\in \N$. Then $\{x_n\}_{n\in\N}$ contains a weakly convergent subsequence, and every weak accumulation point $\bar x$ is a minimizer of $J_{\alpha}(\cdot\,; y)$ with $G(x_{n_k})\to G(\bar x)$ along the corresponding subsequence.

    If the minimizer $x_\alpha$ of $J_\alpha(\cdot\,;y)$ is unique, then $x_n\wkto x_\alpha$ and $G(x_n)\to G(x_\alpha)$.
\end{theorem}
\begin{proof}
    By the minimizing property of $x_n$, we have for all $n\in \N$ and arbitrary $x\in \dom G$ that
    \begin{equation*}
        J_\alpha(x_n; y_n) \leq J_\alpha(x; y_n) = \frac1p\norm{Tx-y_n}_Y^p + \alpha G(x).
    \end{equation*}
    Applying \cref{lem:tikh_cvx:p-triangle} twice together with $1\leq 2^{p-1}$ then yields that
    \begin{equation*}
        \begin{aligned}
            J_\alpha(x_n;y) &\leq \frac1p 2^{p-1} \left(\norm{Tx_n-y_n}_Y^p + \norm{y_n-y}_Y^p\right)+\alpha G(x_n) \\
            &\leq 2^{p-1} J_\alpha(x_n;y_n) + 2^{p-1}\frac1p\norm{y_n-y}_Y^p\\
            &\leq 2^{p-1}  J_\alpha(x;y_n) + 2^{p-1}\frac1p\norm{y_n-y}_Y^p\\
            &\leq 2^{2(p-1)} J_\alpha(x;y) + 2^{2(p-1)+1}\frac1p\norm{y_n-y}_Y^p.
        \end{aligned}
    \end{equation*}
    Since $y_n\to y$, the right-hand side and hence also the left-hand side is bounded in $n\in\N$. The coercivity of $J_\alpha(\cdot\,;y)$ thus implies that $\{x_n\}_{n\in\N}$ is bounded, and \cref{thm:banachal} yields a weakly convergent subsequence.

    Let now $\{x_{n_k}\}_{k\in\N}$ be an arbitrary weakly convergent subsequence with limit $\bar x\in X$. Since $T$ is weakly continuous and $y_n\to y$, we have that $T x_{n_k}-y_{n_k}\wkto T\bar x-y$ (even strongly). The weak lower semicontinuity of the norm and of $G$ (by \cref{lem:konvex:lsc}) together with the minimizing property of $x_{n_k}$ then implies for arbitrary $x\in X$ that
    \begin{equation}\label{eq:tikh_cvx:stabil1}
        \begin{aligned}[t]
            J_\alpha(\bar x;y) &\leq \liminf_{k\to\infty}\left( \frac1p\norm{Tx_{n_k}-y_{n_k}}_Y^p+\alpha G(x_{n_k})\right)\\
            &\leq \limsup_{k\to\infty} \left(\frac1p\norm{Tx_{n_k}-y_{n_k}}_Y^p +\alpha G(x_{n_k})\right)\\
            &\leq \limsup_{k\to\infty} \left(\frac1p\norm{Tx-y_{n_k}}_Y^p +\alpha G(x)\right)\\
            &= \frac1p\norm{Tx-y}_Y^p +\alpha G(x) = J_{\alpha}(x;y).
        \end{aligned}
    \end{equation}
    Hence $\bar x\in \dom G$ is a minimizer of $J_\alpha(\cdot\,;y)$. Choosing $x = \bar x$ in \eqref{eq:tikh_cvx:stabil1} further shows that
    \begin{equation*}
        J_\alpha(\bar x;y) = \lim_{k\to\infty} \left(\frac1p\norm{Tx_{n_k}-y_{n_k}}_Y^p +\alpha G(x_{n_k})\right).
    \end{equation*}
    Together with the calculus for convergent sequences, we obtain from this that
    \begin{equation*}
        \begin{aligned}
            \limsup_{k\to \infty} \alpha G(x_{n_k}) &\leq \limsup_{k\to\infty}\left(\frac1p\norm{Tx_{n_k}-y_{n_k}}_Y^p +\alpha G(x_{n_k})\right) - \liminf_{k\to\infty} \frac1p\norm{Tx_{n_k}-y_{n_k}}_Y^p\\
            &\leq \frac1p\norm{T\bar x-y}_Y^p +\alpha G(\bar x)-\frac1p\norm{T\bar x-y}_Y^p  = \alpha G(\bar x).
        \end{aligned}
    \end{equation*}
    The weak lower semicontinuity of $G$ now implies that
    \begin{equation*}
        G(\bar x) \leq \liminf_{k\to\infty} G(x_{n_k})\leq \limsup_{k\to\infty} G(x_{n_k}) \leq G(\bar x),
    \end{equation*}
    and hence that $G(\bar x)= \lim_{k\to\infty}G(x_{n_k})$.

    Finally, if $J_\alpha(\cdot\,; y)$ has a unique minimizer, every weakly convergent subsequence has the same limit, and a subsequence-subsequence argument shows that the whole sequence converges.
\end{proof}

\bigskip

We now consider the convergence for $\delta\to 0$, where we again denote by $x_\alpha^\delta$ a minimizer of $J_\alpha(\cdot\,;y^\delta)$ for noisy data $y^\delta\in B_\delta(y)$.
Since $T$ need not be injective, the equation $Tx=y$ for $y\in\calR(T)$ can have more than one solution. In analogy to the $x_0$-minimum norm solution, we call $x^\dag\in X$ a \term[solution!G-minimizing@$G$-minimizing]{$G$-minimizing solution} if $Tx^\dag = y$ and
\begin{equation*}
    G(x^\dag) = \min\setof{G(x)}{x\in X,\ Tx=y}.
\end{equation*}
Since $T$ is linear and continuous, the set $\setof{x\in X}{Tx=y}$ is convex and closed, and \cref{thm:konvex:existenz} yields the existence of a $G$-minimizing solution if this set intersects $\dom G$. (This need not be unique if $G$ is not strictly convex.)
\begin{theorem}\label{thm:tikh_cvx:konvergenz}
    Let $y\in \calR(T)$ and $\{y^{\delta_n}\}_{n\in\N} \subset Y$ with $y^{\delta_n}\in B_{\delta_n}(y)$ and $\delta_n \to 0$. If $\alpha_n := \alpha(\delta_n)$ satisfies
    \begin{equation*}
        \lim_{n\to\infty} \alpha_n = 0\qquad\text{and}\qquad\lim_{n\to\infty} \frac{\delta_n^{p}}{\alpha_n} =0,
    \end{equation*}
    then every sequence $\{x_{\alpha_n}^{\delta_n}\}_{n\in\N}$ of minimizers of $J_{\alpha_n}(\cdot\,;y^{\delta_n})$ contains a weakly convergent subsequence, and every weak accumulation point $\bar x$ is a $G$-minimizing solution of $Tx=y$ with $G(x_{\alpha_{n_k}}^{\delta_{n_k}})\to G(\bar x)$ along the corresponding subsequence.

    If the $G$-minimizing solution $x^\dag$ is unique, then $x_{\alpha_n}^{\delta_n}\wkto x^\dag$ and $G(x_{\alpha_n}^{\delta_n})\to G(x^\dag)$.
\end{theorem}
\begin{proof}
    Let $x^\dag$ be a $G$-minimizing solution and set $x_n:=x_{\alpha_n}^{\delta_n}$. The minimizing property of $x_n$ then implies that
    \begin{equation}\label{eq:tikh_cvx:konvergenz_min}
        J_{\alpha_n}(x_n;y^{\delta_n})\leq \frac1p\norm{Tx^\dag -y^{\delta_n}}_Y^p + \alpha_n G(x^\dag) \leq \frac{\delta_n^p}p + \alpha_n G(x^\dag).
    \end{equation}
    Together with \cref{lem:tikh_cvx:p-triangle}, this yields that
    \begin{equation}\label{eq:tikh_cvx:konvergenz_T}
        \begin{aligned}[t]
            \frac1p\norm{Tx_n-y}_Y^p \leq 2^{p-1}\left(\frac1p\norm{Tx_n-y^{\delta_n}}_Y^p +\frac{\delta_n^p}p\right) &\leq 2^{p-1}J_{\alpha_n}(x_n;y^{\delta_n}) + 2^{p-1}\frac{\delta_n^p}p \\
            &\leq 2^{p-1}\alpha_n G(x^\dag) + 2^{p}\frac{\delta_n^p}p \to 0
        \end{aligned}
    \end{equation}
    as well as
    \begin{equation}\label{eq:tikh_cvx:konvergenz_G}
        G(x_n) \leq  \frac1{\alpha_n}J_{\alpha_n}(x_n;y^{\delta_n}) \leq \frac{\delta_n^p}{p\,\alpha_n} +  G(x^\dag)\to G(x^\dag).
    \end{equation}
    Hence $\{G(x_n)\}_{n\in\N}$ is bounded, and the coercivity of $G$ implies that $\{x_n\}_{n\in\N}$ is bounded as well. By \cref{thm:banachal}, there thus exists a weakly convergent subsequence.

    Let now $\{x_{n_k}\}_{k\in\N}$ be an arbitrary weakly convergent subsequence with limit $\bar x\in X$. Then \eqref{eq:tikh_cvx:konvergenz_T} and the weak continuity of $T$ imply that $Tx_{n_k}-y \wkto T\bar x-y$ and $Tx_{n_k}-y\to 0$, i.e., $T\bar x=y$. Furthermore, \eqref{eq:tikh_cvx:konvergenz_G} and the weak lower semicontinuity of $G$ imply that
    \begin{equation*}
        G(\bar x) \leq \liminf_{k\to\infty} G(x_{n_k}) \leq \limsup_{k\to\infty} G(x_{n_k}) \leq G(x^\dag)\leq G(\bar x),
    \end{equation*}
    where the last inequality follows from the definition of $x^\dag$. Hence $\bar x$ is a $G$-minimizing solution of $Tx=y$ with $G(x_{n_k})\to G(\bar x)$.

    The convergence of the whole sequence for a unique $G$-minimizing solution again follows from a subsequence-subsequence argument.
\end{proof}

The parameter choice rules from \cref{def:parameterwahl} can also be applied here. By \cref{thm:tikh_cvx:konvergenz}, any a priori choice rule $\alpha(\delta) = c\delta^{p-\eps}$ for $c>0$ and $\eps\in(0,p)$ leads to convergence. For the discrepancy principle\index{discrepancy principle!for Tikhonov regularization in Banach spaces}\index{parameter choice rule!discrepancy principle!for Tikhonov regularization in Banach spaces} \eqref{eq:parameter:morozov}, one usually chooses for given $1\leq\tau_1\leq\tau_2$ the parameter $\alpha=\alpha(\delta,y^\delta)$ such that
\begin{equation}\label{eq:tikh_cvx:morozov}
    \tau_1\delta \leq \norm{Tx_\alpha^\delta-y^\delta}_Y \leq \tau_2 \delta,
\end{equation}
since the residual need not depend continuously on $\alpha$ if the minimizers are not unique. (If $\alpha\mapsto \norm{Tx_\alpha^\delta-y^\delta}_Y$ is continuous, which is the case if the minimizers of $J_\alpha(\cdot\,;y^\delta)$ are unique, then such an $\alpha$ exists under mild conditions even for $\tau_1=\tau_2$; see \cite[Theorem~3.13]{ItoJin}.) The lower bound in \eqref{eq:tikh_cvx:morozov} is the key to the following convergence result, which does not require any additional condition on $\alpha(\delta,y^\delta)$.
\begin{theorem}\label{thm:tikh_cvx:morozov}
    Let $y\in \calR(T)$ and $\{y^{\delta_n}\}_{n\in\N} \subset Y$ with $y^{\delta_n}\in B_{\delta_n}(y)$ and $\delta_n \to 0$, and let $\alpha_n := \alpha(\delta_n,y^{\delta_n})$ be chosen according to \eqref{eq:tikh_cvx:morozov}. Then every sequence $\{x_{\alpha_n}^{\delta_n}\}_{n\in\N}$ of minimizers of $J_{\alpha_n}(\cdot\,;y^{\delta_n})$ contains a weakly convergent subsequence, and every weak accumulation point $\bar x$ is a $G$-minimizing solution of $Tx=y$ with $G(x_{\alpha_{n_k}}^{\delta_{n_k}})\to G(\bar x)$ along the corresponding subsequence.

    If the $G$-minimizing solution $x^\dag$ is unique, then $x_{\alpha_n}^{\delta_n}\wkto x^\dag$ and $G(x_{\alpha_n}^{\delta_n})\to G(x^\dag)$.
\end{theorem}
\begin{proof}
    Let $x^\dag$ be a $G$-minimizing solution and set $x_n:=x_{\alpha_n}^{\delta_n}$. The minimizing property of $x_n$ again yields \eqref{eq:tikh_cvx:konvergenz_min}, which together with the lower bound in \eqref{eq:tikh_cvx:morozov} implies that
    \begin{equation*}
        \frac{(\tau_1\delta_n)^p}{p} + \alpha_n G(x_n) \leq \frac1p\norm{Tx_n-y^{\delta_n}}_Y^p + \alpha_n G(x_n) \leq \frac{\delta_n^p}{p} + \alpha_n G(x^\dag),
    \end{equation*}
    and hence $G(x_n)\leq G(x^\dag)$ due to $\tau_1\geq 1$. Furthermore, the upper bound in \eqref{eq:tikh_cvx:morozov} implies that
    \begin{equation}\label{eq:tikh_cvx:morozov1}
        \norm{Tx_n-y}_Y \leq \norm{Tx_n-y^{\delta_n}}_Y + \norm{y^{\delta_n}-y}_Y \leq (1+\tau_2)\delta_n \to 0.
    \end{equation}
    These two estimates take the place of \eqref{eq:tikh_cvx:konvergenz_G} and \eqref{eq:tikh_cvx:konvergenz_T}, respectively, and the claim follows exactly as in the proof of \cref{thm:tikh_cvx:konvergenz}.
\end{proof}

Finally, heuristic choice rules such as the Hanke--Raus rule, which chooses $\alpha$ as a minimizer of $\alpha\mapsto \alpha^{-1}\norm{Tx_\alpha^\delta-y^\delta}_Y^p$, are again subject to the Bakushinski\u{\i} veto (\cref{thm:parameter:bakushinskii}) and can only lead to convergence under additional assumptions on the noise.

\section{Convergence rates}

We now derive error estimates for $x_\alpha^\delta$. Since \cref{thm:tikh_cvx:konvergenz} only shows weak convergence of $x_\alpha^\delta$ and convergence of $G(x_\alpha^\delta)$, we cannot expect estimates in the norm of $X$. Instead, we measure the error using the \term{Bregman distance}
\begin{equation*}
    D_\xi(x,\bar x) := G(x)-G(\bar x) - \dual{\xi,x-\bar x}_X \qquad\text{for }\xi\in \partial G(\bar x).
\end{equation*}
By definition of the subdifferential, $D_\xi(x,\bar x) \geq 0$ for all $x,\bar x\in X$, and obviously $D_\xi(\bar x,\bar x)=0$. However, the Bregman distance is not symmetric, $D_\xi(x,\bar x)=0$ does not imply $x=\bar x$ in general, and it is only defined if $\partial G(\bar x) \neq\emptyset$; it is therefore not a metric. (For $G(x) = \frac12\norm{x-x_0}_X^2$ in a Hilbert space $X$, however, we have $\partial G(\bar x) = \{\bar x-x_0\}$ by \cref{bsp:konvex:subdiff}\,(i) and hence $D_{\bar x-x_0}(x,\bar x) = \frac12\norm{x-\bar x}_X^2$.)

As in \cref{chap:tikhonov-nl}, we impose a \term[source condition!variational]{variational source condition} on a $G$-minimizing solution $x^\dag$: There exist $\xi\in\partial G(x^\dag)$, $\beta_1\in [0,1)$, and $\beta_2\geq 0$ such that
\begin{equation}\label{eq:tikh_cvx:var_source}
    \dual{\xi,x^\dag-x}_X\leq \beta_1 D_\xi(x,x^\dag) + \beta_2 \norm{Tx-Tx^\dag}_Y\qquad\text{for all }x\in X.
\end{equation}
For $G(x) = \frac12\norm{x-x_0}_X^2$ in Hilbert spaces, this coincides with \eqref{eq:tikh_nl:var_source} for $F=T$.
As in the Hilbert space setting, the variational source condition follows from a classical source condition\index{source condition!in Banach spaces}.
\begin{lemma}\label{lem:tikh_cvx:quell}
    If there exists a $w\in Y^*$ with
    \begin{equation}\label{eq:tikh_cvx:source}
        x^\dag \in \partial G^*(T^*w),
    \end{equation}
    then the variational source condition \eqref{eq:tikh_cvx:var_source} holds for $\xi = T^*w$.
\end{lemma}
\begin{proof}
    By \cref{lem:konvex:fenchel-young}, \eqref{eq:tikh_cvx:source} is equivalent to $\xi:=T^*w\in \partial G(x^\dag)$. The definition of the adjoint and \eqref{eq:funktan:cs_banach} then immediately imply that
    \begin{equation*}
        \dual{\xi,x^\dag-x}_X = \dual{w,Tx^\dag-Tx}_Y \leq \norm{w}_{Y^*}\norm{Tx-Tx^\dag}_Y,
    \end{equation*}
    i.e., \eqref{eq:tikh_cvx:var_source} with $\beta_1=0$ and $\beta_2 =\norm{w}_{Y^*}$.
\end{proof}
For $G(x)=\frac12\norm{x}_X^2$ in Hilbert spaces, we have $G^*=G$ by \cref{bsp:konvex:fenchel}\,(i), and \eqref{eq:tikh_cvx:source} reduces to the classical source condition $x^\dag = T^*w \in \calR(T^*)$. The source condition \eqref{eq:tikh_cvx:source} can thus again be interpreted as an abstract smoothness condition.

\begin{theorem}\label{thm:tikh_cvx:rate}
    Let $y\in \calR(T)$, $y^\delta\in B_\delta(y)$, and let $x^\dag$ be a $G$-minimizing solution satisfying the variational source condition \eqref{eq:tikh_cvx:var_source}.
    \begin{enumerate}
        \item If $p=1$ and $\alpha(\delta)$ is a parameter choice rule with
            \begin{equation*}
                c \delta^{\eps} \leq \alpha(\delta) \leq C\delta^{\eps}\qquad\text{for }C>c>0 \text{ and }\eps\in (0,1),
            \end{equation*}
            then
            \begin{equation*}
                D_\xi (x_{\alpha(\delta)}^\delta,x^\dag) = \mathcal{O}(\delta^{1-\eps})\qquad\text{and}\qquad
                \norm{Tx_{\alpha(\delta)}^\delta-y}_Y = \mathcal{O}(\delta).
            \end{equation*}
        \item If $p>1$ and $\alpha(\delta)$ is a parameter choice rule with
            \begin{equation*}
                c \delta^{p-1} \leq \alpha(\delta) \leq C\delta^{p-1}\qquad\text{for }C>c>0,
            \end{equation*}
            then
            \begin{equation*}
                D_\xi (x_{\alpha(\delta)}^\delta,x^\dag) = \mathcal{O}(\delta)\qquad\text{and}\qquad
                \norm{Tx_{\alpha(\delta)}^\delta-y}_Y = \mathcal{O}(\delta).
            \end{equation*}
    \end{enumerate}
\end{theorem}
\begin{proof}
    Set $\alpha:=\alpha(\delta)$. The minimizing property of $x_\alpha^\delta$ and $y^\delta\in B_\delta(y)$ again imply that
    \begin{equation*}
        \frac1p\norm{Tx_\alpha^\delta - y^\delta}_Y^p + \alpha G(x_\alpha^\delta) \leq \frac{\delta^p}{p} + \alpha G(x^\dag).
    \end{equation*}
    Using the definition of the Bregman distance, the variational source condition, and the triangle inequality, we obtain from this that
    \begin{equation}\label{eq:tikh_cvx:rate1}
        \begin{aligned}[t]
            \frac1p\norm{Tx_\alpha^\delta - y^\delta}_Y^p + \alpha D_\xi(x_\alpha^\delta,x^\dag)
            &\leq \frac{\delta^p}{p} + \alpha \dual{\xi,x^\dag-x_\alpha^\delta}_X\\
            &\leq \frac{\delta^p}{p} + \alpha\left(\beta_1 D_\xi(x_\alpha^\delta,x^\dag) + \beta_2 \norm{Tx_\alpha^\delta-y}_Y\right)\\
            &\leq \frac{\delta^p}{p} + \alpha\left(\beta_1 D_\xi(x_\alpha^\delta,x^\dag) + \beta_2 \norm{Tx_\alpha^\delta-y^\delta}_Y+ \beta_2 \delta\right).
        \end{aligned}
    \end{equation}
    We now distinguish the two cases.
    \begin{enumerate}
        \item For $p=1$, we can rearrange \eqref{eq:tikh_cvx:rate1} to
            \begin{equation}\label{eq:tikh_cvx:rate2}
                (1-\alpha\beta_2)\norm{Tx_\alpha^\delta - y^\delta}_Y + \alpha(1-\beta_1) D_\xi(x_\alpha^\delta,x^\dag) \leq \delta(1+\alpha\beta_2).
            \end{equation}
            If $\delta$ is sufficiently small that $C\delta^{\eps}\beta_2 < 1$, both terms on the left-hand side are nonnegative due to $\beta_1<1$. Dropping either term and using the parameter choice then yields
            \begin{align*}
                D_\xi(x_\alpha^\delta,x^\dag) &\leq \frac{\delta(1+\alpha\beta_2)}{\alpha(1-\beta_1)} \leq \delta^{1-\eps} \, \frac{1+C\delta^\eps \beta_2}{c(1-\beta_1)}
                \shortintertext{as well as}
                \norm{Tx_\alpha^\delta-y^\delta}_Y &\leq \frac{\delta(1+\alpha\beta_2)}{1-\alpha\beta_2}\leq \delta \, \frac{1+C\delta^\eps \beta_2}{1-C\delta^\eps \beta_2},
            \end{align*}
            where the fractions on the right-hand side are bounded for $\delta\to0$.

        \item For $p>1$, we first apply the generalized Young inequality $ab \leq \frac{\eps^p}p a^p + \frac1{\eps^{p'}p'} b^{p'}$ for $p'=\frac{p}{p-1}$ and $\eps\in(0,1)$ to $a=\norm{Tx_\alpha^\delta-y^\delta}_Y$ and $b= \alpha\beta_2$. Inserting this into \eqref{eq:tikh_cvx:rate1} and rearranging yields
            \begin{equation*}
                \frac{1-\eps^p}{p} \norm{Tx_\alpha^\delta-y^\delta}_Y^p +  \alpha(1-\beta_1) D_\xi(x_\alpha^\delta,x^\dag) \leq
                \frac{\delta^p}{p} +\alpha\beta_2 \delta +  \frac{1}{\eps^{p'}p'}(\alpha\beta_2)^{p'}.
            \end{equation*}
            Again both terms on the left-hand side are nonnegative, and we obtain using the parameter choice and $(\delta^{p-1})^{p'}=\delta^p$ that
            \begin{equation*}
                D_\xi(x_\alpha^\delta,x^\dag) \leq \frac{\frac1p\delta^p+\alpha\beta_2\delta+\frac{1}{\eps^{p'}p'}(\alpha\beta_2)^{p'}}{\alpha(1-\beta_1)} \leq\delta \,\frac{\frac1p+C\beta_2 + \frac{(C\beta_2)^{p'}}{\eps^{p'}p'}}{c(1-\beta_1)}
            \end{equation*}
            as well as
            \begin{equation*}
                \norm{Tx_\alpha^\delta-y^\delta}_Y \leq
                \left(\frac{p\left(\frac1p\delta^p+\alpha\beta_2\delta+\frac{1}{\eps^{p'}p'}(\alpha\beta_2)^{p'}\right)}{1-\eps^p}\right)^{\frac1p} \leq \delta  \left(\frac{1+pC\beta_2+(p-1)\frac{(C\beta_2)^{p'}}{\eps^{p'}}}{1-\eps^p}\right)^{\frac1p}.
            \end{equation*}
    \end{enumerate}
    This shows the claimed estimates for the Bregman distance; those for the residual follow from
    \begin{equation*}
        \norm{Tx_\alpha^\delta - y}_Y\leq \norm{Tx_\alpha^\delta - y^\delta}_Y + \delta.
        \qedhere
    \end{equation*}
\end{proof}

For $p=1$, the \enquote{$\eps$-gap} can be closed: In this case, $\alpha$ need not tend to zero at all.
\begin{corollary}
    Let $y\in \calR(T)$, $y^\delta\in B_\delta(y)$, and let $x^\dag$ be a $G$-minimizing solution satisfying the variational source condition \eqref{eq:tikh_cvx:var_source}. If $p=1$, then for all $0<\alpha < \beta_2^{-1}$ and $\delta\geq 0$,
    \begin{equation*}
        D_\xi(x_\alpha^\delta,x^\dag) \leq C_1(\alpha) \delta\qquad\text{and}\qquad
        \norm{Tx_{\alpha}^\delta-y}_Y \leq C_2(\alpha) \delta.
    \end{equation*}
    In particular, $T x_\alpha = y$ for all $0<\alpha<\beta_2^{-1}$.
\end{corollary}
\begin{proof}
    By assumption, the factor $1-\alpha\beta_2$ in \eqref{eq:tikh_cvx:rate2} is positive. Rearranging as in the proof of \cref{thm:tikh_cvx:rate} then yields the claimed estimates with
    \begin{equation*}
        C_1(\alpha) := \frac{1+\alpha\beta_2}{\alpha(1-\beta_1)},\qquad
        C_2(\alpha) := 1+\frac{1+\alpha\beta_2}{1-\alpha\beta_2}.
    \end{equation*}
    For $\delta=0$, this immediately implies the last claim.
\end{proof}
For $p=1$, the discrepancy term is therefore an \term[penalization, exact]{exact penalization} of the constraint $Tx=y$.

\bigskip

For the discrepancy principle, the variational source condition directly yields a convergence rate without requiring knowledge of $\beta_1,\beta_2$ or a case distinction in $p$.
\begin{theorem}\label{thm:tikh_cvx:morozov_rate}
    Let $y\in \calR(T)$, $y^\delta\in B_\delta(y)$, and let $x^\dag$ be a $G$-minimizing solution satisfying the variational source condition \eqref{eq:tikh_cvx:var_source}. If $\alpha=\alpha(\delta,y^\delta)$ is chosen according to \eqref{eq:tikh_cvx:morozov}, then
    \begin{equation*}
        D_\xi(x_\alpha^\delta,x^\dag) \leq \frac{\beta_2(1+\tau_2)}{1-\beta_1}\,\delta\qquad\text{and}\qquad
        \norm{Tx_\alpha^\delta-y}_Y \leq (1+\tau_2)\delta.
    \end{equation*}
\end{theorem}
\begin{proof}
    As in the proof of \cref{thm:tikh_cvx:morozov}, we obtain $G(x_\alpha^\delta)\leq G(x^\dag)$ and $\norm{Tx_\alpha^\delta-y}_Y\leq (1+\tau_2)\delta$. Together with the definition of the Bregman distance and the variational source condition, this implies that
    \begin{equation*}
        \begin{aligned}
            D_\xi(x_\alpha^\delta,x^\dag) = G(x_\alpha^\delta)-G(x^\dag) + \dual{\xi,x^\dag-x_\alpha^\delta}_X &\leq \dual{\xi,x^\dag-x_\alpha^\delta}_X\\
            &\leq \beta_1 D_\xi(x_\alpha^\delta,x^\dag) + \beta_2 \norm{Tx_\alpha^\delta-y}_Y\\
            &\leq \beta_1 D_\xi(x_\alpha^\delta,x^\dag) + \beta_2 (1+\tau_2)\delta.
        \end{aligned}
    \end{equation*}
    Since $\beta_1<1$, we can absorb the first term on the right-hand side into the left-hand side, which yields the claim.
\end{proof}
In particular, the discrepancy principle yields for all $1\leq p<\infty$ the same rate $\mathcal{O}(\delta)$ as the a priori choices in \cref{thm:tikh_cvx:rate} for $p>1$ and in the previous corollary for $p=1$, without requiring knowledge of $\beta_1$ and $\beta_2$.

\bigskip

For heuristic choice rules like the Hanke--Raus rule, convergence rates in the Bregman distance can again be shown under assumptions that sufficiently separate the noise from the range of $T$; see, e.g., \cite[Theorem 3.20]{ItoJin}.

\part{Statistical inverse problems}

\chapter{Frequentist inverse problems}\label{chap:frequentist}

Until now, we have treated the noisy data $y^\delta$ as an arbitrary element of the data space $Y$ of which we know nothing besides the noise level $\delta = \norm{y^\delta-y}_Y$.
In \term[inverse problem!statistical]{statistical inverse problems}, the noisy data are instead considered as a random variable with a (hopefully) known distribution. Rather than the worst-case regularization error over all data that are compatible with the noise level, one can then study the \emph{average} regularization error over \emph{all} data with respect to this distribution. (In contrast, the first approach is sometimes referred to as \term[inverse problem!deterministic]{deterministic inverse problems}.)

In statistics, this is also called \term{inference}, and one distinguishes two approaches:%
\footnote{This distinction is based on fundamental philosophical differences on the nature of probability: for \enquote{frequentists}, probability arises as the limit of relative frequencies of outcomes of random experiments as the number of repetitions tends to infinity (which may make sense for measurements but not for the sought-for exact solution), while \enquote{Bayesians} consider probability as a quantification of ignorance (which is also -- and especially -- relevant for the exact solution). But regardless of philosophical interpretation, these two approaches lead to practical differences in methodology.}
\begin{enumerate}
    \item \term[inference!frequentist]{frequentist inference}, where the exact (minimum norm) solution $x^\dag$ is treated as a fixed but unknown element, and
    \item \term[inference!Bayesian]{Bayesian inference}, where the exact solution is treated as a random variable with a given distribution as well.
\end{enumerate}

\bigskip

In this chapter, we will consider the first approach, using a formulation (but not notation) that follows statistical conventions. As in the previous chapters, we focus here on infinite-dimensional problems; this is referred to as \term[statistics!nonparametric]{nonparametric statistics}, in contrast to \term[statistics!parametric]{parametric} statistics, which studies models such as linear, polynomial, or exponential models that can be described by finitely many scalar parameters which are then estimated.

\section{Statistical noise model and estimators}

We first have to establish what we mean by random noise in infinite-dimensional spaces.
Let $X$ and $Y$ again be Hilbert spaces and $T\in \linop(X,Y)$ be a bounded linear operator. We assume that for given $x\in X$ we have measured
\begin{equation}\label{eq:frequentist:gleichung}
    y^\delta := Tx + \delta \xi,
\end{equation}
where $\xi$ is a random perturbation and $\delta$ is the noise level. We consider here \term[noise!Gaussian white]{Gaussian white noise}; such noise not only occurs in a wide range of practical applications but is also the simplest case.
(Alternatives would be, e.g., \term[noise!Laplace]{Laplace} or \term[noise!Poisson]{Poisson noise}.)
The fundamental difficulty is here that $\xi$ -- and hence $y^\delta$ in general\footnote{and, indeed, with probability $1$} -- is \emph{not} an element of $Y$. It is instead a \term[process, stochastic]{stochastic process}, which maps any $y\in Y$ linearly and continuously to a real random variable. (The noise $\xi$ is thus only defined by how it acts on the \enquote{exact} measurement $Tx$.)

The proper definition requires some concepts from probability theory, which we assume to be familiar and refer to standard textbooks like \cite{Klenke,Kallenberg} for a rigorous introduction.
Let $\Omega$ be a sample space, $\Sigma$ be a sigma-algebra on $\Omega$, and $\mu:\Sigma\to[0,1]$ be a probability measure. Then $L^2(\Omega; \mu)$ denotes the space of all random variables that are square-integrable with respect to $\mu$. For random variables $X,Y:\Omega\to \R$ on the probability space $(\Omega,\Sigma,\mu)$, the expectation and covariance are then defined, respectively, as 
\begin{equation*}
    \Exp[X] := \int_\Omega X\, d\mu,\qquad
    \Cov[X,Y] := \Exp\left[(X- \Exp[X])(Y-\Exp[Y])\right].
\end{equation*}
If -- as usual in functional analysis -- we write $\dual{\xi,y}_Y$ for the action of $\xi\in Y^*$ on $y\in Y$, we can define Gaussian white noise as follows.
\begin{definition}
    Let $(\Omega,\Sigma,\mu)$ be a probability space.
    A bounded linear operator $\xi:Y\to L^2(\Omega; \mu)$ is called a \term[noise!white noise process]{white noise process} (with expectation $0$ and covariance $\Id$) on $Y$ if
    \begin{enumerate}
        \item$\displaystyle\Exp[\dual{\xi,y}_Y] = 0$  for all $y\in Y$;
        \item$\displaystyle\Cov[\dual{\xi,y_1}_Y,\dual{\xi,y_2}_Y] = \inner{y_1}{y_2}_Y$ for all $y_1,y_2\in Y$.
    \end{enumerate}
    A white noise process $\xi$ is called \term[noise!Gaussian white]{Gaussian} if for all $n\in \N$ and any pairwise distinct $y_1,\dots,y_n\in Y$, the scalar random variables $\dual{\xi,y_1}_Y,\dots,\dual{\xi,y_n}_Y$ are jointly normally distributed.
\end{definition}

If $K:X\to Y$ is a compact linear operator, we can take for $y\in Y$ in particular the sequence $\{u_n\}_{n\in\N}\subset Y$ of its singular vectors. In this case, the scalar random variables
\begin{equation*}
    \xi_n := \dual{\xi,u_n}_Y, \qquad\text{for all } n\in\N,
\end{equation*}
are independently and identically normally distributed with expectation $0$ and standard deviation $1$ since $\Exp[\xi_n] = 0$ and $\Cov[\xi_m,\xi_n] = \inner{u_n}{u_m}_Y = 1$ for $n=m$ and $0$ otherwise.
Similarly, we obtain from \eqref{eq:frequentist:gleichung} the random variables
\begin{equation*}
    y_n^\delta := \dual{y^\delta,u_n}_Y = \inner{Kx}{u_n}_Y + \delta\xi_n = \sigma_n \inner{x}{v_n}_X + \delta\xi_n\qquad\text{for all }n\in\N,
\end{equation*}
or equivalently since $\sigma_n>0$,
\begin{equation}\label{eq:frequentist:folgenraum}
    x_n^\delta := \sigma_n^{-1} y_n^\delta = \inner{x}{v_n}_X + \frac{\delta}{\sigma_n} \xi_n \qquad\text{for all } n\in\N.
\end{equation}
This is referred to in statistics as the \term[model, sequence space]{sequence space model}; in this formulation, inference entails  estimating $x_n := \inner{x}{v_n}_X$ given $x_n^\delta$. Since $\sigma_n\to 0$ as $n\to \infty$, this is of course not possible in a stable way for all $n\in \N$. However, for $\delta = 0$ we recover from \eqref{eq:frequentist:folgenraum} with
\begin{equation*}
    \sum_{n\in\N} x_n v_n = \sum_{n\in\N} \sigma_n^{-1} y_n v_n = \sum_{n\in\N} \sigma_n^{-1} \inner{y}{u_n}_Y v_n = x^\dag
\end{equation*}
the minimum norm solution; compare \cref{thm:inverse:picard}. Hence $x_n = \inner{x}{v_n}_X = \inner{x^\dag}{v_n}_X$ for every solution $x\in X$ to $Kx=y$.

\bigskip

Similarly to \cref{chap:regularization}, for $\delta>0$ we instead define a \term[estimator!linear]{linear estimator}
\begin{equation*}
    x^\delta_\gamma := \sum_{n\in\N} \gamma_n x_n^\delta v_n
\end{equation*}
for a suitable sequence $\{\gamma_n\}_{n\in\N} \subset [0,\infty)$.

Clearly, every filter $\{\phi_\alpha\}_{\alpha>0}$ together with a parameter choice rule $\alpha$ defines a linear estimator via
\begin{equation}\label{eq:frequentist:filter}
    \gamma_n := \phi_{\alpha(\delta,y^\delta)}(\sigma_n^2)\sigma_n^2\qquad \text{for all }n\in\N.
\end{equation}
(Hence estimators correspond to regularization \emph{methods} rather than operators in deterministic inverse problems.)
In particular,
\begin{enumerate}
    \item truncated singular value decomposition corresponds to $\gamma_n = \begin{cases} 1 &\text{if } \sigma_n^2\geq \alpha(\delta,y^\delta),\\ 0 &\text{otherwise;}\end{cases}$
    \item Tikhonov regularization corresponds to $\gamma_n = \frac{\sigma_n^2}{\sigma_n^2+\alpha(\delta,y^\delta)}$;
    \item Landweber regularization corresponds to
        $\gamma_n = 1-(1-\omega\sigma_n^2)^{1/\alpha(\delta,y^\delta)}$.
\end{enumerate}
(For the practical implementation, of course, one would use the alternative formulations from the corresponding chapters instead of the sequence space model.)

However, filters are not the only way of defining linear estimators. Of special theoretical (albeit not practical) importance is the \term[estimator!Pinsker]{Pinsker estimator}, defined by setting
\begin{equation*}
    \gamma_n := \max\{0,1-\kappa_\delta a_n\}
\end{equation*}
for a monotonically increasing sequence $\{a_n\}_{n\in\N}\subset (0,\infty)$ and the solution $\kappa_\delta>0$ to
\begin{equation}\label{eq:frequentist:pinsker}
    \kappa_\delta \rho^2- \delta^2 \sum_{n\in\N} \frac{a_n}{\sigma_n^2} \max\{0, 1-\kappa_\delta a_n\} = 0
\end{equation}
for a further constant $\rho>0$.
Here, the parameter $\kappa_\delta$ corresponds to the regularization parameter; the Pinsker estimator thus includes a specific (a priori) parameter choice strategy.
This parameter can be more explicitly given as
\begin{equation}\label{eq:frequentist:pinsker_parameter}
    \kappa_\delta = \frac{\sum_{n=1}^{N_\delta} \sigma_n^{-2}a_n}{\frac{\rho^2}{{\delta}^2} + \sum_{n=1}^{N_\delta} \sigma_n^{-2}a_n^2}
    \quad\text{for}\quad
    N_\delta := \max\setof{N\in\N}{\delta^2 \sum_{n=1}^N \sigma_n^{-2} a_n(a_N-a_n)\leq \rho^2}.
\end{equation}

\section{Risk and order optimality}

Since the noise and therefore the data are random, the corresponding estimator is a random variable as well. Instead of the worst-case regularization error \eqref{eq:regularisierung:paramwahl}, we now define for a (not necessarily linear) estimator $x(y^\delta)$ the \term{risk}
\begin{equation}
    R(x(y^\delta),x^\dag) := \Exp[\norm{x(y^\delta)-x^\dag}_X^2],
\end{equation}
where the expectation is taken over all possible perturbations $\xi$. (Note that the risk is based on the \emph{squared} regularization error.)
As for deterministic noise, the analysis is based on the fundamental decomposition of the regularization error.
\begin{theorem}\label{thm:frequentist:gesamtfehler}
    Let $x^\delta_\gamma$ be a linear estimator for the sequence space model \eqref{eq:frequentist:folgenraum}. Then
    \begin{equation}\label{eq:frequentist:gesamtfehler}
        R(x^\delta_\gamma,x^\dag) = \sum_{n\in\N}\left( (1-\gamma_n)^ 2 x_n^2 + \delta^2 \frac{\gamma_n^2}{\sigma_n^2}\right).
    \end{equation}
\end{theorem}
\begin{proof}
    First, the definition of $x^\delta_\gamma$ and the characterization of $x^\dag$ from \cref{thm:inverse:picard} imply that
    \begin{equation*}
        \norm{x^\delta_\gamma - x^\dag}_X^2 = \sum_{n\in\N} (\gamma_n x_n^\delta - x_n)^2.
    \end{equation*}
    We now insert the productive zero $\gamma_n x_n - \gamma_n x_n$, apply the definition \eqref{eq:frequentist:folgenraum}, expand the square, and use the linearity of the expectation to obtain
    \begin{equation*}
        \begin{aligned}
            \Exp[\norm{x^\delta_\gamma-x^\dag}_X^2]
            &= \sum_{n\in\N} \Exp\left[\left((\gamma_nx_n^\delta - \gamma_n x_n) + (\gamma_n x_n - x_n)\right)^2\right] \\
            &= \sum_{n\in\N} \Exp\left[\left( \gamma_n \frac\delta{\sigma_n}\xi_n+(\gamma_n-1)x_n \right)^2\right] \\
            &= \sum_{n\in\N} \left((1-\gamma_n)^2 x_n^2 + 2 (\gamma_n-1)x_n \cdot \gamma_n \frac{\delta}{\sigma_n}\Exp[\xi_n] +
            \delta^2 \frac{\gamma_n^2}{\sigma_n^2} \Exp[\xi_n^2]\right).
        \end{aligned}
    \end{equation*}
    Since the $\xi_n$ by assumption are normally distributed with expectation $0$ and variance $1$, we have $\Exp[\xi_n] = 0$ and $\Exp[\xi_n^2] = \Cov[\xi_n,\xi_n]=1$, which yields the claim.
\end{proof}
Comparing \eqref{eq:frequentist:gesamtfehler} to \eqref{eq:regularisierung:splitting}, the first term corresponds exactly to the approximation error, here called \term{bias}, while the second term corresponds to the propagated data error, here called \term{variance}.
As in spectral regularization, our job is now to choose the weights $\gamma_n$ (corresponding to the parameter choice for $\alpha$) such that these two terms are optimally balanced.
For example, if the estimator is defined as in \eqref{eq:frequentist:filter} via a filter $\{\phi_\alpha\}_{\alpha>0}$ and we again set $r_\alpha(\lambda) = 1-\lambda\phi_\alpha(\lambda)$, then
\begin{equation*}
    R(x^\delta_\gamma,x^\dag) = \sum_{n\in\N} \left(r_{\alpha(\delta,y^\delta)}(\sigma_n^2)^2 x_n^2 + \delta^2  \phi_{\alpha(\delta,y^\delta)}(\sigma_n^2)^2\sigma_n^2\right)
\end{equation*}
(where in contrast to the worst-case error we even have equality).
Here the first (purely deterministic) term can be estimated exactly as in spectral regularization. However, the second term has to be treated differently: we cannot as in \cref{lem:spektral:beschraenkt} simply use that $\sigma_n^2\phi_\alpha(\sigma_n^2)\leq C_\phi^2$ 
and then bound the remaining term via Bessel's inequality by $\norm{y-y^\delta}_Y^2\leq \delta^2$ -- this is exactly the price that we have to pay for considering stochastic noise $\xi\notin Y$.

\bigskip

As for deterministic inverse problems, one is now interested in bounding the risk under a source condition. Similarly to \cref{lem:regularisierung:quell}, it is straightforward to show that for the sequence space model \eqref{eq:frequentist:folgenraum}, the source condition $x\in X_{\nu,\rho}$ holds for some $\nu,\rho>0$ if and only if
\begin{equation}
    \sum_{n\in\N} \sigma_n^{-2\nu} x_n^2 \leq \rho^2.
\end{equation}
(Of course one could consider -- also for deterministic problems -- more general weights $a_n\to\infty$ instead of $\sigma_n^{-\nu}$; one then refers to $X_{\nu,\rho}$ as a \term[ellipsoid, coefficient]{coefficient ellipsoid}.)
The consequence of the different variance term in \cref{thm:frequentist:gesamtfehler} is now that for stochastic inverse problems, the singular values of $K$ directly influence the convergence rate. We show this for the example of truncated singular value decomposition for moderately ill-posed problems.
\begin{theorem}\label{thm:frequentist:tsvd}
    Let $K\in \calK(X,Y)$ have singular values satisfying
    \begin{equation*}
        c_\mu n^{-\mu} \leq \sigma_n \leq C_\mu n^{-\mu}\qquad\text{for all }n\in\N
    \end{equation*}
    for some $\mu>0$ and $C_\mu>c_\mu>0$, and let $x^\dag\in X_{\nu,\rho}$ for $\nu,\rho>0$. If $x_\gamma^\delta$ is a linear estimator defined through
    \begin{equation}\label{eq:frequentist:tsvd_apriori}
        \gamma_n = \begin{cases} 1 & \text{if } n\leq N(\delta),\\ 0 &\text{otherwise},\end{cases}
        \qquad \text{with}\qquad
        c_N\delta^{\frac{-2}{2\mu(\nu+1)+1}} \leq N(\delta) \leq C_N\delta^{\frac{-2}{2\mu(\nu+1)+1}}
    \end{equation}
    for some $C_N>c_N>0$, then there exists a $C>0$ such that
    \begin{equation}\label{eq:frequentist:tsvd_rate}
        R(x_\gamma^\delta,x^\dag) \leq C \delta^\frac{4\mu\nu}{2\mu(\nu+1)+1} \qquad\text{as }\delta\to 0.
    \end{equation}
\end{theorem}
\begin{proof}
    Applying \cref{thm:frequentist:gesamtfehler} to this choice of $\gamma_n$, we obtain with $N:=N(\delta)$
    \begin{equation*}
        R(x_\gamma^\delta, x^\dag)
        = \sum_{n = N+1}^\infty x_n^2 + \delta^2 \sum_{n=1}^N \sigma_n^{-2}.
    \end{equation*}
    For the first term, we can use the source condition together with the assumption on the singular values to estimate
    \begin{equation*}
        \sum_{n = N+1}^\infty x_n^2 = \sum_{n = N+1}^\infty \sigma_n^{-2\nu} x_n^2\sigma_n^{2\nu} \leq \sigma_N^{2\nu} \rho^2 \leq C_\mu^{2\nu} N^{-2\mu\nu}\rho^2.
    \end{equation*}
    For the second term, we can similarly use the assumption on the singular values to obtain%
    \footnote{For given $N\in\N$, this is of course a rather generous upper bound. However, \term{Faulhaber's formula} states that $\sum_{n=1}^N n^r = \frac1{r+1}N^{r+1} + \mathcal{O}(N^r)$ as $N\to \infty$ for any $r>{-1}$; this formula can be derived via, e.g., the Euler--Maclaurin formula (and for $r\in\N$ directly via induction and the binomial formula).}
    \begin{equation*}
        \delta^2 \sum_{n=1}^N \sigma_n^{-2} \leq \delta^2 c_\mu^{-2} \sum_{n=1}^N n^{2\mu} \leq \delta^2 c_\mu^{-2} N\cdot N^{2\mu} = c_\mu^{-2} \delta^2 N^{2\mu+1}.
    \end{equation*}
    Inserting the parameter choice rule \eqref{eq:frequentist:tsvd_apriori} in both estimates and simplifying then yields \eqref{eq:frequentist:tsvd_rate}.
\end{proof}
Compare this to the \enquote{deterministic} order-optimal rate $\delta^{\frac{2\nu}{\nu+1}} = \delta^{\frac{4\mu\nu}{2\mu(\nu+1)}}$ from \cref{thm:spektral:apriori,ex:spectral:tsvd} of the squared(!) worst-case error for the truncated singular value decomposition: the \enquote{statistical} rate \eqref{eq:frequentist:tsvd_rate} is not only slower but also explicitly depends on the decay rate of the singular values.
Since the rate -- and hence the parameter choice rule -- therefore depends even more on information that is usually not available in practice, in statistics one is less interested in the concrete rate for an estimator and instead defines order optimality directly as the smallest possible risk over a given class of estimators. Hence an estimator $x(y^\delta)$ is called a \term[estimator!linear minimax]{linear minimax estimator} for $X_{\nu,\rho}$ if
\begin{equation*}
    \sup_{x^\dag\in X_{\nu,\rho}} R(x(y^\delta),x^\dag) = \inf_{\gamma}\sup_{x^\dag\in X_{\nu,\rho}} R(x_\gamma^\delta,x^\dag),
\end{equation*}
where the infimum is taken over all sequences $\{\gamma_n\}_{n\in\N}\subset [0,\infty)$.
If we only have that
\begin{equation*}
    \sup_{x^\dag\in X_{\nu,\rho}} R(x(y^\delta),x^\dag) \leq C \inf_{\gamma}\sup_{x^\dag\in X_{\nu,\rho}} R(x_\gamma^\delta,x^\dag)
\end{equation*}
for some constant $C>0$, then $x(y^\delta)$ is said to attain the \term[rate!minimax]{(linear) minimax rate}.

Surprisingly, it is even possible to explicitly determine the linear minimax estimator.
\begin{theorem}\label{thm:frequentist:pinsker}
    The Pinsker estimator is a linear minimax estimator. Specifically, for all $\nu,\rho>0$ it holds that
    \begin{equation*}
        \inf_{\gamma}\sup_{x^\dag\in X_{\nu,\rho}} R(x_\gamma^\delta,x^\dag)
        = \delta^2\sum_{n\in\N} \sigma_n^{-2}\max\{0,1-\kappa_\delta \sigma_n^{-\nu}\}
    \end{equation*}
    for the solution $\kappa_\delta$ to \eqref{eq:frequentist:pinsker} with the given $\rho$ and $a_n := \sigma_n^{-\nu}$,
    where the infimum and supremum are attained, respectively, for
    \begin{align*}
        \bar \gamma_n &= \max\{0,1-\kappa_\delta \sigma_n^{-\nu}\},\\
        \bar x_n^2 &=  \frac{\delta^2}{\kappa_\delta}\sigma_n^{-2+\nu}\max\{0,1-\kappa_\delta\sigma_n^{-\nu}\}.
    \end{align*}
\end{theorem}
\begin{proof}
    By \cref{thm:frequentist:gesamtfehler}, any $x^\dag$ and any linear estimator $x_\gamma^\delta$ satisfy
    \begin{equation*}
        R(x_\gamma^\delta,x^\dag) = \sum_{n\in\N}\left((1-\gamma_n)^2x_n^2 + \delta^2 \frac{\gamma_n^2}{\sigma_n^2}\right).
    \end{equation*}
    We now first take the infimum over all admissible $\gamma_n$ for fixed $x^\dag$, which we can do term by term due to their nonnegativity. Each term is of the form $f(t) = a (1-t)^2 + bt^2$ with $a,b>0$, for which we find the minimum $\min_t f(t) = \frac{ab}{a+b}$ by straightforward calculus.
    Hence
    \begin{equation*}
        \inf_{\gamma} R(x_\gamma^\delta,x^\dag) = \sum_{n\in\N} \frac{\delta^2\sigma_n^{-2} x_n^2}{x_n^2 + \delta^2\sigma_n^{-2}}.
    \end{equation*}
    Furthermore, we have for $0<\bar\gamma_n = 1-\kappa_\delta \sigma_n^{-\nu}$ and hence for $\bar x_n^2 = \frac{\delta^2}{\kappa_\delta}\sigma_n^{-2+\nu}\bar\gamma_n$ that
    \begin{equation*}
        \frac{\bar x_n^2}{\bar x_n^2 + \delta^2 \sigma_n^{-2}} =
        \frac{\bar \gamma_n}{\bar \gamma_n + \kappa_\delta \sigma_n^{-\nu}}
        = 1 - \kappa_\delta \sigma_n^{-\nu} = \bar \gamma_n.
    \end{equation*}
    (If $\bar\gamma_n = 0$, then $\bar x_n=0$ as well and hence this relation holds trivially.)
    Finally, \eqref{eq:frequentist:pinsker} yields
    \begin{equation*}
        \sum_{n\in\N} \sigma_n^{-2\nu} \bar x_n^2 = \frac{\delta^2}{\kappa_\delta} \sum_{n\in\N} \sigma_n^{-\nu-2} \max\{0,1-\kappa_\delta \sigma_n^{-\nu}\} = \rho^2
    \end{equation*}
    and hence $\bar x := \sum_{n \in \N} \bar x_n v_n \in X_{\nu,\rho}$. Since $\inf\sup \geq \sup\inf$ in general, we thus obtain that
    \begin{equation}\label{eq:frequentist:pinsker:minimax1}
        \begin{aligned}[t]
            \inf_{\gamma}\sup_{x^\dag\in X_{\nu,\rho}} R(x_\gamma^\delta,x^\dag)
            &\geq
            \sup_{x^\dag\in X_{\nu,\rho}} \inf_{\gamma}R(x_\gamma^\delta,x^\dag)
            =
            \sup_{x^\dag\in X_{\nu,\rho}} \sum_{n\in\N}  \frac{\delta^2 \sigma_n^{-2} x_n^2}{x_n^2 + \delta^2\sigma_n^{-2}} \\
            &\geq
            \sum_{n\in\N}  \frac{\delta^2\sigma_n^{-2} \bar x_n^2}{\bar x_n^2 + \delta^2\sigma_n^{-2}}
            =
            \delta^2  \sum_{n\in\N} \sigma_n^{-2} \bar\gamma_n.
        \end{aligned}
    \end{equation}

    For the converse inequality, we use that for all $x^\dag \in X_{\nu,\rho}$,
    \begin{equation*}
        \sum_{n\in\N} (1-\bar\gamma_n)^2 x_n^{2}
        \leq \sup_{n\in \N}\{ (1-\bar\gamma_n)^2\sigma_n^{2\nu}\} \sum_{n\in\N} \sigma_n^{-2\nu} x_n^2
        \leq \sup_{n\in \N}\{ (1-\bar\gamma_n)^2\sigma_n^{2\nu}\} \rho^2.
    \end{equation*}
    We again follow the case distinction in the definition of $\gamma_n$ to obtain
    \begin{equation*}
        (1-\bar\gamma_n)^2 \sigma_n^{2\nu}
        =
        \begin{cases}
            \sigma_n^{2\nu} &\text{if }  1-\kappa_\delta \sigma_n^{-\nu} \leq 0,\\
            \kappa_\delta^2 &\text{if }  1-\kappa_\delta \sigma_n^{-\nu} >0,
        \end{cases}
    \end{equation*}
    where both cases can be estimated from above by $\kappa_\delta^2$. Hence \eqref{eq:frequentist:pinsker} implies that
    \begin{equation}\label{eq:frequentist:pinsker:minimax2}
        \begin{aligned}[t]
            \inf_{\gamma}\sup_{x^\dag\in X_{\nu,\rho}} R(x_\gamma^\delta,x^\dag)
            &\leq
            \sup_{x^\dag\in X_{\nu,\rho}} \sum_{n\in\N}\left((1-\bar \gamma_n)^2x_n^2 + \delta^2 \frac{\bar \gamma_n^2}{\sigma_n^2}\right)
            \leq \kappa_\delta^2\rho^2 + \delta^2\sum_{n\in \N} \sigma_n^{-2}\bar\gamma_n^2\\
            &= \delta^2\sum_{n\in\N}\sigma_n^{-2} (\kappa_\delta \sigma_{n}^{-\nu}\bar\gamma_n + \bar\gamma_n^2)
            =
            \delta^2  \sum_{n\in\N} \sigma_n^{-2} \bar\gamma_n,
        \end{aligned}
    \end{equation}
    where the final equality is once more obtained by case distinction for $\gamma_n$.
\end{proof}
In fact, the KKT conditions for the linear minimax estimator (considered as a convex minimization problem for $\gamma_n$ under the constraint $x^\dag\in X_{\nu,\rho}$) can be used to derive \eqref{eq:frequentist:pinsker} as well as its given solution.
With significantly more effort, one can further show that the Pinsker estimator even attains the minimax rate over \emph{all} estimators; see \cite{BelitserLevit}.

The Pinsker estimator can now serve as a benchmark for other linear estimators. For example, the fact that the singular values form a monotonically decreasing null sequence implies that $\bar\gamma_n = 0$ for all $n>N$ with $\sigma_{N+1}^\nu \leq \kappa_\delta < \sigma_N^\nu$, similarly to the truncated singular value decomposition.
In fact, the optimality of the Pinsker estimator fundamentally relies on the optimality of this parameter choice.
\begin{corollary}
    If $\nu,\rho>0$ and $x_{\hat\gamma}^\delta$ is a linear estimator with
    \begin{equation}\label{eq:frequentist:tsvd_pinsker}
        \hat\gamma_n = \begin{cases} 1 & \text{if } n\leq N(\delta),\\ 0 &\text{otherwise},\end{cases}
        \qquad \text{with}\qquad
        \sigma_{n}^\nu < 2\kappa_\delta \quad\text{for all } n>N(\delta)
    \end{equation}
    for the solution $\kappa_\delta$ to \eqref{eq:frequentist:pinsker} with the given $\rho$ and $a_n := \sigma_n^{-\nu}$, then
    \begin{equation*}
        \sup_{x^\dag\in X_{\nu,\rho}} R(x_{\hat\gamma}^\delta,x^\dag)\leq 4 \inf_{\gamma}\sup_{x^\dag\in X_{\nu,\rho}} R(x_{\gamma}^\delta,x^\dag).
    \end{equation*}
\end{corollary}
\begin{proof}
    We have shown in the proof of \cref{thm:frequentist:tsvd} that
    \begin{equation*}
        \sup_{x^\dag\in X_{\nu,\rho}} R(x_{\hat\gamma}^\delta,x^\dag)\leq
        \sigma_{N(\delta)}^{2\nu} \rho^2 + \delta^2\sum_{n=1}^{N(\delta)} \sigma_n^{-2}.
    \end{equation*}
    By definition, $\bar \gamma_n = 1 - \kappa_\delta \sigma_n^{-\nu}\geq \frac12$ for all $n\leq N(\delta)$, and hence
    \begin{equation*}
        \sigma_{N(\delta)}^{2\nu} \rho^2 + \delta^2\sum_{n=1}^{N(\delta)} \sigma_n^{-2}
        \leq
        (2\kappa_\delta)^2\rho^2 + \delta^2\sum_{n\in \N} \sigma_n^{-2}(2\bar\gamma_n)^2.
    \end{equation*}
    Furthermore, \cref{thm:frequentist:pinsker} implies that all inequalities in \eqref{eq:frequentist:pinsker:minimax1} and \eqref{eq:frequentist:pinsker:minimax2} hold with equality and hence that
    \begin{equation*}
        \kappa_\delta^2\rho^2 + \delta^2\sum_{n\in \N} \sigma_n^{-2}\bar\gamma_n^2
        =
        \inf_{\gamma}\sup_{x^\dag\in X_{\nu,\rho}} R(x_{\gamma}^\delta,x^\dag),
    \end{equation*}
    which together with the previous inequalities yields the claim.
\end{proof}
The next step is now to use the ansatz $\sigma_{N(\delta)}^\nu \approx \kappa_\delta$ in \eqref{eq:frequentist:pinsker} together with the assumption that $\sigma_n \approx n^{-\mu}$ to derive an estimate of $N(\delta) \approx \delta^{\frac{-2}{2\mu(\nu+1)+1}}$ and thus show that the truncated singular value decomposition attains the linear minimax rate with the a priori choice \eqref{eq:frequentist:tsvd_apriori} as well.
We will not do this here and merely make a plausibility estimate (by simply ignoring constants and lower-order terms):
\begin{equation*}
    N(\delta)^{-\mu\nu} \approx \sigma_{N(\delta)}^{\nu} \approx \kappa_\delta \approx \delta^2 \sum_{n\in \N} \sigma_n^{-2-\nu}\bar \gamma_n
    \approx \delta^2\sum_{n=1}^{N(\delta)} n^{2\mu+\mu\nu} \approx \delta^2 N(\delta)^{2\mu+\mu\nu+1},
\end{equation*}
where in the last step we have used the same estimate as in the proof of \cref{thm:frequentist:tsvd} (which is asymptotically tight due to Faulhaber's formula).
Solving for $N(\delta)$ now gives the desired rate in $\delta$.
In general, one can show that any regularizing filter together with the appropriate a priori choice rule  defines an estimator that attains the linear minimax rate up to the qualification of the filter; see \cite{BissantzHohageMunkRuymgaart:2007}.

\chapter{Bayesian inverse problems}

We conclude with a very brief outlook on the Bayesian approach to inverse problems; a good introduction to these topics can be found in \cite[Part I]{Sanz-Alonso:2023}. As mentioned in the last chapter, the general idea is to not consider a fixed exact (minimum norm) solution $x^\dag$ and instead only assume that such a solution follows a certain probability distribution -- according to the tenet that \enquote{randomness is lack of information}.
An alternative interpretation is that we replace the \enquote{hard} source condition $x^\dag\in X_\nu$ with a probability distribution for the true solution, just like in the last chapter we have replaced the hard assumption that $y^\delta\in B_\delta(Kx^\dag)$ with a distribution for the noise.

\bigskip

We again need some concepts from probability theory. Let $(\Omega,\Sigma,\mu)$ be a probability space. For technical reasons, which we only address briefly here, we will assume that $\Omega\subset\R^N$ for some (possibly very large) $N\in \N$; we can thus fix $\Sigma$ as the Borel algebra and $\mu$ as the scaled Lebesgue measure on $\R^N$. If $X:\Omega\to\R$ is a random variable, then we can define a new probability measure on $\R$, endowed with the Borel algebra, by setting
\begin{equation*}
    \mu_X(A) := \mu(X^{-1}(A))\qquad\text{for all Borel measurable } A\subset \R,
\end{equation*}
i.e., the probability that $X\in A$. We then say that $X$ \term[distribution]{is distributed according to} $\mu_X$ and write $X\sim \mu_X$. Furthermore, if there exists a $\rho_X\in L^1(\R)$ such that
\begin{equation*}
    \mu_X(A) = \int_A \rho_X(x)\,dx \qquad\text{for all Borel measurable } A\subset \R,
\end{equation*}
then $\rho_X$ is called the \term{density} of $\mu_X$ (with respect to the Lebesgue measure). (If $\Omega\subset X$ for an infinite-dimensional Hilbert space $X$, there need not exist a reference measure like the Lebesgue measure with respect to which one can define densities in general.)

Let now $X,Y:\Omega\to\R$ be random variables. Their \term[distribution!joint]{joint distribution} is then given by
\begin{equation*}
    \mu_{X,Y}(A\times B) := \mu(X^{-1}(A)\cap Y^{-1}(B)) \qquad\text{for all Borel measurable } A,B\subset \R.
\end{equation*}
If
\begin{equation*}
    \mu_{X,Y}(A\times B) = \mu_X(A)\mu_Y(B)  \qquad\text{for all Borel measurable } A,B\subset \R,
\end{equation*}
then $X$ and $Y$ are called \term[random variables, independent]{independent}; in this case, $\mu_{X,Y}$ has density $\rho_{X,Y}(x,y) = \rho_{X}(x)\rho_Y(y)$. This allows defining random variables taking values in $\R^N$. Conversely, if we have a joint distribution $\mu_{X,Y}$ with density $\rho_{X,Y}$, we can extract the \term[distribution!marginal]{marginal}
\begin{equation*}
    \mu_{X}(A) = \mu_{X,Y}(A\times \R) \qquad\text{for all Borel measurable } A\subset \R.
\end{equation*}
Similarly, we define the \term[distribution!conditional probability]{conditional probability distribution} (for $X\in A$ given that $Y\in B$) as
\begin{equation*}
    \mu_{X|Y\in B}(A) := \frac{\mu_{X,Y}(A\times B)}{\mu_Y(B)}\qquad\text{for all Borel measurable } A,B\subset \R.
\end{equation*}
This definition can be extended to singleton sets of the form $B=\{y_0\}$ for some $y_0\in \R$ (which have Lebesgue measure zero) via so-called \term[probability, regular conditional]{regular conditional probabilities}. In this case, we write $\mu_{X|y_0}$ for the \emph{conditional probability distribution of $X$ given $y_0$}.

The \term{Bayes theorem} now allows characterizing these conditional probability distributions through the joint distribution and the marginal for $X$. (Using the right definition of all occurring objects, its proof reduces to a simple application of Fubini's theorem.)
There exist several different versions; we here give one in terms of densities that will be used in the following.
\begin{theorem}[Bayes]\label{thm:bayes}
    Let $X:\Omega\to \R^N$ and $Y:\Omega\to\R^M$ be random variables. If $\mu_X$ has density $\rho_X$, $\mu_Y$ has density $\rho_Y$, and $\mu_{Y|x}$ has for $\mu_X$-almost every $x\in \R^N$ density $\rho_{Y|x}$, then $\mu_{X|y}$ has for $\mu_Y$-almost every $y\in \R^M$ density
    \begin{equation*}
        \rho_{X|y}(x) = \frac{\rho_{Y|x}(y)\rho_X(x)}{\rho_Y(y)}\qquad\text{for $\mu_X$-almost every }x\in \R^N
    \end{equation*}
    with
    \begin{equation*}
        \rho_Y(y) = \int_{\R^N} \rho_{Y|x}(y)\,d\mu_X = \int_{\R^N} \rho_{Y|x}(y)\rho_X(x) \,dx.
    \end{equation*}
\end{theorem}

\bigskip

We will now apply this theorem to inverse problems of the form $Tx=y$ for $T\in\linop(\R^N,\R^M)$.
To this end, we have to choose
\begin{itemize}
    \item for the unknown exact solution $x^\dag\in\R^N$ a \term[distribution!prior]{prior distribution} $\mu_X$, based on prior knowledge we have independent of any measurement, and
    \item for the measurement $y^\delta\in \R^M$ a \term{likelihood} $\mu_{Y|x}$, which quantifies how likely any measurement $y^\delta$ is given some hypothetical exact data $y:=Tx\in \R^M$ for given $x\in\R^N$.
\end{itemize}
The \term[distribution!posterior]{posterior distribution} $\mu_{X|y^\delta}$ then tells us how consistent any $x\in\R^N$ is with both our prior knowledge and the actual measurement $y^\delta$ and thus quantifies all remaining uncertainty about the true solution after the measurement. The \term{evidence} $\rho_Y(y^\delta)$ is a normalization constant that ensures that the posterior is indeed a probability distribution (i.e., is nonnegative and has total mass $1$).

The central questions are now:
\begin{enumerate}[label=\arabic*.]
    \item How should we choose prior distribution and likelihood in a problem-specific way?
    \item How can we extract practically useful information from the posterior distribution?
\end{enumerate}
In the following, we will only sketch the simplest approaches to these two questions, although it should be pointed out that this does not even begin to do justice to the strength of the Bayesian approach.

We start with modeling, where we assume -- both for simplicity and to make the connection to the results of the previous chapters -- that both the true solution and the measurement are normally distributed.
Recall that a random variable $X:\Omega\to\R$ is \term[distribution!normal]{normally distributed} with expectation $t_0$ and variance $\sigma^2$ if and only if it has density
\begin{equation*}
    \rho_X(t) = \frac{1}{\sqrt{2\pi\sigma^2}}\, e^{-\frac{1}{2\sigma^2} (t-t_0)^2}\qquad\text{for all }t\in \R.
\end{equation*}
We then write $X \sim \calN(t_0,\sigma^2)$. If we are given $N$ independently and identically normally distributed variables $X_i$, then $X:=(X_1,\dots,X_N)^T:\Omega\to\R^N$ is normally distributed as well with density (for $x_0:=(t_0,\dots,t_0)^T$)
\begin{equation*}
    \rho_X(x) = \prod_{i=1}^N\frac{1}{\sqrt{2\pi\sigma^2}} \, e^{-\frac{1}{2\sigma^2} (x_i-t_0)^2} = \frac{1}{\sqrt{(2\pi\sigma^2)^N}} \, e^{-\frac{1}{2\sigma^2} \norm{x-x_0}_2^2} \qquad\text{for all }x\in\R^N.
\end{equation*}
Generally, if $x_0\in\R^N$ and if $C\in\R^{N\times N}$ is selfadjoint and positive definite, we say that the random variable $X:\Omega\to\R^N$ is normally distributed with expectation $x_0$ and \term[matrix, covariance]{covariance matrix} $C$, and write $X\sim \calN(x_0,C)$, if it has density
\begin{equation}\label{eq:bayes:gaussian}
    \rho_X(x) =  \frac{1}{\sqrt{(2\pi)^N\det C}} \, e^{-\frac{1}{2} \inner{C^{-1}(x-x_0)}{(x-x_0)}_2} \qquad\text{for all }x\in\R^N.
\end{equation}
(If $K$ is injective, then $x\sim \calN(0,C)$ for $C=|K|^{-2\nu}=(K^*K)^{-\nu}$ may be considered as the Bayesian equivalent to the source condition $x\in X_\nu$.)

We now assume for prior distribution and likelihood, respectively, that
\begin{itemize}
    \item $\mu_X = \calN(0,\sigma^2\Id)$ for some $\sigma>0$;
    \item $\mu_{Y|x} = \calN(Tx,\delta^2 \Id)$ for some $\delta>0$.
\end{itemize}
(The latter corresponds exactly to the assumption on the noise in \cref{chap:frequentist}; since we here only consider finite-dimensional problems, we do not have to make the distinction between random variables and processes. In practice, of course, these assumptions should be based on careful statistical modeling of the problem at hand.)

By \cref{thm:bayes} and the calculus for exponential functions, the posterior distribution $\mu_{X|y^\delta}$ for a given measurement $y^\delta\in \R^M$ then has density
\begin{equation}\label{eq:bayes:posterior}
    \rho_{X|y^\delta}(x) =
    \frac{e^{-\frac{1}{2\delta^2}\norm{Tx-y^\delta}_2^2 - \frac{1}{2\sigma^2}\norm{x}_2^2}}{\int_{\R^N}e^{-\frac1{2\delta^2}\norm{Tx-y^\delta}_2^2 - \frac{1}{2\sigma^2}\norm{x}_2^2}\,dx} \qquad\text{for all }x\in\R^N.
\end{equation}
This density can now be used to make statements about which points $x\in \R^N$ have a particularly high probability (based on the given distributions) to give rise to this measurement; such points are called \term[estimator!point]{point estimators}.
The most common point estimator is the \term[estimator!maximum a posteriori]{maximum a posteriori (MAP) estimator}, which is the point that has (in a suitable sense) the maximal probability under the posterior distribution. In our setting, this is the (global) maximizer of the posterior density, i.e.,
\begin{equation*}
    \begin{aligned}[t]
        x_{\mathrm{MAP}} &:= \arg\max_{x\in\R^N} \rho_{X|y^\delta}(x) = \arg\min_{x\in\R^N} -\log \rho_{X|y^\delta}(x)\\
        &= \arg\min_{x\in\R^N} \frac{1}{2\delta^2}\norm{Tx-y^\delta}_2^2 +                              \frac{1}{2\sigma^2}\norm{x}_2^2.
    \end{aligned}
\end{equation*}
In other words, $x_{\mathrm{MAP}}$ coincides -- for this choice of prior distribution and likelihood!%
\footnote{Conversely, deliberately choosing prior distribution and likelihood for the sole purpose of computing the MAP estimator via Tikhonov regularization amounts to committing a \term[crime, Bayesian]{Bayesian crime}.}
-- by \cref{thm:tikhonov:funktional} exactly with Tikhonov regularization with $\alpha = \frac{\delta^2}{\sigma^2}$, which by \cref{lem:tikhonov:normalen} can be written as
\begin{equation}\label{eq:bayes:map2}
    x_{\mathrm{MAP}} = \left(T^*T + \frac{\delta^2}{\sigma^2} \Id\right)^{-1} T^* y^\delta.
\end{equation}

An alternative is the \term[mean, conditional]{conditional mean}
\begin{equation*}
    x_{\mathrm{CM}} := \Exp[X| Y = y^\delta] = \int_{\R^N} x\,d\mu_{X|y^\delta} = \int_{\R^N} x \rho_{X|y^\delta}(x) \,dx.
\end{equation*}
To calculate this point estimator in our concrete setting, we rewrite the posterior density \eqref{eq:bayes:posterior} slightly by expanding the square, using \eqref{eq:bayes:map2}, and completing the square again. This yields
\begin{equation*}
    \begin{aligned}
        \frac{1}{2\delta^2}\norm{Tx-y^\delta}_2^2 + \frac{1}{2\sigma^2}\norm{x}_2^2
        &=
        \frac1{2\delta^2}\inner{(T^*T+\tfrac{\delta^2}{\sigma^2}\Id)x}{x}_2 - \frac{1}{\delta^2}\inner{T^*y^\delta}{x}_2 + \frac{1}{2\delta^2}\norm{y^\delta}_2^2\\
        &= \frac1{2\delta^2}\inner{(T^*T+\tfrac{\delta^2}{\sigma^2}\Id)(x-x_{\mathrm{MAP}})}{x}_2 \\
        \MoveEqLeft[-1] - \frac{1}{2\delta^2}\inner{(T^*T+\tfrac{\delta^2}{\sigma^2}\Id)x}{x_{\mathrm{MAP}}}_2 + \frac{1}{2\delta^2}\norm{y^\delta}_2^2\\
        &= \frac1{2\delta^2}\inner{(T^*T+\tfrac{\delta^2}{\sigma^2}\Id)(x-x_{\mathrm{MAP}})}{x-x_{\mathrm{MAP}}}_2 \\
        \MoveEqLeft[-1] - \frac{1}{2\delta^2}\inner{(T^*T+\tfrac{\delta^2}{\sigma^2}\Id)x_{\mathrm{MAP}}}{x_{\mathrm{MAP}}}_2 + \frac{1}{2\delta^2}\norm{y^\delta}_2^2.
    \end{aligned}
\end{equation*}
Now the last two terms are constant in $x$. Furthermore, $T^*T+\frac{\delta^2}{\sigma^2}\Id$ is selfadjoint and positive definite, and thus has an inverse that is selfadjoint and positive definite as well. Up to a constant,  \eqref{eq:bayes:posterior} therefore is exactly of the form \eqref{eq:bayes:gaussian}, i.e., the posterior distribution is also normally distributed with
\begin{equation}\label{eq:bayes:posterior_gauss}
    \mu_{X|y^\delta} \sim \calN\left(x_{\mathrm{MAP}},\delta^2(T^*T+\tfrac{\delta^2}{\sigma^2}\Id)^{-1}\right).
\end{equation}
In particular, $x_{\mathrm{CM}} = x_{\mathrm{MAP}}$.

However, if the prior distribution or the likelihood is not normally distributed, or if the inverse problem is nonlinear, then in general the conditional mean does not coincide with the MAP estimator. In this case, one usually does not have a closed expression for $x_{\mathrm{CM}}$ and has to resort to a numerical approximation.
One possibility is the following. \Cref{thm:bayes} implies that
\begin{equation*}
    x_{\mathrm{CM}} = \int_{\R^N} x\rho_{X|y^\delta}(x) \,dx
    = \frac{\int_{\R^N} x \rho_{Y|x}(y^\delta)\rho_X(x) \,dx}{\int_{\R^N} \rho_{Y|x}(y^\delta)\rho_X(x)\,dx}
    = \frac{\int_{\R^N} x \rho_{Y|x}(y^\delta) \,d\mu_X}{\int_{\R^N} \rho_{Y|x}(y^\delta)\,d\mu_X}
\end{equation*}
by definition of the density. However, these are very high-dimensional integrals so that standard quadrature is not feasible. Instead, one uses \term[integration, Monte Carlo]{Monte Carlo integration}: If we have samples $x_1,\dots,x_n$ that are independently and identically distributed according to $\mu_X$ (which are straightforward to generate at least if $\mu_X$ is a normal distribution), then
\begin{equation*}
    \hat x_{\mathrm{CM}} := \frac{\frac1n\sum_{i=1}^n x_i \rho_{Y|x_i}(y^\delta)}{\frac1n\sum_{i=1}^n \rho_{Y|x_i}(y^\delta)}
    =
    \frac{\sum_{i=1}^n x_i e^{-\frac1{2\delta^2}\norm{Tx_i-y^\delta}_2^2}}{\sum_{i=1}^n e^{-\frac1{2\delta^2}\norm{Tx_i-y^\delta}_2^2}}
\end{equation*}
defines an approximation to $x_{\mathrm{CM}}$ that by the law of large numbers converges to $x_{\mathrm{CM}}$ at the rate $\mathcal{O}(1/\sqrt{n})$. (If it is impossible or prohibitive to generate samples from the prior distribution, one can instead use \term[method!Markov chain Monte Carlo]{Metropolis--Hastings Markov chain Monte Carlo methods} that directly generate a sequence of samples that are distributed according to the \emph{posterior} distribution with high probability.)

Monte Carlo methods can also be used to compute \term[estimator!region]{region estimators}; such estimators quantify the remaining uncertainty after the measurement. (Accordingly, one speaks of \term[quantification, uncertainty]{uncertainty quantification}.) One class of examples of such estimators are \term[estimator!region!credible set]{credible sets}; these are sets $C_\alpha\subset\R^N$ for given $\alpha\in (0,1)$ that satisfy
\begin{equation}\label{eq:bayes:credibleset}
    \mu_{X|y^\delta}(C_\alpha) = \int_{C_\alpha}\rho_{X|y^\delta}(x)\,dx = 1-\alpha,
\end{equation}
i.e., $C_\alpha$ contains $1-\alpha$ (e.g., $0.95$) of the posterior distribution's mass. 
The larger such a set, the less certain we are about the true solution (which however need not lie in $C_\alpha$!). However, these sets are not unique for a given $\alpha$; one possibility is to consider as credible sets only \term[estimator!region!highest posterior density set]{highest posterior density sets} of the form
\begin{equation*}
    C_\alpha^* =  \setof{x\in \R^N}{ -\log \rho_{X|y^\delta}(x) \leq \eta_\alpha},
\end{equation*}
where for given $\alpha$ one only has to find $\eta_\alpha\in (0,\infty)$ such that \eqref{eq:bayes:credibleset} holds.

\backmatter

\printbibliography[heading=bibintoc]

@book{Alt,
  author = {{Alt}, Hans Wilhelm},
  location = {London},
  publisher = {Springer},
  date = {2016},
  doi = {10.1007/978-1-4471-7280-2},
  series = {Universitext},
  subtitle = {An Application-Oriented Introduction},
  title = {Linear Functional Analysis},
}

@book{Brezis:2010a,
  author = {Brezis, Haïm},
  location = {New York},
  publisher = {Springer},
  date = {2010},
  doi = {10.1007/978-0-387-70914-7},
  title = {{Functional Analysis, Sobolev Spaces and Partial Differential Equations}},
}

@article{Bakushinskii,
  author = {Bakushinskiĭ, Anatoliĭ Borisovich},
  date = {1985},
  doi = {10.1016/0041-5553(84)90253-2},
  journaltitle = {USSR Comput. Math. Math. Phys.},
  number = {4},
  pages = {181--182},
  title = {Remarks on choosing a regularization parameter using the quasioptimality and ratio criterion},
  volume = {24},
}

@article{Bakushinskii:1992,
  author = {Bakushinskiĭ, Anatoliĭ Borisovich},
  date = {1992},
  journaltitle = {USSR Comput. Math. Math. Phys.},
  number = {9},
  pages = {1503--1509},
  title = {On a convergence problem of the iterative-regularized {G}auss-{N}ewton method},
  volume = {32},
}

@article{Bauer,
  author = {Bauer, Frank and Lukas, Mark A.},
  date = {2011},
  doi = {10.1016/j.matcom.2011.01.016},
  journaltitle = {Mathematics and Computers in Simulation},
  number = {9},
  pages = {1795--1841},
  title = {Comparing parameter choice methods for regularization of ill-posed problems},
  volume = {81},
}

@unpublished{Burger,
  author = {Burger, Martin},
  url = {http://wwwmath.uni-muenster.de/num/Vorlesungen/IP_WS07/skript.pdf},
  date = {2007},
  howpublished = {Lecture notes, Institut für Numerische und Angewandte Mathematik, Universität Münster},
  title = {Inverse Problems},
}

@article{Du:2008,
  author = {Du, Nailin},
  date = {2008},
  doi = {10.1137/060661120},
  journaltitle = {SIAM J. Numer. Anal.},
  number = {3},
  pages = {1454--1482},
  title = {Finite-dimensional approximation settings for infinite-dimensional {M}oore--{P}enrose inverses},
  volume = {46},
}

@article{EKN,
  author = {Engl, Heinz W. and Kunisch, Karl and Neubauer, Andreas},
  date = {1989},
  doi = {10.1088/0266-5611/5/4/007},
  journaltitle = {Inverse Problems},
  number = {4},
  pages = {523--540},
  title = {Convergence rates for {T}ikhonov regularisation of nonlinear ill-posed problems},
  volume = {5},
}

@book{Engl,
  author = {Engl, Heinz W. and Hanke, Martin and Neubauer, Andreas},
  publisher = {Kluwer Academic Publishers Group, Dordrecht},
  date = {1996},
  doi = {10.1007/978-94-009-1740-8},
  series = {Mathematics and its Applications},
  title = {Regularization of Inverse Problems},
  volume = {375},
}

@article{HankeNeubauerScherzer,
  author = {Hanke, Martin and Neubauer, Andreas and Scherzer, Otmar},
  date = {1995},
  doi = {10.1007/s002110050158},
  journaltitle = {Numerische Mathematik},
  number = {1},
  pages = {21--37},
  title = {A convergence analysis of the Landweber iteration for nonlinear ill-posed problems},
  volume = {72},
}

@unpublished{Harrach,
  options = {useprefix=false},
  author = {von Harrach, Bastian},
  date = {2014},
  howpublished = {Lecture notes, Fachbereich Mathematik, Universität Stuttgart},
  title = {Regularisierung Inverser Probleme},
}

@book{HewittStromberg,
  author = {Hewitt, Edwin and Stromberg, Karl},
  location = {New York and Heidelberg},
  publisher = {Springer-Verlag},
  date = {1975},
  doi = {10.1007/978-3-662-29794-0},
  title = {Real and Abstract Analysis},
}

@book{HKKS,
  author = {Schuster, Thomas and Kaltenbacher, Barbara and Hofmann, Bernd and Kazimierski, Kamil S.},
  location = {Berlin},
  publisher = {De Gruyter},
  date = {2012},
  doi = {10.1515/9783110255720},
  series = {Radon Series on Computational and Applied Mathematics},
  title = {Regularization Methods in {B}anach Spaces},
  volume = {10},
}

@book{Hofmann,
  author = {Hofmann, Bernd},
  publisher = {B. G. Teubner Verlagsgesellschaft mbH, Stuttgart},
  date = {1999},
  series = {Mathematik für Ingenieure und Naturwissenschaftler},
  title = {Mathematik inverser {P}robleme},
}

@unpublished{Hohage,
  author = {Hohage, Thorsten},
  date = {2002},
  howpublished = {Lecture notes, Institut für Numerische und Angewandte Mathematik, Universität Göttingen},
  title = {Inverse Problems},
}

@article{Hohage:2000,
  author = {Hohage, Thorsten},
  date = {2000},
  doi = {10.1080/01630560008816965},
  journaltitle = {Numer. Funct. Anal. Optim.},
  number = {3-4},
  pages = {439--464},
  title = {Regularization of exponentially ill-posed problems},
  volume = {21},
}

@book{ItoJin,
  author = {Ito, Kazufumi and Jin, Bangti},
  location = {Singapore},
  publisher = {World Scientific},
  date = {2014},
  doi = {10.1142/9789814596206_0001},
  series = {Series on Applied Mathematics},
  title = {Inverse Problems: Tikhonov Theory and Algorithms},
  volume = {22},
}

@book{Kaltenbacher,
  author = {Kaltenbacher, Barbara and Neubauer, Andreas and Scherzer, Otmar},
  location = {Berlin},
  publisher = {De Gruyter},
  date = {2008},
  doi = {10.1515/9783110208276},
  series = {Radon Series on Computational and Applied Mathematics},
  title = {Iterative Regularization Methods for Nonlinear Ill-posed Problems},
  volume = {6},
}

@article{Kindermann,
  author = {Kindermann, Stefan},
  url = {http://etna.mcs.kent.edu/volumes/2011-2020/vol38/abstract.php?vol=38&pages=233-257},
  date = {2011},
  journaltitle = {Electron. Trans. Numer. Anal.},
  pages = {233--257},
  title = {Convergence analysis of minimization-based noise level-free parameter choice rules for linear ill-posed problems},
  volume = {38},
}

@article{Kindermann:2016,
  author = {Kindermann, Stefan},
  date = {2016},
  doi = {10.1515/cmam-2015-0036},
  journaltitle = {Comput. Methods Appl. Math.},
  number = {2},
  pages = {257--276},
  title = {Projection methods for ill-posed problems revisited},
  volume = {16},
}

@book{Kirsch,
  author = {Kirsch, Andreas},
  publisher = {Springer, New York},
  date = {2011},
  doi = {10.1007/978-1-4419-8474-6},
  edition = {2},
  title = {An Introduction to the Mathematical Theory of Inverse Problems},
}

@article{Landweber,
  author = {Landweber, L.},
  date = {1951},
  doi = {10.2307/2372313},
  journaltitle = {Amer. J. Math.},
  pages = {615--624},
  title = {An iteration formula for {F}redholm integral equations of the first kind},
  volume = {73},
}

@book{Louis,
  author = {Louis, Alfred Karl},
  publisher = {B. G. Teubner, Stuttgart},
  date = {1989},
  doi = {10.1007/978-3-322-84808-6},
  series = {Teubner Studienbücher Mathematik},
  title = {Inverse und schlecht gestellte {P}robleme},
}

@article{Poeschl:2007a,
  author = {Hofmann, B. and Kaltenbacher, B. and Pöschl, C. and Scherzer, O.},
  date = {2007},
  doi = {10.1088/0266-5611/23/3/009},
  journaltitle = {Inverse Problems},
  number = {3},
  pages = {987--1010},
  title = {A convergence rates result for {T}ikhonov regularization in {B}anach spaces with non-smooth operators},
  volume = {23},
}

@book{Rieder,
  author = {Rieder, Andreas},
  publisher = {Friedr. Vieweg \& Sohn, Braunschweig},
  date = {2003},
  doi = {10.1007/978-3-322-80234-7},
  note = {Eine Einführung in ihre stabile Lösung},
  title = {Keine {P}robleme mit inversen {P}roblemen},
}

@book{Scherzer:2009,
  author = {Scherzer, Otmar and Grasmair, Markus and Grossauer, Harald and Haltmeier, Markus and Lenzen, Frank},
  location = {New York},
  publisher = {Springer},
  date = {2009},
  doi = {10.1007/978-0-387-69277-7},
  isbn = {978-0-387-30931-6},
  series = {Applied Mathematical Sciences},
  title = {{Variational Methods in Imaging}},
  volume = {167},
}

@article{Scherzer:2014,
  author = {Andreev, R. and Elbau, P. and de Hoop, M. V. and Qiu, L. and Scherzer, O.},
  date = {2015},
  doi = {10.1080/01630563.2015.1021422},
  journaltitle = {Numer. Funct. Anal. Optim.},
  number = {5},
  pages = {549--566},
  title = {Generalized convergence rates results for linear inverse problems in {H}ilbert spaces},
  volume = {36},
}

@article{Tikhonov1,
  author = {Tikhonov, A. N.},
  date = {1963},
  journaltitle = {Dokl. Akad. Nauk SSSR},
  pages = {501--504},
  title = {On the solution of ill-posed problems and the method of regularization},
  volume = {151},
}

@article{Tikhonov2,
  author = {Tikhonov, A. N.},
  date = {1963},
  journaltitle = {Dokl. Akad. Nauk SSSR},
  pages = {49--52},
  title = {On the regularization of ill-posed problems},
  volume = {153},
}

@article{MatheHofmann:2008,
  author = {Mathé, Peter and Hofmann, Bernd},
  date = {2008},
  doi = {10.1088/0266-5611/24/1/015009},
  journaltitle = {Inverse Problems},
  number = {1},
  pages = {015009,5},
  title = {How general are general source conditions?},
  volume = {24},
}

@book{Clason,
  author = {Clason, Christian},
  location = {Basel},
  publisher = {Birkhäuser},
  date = {2026},
  doi = {10.1007/978-3-032-13719-7},
  edition = {2},
  series = {Compact Textbooks in Mathematics},
  title = {Introduction to Functional Analysis},
}

@book{ClasonValkonen,
  author = {Clason, Christian and Valkonen, Tuomo},
  location = {Philadelphia},
  publisher = {Society for Industrial and Applied Mathematics},
  date = {2026},
  doi = {10.1137/1.9781611978995},
  eprint = {2001.00216},
  eprinttype = {arxiv},
  series = {MOS-SIAM Series on Optimization},
  title = {Introduction to Nonsmooth Analysis and Optimization},
  volume = {38},
}

@book{renardyrogers2004,
  author = {Renardy, Michael and Rogers, Robert C.},
  publisher = {Springer-Verlag, New York},
  date = {2004},
  doi = {10.1007/b97427},
  edition = {2},
  series = {Texts in Applied Mathematics},
  title = {An Introduction to Partial Differential Equations},
  volume = {13},
}

@incollection{Cavalier,
  author = {Cavalier, Laurent},
  publisher = {Springer, Heidelberg},
  booktitle = {Inverse Problems and High-Dimensional Estimation},
  date = {2011},
  doi = {10.1007/978-3-642-19989-9_1},
  pages = {3--96},
  series = {Lect. Notes Stat. Proc.},
  title = {Inverse problems in statistics},
  volume = {203},
}

@article{BissantzHohageMunkRuymgaart:2007,
  author = {Bissantz, N. and Hohage, T. and Munk, A. and Ruymgaart, F.},
  date = {2007},
  doi = {10.1137/060651884},
  journaltitle = {SIAM J. Numer. Anal.},
  number = {6},
  pages = {2610--2636},
  title = {Convergence rates of general regularization methods for statistical inverse problems and applications},
  volume = {45},
}

@article{BelitserLevit,
  author = {Belitser, Eduard N. and Levit, Boris Y.},
  date = {1995},
  journaltitle = {{Math. Methods Stat.}},
  number = {3},
  pages = {259--273},
  title = {On minimax filtering over ellipsoids},
  volume = {4},
}

@unpublished{Kekkonen,
  author = {Kekkonnen, Hanne},
  url = {http://www.damtp.cam.ac.uk/research/cia/files/teaching/Inverse_Problems_Lent_2018/19y_01m_16d_LectureNotes.pdf},
  date = {2019},
  howpublished = {Lecture notes, University of Cambridge},
  title = {Bayesian Inverse Problems},
}

@book{Klenke,
  author = {Klenke, Achim},
  publisher = {Springer International Publishing},
  date = {2020},
  doi = {10.1007/978-3-030-56402-5},
  edition = {3},
  subtitle = {A Comprehensive Course},
  title = {Probability Theory},
}

@book{Kallenberg,
  author = {Kallenberg, Olav},
  publisher = {Springer International Publishing},
  date = {2021},
  doi = {10.1007/978-3-030-61871-1},
  edition = {3},
  title = {Foundations of Modern Probability},
}

@book{Sanz-Alonso:2023,
  author = {Sanz-Alonso, Daniel and Stuart, Andrew and Taeb, Armeen},
  location = {Cambridge},
  publisher = {Cambridge University Press},
  date = {2023},
  doi = {10.1017/9781009414319},
  eprint = {1810.0619},
  eprinttype = {arxiv},
  series = {Lond. Math. Soc. Stud. Texts},
  title = {Inverse Problems and Data Assimilation},
  volume = {107},
}

@online{FormalizingIP,
  author = {Clason, Christian and Helin, Tapio and Nursultanov, Medet and Oksanen, Lauri},
  date = {2026},
  eprint = {2611.06962},
  eprinttype = {arxiv},
  title = {Formalizing inverse problems in the age of {AI}: possibilities, pitfalls, and an invitation},
}

@article{Hofmann:2018,
  author = {Egger, H. and Hofmann, B.},
  date = {2018},
  doi = {10.1088/1361-6420/aadef4},
  journaltitle = {Inverse Problems},
  number = {11},
  pages = {115015,17},
  title = {Tikhonov regularization in {H}ilbert scales under conditional stability assumptions},
  volume = {34},
}

@unpublished{Clason:2019,
  author = {Clason, Christian},
  url = {https://imsc.uni-graz.at/clason/skripte/MathBildSkript18.pdf},
  date = {2019},
  howpublished = {Lecture notes, Faculty of Mathematics, University Duisburg-Essen},
  title = {Mathematische Bildverarbeitung},
}

\printindex

\end{document}